\documentclass[reqno]{amsart}
\usepackage{amssymb,amsmath,amsthm}
\usepackage{mathrsfs}
\usepackage{amscd}
\usepackage{amsfonts}
\usepackage{amssymb}
\usepackage{latexsym}
\usepackage{color}
\usepackage{esint}
\usepackage{graphicx}
\usepackage{float}
\graphicspath{{Figures/}}
\usepackage{todonotes}
\usepackage[colorlinks,linkcolor=blue,anchorcolor=red,citecolor=blue]{hyperref}
\usepackage{cleveref}
\usepackage[margin=1in]{geometry} 
\usepackage{marginnote}

\usepackage{color}

\newcommand{\dis}{\displaystyle}
\allowdisplaybreaks
\usepackage[makeroom]{cancel}

\newtheorem{theorem}{Theorem}

\newtheorem{definition}{Definition}
\newtheorem{lemma}{Lemma}
\newtheorem{proposition}[theorem]{Proposition}
\newtheorem{remark}{Remark}

\let\e=\varepsilon
\let\d=\delta
\let\h=v

\let\p=\partial

\let\O=\Omega

\numberwithin{equation}{section}

\let\hide\iffalse
\let\unhide\fi

\newcommand{\R}{\mathbb{R}}
\renewcommand{\P}{\mathbf{P}}

\newcommand{\be}{\begin{equation}}
\newcommand{\bm}{\begin{multline}}
\newcommand{\ee}{\end{equation}}
\newcommand{\dd}{\mathrm{d}}

\newcommand{\xb}{x_{\mathbf{b}}}
\newcommand{\tb}{t_{\mathbf{b}}}

\newcommand{\Bes}{\begin{eqnarray*}}
\newcommand{\Ees}{\end{eqnarray*}}
\newcommand{\Be}{\begin{equation} }
\newcommand{\Ee}{\end{equation}}

\def\p{\partial}

\def\O{\Omega}
\def\R{\mathbb{R}}
\def\d{\mathrm{d}}
\def\B{\begin{equation}}
\def\E{\end{equation}}
\def\BN{\begin{eqnarray*}}
\def\EN{\end{eqnarray*}}

\usepackage{cite}

\makeatletter
\@namedef{subjclassname@2020}{%
  \textup{2020} Mathematics Subject Classification}
\makeatother
\begin{document}
\title[Isothermal rarefied gas flows in an infinite layer]{Global dynamics of isothermal rarefied gas flows in an infinite layer}

\author[H.-X. Chen]{Hongxu Chen}
\address[HXC]{Department of Mathematics, The Chinese University of Hong Kong, Shatin, Hong Kong}
\email{hxchen@math.cuhk.edu.hk}

\author[R.-J. Duan]{Renjun Duan}
\address[RJD]{Department of Mathematics, The Chinese University of Hong Kong, Shatin, Hong Kong.}
\email{rjduan@math.cuhk.edu.hk}

\author[J.-H. Zhang]{Junhao Zhang}
\address[JHZ]{Department of Mathematics, The Chinese University of Hong Kong, Shatin, Hong Kong}
\email{jhzhang@math.cuhk.edu.hk}

\begin{abstract}
Let rarefied gas be confined in an infinite layer with diffusely reflecting boundaries that are isothermal and non-moving. The initial-boundary value problem on the nonlinear Boltzmann equation governing the rarefied gas flow in such setting is challenging due to unboundedness of both domain and its boundaries as well as the presence of physical boundary conditions. In the paper, we establish the global-in-time dynamics of such rarefied gas flows near global Maxwellians in three or two-dimensions. For the former case, we also prove that the solutions decay in time at a polynomial rate which is the same as that of solutions to the two-dimensional heat equation. This is the first result on global solutions of the Boltzmann equation with non-compact and diffuse boundaries.

%\Red{To be added}Difficulties:
%\begin{itemize}
%  \item domain is unbounded, Poincar\'e inequality fails, no exponential time-decay, macro component over the low-frequency regime behaves only diffusive, in particular, motivated by Kagei's work \cite{kagei2008large,kagei2007asymptotic,kagei2007resolvent} it should be two-dimensional diffusive wave, very low time-decay in large time 
%  \item domain involves a physical boundary, diffusive reflection boundary, boundary singularity makes it possible to consider only low-regularity solutions
  %\item 
%\end{itemize}
%Key points:
%\begin{itemize}
%  \item take the Fourier transform in unbounded tangent variables, so as to consider the 1D problem in bounded domain with two Fourier variables as extra parameters
%  \item develop $L^\infty$ estimates with continuous Fourier variables as parameter 
%  \item develop the dual argument to obtain the macro dissipation, in particular, in the regime where the frequnecy variable is near zero where the Poincar\'e inequality fails
%  \item $L^1_k\cap L^p_k$ approach of treating nonlinear estimates for the unbounded domain case, applied to the current IBVP with physical boundary
%  \item this opens a new direction for the problem on the Boltzmann equation in unbounded domain with unbounded boundaries
%\end{itemize}
\end{abstract}

\date{\today}
	%35Q20  	Boltzmann equations
	%35B35  	Stability in context of PDEs
\maketitle

\thispagestyle{empty}
\tableofcontents

\section{Introduction}
There exist various fundamental physical problems of rarefied gas flows, such as Couette-flow and heat-transfer problems between two parallel plates, and Poiseuille flow and thermal transpiration through a channel or pipe by a pressure and temperature gradient along it, respectively, cf.~\cite{KoganBook,SoneBook}. The mathematical study of the nonlinear Boltzmann equation for such problems is significantly important and challenging in kinetic theory, cf.~\cite{ELM-94, ELM-95} and \cite{duan20243d,DLYZ-heat,DZ-nsf} as well as references therein. In this paper, we study the initial-boundary value problem on the Boltzmann equation in a three-dimensional infinite layer $\mathbb{R}^2 \times (-1,1)$ with diffuse reflection boundary conditions at the planes $x_3 = \pm 1$ that are also isothermal and non-moving. We construct global-in-time solutions that are close to global Maxwellians and analyze the large-time behavior of these solutions. Our result demonstrates that the solutions decay in time at a polynomial rate which is the same as that of solutions to the two-dimensional heat equation. Such long time behavior of solutions is consistent with the result by Kagei in his series work \cite{kagei2007asymptotic,kagei2007resolvent,kagei2008large} for the study of isentropic compressible Navier-Stokes equations in the infinite layer. We are devoted to developing an analogous theory for the Boltzmann equation. For the proof, we apply the Fourier transform in the horizontal variable $\bar{x}\in \R^2$ and utilize the $L^1_k \cap L^p_k$ approach in the Fourier space with $2 < p \leq \infty$, ensuring that the additional $L^p_k$ norm provides sufficient time decay. Meanwhile, we employ the interplay technique in $L^2_{x_3,v} \cap L^\infty_{x_3,v}$ to control the nonlinear terms. In particular, a key ingredient of our proof is to develop a dual argument for this infinite layer problem with physical boundaries to address the macroscopic dissipation estimates with respect to the mixed Fourier and physical variables $(k,x_3)$. We also study the problem in a two-dimensional infinite layer $\mathbb{R} \times (-1,1)$. The same method is not applicable due to the slower time decay property of solutions along the one-dimensional horizontal direction. Instead, we utilize a time-derivative combined with a direct $L^2 \cap L^\infty$ approach in physical space to establish the global existence of solutions, but the large-time behavior is left unknown. The current work also provides possible insights to further understand the problem on kinetic shear flow in such an infinite channel domain when the boundaries are moving along the tangent planes relative to each other.

\subsection{The problem}
We consider the initial-boundary value problem on the Boltzmann equation for rarefied gas contained in an infinite layer $\Omega=\mathbb{R}^2\times (-1,1)$ of three dimensions: 
\begin{equation}\label{proF}
\p_t F + v\cdot\nabla_x F=Q(F,F), \ (t,x,v)\in [0,\infty) \times \O \times \mathbb{R}^3. 
\end{equation}
Here, $F=F(t,x,v)\geq 0$ stands for the velocity distribution function of gas particles with velocity $v=(v_1,v_2,v_3)\in \R^3$ at time $t\geq 0$ and position $x=(x_1,x_2,x_3)\in \Omega\subset \R^3$, and the initial and boundary conditions are to be specified later. The Boltzmann collision term is a bilinear integral operator acting only on velocity variable and for the hard sphere model it reads as 
\begin{equation*}
Q(F,G)=\int_{\mathbb{R}^3}\int_{\mathbb{S}^2}|(v-u)\cdot \omega|[F(u')G(v')-F(u)G(v)]\,\d\omega \dd u,
\end{equation*}
where the velocity pairs $(v,u)$ and $(v',u')$ satisfy 
\begin{equation*}
  v'=v+[(u-v)\cdot\omega]\omega,\quad u'=u-[(u-v)\cdot\omega]\omega,\quad \omega\in\mathbb{S}^2,
\end{equation*}
that's the $\omega$-representation in terms of the conservation of momentum and energy for elastic collisions between molecules:
\begin{equation*}
  v+u=v'+u',\quad |v|^2+|u|^2=|v'|^2+|u'|^2.
\end{equation*}

Different from the case of the pure whole space, the gas particles also interact with the physical boundary at the infinite planes $\partial\Omega=\mathbb{R}^2 \times \{x_3=\pm 1\}$. To describe the boundary condition, we split the boundary phase space $\partial\Omega\times \R^3_v$ as
\begin{align*}
\gamma_{+}^\pm=&\{(x,v)\in\mathbb{R}^2 \times \{x_3=\pm 1\}\times\mathbb{R}^{3}: v_3 \gtrless 0 \} , \\
\gamma_{-}^\pm=&\{(x,v)\in\mathbb{R}^2 \times \{x_3=\pm 1\}\times\mathbb{R}^{3}:v_3 \lessgtr 0\}, \\
\gamma_{0}^\pm=&\{(x,v)\in\mathbb{R}^2 \times \{x_3=\pm 1\}\times\mathbb{R}^{3}:v_3=0\}.
\end{align*}
We are interested in the infinite layer problem with the isothermal diffuse reflection boundary condition:
\begin{equation}\label{F.bc}
F(t,x,v)|_{\gamma_-^\pm}=c_\mu \mu(v)\int_{u_3\gtrless 0}F(t,x,u)|u_3| \dd u,
\end{equation}
where
\[
\mu:=\frac{1}{(2\pi)^{3/2}}e^{-\frac{|v|^2}{2}}
\]
is a normalized global Maxwellian with zero bulk velocity and the constant $c_\mu=\sqrt{2\pi}$ is chosen to satisfy $\int_{v_3\gtrless 0}c_\mu\mu(v) |v_3|dv=1$ so that $c_\mu\mu(v) |v_3|$ is a probability measure on the half velocity spaces $\{\R^3: v_3\gtrless 0\}$. Note that the mass flux is vanishing at the boundaries, namely
\begin{equation*}
\int_{\R^3} v_3F(t,x_1,x_2,x_3=\pm 1, v)\d v=0.
\end{equation*}

In the standard perturbation framework, we seek for the solution of the form $F = \mu + \sqrt{\mu}f$. Then plugging this to \eqref{proF} and \eqref{F.bc}, the IBVP on $f$ is reformulated as
\begin{align}
\begin{cases}
     &\dis \p_t f + v\cdot \nabla_x f + \mathcal{L}f = \Gamma(f,f), \\
    &\dis  f(t,x,v)|_{\gamma_-^\pm } = c_\mu \sqrt{\mu(v)} \int_{u_3\gtrless 0} f(t,x,u)\sqrt{\mu(u)}|u_3|\dd u, \\   
    &\dis  f(0,x,v) = f_0(x,v) := (F(0,x,v)-\mu)/\sqrt{\mu}.  
\end{cases}\label{linear_f}
\end{align}
Here $\mathcal{L}f$ and $\Gamma(f,f)$ denote the linearized collision term and nonlinear term respectively:
\begin{align}
    &  \mathcal{L}f:=  -\mu^{-1/2}[Q(\mu,\sqrt{\mu}f) + Q(\sqrt{\mu}f,\mu)], \label{def.L}\\
    &   \Gamma(f,f) := \mu^{-1/2} Q(\sqrt{\mu}f,\sqrt{\mu}f).\label{def.Ga}
\end{align}

We aim at constructing the global in time solutions $f(t,x,v)$ to \eqref{linear_f} for suitably small initial data $f_0(x,v)$ and also obtaining the long time behavior of these solutions, in particular, the explicit time-decay rate. In what follows we provide a brief review of the literature with emphasis on most relevant works in the perturbation framework before stating our results. Since we are addressing the boundary value problem in the unbounded domain, we discuss the issue in both scenarios: with and without the presence of boundaries.\\

%\Purple {Literature review:} 
\begin{itemize}
    \item {\bf Whole space and torus:} Both cases are well understood; for instance, we may refer to a recent closely related work \cite{duan2021global} for a complete review. In \cite{duan2021global}, a class of low-regularity global in time solutions with exponential decay based on the Wiener algebra was constructed for the non-cutoff Boltzmann equation in the torus $\mathbb{T}^3$. Note that the idea of introducing the Wiener algebra or $L^1_k$ was motivated by Lei-Lin's work \cite{lei2011global} for the construction of global mild solutions to the incompressible Navier-Stokes equations.
    
    In the whole space $\mathbb{R}^3$, only polynomial decay rates are expected; see \cite{bouin2020hypocoercivity} for a general hypocoercivity approach. Guo \cite{guo2004boltzmann,guo2010bounded} constructed a global solution without time decay via a nonlinear energy method and entropy method,  respectively. Ukai and Yang \cite{ukai2006boltzmann} obtain the optimal decay rate through spectral analysis and semi-group method. In contrast with \cite{duan2021global},  an $L^1_k \cap L^p_k$ method in the Fourier frequency space was proposed in \cite{duan_SIMA} for obtaining the almost optimal decay rate for the non-cutoff Boltzmann equation without relying on the embedding $H^2(\mathbb{R}^3)\subset L^\infty(\mathbb{R}^3)$. Guo and Wang \cite{guo2012decay} obtained the optimal decay rate in high order Sobolev spaces for initial data in a negative Sobolev space. 
    
    When the domain exhibits both bounded and unbounded properties, one may consider a domain as an infinite channel $\mathbb{R}\times \mathbb{T}^2$. Wang and Wang \cite{Teng2019} investigated the Boltzmann equation, while \cite{duan20213d} studied the Vlasov-Poisson-Landau system, both utilizing high-order Sobolev energy methods in such domains.  However, it is important to note that most of these methods cannot be directly applied in the presence of physical boundaries, for instance, diffuse reflection boundary under consideration. \\

    \item {\bf Bounded domain:} The boundary effect plays an important role in kinetic theory, and there have been numerous contributions to the mathematical study of boundary value problems, including \cite{guiraud1975h, desvillettes2005trend, liu2007initial, yang2005half, desvillettes1990convergence, mischler2000initial, cercignani1992initial, hamdache1992initial} as well as recent progress \cite{ouyang2024conditional,bernou2022hypocoercivity}. In a general bounded domain, high-regularity solutions may not be expected in general due to the singularity near boundary, as noted in \cite{kim2011formation,GKTT,GKTT2,CK,chen2024gradient}. The geometric complexities make it challenging to apply Fourier transform techniques effectively. 
    
    In 2010, Guo \cite{G} proposed an $L^2-L^\infty$ framework to establish a global solution with exponential convergence rate under boundary conditions including the diffusive reflection and specular reflection. This breakthrough has led to substantial advancements in the study of boundary value problems within kinetic theory \cite{KL, CKL, duan2019effects, EGKM, EGKM2}. In the argument of \cite{G}, the dissipation estimate for the macroscopic component is obtained by crucially using the Poincar\'e inequality and the $L^\infty$ estimate is obtained by method of characteristics with repeated boundary interactions. 
    
    To adapt the $L^2-L^\infty$ argument in the infinite layer domain, one may first consider the scenario where the tangent variable $(x_1,x_2)$ is bounded with specular boundary condition, or where $(x_1,x_2)\in \mathbb{T}^2$, and then extend it to case of $(x_1,x_2)\in \mathbb{R}^2$. We mention \cite{chen_mixed} and \cite{duan20243d}. Here, the first work examines a scenario in which particles are specularly reflected between two parallel plates, while diffusive reflection occurs in the remaining area between these two specular regions. The second work investigates the Couette flow in the region $\mathbb{T}^2 \times (-1,1)$ with diffusive boundary condition. 
    
    Note that the approaches in those studies for the finite layer cannot be adapted to the situation where the domain becomes unbounded along the horizontal directions for which the exponential decay rate would be lost and only polynomial decay rate can be expected similar to the whole space case. \\

    \item {\bf Exterior problem:} Since we are concerning the boundary value problem in an unbounded domain, the exterior problem is also closely related. Ukai and Asano \cite{ukai1983steady,ukai1986steady} investigated the exterior problem in both steady and unsteady cases using the delicate spectral analysis together with the observation that the exterior problem can be viewed as a compact perturbation of the whole space problem. An alternative approach to studying the exterior problem involves obtaining an $L^6$ control of macroscopic quantities via the weak formulation, that was initially proposed by Esposito-Guo-Kim-Marra \cite{EGKM,EGKM2} for the problem in bounded domains. By crucially using the Sobolev embedding $W^{2,\frac{6}{5}}\subset H^1 \subset L^6$ and compactness of the boundary, one can obtain proper control over the trace in macroscopic estimate. We refer such an argument to \cite{esposito2018hydrodynamic} for the steady exterior problem with the stability of steady solutions in the time-evolutionary case left unknown; see also recent progress \cite{jung2023diffusive,jang2021incompressible,cao2023passage,guo2024diffusive} for the dynamical problems. \\

    \item {\bf Infinite layer problem for compressible viscous fluid:} It is well known that the compressible Navier-Stokes equations can be deduced via the Chapman-Enskog expansion from the Boltzmann equation. One can expect that solutions of both equations in an infinite layer may share some similar qualitative properties, in particular, the large time behavior of solutions. Indeed, Kagei \cite{kagei2007asymptotic,kagei2007resolvent,kagei2008large} studied the isentropic compressible Navier-Stokes equations in the infinite layer $\Omega=\mathbb{R}^2\times (-1,1)$ with no-slip boundary condition. In this series of work, through taking Fourier transform in the horizontal direction and the spectral analysis to the linearized Navier-Stokes operator, it was proved that the leading part of the solution satisfies a two-dimensional heat equation. \\

\end{itemize}

Although significant progress has been made on the boundary value problem of the Boltzmann equation mentioned above, the long-time asymptotic stability remains largely open, particularly when the boundary is not compact; even the question of global existence is still unresolved. In the current work, we address this issue by investigating the initial-boundary value problem in an infinite layer with diffuse boundary conditions. We develop a new approach of overcoming difficulties from unboundedness of both domain and its boundary as well as appearance of physical boundary conditions.    

\subsection{Main results}

Motivated by the series work of Kagei \cite{kagei2007asymptotic,kagei2007resolvent,kagei2008large} and the $L^1_k \cap L_k^p$ approach for whole space \cite{duan_SIMA}, we take Fourier transform to \eqref{linear_f} in the horizontal direction. To the end we denote $k: = (k_1,k_2)\in \R^2$ to be a two-dimensional Fourier variable of the tangent physical variable $\bar{x}:=(x_1,x_2)\in \R^2$. Thus, the Fourier transform of $f(t,x,v)$ is defined as
\begin{align*}
&  \hat{f}(t,k,x_3,v) = \int_{\mathbb{R}^2} f(t,x,v) e^{-i k\cdot \bar{x}} \dd \bar{x}.    
\end{align*}
Denote $\bar{v}:=(v_1,v_2)\in \mathbb{R}^2$ as the tangent velocity variables. The problem for $\hat{f}=\hat{f}(t,k,x_3,v)$ can be formulated as
\begin{align}
\begin{cases}
    & \dis    \p_t \hat{f} + i\bar{v}\cdot k \hat{f} + v_3 \p_{x_3} \hat{f} + \mathcal{L}\hat{f} = \hat{\Gamma}(\hat{f},\hat{f}),\\
    & \dis  \hat{f}(t,k,\pm 1,v)|_{v_3 \lessgtr 0} = c_\mu \sqrt{\mu(v)} \int_{u_3 \gtrless 0 } \hat{f}(t,k,x_3,u) |u_3| \sqrt{\mu(u)} \dd u, \\
    & \dis  \hat{f}(0,k,x_3,v) = \hat{f_0}(k,x_3,v).      
\end{cases} \label{f_eqn}
\end{align}
Here, recalling \eqref{def.Ga}, $\hat{\Gamma}(\hat{f},\hat{g})$ is the Fourier transform of the nonlinear term $\Gamma(f,g)$ with respect to $\bar{x}=(x_1,x_2)$:
\begin{align}
    &  \hat{\Gamma}(\hat{f},\hat{g}) := \int_{\mathbb{R}^3}\int_{\mathbb{S}^2} |(v-u)\cdot \omega|\sqrt{\mu(u)}[\hat{f}(x_3,v')*_k \hat{g}(x_3,u')-f(x_3,v)*_k g(x_3,u)]  \dd \omega \dd u,\label{gamma_hat}
\end{align}
namely, it can be explicitly written as
\begin{align}
     \hat{\Gamma}(\hat{f},\hat{g})(t,k,x_3,v)= \int_{\mathbb{R}^3}\int_{\mathbb{S}^2}\int_{\R^2} |(v-u)\cdot \omega|\sqrt{\mu(u)}[&\hat{f}(t,k-\ell,x_3,v') \hat{g}(t,\ell,x_3,u')\notag\\
     &-f(t,k-\ell,x_3,v) g(t,\ell,x_3,u)]\d \ell  \dd \omega \dd u.
     \label{gamma_hat1}
\end{align}

Then the original problem \eqref{linear_f} is reformulated to be the one-dimensional IBVP problem \eqref{f_eqn} in the bounded domain $(-1,1)$ involving the two-dimensional Fourier variable $k\in \R^2$ as an extra continuous parameter. Since the only physical variable $x_3$ is bounded, one can expect to employ the $L^2_{x_3,v}-L^\infty_{x_3,v}$ argument in \cite{G}. For this purpose, we define several notations. Denote the macroscopic component as $\mathbf{P}{\hat{f}}$, which represents the projection from $L^2_v$ to $\ker \mathcal{L}=\text{span} (\{\sqrt{\mu(v)},v\sqrt{\mu(v)},\frac{1}{2}(|v|^2-3)\sqrt{\mu(v)}\}) $:
\[\mathbf{P}\hat{f}:= \Big(\hat{a}+\hat{\mathbf{b}}\cdot v + \hat{c}\frac{|v|^2-3}{2} \Big)\sqrt{\mu(v)},\]
where $\hat{a}$, $\hat{\mathbf{b}}=(\hat{b}_1,\hat{b}_2,\hat{b}_3)$ and $\hat{c}$ are functions of $(t,k,x_3)$ for $\hat{f}=\hat{f}(t,k,x_3,v)$. Denote an exponential velocity weight as
\begin{align}\label{weight_w}
w(v) := e^{\theta |v|^2}, \ 0<\theta< \frac{1}{4}.
\end{align}
Denote the $P_\gamma {\hat{f}}$ as the projection to the diffuse reflection at the planes $x_3=\pm 1$:
\begin{align*}
    &   P_\gamma {\hat{f}} (t,k,\pm 1, v) := c_\mu \sqrt{\mu(v)} \int_{u_3\gtrless 1} \hat{f}(t,k,\pm 1,u) \sqrt{\mu(u)} |u_3| \dd u,\quad v_3 \lessgtr 0.
\end{align*}

Below we state the main results of this paper. We refer all norm notations to Section \ref{sec:notation} later on. 

\begin{theorem}\label{thm:l1k_lpk}
Let $2<p\leq \infty$ and  $\sigma = 2(1-\frac{1}{p})-2\e > 1$ with $\e>0$ small enough, then there exist constants $\delta>0$ and $C>0$ such that if the initial data $\hat{f}_0(k,x_3,v)$ with $F_0(x,v):=\mu+\sqrt{\mu}f_0(x,v)\geq 0$ satisfies
%\begin{equation}
% \int_{\O}\int_{\mathbb{R}^3} \sqrt{\mu(v)}f_0(x,v) \dd v \dd x = 0, \label{average_0}
%\end{equation}
%and 
\begin{equation}
\Vert w\hat{f}_0\Vert_{L^1_k L^\infty_{x_3,v}} + \Vert \hat{f}_0\Vert_{L^p_k L^2_{x_3,v}} < \delta, \label{initial_assumption}
\end{equation}
then there exists a unique solution $\hat{f}(t,k,x_3,v)$ to \eqref{f_eqn} such that $F(t,x,v)=\mu+\sqrt{\mu}f(t,x,v)\geq 0$ 
%with $\int_{\O}\int_{\mathbb{R}^3}\sqrt{\mu(v)}f(t,x,v) \dd v \dd x=0$  
and the following estimate is satisfied:
\begin{align}
&\Vert (1+t)^{\sigma/2} w \hat{f}\Vert_{L^1_k L^\infty_{T,x_3,v}} + \Vert \hat{f}\Vert_{L^p_k L^\infty_T L^2_{x_3,v}}  \leq C\Vert w\hat{f}_0\Vert_{L^1_k L^\infty_{x_3,v}} + C\Vert \hat{f}_0\Vert_{L^p_k L^2_{x_3,v}}, 
\label{f_estimate}
\end{align}
for any $T>0$. Moreover, it also holds that
\begin{align}
    & \Vert (1+t)^{\sigma/2} \hat{f}\Vert_{L^1_k L^\infty_T L^2_{x_3,v}} + \Vert (1+t)^{\sigma/2} (\mathbf{I}-\mathbf{P}) \hat{f} \Vert_{L^1_k L^2_{T,x_3,\nu}} + |(1+t)^{\sigma/2}(I-P_\gamma)\hat{f}|_{L^1_k L^2_{T,\gamma_+}} \notag \\
    &  + \Big\Vert (1+t)^{\sigma/2} \frac{|k|}{\sqrt{1+|k|^2}}(\hat{a},\hat{\mathbf{b}},\hat{c})\Big\Vert_{L^1_k L^2_{T,x_3}} \leq C\Vert w\hat{f}_0\Vert_{L^1_k L^\infty_{x_3,v}} + C\Vert \hat{f}_0\Vert_{L^p_k L^2_{x_3,v}}. \label{f_estimate_2}
\end{align}
%Here $C>1$ is constant that does not depend on $T$.
\end{theorem}

\begin{remark}
The restriction that $p$ is strictly larger than 2 comes from the condition $\sigma = 2(1-\frac{1}{p})-\e > 1$ with $\e>0$. Indeed, it is used to guarantee the time integrability of $(1+t)^{-\sigma}$
%$\int_0^T \frac{1}{(1+t)^\sigma}\dd t \lesssim 1$ 
for controlling the nonlinear collision term when the $L^1_k\cap L^p_k$ approach introduced in \cite{duan_SIMA} is applied. It will be interesting to construct the global in time solutions for initial data $\hat{f}_0\in (L^1_k \cap L^2_k) L^\infty_{x_3,v}$ without relying on the time-decay properties of solutions.
\end{remark}

\begin{remark}
We finally close the nonlinear estimate using $L^1_k L^\infty_{T,x_3,v}$ control in \eqref{f_estimate} in the spirit of the $L^2_{x_3,v}-L^\infty_{x_3,v}$ argument. We refer description of the other norms in \eqref{f_estimate_2} to Section \ref{sec:proof_strategy}.
\end{remark}

To the best of our knowledge, Theorem \ref{thm:l1k_lpk} provides the first result on the global decay-in-time solution to the Boltzmann equation with non-compact and diffuse reflection boundary condition. Moreover, the decay rate $t^{-(1-\frac{1}{p}-\frac{\e}{2})}$ of the $L^1_k$ norms in \eqref{f_estimate} or \eqref{f_estimate_2}  is almost optimal in the sense that solutions to the two dimensional heat equation decay in time with a polynomial rate as $t^{-(1-\frac{1}{p})}$ via the usual $L^\infty_x-L^{p'}_x$ time decay estimates in the physical variables, where $p'$ is the conjugate to $p$. Such decay rate is also consistent with that of solutions to the compressible Navier-Stokes equations in the infinite layer $\Omega=\R^2\times (-1,1)$ studied in \cite{kagei2007asymptotic,kagei2007resolvent,kagei2008large}, as mentioned in the previous paragraphs. In fact, through the spectral analysis, \cite{kagei2007asymptotic,kagei2007resolvent,kagei2008large} also proved a much stronger result on the large time behavior of solutions, namely, it turns out that the rate of convergence of Navier-Stokes solutions to heat equation solutions is much faster. Thus it would be interesting to further obtain an analogous result for the global decay-in-time Boltzmann solution obtained in Theorem \ref{thm:l1k_lpk}; this will be left for our future study.

The macroscopic dissipation estimate in \eqref{f_estimate_2} degenerates when $|k|\to 0$. Such degeneracy is justified in the case of the whole space by the previous literature \cite{duan_SIMA,ukai2006boltzmann}. Since Poincar\'e inequality holds in our domain, similar to the velocity field in the Navier-Stokes equation, we expect to control the dissipation estimate of $\hat{\mathbf{b}}=(\hat{b}_1,\hat{b}_2,\hat{b}_3)$ and $\hat{c}$ even in the low-frequency regime. 

To fully recover dissipation of $\hat{b}_3$ and $\hat{c}$, we need to leverage the time derivative estimate. We take the $t$-derivative to the original equation \eqref{f_eqn} and obtain
\begin{align}
\begin{cases}
    & \dis    \p_t (\p_t \hat{f}) + i\bar{v}\cdot k \p_t \hat{f} + v_3 \p_{x_3} \p_t \hat{f} + \mathcal{L}\p_t \hat{f} = \hat{\Gamma}(\p_t \hat{f},\hat{f}) + \hat{\Gamma}(\hat{f},\p_t \hat{f}), \\
    & \dis  \p_t \hat{f}(t,k,\pm 1, v) = c_\mu \sqrt{\mu(v)} \int_{u_3\gtrless 0} \p_t \hat{f}(t,k,x_3,u) |u_3|\sqrt{\mu(u)} \dd u, \\
    & \dis  \p_t \hat{f}(0,k,x_3,v) = \p_t \hat{f}_0(k,x_3,v) := - i\bar{v}\cdot k \hat{f}_0 - v_3 \p_{x_3} \hat{f}_0 - \mathcal{L}(\hat{f}_0) + \hat{\Gamma}(\hat{f}_0,\hat{f}_0),   
\end{cases} \label{p_t_f_eqn}
\end{align}
where initial data for $\p_t \hat{f}$ is defined in terms of the first equation of \eqref{f_eqn} at $t=0$.

In the next theorem, we address this issue by providing a refined $\hat{\mathbf{b}},\hat{c}$ dissipation estimate.

\begin{theorem}\label{thm:bc_refined}
Let all the assumptions of Theorem \ref{thm:l1k_lpk} be satisfied, then
%Assume \eqref{initial_assumption} is satisfied. 
for the macroscopic quantities $\hat{\mathbf{b}}=(\hat{b}_1,\hat{b}_2,\hat{b}_3)$ and $\hat{c}$, the time-weighted dissipation estimate in \eqref{f_estimate_2}, particularly when $k$ is near zero, can be refined as 
\begin{align}
\Vert (1+t)^{\sigma/2} (\hat{b}_1,\hat{b}_2) \Vert_{L^1_k L^2_{T,x_3}}\leq C\Vert w\hat{f}_0\Vert_{L^1_k L^\infty_{x_3,v}} + C\Vert \hat{f}_0\Vert_{L^p_k L^2_{x_3,v}},\label{b_1_b_2_refined}
\end{align}
\begin{align}\label{b3_refined}
    \Big\Vert (1+t)^{\sigma/2} \frac{\sqrt{|k|}}{(1+|k|^2)^{1/4}} \hat{b}_3 \Big\Vert_{L^1_k L^2_{T,x_3}} \leq C\Vert w\hat{f}_0\Vert_{L^1_k L^\infty_{x_3,v}} + C\Vert \hat{f}_0\Vert_{L^p_k L^2_{x_3,v}},
\end{align}
\begin{align}\label{c_refined}
     \Big\Vert (1+t)^{\sigma/2} \frac{|k|^{1/4}}{(1+|k|^2)^{1/8}} \hat{c} \Big\Vert_{L^1_k L^2_{T,x_3}} \leq C\Vert w\hat{f}_0\Vert_{L^1_k L^\infty_{x_3,v}} + C\Vert \hat{f}_0\Vert_{L^p_k L^2_{x_3,v}},
\end{align}
for any $T>0$, where $C>0$ is independent of $T$.

If it is further assumed that 
\begin{align}
    &   \Vert w\p_t \hat{f}_0\Vert_{L^1_k L^\infty_{x_3,v}} + \Vert \p_t\hat{f}_0\Vert_{L^p_k L^2_{x_3,v}} < \delta,  \ 2<p\leq \infty, \label{t_derivative_assumption}
\end{align}
then the estimates on $\hat{b}_3$ and $\hat{c}$ can be further refined as
\begin{align}
&   \Vert (1+t)^{\sigma/2} (\hat{b}_3,\hat{c})\Vert_{L^1_k L^2_{T,x_3}} \leq  C\Vert w\p_t\hat{f}_0\Vert_{L^1_k L^\infty_{x_3,v}} + C\Vert \p_t\hat{f}_0\Vert_{L^p_k L^2_{x_3,v}} \label{b3_c_refined}.
\end{align}

\end{theorem}

\begin{remark}
The dissipation estimate for $\hat{b}_1$ and $\hat{b}_2$ \eqref{b_1_b_2_refined} holds true without relying on the assumption regarding the time derivative \eqref{t_derivative_assumption}, as the conservation law of $\hat{b}_1$ and $\hat{b}_2$ provides an additional $k$ factor on $\hat{a}$. However, we do not obtain this gain in $k$ for the conservation law of $\hat{b}_3$ and $\hat{c}$, and we need to include extra $k-$weight in the dissipation estimates \eqref{b3_refined} and \eqref{c_refined}. Notably, we can achieve an improvement in the $k$-weight in estimates of $\hat{b}_3$ and $\hat{c}$. 

It remains uncertain whether it is possible to recover the dissipation estimate in \eqref{b3_c_refined} without utilizing the time-derivative estimate.
\end{remark}

When the tangent variable is two-dimensional, we have obtained in Theorem \ref{thm:l1k_lpk} the global existence of solutions by using the time-decay of solutions with an extra smallness condition on initial data $\hat{f}_0$ in $L^p_kL^2_{x_3,v}$ for $2<p\leq \infty$. In what follows we consider the situation where the tangential direction is only one-dimensional. For this case the parameter $\sigma$ takes the form of $\sigma=(1-1/p)-\e$, then $\sigma$ is strictly less than $1$ for any $1\leq p\leq \infty$ and $\e>0$ small enough, and thus the time-weighted energy method fails to obtain a global existence result basing on time-decay of solutions. This is also similar to the situation in \cite{duan_SIMA} where the one-dimension case is left open. The main difficulty is that even if $p=\infty$, the decay rate $\sigma=(1-1/p)-\e < 1$ is too slow to control the nonlinear part $\hat{\Gamma}(\mathbf{P}\hat{f},\mathbf{P}\hat{f})$ contributed by the pure macro component in the large time scale. This motivates us to re-consider the two-dimensional problem when the tangent variable is one-dimensional:
\begin{equation}\label{2dprob}
\partial_t F+v_1\partial_{x_1}F+v_3\partial_{x_3}F=Q(F,F),
\end{equation}
for $F=F(t,x_1,x_3,v)$ with $t\geq 0$, $x_1\in \mathbb{R}$, $x_3\in (-1,1)$ and $v=(v_1,v_2,v_3)\in \mathbb{R}^3$. We still simply denote $x=(x_1,x_3)\in \O=\R\times (-1,1)$. The problem on the corresponding perturbation $f$ with $F=\mu+ \sqrt{\mu}f$ is given by
\begin{align}
\begin{cases}
     &\dis \p_t f + v_1 \p_{x_1}f+ v_3 \p_{x_3}f + \mathcal{L}f = \Gamma(f,f), \\
    &\dis  f(t,x,v)|_{\gamma_-^\pm } = c_\mu \sqrt{\mu(v)} \int_{u_3\gtrless 0} f(t,x,u)\sqrt{\mu(u)}|u_3|\dd u, \\   
    &\dis  f(0,x,v) = f_0(x,v) := (F(0,x,v)-\mu)/\sqrt{\mu}.  
\end{cases}\label{f_eqn_2d}
\end{align}

To obtain the macroscopic dissipation estimate in the physical space, similar to Theorem \ref{thm:bc_refined}, we turn to leverage the time derivative estimate. Then we take time derivative to \eqref{f_eqn_2d} and obtain
\begin{align}
\begin{cases}
    & \dis   \p_t (\p_t f) +  v_1 \p_{x_1}(\p_t f) + v_3 \p_{x_3}(\p_t f)+ \mathcal{L}\p_t f = \Gamma(\p_t f, f) + \Gamma(f, \p_t f), \\
    & \dis  \p_t f(t,x,v)|_{\gamma_-^\pm} = c_\mu\sqrt{\mu(v)} \int_{u_3 \gtrless 0} \p_t f(t,x,u)\sqrt{\mu(u)} |u_3| \dd u, \\
    & \dis  \p_t f(0,x,v):= \p_t f_0(x,v) = -v_1 \p_{x_1}f_0-v_3 \p_{x_3}f_0 - \mathcal{L}f_0 + \Gamma(f_0,f_0), 
\end{cases} \label{p_t_eqn_2d}
\end{align}
where initial data for $\p_t f$ is defined in terms of the first equation of \eqref{f_eqn_2d} at $t=0$.

Below we state the result on the global existence to the two-dimensional problem \eqref{f_eqn_2d} without taking the Fourier transform. $\mathbf{P}f$ and $P_\gamma f$ under this setting are defined as 
\begin{align*}
    & \mathbf{P}{f}:= \Big({a} + {\mathbf{b}} \cdot v + {c} \frac{|v|^2-3}{2}\Big)\sqrt{\mu(v)},
\end{align*}
where $a$, ${\mathbf{b}}=({b}_1,{b}_2,{b}_3)$ and $c$ are functions of $(t,x)$ for $f=f(t,x,v)$. 
\begin{align*}
    &   P_\gamma {f} (t,x_1,\pm 1, v) := c_\mu \sqrt{\mu(v)} \int_{u_3\gtrless 1} f(t,x_1,\pm 1,u) \sqrt{\mu(u)} |u_3| \dd u,\quad v_3 \lessgtr 0.
\end{align*}

We refer norm notations to Section \ref{sec:notation} later on.

\begin{theorem}\label{thm:2d}
There exist constants $\delta>0$ and $C$ such that if the initial data $f_0(x,v)$ with $F_0(x,v):=\mu+\sqrt{\mu}f_0(x,v)\geq 0$ satisfies
\begin{align}
    &        \Vert f_0\Vert_{L^2_{x,v}} + \Vert \p_t f_0\Vert_{L^2_{x,v}} + \Vert wf_0\Vert_{L^\infty_{x,v}} + \Vert w\p_t f_0\Vert_{L^\infty_{x,v}} <\delta, \label{f_initial_2d}
\end{align}
then there exists a unique solution $f(t,x,v)$ to \eqref{f_eqn_2d} such that $F(t,x,v)=\mu+\sqrt{\mu}f(t,x,v)\geq 0$ and the following estimate is satisfied
\begin{align}
    &    \Vert f\Vert_{L^\infty_T L^2_{x,v}} + \Vert \p_t f\Vert_{L^\infty_T L^2_{x,v}} + \Vert wf\Vert_{L^\infty_{T,x,v}} + \Vert w\p_t f\Vert_{L^\infty_{T,x,v}} \notag \\
    & \leq C[\Vert f_0\Vert_{L^2_{x,v}} + \Vert \p_t f_0\Vert_{L^2_{x,v}} + \Vert wf_0\Vert_{L^\infty_{x,v}} + \Vert w\p_t f_0\Vert_{L^\infty_{x,v}}],\label{thm:2d:e1}
\end{align}
for any $T>0$. Moreover, it holds that
\begin{align}
    &    |(I-P_\gamma)f|_{L^2_{T,\gamma_+}} + |(I-P_\gamma)\p_t f|_{ L^2_{T,\gamma_+}}+ \Vert (\mathbf{I}-\mathbf{P})f\Vert_{L^2_{T,x,\nu}} + \Vert (\mathbf{I}-\mathbf{P})\p_t f\Vert_{L^2_{T,x,\nu}} + \Vert ( \mathbf{b},c)\Vert_{L^2_{T,x}} \notag\\
    & \leq C[\Vert f_0\Vert_{L^2_{x,v}} + \Vert \p_t f_0\Vert_{L^2_{x,v}} + \Vert wf_0\Vert_{L^\infty_{x,v}} + \Vert w\p_t f_0\Vert_{L^\infty_{x,v}}].\label{thm:2d:e2}
\end{align}
\end{theorem}

\begin{remark}
We employ the $L^2_{x,v}-L^\infty_{x,v}$ argument in the physical space to both $f$ and $\p_t f$ for proving Theorem \ref{thm:2d}. In other words, our proof does not rely on the Fourier transform or any spatial regularity, or any Sobolev embedding. Thus this result also holds true for the three-dimensional problem when the tangent variable is two-dimension.     
\end{remark}

\begin{remark}
Based on the dissipation estimates on $(\mathbf{I}-\mathbf{P})f$ and its time derivative in \eqref{thm:2d:e2}, it follows that 
$$
\Vert (\mathbf{I}-\mathbf{P})f(t)\Vert_{L^2_{x,v}}\to 0,
$$ 
as $t\to \infty$.  However, Theorem \ref{thm:2d} does not establish any decay rate of the solution or even give the large time behavior of the macroscopic part $\mathbf{P} f$. It remains an open problem to prove that the Boltzmann equation in the two-dimensional infinite layer shares the same decay rate as the one-dimensional heat equation. This was noted by Kagei in \cite{kagei2007asymptotic, kagei2008large, kagei2007resolvent} for the compressible Navier-Stokes equations.
%in the two-dimensional infintie layer.
\end{remark}

\begin{remark}
The construction of the solution in Theorem \ref{thm:2d} requires the time derivative estimate and the assumption on $\p_t f_0$ in \eqref{f_initial_2d}. This differs from Theorem \ref{thm:l1k_lpk}, where the solution construction does not depend on any assumptions regarding the time derivative. The time-derivative estimate in Theorem \ref{thm:bc_refined} is solely utilized to obtain the refined estimate \eqref{b3_c_refined}.
\end{remark}

To the best of our knowledge, although time decay of solutions is left unknown, Theorem \ref{thm:2d} provides the first result on the global existence of a class of $L^2_{x,v}-L^\infty_{x,v}$ solutions for both two and three dimensional Boltzmann equation with unbounded boundaries. In the whole space $\mathbb{R}^3$, Guo used the $L^3-L^6$ Young's inequality and the Sobolev embedding $H^2(\mathbb{R}^3)\subset L^\infty(\mathbb{R}^3)$ in the $L^2_x$ framework in \cite{guo2004boltzmann}. The $H^2(\mathbb{R}^3)$ argument fails with the presence of the boundary and the $L^6$ estimate fails in the two-dimension problem. In the exterior problem with the domain as the exterior of a bounded domain $\Omega$ in $\R^3$,    
%$\mathbb{R}^3/\O$, 
\cite{jung2023diffusive} proposed an $L^2-L^6-L^\infty$ argument to overcome the difficulty from the lack of Poincar\'e inequality. Similar to \cite{jang2021incompressible} and \cite{cao2023passage}, the derivation of the $L^6$ estimate in the unbounded domain heavily relies on the compactness of the boundary and makes use of the Sobolev embedding $W^{2,\frac{6}{5}}(\bar{\O}^c) \subset H^1(\bar{\O}^c) \subset L^6(\bar{\O}^c)$ in case of three dimensions. Such argument fails when the boundary is non-compact or when the problem becomes of two dimensions.

\subsection{Proof strategy}\label{sec:proof_strategy}

We sketch key points in the proof of the main results. For Theorem \ref{thm:l1k_lpk}, 
%\begin{itemize}
%    \item Theorem \ref{thm:l1k_lpk}.
%\end{itemize}
we begin with a basic $L^1_k L^\infty_T L^2_{x_3,v}$ energy estimate to the equation \eqref{f_eqn}. To control the pure macroscopic components $\hat{\Gamma}(\mathbf{P}\hat{f},\mathbf{P}\hat{f})$ in the nonlinear term, we need a time-dissipation estimate in $L^1_k L^2_T L^\infty_{x_3}L^2_\nu$. Motivated by \cite{kagei2007asymptotic,kagei2007resolvent,kagei2008large}, the macroscopic fluid part should behave as solutions to the two-dimensional heat equations. To control $L^2_T$ dissipation estimate for $\mathbf{P}\hat{f}$, we seek for the $(1+t)^{\sigma/2}$ time weighted $L^1_k L^\infty_T L^2_{x_3,v}$ estimate with $\sigma>1$, and this then causes us to further make the $L^p_k L^\infty_T L^2_{x_3,v}$ estimates with $2<p\leq \infty$.  The equation of $(1+t)^{\sigma/2}\hat{f}$ generates an extra term $(1+t)^{\sigma-1}\hat{f}$ to be controlled. In the low-frequency regime $|k|<1$, we apply the $L^1_k\cap L^p_k$ interpolation argument in \cite{duan_SIMA,kawashima2004lp}. In the high-frequency regime $|k|\geq 1$, we need to derive a time-weighted dissipation macroscopic estimate. 

Since the domain is bounded in the physical variable $x_3$ and we have imposed the diffuse boundary condition, we employ the test function method \eqref{weak_formulation} proposed in \cite{EGKM,EGKM2} for the macroscopic estimate with treating $k$ as extra variables. In this method, we crucially construct the test functions using an extra frequency weight $|k|^2/(1+|k|^2)$. For instance,
%see \eqref{elliptic_a} in Lemma \ref{lemma:macroscopic}. 
to estimate $\hat{a}$, we choose a test function as $\psi_a =   \sqrt{\mu} (|v|^2-10) (-i\bar{v}\cdot k  +  v_3 \p_{x_3}  )\phi_a$ with $\phi_a$ satisfying the elliptic boundary-value problem
\begin{align*}
\begin{cases}
    & \dis (|k|^2 - \p_{x_3}^2)\phi_a(k,x_3) = \bar{\hat{a}}(k,x_3) \frac{|k|^2}{1+|k|^2} , \ x_3\in (-1,1), \\
    &\dis    \p_{x_3} \phi_a(k, \pm 1) = 0.   
\end{cases}
%\label{}
\end{align*}
%for $k\neq 0$. 
This leads to a time-weighted macroscopic dissipation estimate with weight in $|k|/\sqrt{1+|k|^2}$, which provides the desired $L^1_k L^\infty_TL^2_{x_3,v}$ control in the non-zero frequency regime. 
%Note that for $k=0$, we still have to verify that $\|\hat{a}(t,0)\|_{L^2_{T,x_3}}$ is bounded make a slight modification to the test function

The test function method provides $L^2_{x_3,v}$ estimates. To obtain the $L^\infty_{x_3}$ control from the nonlinear operator, we apply the $L^\infty_{x_3,v}$ bootstrap argument, with treating $k$ as extra parameters. We apply method of characteristics with repeated interaction in one dimensional physical variable $x_3\in (-1,1)$ for diffuse boundary and thus obtain the $(1+t)^{\sigma/2}w(v)$ weighted $L^1_k L^\infty_{T,x_3,v}$ estimate. 

%\begin{itemize}
%\item 
As for Theorem \ref{thm:bc_refined}, 
%\end{itemize}
since Poincar\'e inequality holds true, we can construct test functions for $\hat{b}_1,\hat{b}_2$ without involving weights in $k$, see \eqref{phi_b}. We can achieve $H^2_{x_3}$ estimate for these test functions by applying the Poincar\'e inequality. In the estimate of $\hat{b}_1,\hat{b}_2$, there are extra $k$ factors to $\hat{a}$ in the conservation law \eqref{conservation_momentum}. These additional factors ensure control of $\hat{b}_1,\hat{b}_2$ in the low-frequency regimes, while the control in high-frequency regime has been already established. In contrast, the conservation law for $\hat{b}_3$ \eqref{momentum_conservation} is different, as the term $\p_{x_3} \hat{a}$ does not provide a gain in $|k|$. Additional weight in $k$ needs to be introduced to control this $\hat{a}$ factor. We have similar issue for $\hat{c}$ since the conservation law \eqref{energy_conservation} contains $\p_{x_3}\hat{b}_3$ without extra $k$. This leads to the refined estimates for $\hat{b}_3,\hat{c}$ in \eqref{b3_refined} and \eqref{c_refined}. 

To completely remove the degenerate $k$ factor, we observe that the difficulty originates from the extra $\hat{a}$ from the conservation law. Then we use a variant test function method \eqref{weak_formulation_2}, where time derivative only acts on $\p_t f$. With a proper choice of the test function $\psi$ orthogonal to $\ker \mathcal{L}$ for $\hat{b}_3,\hat{c}$, it suffices to obtain the dissipation estimate for $\p_t (\mathbf{I}-\mathbf{P})\hat{f}$. This estimate can be done using the same spirit of Theorem \ref{thm:l1k_lpk}, with an additional assumption on the initial condition $\p_t \hat{f}_0$.

%\begin{itemize}
%\item 
For Theorem \ref{thm:2d} regarding the two-dimensional infinite layer problem, 
%\end{itemize}
the method of treating Theorem \ref{thm:l1k_lpk} via the Fourier transform is no longer applicable as we explained before. Instead, we begin with the basic $L^\infty_T L^2_{x,v}$ energy estimate to \eqref{f_eqn_2d}. In the estimate of the nonlinear operator, we observe that only $\mathbf{b}$ and $c$ remain in the nonlinear term $\Gamma(\mathbf{P}f,\mathbf{P}f)$ when the pure macroscopic component is involved. In fact, with the help of Poincar\'e inequality, we can employ the similar test function method to obtain the $\mathbf{b}$ and $c$ dissipation estimates. Similar to the issue mentioned in the previous paragraph, the conservation law involves $a$. Therefore, we avoid the conservation law by using the weak formulation \eqref{weak_formula} and utilizing the estimate to $\p_t (\mathbf{I}-\mathbf{P})f$, namely, we obtain the dissipation estimate:
 \begin{align*}
    &    \Vert c\Vert_{L^2_{x}}^2+  \Vert \mathbf{b}\Vert_{L^2_{x}}^2 \lesssim \Vert (\mathbf{I}-\mathbf{P})f\Vert_{L^2_{x,v}}^2 + |(I-P_\gamma)f|_{L^2_{\gamma_+}}^2 + \Vert \nu^{-1/2}\Gamma(f,f)\Vert_{L^2_{x,v}}^2 + \Vert \p_t (\mathbf{I}-\mathbf{P})f\Vert_{L^2_{x,v}}^2.
\end{align*}
We close the nonlinear estimate by combining estimates of both $f$ and $\p_t f$ in $L^2_{x,v}\cap L^\infty_{x,v}$.

\subsection{Outline} In Section \ref{sec:prelim}, we list several properties of the linear and nonlinear collision operator. In Section \ref{sec:energy}, we employ the weak formulation with the frequency $k$ to obtain the crucial macroscopic dissipation estimate and the energy estimate. In Section \ref{sec:time_decay}, we derive the time decay property of the energy estimate and apply the $L^2_{x_3,v}-L^\infty_{x_3,v}$ argument in the one-dimensional physical space to derive the time-weighted $L^1_k L^\infty_{T,x_3,v}$ estimate. This estimate is then used to control the nonlinear operator within the energy estimate, leading to the conclusion of Theorem \ref{thm:l1k_lpk}. In Section \ref{sec:time_derivative}, we derive the refined macroscopic dissipation estimate for $\hat{\mathbf{b}},\hat{c}$ by leveraging the control of the $\p_t f$ and conclude Theorem \ref{thm:bc_refined}. In Section \ref{sec:2d}, we consider the two-dimensional problem and employ $L^2_{x,v}-L^\infty_{x,v}$ argument in the physical space, along with time-derivative estimates, to conclude Theorem \ref{thm:2d}.

\subsection{Notation}\label{sec:notation}

We use general norms:
\begin{align*}
    &  \Vert f\Vert_{L^2_\nu} := \Vert \nu^{1/2}f(v)\Vert_{L^2_v}=\Big(\int_{\R^3}\nu(v)|f(v)|^2\dd v\Big)^{1/2}, \\
    & \Vert f\Vert_{L^2_T} := \Big(\int_0^T |f(t)|^2 \dd t \Big)^{1/2}, \\ 
    & \Vert f\Vert_{L^\infty_T} := \sup_{0\leq t\leq T}|f(t)|. 
    %\\
    \end{align*}
Moreover, $f \lesssim g$  means that there exists $C>1$ such that $f\leq C g$, and $f\leq o(1)g$ and $f \lesssim o(1)g$ both mean that there exists $0<\delta\ll 1$ such that $f\leq \delta g$. 

In Theorem \ref{thm:l1k_lpk} and Theorem \ref{thm:bc_refined} we use norms:
\begin{align*}
    & | \hat{f}|_{L^2_{\gamma_+}} := \Big(\int_{v_3>0} |f(k,1,v)|^2 |v_3| \dd v +  \int_{v_3<0} |f(k,-1,v)|^2 |v_3| \dd v\Big)^{1/2}, \\
    &  \Vert \hat{f}\Vert_{L^1_k L^\infty_{T,x_3,v}}:=      \int_{\mathbb{R}^2}  \sup_{0\leq t\leq T,x_3\in (-1,1),v\in \mathbb{R}^3} |\hat{f}(t,k,x_3,v)|  \dd k,          \\
    &  \Vert \hat{f} \Vert_{L^1_k L^2_{T,x_3,v}} := \int_{\mathbb{R}^2} \Big(\int_{0}^T \int_{-1}^1 \int_{\mathbb{R}^3} |\hat{f} (t,k,x_3,v)|^2 \dd v \dd x_3 \dd t \Big)^{1/2}  \dd k , \\
    &    \Vert \hat{f} \Vert_{L^1_k L^\infty_T L^2_{x_3,v}}:= \int_{\mathbb{R}^2}  \sup_{0\leq t \leq T} \Big(\int_{-1}^1 \int_{\mathbb{R}^3} |\hat{f}(t,k,x_3,v)|^2 \dd v \dd x_3 \Big)^{1/2}   \dd k,                  \\
    & |\hat{f}|_{L^1_k L^2_{T,\gamma_+}} :=  \int_{\mathbb{R}^2}  \Big( \int_0^T \int_{v_3>0} |\hat{f}(t,k,1,v)|^2 |v_3| \dd v \dd t +  \int_0^T \int_{v_3<0} |\hat{f}(t,k,-1,v)|^2 |v_3| \dd v \dd t \Big)^{1/2}   \dd k, 
\end{align*}
and 
\begin{align*}
    &   \Vert \hat{f}\Vert_{L^p_k L^\infty_T L^2_{x_3,v}}:= \Big(\int_{\mathbb{R}^2}  \sup_{0\leq t \leq T} \Big(\int_{-1}^1 \int_{\mathbb{R}^3} |\hat{f}(t,k,x_3,v)|^2 \dd v \dd x_3 \Big)^{p/2}   \dd k \Big)^{1/p}, \\
    & \Vert \hat{f}\Vert_{L^p_k L^2_{x_3,v}} := \Big(\int_{\mathbb{R}^2} \Big( \int_{-1}^1 \int_{\mathbb{R}^3} |\hat{f} (k,x_3,v)|^2 \dd v \dd x_3  \Big)^{p/2}  \dd k\Big)^{1/p},
\end{align*}
with $1\leq p<\infty$, and for $p=\infty$, the norms of $L^\infty_k L^\infty_T L^2_{x_3,v}$ and $L^\infty_k L^2_{x_3,v}$ are similarly defined in the standard way. 

In Theorem \ref{thm:2d} we use norms:
\begin{align*}
    &      | f|_{L^2_{\gamma_+}} := \Big(\int_{\mathbb{R}}\int_{v_3>0} |f(x_1,1,v)|^2 |v_3| \dd v \dd x_1 +  \int_{\mathbb{R}}\int_{v_3<0} |f(x_1,-1,v)|^2 |v_3| \dd v \dd x_1\Big)^{1/2}, \\
    &  \Vert f\Vert_{L^\infty_{T}L^2_{x,v}}:= \sup_{0\leq t\leq T}\Big(\int_{\O} \int_{\mathbb{R}^3} |f(t,x,v)|^2 \dd v \dd x \Big)^{1/2}, \\
    & \Vert f\Vert_{L^\infty_{T,x,v}}: = \sup_{(t,x,v)\in [0,T]\times \O\times \mathbb{R}^3} |f(t,x,v)|, \\
    & \Vert f\Vert_{L^2_{T,x,v}} := \Big(\int_{0}^T \int_{\O}\int_{\mathbb{R}^3} |f(t,x,v)|^2 \dd v \dd x \dd t \Big)^{1/2}, \\
    & |f|_{L^2_{T,\gamma_+}}:=    \Big(\int_0^T \int_{\mathbb{R}} \int_{v_3>0}|f(t,x_1,1,v)|^2|v_3| \dd v \dd x_1 \dd t + \int_0^T \int_{\mathbb{R}}\int_{v_3<0}|f(t,x_1,-1,v)|^2 |v_3| \dd v \dd x_1 \dd t \Big)^{1/2}   .
\end{align*}

\section{Preliminary}\label{sec:prelim}

In this section, we give basic estimates on the linearized collision operator $\mathcal{L}$ and nonlinear collision operator $\Gamma(\cdot,\cdot)$ as in \eqref{def.L} and \eqref{def.Ga}.

First of all, for $\mathcal{L}$, we have the following two lemmas.

\begin{lemma}[\cite{R}]%\label{lemma:k_gamma}
It holds that $\mathcal{L}=\nu(v)-K$, where 
\begin{equation*}
\nu(v)=\int_{\mathbb{R}^3}\int_{\mathbb{S}^2}|(v-u)\cdot \omega|\mu(u)\,\d\omega\d u,
\end{equation*} 
and
\begin{equation*}
Kf(v)=\int_{\mathbb{R}^3}\int_{\mathbb{S}^2}B(v-u,\omega)[\sqrt{\mu(v)\mu(u)}f(u)-\sqrt{\mu(u)\mu(u')}f(v')-\sqrt{\mu(u)\mu(v')}f(u')]\,\d\omega\d u.
\end{equation*}
Here, the collision frequency $\nu(v)$ satisfies
\begin{equation}\label{nu_bdd}
\nu(v) \geq \nu_0 \sqrt{|v|^2+1} \geq \nu_0
\end{equation}
for a positive constant $\nu_0>0$. The integral operator $K$ is given by
\[
Kf(x,v)=\int_{\mathbb{R}^3}\mathbf{k}(v,u)f(x,u)\,\dd u,
\]
with the integral kernel $\mathbf{k}(v,u)$ satisfying
\Be\notag%\label{k_varrho}
 |\mathbf{k}  (v,u)| \lesssim \mathbf{k}_\varrho (v,u), \   \ \mathbf{k}_\varrho (v,u) := e^{- \varrho |v-u|^2}/ |v-u|,
\Ee
for a constant $\varrho>0$.
\end{lemma}

\begin{lemma}\label{lemma:k_theta}
Let $0\leq \theta < \frac{1}{4}$, and $\mathbf{k}_\theta(v,u) := \mathbf{k}(v,u) \frac{e^{\theta |v|^2}}{e^{\theta |u|^2}}$, then there exists $C_\theta > 0$ such that
\begin{equation}\label{k_theta}
\int_{\mathbb{R}^3}  \mathbf{k}(v,u) \frac{e^{\theta |v|^2}}{e^{\theta |u|^2}}  \,\dd u  \leq \frac{C_\theta}{1+|v|}.
\end{equation}
Moreover, for $N\gg 1$, we have
\begin{equation}\label{k_N_upper_bdd}
\mathbf{k}_\theta(v,u) \mathbf{1}_{|v-u|> \frac{1}{N}} \leq C_N,
\end{equation}
and
\begin{equation}\label{K_N_small}
\int_{|u|>N \text{ or } |v-u|\leq \frac{1}{N}} \mathbf{k}_\theta(v,u) \,\dd u \lesssim \frac{1}{N} \leq o(1).
\end{equation}

\end{lemma}

\begin{proof}
The proof mostly follows from Lemma 3 in \cite{G}, where for $0\leq \theta < \frac{1}{4}$, we can find $\e = \e(\theta)$ such that
\begin{align}
    \mathbf{k}_\theta(v,u) \leq \big[\frac{1}{|v-u|}+ |v-u| \big]e^{-\e \big[|v-u|^2 + |v\cdot (v-u)| \big]}. \label{k_theta_bdd}
\end{align}    
Thus \eqref{k_theta} follows by the factor $e^{-\e|v\cdot (v-u)|}$.

Clearly, with the exponential decay in $|v-u|$, we conclude \eqref{k_N_upper_bdd}.

For \eqref{K_N_small}, directly applying \eqref{k_theta_bdd} we have
\begin{align*}
    &\int_{|v-u|\leq \frac{1}{N}} \mathbf{k}_\theta(v,u) \,\dd u \lesssim o(1)   .
\end{align*}
When $|u|>N$, we split the cases into $|v|>\frac{N}{2}$ and $|v|\leq \frac{N}{2}$. In the first case, \eqref{K_N_small} follows by applying \eqref{k_theta}. For the other case, we have $|v-u|>\frac{N}{2}$, then \eqref{K_N_small} follows from \eqref{k_theta_bdd}.
\end{proof}

The estimate for the nonlinear operator is given by the following lemma.

\begin{lemma}\label{lemma:gamma_est}
For $1\leq p\leq \infty$, we have the following estimates to the nonlinear operator $\hat{\Gamma}(\hat{f},\hat{g})$:
\begin{align}
    &   \Big| \int_{\mathbb{R}^3} \hat{\Gamma}(\hat{f},\hat{g}) \bar{\hat{h}}(k) \dd v \Big| \lesssim \int_{\mathbb{R}^2} \Vert \hat{f}(k-\ell)\Vert_{L^2_v} \Vert \hat{g}(\ell)\Vert_{L^2_\nu} \Vert (\mathbf{I}-\mathbf{P})\hat{h}\Vert_{L^2_\nu}  \dd \ell, 
    \label{gamma_product}
\end{align}
\begin{align}
    & \Big\Vert \Big|\int_0^T \int_{-1}^1 \int_{\mathbb{R}^3} \hat{\Gamma}(\hat{f},\hat{g})\bar{\hat{h}}(k) \dd v \dd x_3 \dd t \Big|^{1/2} \Big\Vert_{L^p_k} \lesssim o(1) \Vert (\mathbf{I}-\mathbf{P}) \hat{h}\Vert_{L^p_k L^2_{T,x_3,\nu}} + \Vert \hat{f} \Vert_{L^p_k L^\infty_T L^2_{x_3,v}} \Vert \hat{g}\Vert_{L^1_k L^2_T L^\infty_{x_3} L^2_\nu} \notag \\
    &\lesssim o(1) \Vert (\mathbf{I}-\mathbf{P}) \hat{h}\Vert_{L^p_k L^2_{T,x_3,\nu}} + \Vert \hat{f} \Vert_{L^p_k L^\infty_T L^2_{x_3,v}} \Vert w\hat{g}\Vert_{L^1_k L^2_T L^\infty_{x_3,v}}, \label{gamma_est}
\end{align}
\begin{align}
    &    \Big\Vert \Big|\int_0^T \int_{-1}^1 \int_{\mathbb{R}^3} (1+t)^\sigma \hat{\Gamma}(\hat{f},\hat{g})\bar{\hat{h}}(k) \dd v \dd x_3 \dd t \Big|^{1/2} \Big\Vert_{L^1_k}\notag \\
    &\lesssim o(1) \Vert (1+t)^{\sigma/2}(\mathbf{I}-\mathbf{P}) \hat{h}\Vert_{L^1_k L^2_{T,x_3,\nu}} + \Vert (1+t)^{\sigma/2} \hat{f}\Vert_{L^1_k L^\infty_T L^2_{x_3,v}} \Vert w\hat{g}\Vert_{L^1_k L^2_T L^\infty_{x_3,v}}, 
    \label{gamma_est_time}
\end{align}
and
\begin{align}
    &  \Vert \nu^{-1}(1+t)^{\sigma/2}w\hat{\Gamma}(\hat{f},\hat{g}) \Vert_{L^1_k L^\infty_{T,x_3,v}} \lesssim \Vert (1+t)^{\sigma/2} w \hat{f}\Vert_{L^1_k L^\infty_{T,x_3,v}} \Vert (1+t)^{\sigma/2} w \hat{g}\Vert_{L^1_k L^\infty_{T,x_3,v}}, 
    \label{gamma_est_time_linfty}
\end{align}
where all estimates are independent of $T>0$.
\end{lemma}

\begin{proof}
From the definition of $\hat{\Gamma}(\hat{f},\hat{f})$ in \eqref{gamma_hat} or \eqref{gamma_hat1}, we compute that
\begin{align*}
    &  \Big| \int_{\mathbb{R}^3} \hat{\Gamma}(\hat{f},\hat{g}) \bar{\hat{h}}(k) \dd v \Big| = \Big| \int_{\mathbb{R}^3} \int_{\mathbb{R}^3\times \mathbb{S}^2} |(v-u)\cdot \omega| \sqrt{\mu(u)} (\hat{f}(v')*\hat{g}(u') - \hat{f}(v)*\hat{g}(u)) \dd u \dd \omega \bar{\hat{h}}(k) \dd v \Big| \\
    & = \Big|\int_{\mathbb{R}^3} \int_{\mathbb{R}^3\times \mathbb{S}^2} |(v-u)\cdot \omega| \sqrt{\mu(u)} \int_{\mathbb{R}^2} [\hat{f}(k-\ell,v')\hat{g}(\ell,u')-\hat{f}(k-\ell,v)\hat{g}(\ell,u)] \dd \ell \dd \omega \dd u \bar{\hat{h}}(k) \dd v \Big| \\
    & = \Big|\int_{\mathbb{R}^2} \int_{\mathbb{R}^3} \Gamma(\hat{f}(k-\ell) ,\hat{g}(\ell)) \bar{\hat{h}}(k) \dd v \dd \ell \Big| = \Big| \int_{\mathbb{R}^2} \int_{\mathbb{R}^3} \nu^{-1/2}\Gamma(\hat{f}(k-\ell),\hat{g}(\ell)) \nu^{1/2} (\mathbf{I}-\mathbf{P})\bar{\hat{h}}(k) \dd v \dd \ell \Big| \\
    & \lesssim \int_{\mathbb{R}^2} \Vert \nu^{-1/2}\Gamma(\hat{f}(k-\ell),\hat{g}(\ell)) \Vert_{L^2_v}  \Vert (\mathbf{I}-\mathbf{P})\hat{h}\Vert_{L^2_\nu}   \dd \ell    \\
    & \lesssim \int_{\mathbb{R}^2} \Vert \hat{f}(k-\ell)\Vert_{L^2_v} \Vert \hat{g}(\ell)\Vert_{L^2_\nu} \Vert (\mathbf{I}-\mathbf{P})\hat{h}\Vert_{L^2_\nu}  \dd \ell  .
\end{align*}
In the third line, we have used that $\mathbf{P}\hat{h}(k)$ is orthogonal to $\Gamma$. In the last line, we have used the standard estimate for the nonlinear operator:
\begin{align*}
    \Vert \nu^{-1/2}\Gamma(f,g)\Vert_{L^2_v}\lesssim \Vert f\Vert_{L^2_v} \Vert \nu^{1/2}g\Vert_{L^2_v}.
\end{align*}
This concludes \eqref{gamma_product}.

For proving \eqref{gamma_est}, we apply \eqref{gamma_product} to have
\begin{align}
&  \Big| \int_0^T \int_{-1}^1 \int_{\mathbb{R}^3} \hat{\Gamma}(\hat{f},\hat{g})\bar{\hat{h}}(k) \dd v \dd x_3 \dd t \Big|^{1/2} \notag \\
& \lesssim \Big(\int_0^T \int_{-1}^1 \int_{\mathbb{R}^2} \Vert \hat{f}(k-\ell)\Vert_{L^2_v} \Vert \hat{g}(\ell)\Vert_{L^2_\nu} \Vert (\mathbf{I}-\mathbf{P})\hat{h}\Vert_{L^2_\nu} \dd \ell \dd x_3 \dd t \Big)^{1/2} \notag \\
    & \lesssim \Big( \int_0^T \int_{-1}^1 \Big(\int_{\mathbb{R}^2} \Vert \hat{f}(k-\ell)\Vert_{L^2_v} \Vert \hat{g}(\ell)\Vert_{L^2_\nu} \dd \ell \Big)^2 \dd x_3 \dd t\Big)^{1/4}\Big(\int_0^T\int_{-1}^1 \Vert (\mathbf{I}-\mathbf{P})\hat{h}\Vert_{L^2_\nu}^2 \dd x_3 \dd t  \Big)^{1/4} \notag\\
    &  \lesssim o(1)\Vert (\mathbf{I}-\mathbf{P}) \hat{h}\Vert_{L^2_{T,x_3,\nu}} + \Big( \int_0^T \int_{-1}^1 \Big(\int_{\mathbb{R}^2} \Vert \hat{f}(k-\ell)\Vert_{L^2_v} \Vert \hat{g}(\ell)\Vert_{L^2_\nu} \dd \ell \Big)^2 \dd x_3 \dd t\Big)^{1/2} \notag\\
    & \lesssim o(1)\Vert (\mathbf{I}-\mathbf{P}) \hat{h}\Vert_{L^2_{T,x_3,\nu}} + \int_{\mathbb{R}^2} \Big( \int_0^T \int_{-1}^1 \Vert \hat{f}(k-\ell)\Vert^2_{L^2_v}  \Vert \hat{g}(\ell)\Vert^2_{L^2_\nu} \dd x_3 \dd t \Big)^{1/2} \dd \ell  \notag\\
    & \lesssim o(1)\Vert (\mathbf{I}-\mathbf{P}) \hat{h}\Vert_{L^2_{T,x_3,\nu}} + \int_{\mathbb{R}^2}   \Vert \hat{f}(k-\ell)\Vert_{L^\infty_T L^2_{x_3,v}} \Vert \hat{g}(\ell)\Vert_{L^2_T L^\infty_{x_3}L^2_\nu}    \dd \ell.  \label{minkovski}
\end{align}
In the second last line, we have used the Minkowski inequality.

Last, we take the $k$-integration and use the Young's convolution inequality to have
\begin{align}
&   \Big\Vert \Big|\int_0^T \int_{-1}^1 \int_{\mathbb{R}^3} \hat{\Gamma}(\hat{f},\hat{g})\bar{\hat{h}}(k) \dd v \dd x_3 \dd t \Big|^{1/2} \Big\Vert_{L^p_k} \notag \\
    &\lesssim o(1)\Vert (\mathbf{I}-\mathbf{P})\hat{h}\Vert_{L^p_k L^2_{T,x_3,\nu}} + \Big\Vert   \Vert \hat{f}(k)\Vert_{L^\infty_T L^2_{x_3,v}} *\Vert \hat{g}(k)\Vert_{L^2_T L^\infty_{x_3} L^2_\nu}  \Big\Vert_{L^p_k} \notag \\
& \lesssim o(1)\Vert (\mathbf{I}-\mathbf{P})\hat{h}\Vert_{L^p_k L^2_{T,x_3,\nu}} + \Vert \hat{f} \Vert_{L^p_k L^\infty_T L^2_{x_3,v}} \Vert \hat{g}\Vert_{L^1_k L^2_T L^\infty_{x_3} L^2_\nu}. \label{convolution}
\end{align}
This concludes the first inequality in \eqref{gamma_est}. 

The second inequality in \eqref{gamma_est} follows from the fact that $\Vert \nu^{1/2} \hat{f}\Vert_{L^2_v} \lesssim \Vert w\hat{f}\Vert_{L^\infty_v}$, where $w$ is defined in \eqref{weight_w}.

The proof of \eqref{gamma_est_time} is the same, with placing one $(1+t)^{\sigma/2}$ to $\hat{f}(k-\ell)$, and placing the other $(1+t)^{\sigma/2}$ to $\hat{h}(k)$. These two terms become $\Vert (1+t)^{\sigma/2} \hat{f}(k-\ell)\Vert_{L^\infty_T L^2_{x_3,v}}$ and $\Vert (1+t)^{\sigma/2}(\mathbf{I}-\mathbf{P})\hat{h}\Vert_{L^2_{T,x_3,\nu}}$ respectively.

In the end, we prove \eqref{gamma_est_time_linfty}. In fact, we compute that
\begin{align*}
    & |w(v)\hat{\Gamma}(\hat{f},\hat{g})| = w(v) \Big|\int_{\mathbb{R}^3}\int_{\mathbb{S}^2}|(v-u)\cdot \omega| \sqrt{\mu(u)}[\hat{f}(v')*_k\hat{g}(u') - \hat{f}(v)*_k \hat{g}(u)] \dd \omega \dd u  \Big| \\
    & \lesssim w(v) \Big| \int_{\mathbb{R}^3} |v-u|\sqrt{\mu(u)}w^{-1}(v)w^{-1}(u) \Vert w\hat{f} \Vert_{L^\infty_v}*_k \Vert w\hat{g}\Vert_{L^\infty_v} \dd u  \Big| \\
    & \lesssim \nu(v) \Vert w\hat{f} \Vert_{L^\infty_v}*_k \Vert w\hat{g}\Vert_{L^\infty_v} . 
\end{align*}
Then taking $L^\infty$ in $t$ and $x_3$ with the extra terms $\nu^{-1}$ and $(1+t)^{\sigma/2}$, we have
\begin{align*}
    & \Vert \nu^{-1}(1+t)^{\sigma/2} w\hat{\Gamma}(\hat{f},\hat{g})\Vert_{L^\infty_{T,x_3,v}} \lesssim \Vert (1+t)^{\sigma/2} w\hat{f}\Vert_{L^\infty_{T,x_3,v}}*_k\Vert (1+t)^{\sigma/2} w\hat{g}\Vert_{L^\infty_{T,x_3,v}}.
\end{align*}
Therefore, taking integration in $k$, we conclude \eqref{gamma_est_time_linfty} from the Young's convolution inequality.  
\end{proof}

\section{$L^p_k L^\infty_T L^2_{x_3,v}$ estimate and macroscopic dissipation estimate}\label{sec:energy}

In this section, we construct the following energy estimate. 
\begin{proposition}[\textbf{Energy estimate}]\label{prop:full energy}
Let $\hat{f}$ be the solution to \eqref{f_eqn}, with initial condition $f_0$ satisfying \eqref{initial_assumption}, then
\begin{align*}
    &   \Vert \hat{f}\Vert_{L^1_k L^\infty_T L^2_{x_3,v}} + \Vert (\mathbf{I}-\mathbf{P})\hat{f}\Vert_{L^1_k L^2_{T,x_3,\nu}} + |(I-P_\gamma)\hat{f}|_{L^1_k L^2_{T,\gamma_+}} + \Big\Vert \frac{|k|}{\sqrt{1+|k|^2}} (\hat{a},\hat{\mathbf{b}},\hat{c})\Big\Vert_{L^1_k L^2_{T,x_3}} \\
    &\lesssim \Vert \hat{f}_0\Vert_{L^1_k L^2_{x_3,v}} + \Vert \hat{f}\Vert_{L^1_k L^\infty_T L^2_{x_3,v}} \Vert w\hat{f}\Vert_{L^1_k L^2_T L^\infty_{x_3,v}}.
\end{align*}
\end{proposition}

The proof of Proposition \ref{prop:full energy} follows by combining the basic energy estimate to \eqref{f_eqn} in Lemma \ref{lemma: energy}, which controls $\Vert \hat{f}\Vert_{L^1_k L^\infty_T L^2_{x_3,v}}$, and a crucial macroscopic dissipation estimate in Lemma \ref{lemma:macroscopic}, which controls the macroscopic component $\Big\Vert \frac{|k|}{\sqrt{1+|k|^2}} (\hat{a},\hat{\mathbf{b}},\hat{c})\Big\Vert_{L^1_k L^2_{T,x_3}}$.

\subsection{Basic energy estimate}

\begin{lemma}\label{lemma: energy}
Let $\hat{f}$ satisfy the assumption in Proposition \ref{prop:full energy}, then it holds that
\begin{align}
    &     \Vert \hat{f}\Vert_{L^1_k L^\infty_T L^2_{x_3,v}} + \Vert (\mathbf{I}-\mathbf{P})\hat{f}\Vert_{L^1_k L^2_{T,x_3,\nu}} + |(I-P_\gamma)\hat{f}|^2_{L^1_k L^2_{T,\gamma_+}} \notag\\
    & \lesssim \Vert \hat{f}_0\Vert_{L^1_k L^2_{x_3,v}} + \Vert \hat{f}\Vert_{L^1_k L^\infty_T L^2_{x_3,v}} \Vert w\hat{f}\Vert_{L^1_k L^2_T L^\infty_{x_3,v}},   
    \label{L1k_energy}
\end{align}
and for $1\leq p\leq \infty$,
\begin{align}
    & \Vert \hat{f}\Vert_{L^p_k L^\infty_T L^2_{x_3,v}} + \Vert (\mathbf{I}-\mathbf{P})\hat{f}\Vert_{L^p_k L^2_{T,x_3,\nu}} + |(I-P_\gamma)\hat{f}|^2_{L^p_k L^2_{T,\gamma_+}} \notag\\
    & \lesssim \Vert \hat{f}_0\Vert_{L^p_k L^2_{x_3,v}} + \Vert \hat{f}\Vert_{L^p_k L^\infty_T L^2_{x_3,v}} \Vert w \hat{f}\Vert_{L^1_k L^2_T L^\infty_{x_3,v}}. \label{Lpk_energy}
\end{align}
\end{lemma}

\begin{proof}
    We multiply \eqref{f_eqn} by $\bar{\hat{f}}$, the complex conjugate of $\hat{f}$, and then take the real part: 
\begin{align*}
    &\p_t \int_{-1}^1 \int_{\mathbb{R}^3} |\hat{f}|^2 \dd v \dd x_3 + |(I-P_\gamma)\hat{f}|^2_{L^2_{\gamma_+}} + \textbf{Re}\int_{-1}^1 \int_{\mathbb{R}^3} \mathcal{L}\hat{f} \bar{\hat{f}} \dd v \dd x_3 = \textbf{Re}  \int_{-1}^1 \int_{\mathbb{R}^3} \bar{\hat{f}} \hat{\Gamma}(\hat{f},\hat{f})   \dd v \dd x_3 .
\end{align*}
Taking time integration on $[0,T]$ and taking square root we obtain
\begin{align*}
    &  \Vert \hat{f}(T)\Vert_{L^2_{x_3,v}} + \Vert (\mathbf{I}-\mathbf{P})\hat{f}\Vert_{L^2_{T,x_3,\nu}} + |(I-P_\gamma)\hat{f}|_{L^2_{T,\gamma_+}} \\
    &\lesssim \Vert \hat{f}_0 \Vert_{L^2_{x_3,v}} + \Big(\int_0^T \left|\int_{-1}^1 \int_{\mathbb{R}^3}  \hat{\Gamma}(\hat{f},\hat{f}) \bar{\hat{f}}  \dd v \dd x_3\right| \dd t \Big)^{1/2}  .
\end{align*}

Then, since $T>0$ can be arbitrary, taking $L^p_k$-norm we obtain
\begin{align*}
    &     \Vert \hat{f}\Vert_{L^p_k L^\infty_T L^2_{x_3,v}} + \Vert (\mathbf{I}-\mathbf{P})\hat{f}\Vert_{L^p_k L^2_{T,x_3,\nu}} + |(I-P_\gamma)\hat{f}|_{L^p_k L^2_{T,\gamma_+}} \\
    &   \lesssim  \Vert \hat{f}_0\Vert_{L^p_k L^2_{x_3,v}} + \Vert \hat{f}\Vert_{L^p_k L^\infty_T L^2_{x_3,v}} \Vert w\hat{f}\Vert_{L^1_k L^2_T L^\infty_{x_3,v}}.
\end{align*}
In the last line, we applied \eqref{gamma_est} in Lemma \ref{lemma:gamma_est}.

This concludes both \eqref{L1k_energy} and \eqref{Lpk_energy}.
\end{proof}

\subsection{Macroscopic dissipation estimate}

The macroscopic dissipation estimate is given by the following lemma. We employ the dual argument proposed in \cite{EGKM,EGKM2} for the macroscopic estimate with treating $k=(k_1,k_2)$ as an extra variable. Note that $k$ is the Fourier variable corresponding to the two-dimensional horizontal variable $\bar{x}=(x_1,x_2)\in \R^2$. In the method, we crucially construct the test functions using extra weight $\frac{|k|^2}{1+|k|^2}$, see \eqref{elliptic_a} later on for instance. This leads to the frequency-weighted macroscopic dissipation estimate with the weight function $\frac{|k|}{\sqrt{1+|k|^2}}$ that features that the macroscopic component behaves as two-dimensional diffusion waves in the infinite layer $\Omega=\R^2\times (-1,1)$.

\begin{lemma}\label{lemma:macroscopic}
Let $\hat{f}$ satisfy the assumption in Proposition \ref{prop:full energy}, then it holds that
\begin{align}
    &     \Big\Vert \frac{|k|}{\sqrt{1+|k|^2}}(\hat{a},\hat{\mathbf{b}},\hat{c})\Big\Vert_{L^1_k L^2_{T,x_3}} \lesssim \Vert (\mathbf{I}-\mathbf{P})\hat{f}\Vert_{L^1_k L^2_{T,x_3,v}}  \notag\\
    &+ \Vert \hat{f}\Vert_{L^1_k L^\infty_T L^2_{x_3,v}} \Vert w\hat{f}\Vert_{L^1_k L^2_T L^\infty_{x_3,v}} + |(I-P_\gamma)\hat{f}|_{L^1_k L^2_{T,\gamma_+}} + \Vert \hat{f}\Vert_{L^1_k L^\infty_T L^2_{x_3,v}} + \Vert \hat{f}_0\Vert_{L^1_k L^2_{x_3,v}}. 
    \notag
\end{align}

\hide
When $k=0$, we have
\begin{align*}
    &    \mathbf{1}_{k=0}\Vert (\hat{a},\hat{\mathbf{b}},\hat{c})\Vert^2_{L^2_{T,x_3}} \lesssim \mathbf{1}_{k=0}\Big[ \Vert (\mathbf{I}-\mathbf{P})\hat{f}\Vert^2_{L^2_{T,x_3,v}} + \Vert \nu^{-1/2}\hat{\Gamma}(\hat{f},\hat{f})\Vert^2_{L^2_{T,x_3,v}} \\
    & + |(I-P_\gamma)\hat{f}|^2_{L^2_{T,\gamma_+}} + \Vert \hat{f}\Vert^2_{L^\infty_{T}L^2_{x_3,v}} + \Vert \hat{f}_0\Vert_{L^2_{x_3,v}}   \Big].
\end{align*}

\unhide

\end{lemma}

\begin{proof}

In order to estimate the macroscopic component of $\hat{f}$, we use the following weak formulation of  \eqref{f_eqn} with a test function $\psi$:
\begin{align}
    &  \underbrace{\int_{-1}^1 \int_{\mathbb{R}^3} [\hat{f}\psi(T) - \hat{f}\psi (0)] \dd v \dd x_3 }_{\eqref{weak_formulation}_0} \notag\\
    &+  \underbrace{\int_0^{T}\int_{-1}^{1}\int_{\mathbb{R}^3} i\bar{v}\cdot k \hat{f} \psi \dd v \dd x_3 \dd t}_{\eqref{weak_formulation}_1} \underbrace{- \int_0^T \int_{-1}^1 \int_{\mathbb{R}^3} v_3 \hat{f} \p_{x_3}\psi \dd v \dd x_3 \dd t}_{\eqref{weak_formulation}_2} \notag\\
    &+ \underbrace{ \int_0^T \int_{\mathbb{R}^3} v_3 [\hat{f}(k,1)  \psi(1) - \hat{f}(k,-1) \psi(-1) ]\dd v \dd t}_{\eqref{weak_formulation}_3} \underbrace{- \int_0^T \int_{\mathbb{R}^3} \int_{-1}^1 \hat{f} \p_t \psi \dd v \dd x_3 \dd t}_{\eqref{weak_formulation}_4} \notag  \\
    &+ \underbrace{\int_0^T \int_{-1}^1 \int_{\mathbb{R}^3} \mathcal{L}(\hat{f}) \psi \dd v \dd x_3 \dd t}_{\eqref{weak_formulation}_5}  = \underbrace{\int_0^T\int_{-1}^1 \int_{\mathbb{R}^3} \hat{\Gamma}(\hat{f},\hat{f}) \psi \dd v \dd x_3 \dd t}_{\eqref{weak_formulation}_6}. \label{weak_formulation}
\end{align}

\hide
For conciseness, below we only discuss the case $k\neq 0$. The case of $k=0$ will be discussed in Remark \ref{rmk:k=0}. For ease of notations, in the proof we do not specify the condition $k\neq 0$ and just use the notation $L^1_k$.
\unhide

\textit{Estimate of $\hat{a}$}. 

For the estimate of $\hat{a}$, we choose a test function as
\begin{align}
    &  \psi_a =   \sqrt{\mu} (|v|^2-10) (-i\bar{v}\cdot k  +  v_3 \p_{x_3}  )\phi_a.  \label{test_a}
\end{align}
Here $\phi_a$ satisfies the elliptic equation
\begin{align}
\begin{cases}
    & \dis (|k|^2 - \p_{x_3}^2)\phi_a(k,x_3) = -\bar{\hat{a}}(k,x_3) \frac{|k|^2}{1+|k|^2} , \ x_3\in (-1,1), \\
    &\dis    \p_{x_3} \phi_a(k, \pm 1) = 0.    
\end{cases}\label{elliptic_a}
\end{align}
Here $\bar{\hat{a}}$ stands for the complex conjugate of $\hat{a}$.

Multiplying \eqref{elliptic_a} by $\bar{\phi}_a$, the complex conjugate of $\phi_a$, we obtain
\begin{align*}
    &  |k|^2 \Vert \phi_a \Vert_{L^2_{x_3}}^2 + \Vert \p_{x_3}\phi_a \Vert_{L^2_{x_3}}^2 \lesssim  \frac{|k|^2}{1+|k|^2} \Vert \hat{a}\Vert_{L^2_{x_3}}^2 + o(1)\frac{|k|^2}{1+ |k|^2} \Vert \phi_a \Vert_{L^2_{x_3}}^2, \\
    & \Vert |k| \phi_a\Vert_{L^2_{x_3}}^2 + \Vert \p_{x_3}\phi_a\Vert_{L^2_{x_3}}^2 \lesssim \frac{|k|^2}{1+|k|^2} \Vert \hat{a}\Vert_{L^2_{x_3}}^2.
\end{align*}

Multiplying \eqref{elliptic_a} by $|k|^2 \bar{\phi}_a$ we obtain
\begin{align*}
    &    |k|^4 \Vert \phi_a\Vert_{L^2_{x_3}}^2 + |k|^2 \Vert \p_{x_3}\phi_a\Vert_{L^2_{x_3}}^2  \lesssim o(1) |k|^4 \Vert \phi_a \Vert_{L^2_{x_3}}^2 + \frac{|k|^4}{(1+|k|^2)^2}\Vert \hat{a}\Vert_{L^2_{x_3}}^2, \\
    & \Vert |k|^2 \phi_a\Vert_{L^2_{x_3}}^2 + \Vert |k|\p_{x_3}\phi_a\Vert_{L^2_{x_3}}^2 \lesssim \frac{|k|^4}{1+|k|^4} \Vert \hat{a}\Vert_{L^2_{x_3}}^2 \lesssim \frac{|k|^2}{1+|k|^2}\Vert \hat{a}\Vert_{L^2_{x_3}}^2.
\end{align*}

This leads to the estimate that
\begin{align}
    &    \Vert \p_{x_3}^2 \phi_a \Vert_{L^2_{x_3}} \lesssim |k|^2 \Vert \phi_a\Vert_{L^2_{x_3}} + \frac{|k|^2}{1+|k|^2} \Vert \hat{a}\Vert_{L^2_{x_3}} \lesssim \frac{|k|}{\sqrt{1+|k|^2}} \Vert \hat{a}\Vert_{L^2_{x_3}},  \notag  \\
    & \Vert ( |k|+ |k|^2) \phi_a\Vert_{L^2_{x_3}} + \Vert (1+|k|) \p_{x_3}\phi_a\Vert_{L^2_{x_3}} + \Vert \p_{x_3}^2 \phi_a\Vert_{L^2_{x_3}} \lesssim \frac{|k|}{\sqrt{1+|k|^2}} \Vert \hat{a}\Vert_{L^2_{x_3}}.\label{H2_est_a}
\end{align}
By trace theorem, using \eqref{H2_est_a} we conclude that
\begin{align}
    &    | |k| \phi_a(k,\pm 1)| \lesssim \frac{|k|}{\sqrt{1+|k|^2}} \Vert \hat{a}\Vert_{L^2_{x_3}}, \label{L2_trace_a}\\
    & | \p_{x_3} \phi_a(k,\pm 1)| \lesssim    \frac{|k|}{\sqrt{1+|k|^2}} \Vert \hat{a}\Vert_{L^2_{x_3}}. \label{H1_trace_a}
\end{align}

\hide
For $k=0$, we select $\phi_a$ as
\begin{align*}
    &    -\p_{x_3}^2 \phi_a(k,x_3) = \hat{a}(0,x_3), \ x_3\in (-1,1), \\
    & \p_{x_3} \phi_a(k,\pm 1) = 0,  \ \int^1_{-1} \hat{a}(0,x_3) \dd x_3 = 0.
\end{align*}
Then by the Poincar\'e inequality, we obtain the $H^2$ estimate as
\begin{align*}
    &   \Vert \phi_a(0,x_3)\Vert_{H^2_{x_3}}^2 \lesssim \Vert \hat{a}(0,x_3)\Vert_{L^2_{x_3}}^2.
\end{align*}
By trace theorem, we obtain that
\begin{align*}
    &   \Vert \phi_a(0,\pm 1)\Vert_{L^2_{x_3}}^2 \lesssim \Vert \hat{a}(0,x_3)\Vert_{L^2_{x_3}}^2.
\end{align*}

\unhide

We substitute \eqref{test_a} into \eqref{weak_formulation}. We decompose $\hat{f} = \mathbf{P}\hat{f} + (\mathbf{I}-\mathbf{P})\hat{f}$. Then we have
\begin{align*}
    &     \eqref{weak_formulation}_1  =  -\int_0^T \int_{-1}^1 \int_{\mathbb{R}^3} i (v_1 k_1 + v_2 k_2) \hat{a} \mu (|v|^2-10)i (v_1 k_1 + v_2 k_2) \phi_a \dd v \dd x_3 \dd t\\
    & + \int_0^T \int^1_{-1} \int_{\mathbb{R}^3} i(v_1k_1 + v_2k_2) (\mathbf{I}-\mathbf{P})(\hat{f}) \sqrt{\mu} (|v|^2-10) (-iv_1 k_1 -i v_2 k_2 + v_3\p_{x_3}) \phi_a \dd v \dd x_3 \dd t.
\end{align*}
In the first line, the contribution of $v_3\p_{x_3}\phi_a$ and $\hat{\mathbf{b}},\hat{c}$ vanish from the oddness and 
\begin{align*}
    &  \int_{\mathbb{R}^3} v_i^2 (|v|^2-10)(\frac{|v|^2-3}{2}) \mu \dd v = 0.
\end{align*}
Then from $\int_{\mathbb{R}^3} v_i^2 (|v|^2-10)\mu \dd v = -5$, we further have
\begin{align*}
    &     \eqref{weak_formulation}_1 = -5\int_0^T \int_{-1}^1  |k|^2 \phi_a \hat{a}  \dd x_3 \dd t  \\
    &\underbrace{+ \int_0^T\int_{-1}^1 \int_{\mathbb{R}^3} i(\bar{v}\cdot k)(\mathbf{I}-\mathbf{P})\hat{f} \sqrt{\mu}(|v|^2-10) (-i\bar{v}\cdot k + v_3 \p_{x_3})\phi_a \dd v \dd x_3 \dd t}_{E_1}. 
\end{align*}

Next, from the oddness we have
\begin{align*}
    &  \eqref{weak_formulation}_2 = -\int_0^T \int_{-1}^1 \int_{\mathbb{R}^3} v_3  \hat{a}\mu (|v|^2-10) v_3 \p_{x_3}^2 \phi_a \dd v \dd x_3 \dd t\\
    & \underbrace{- \int_0^T \int_{-1}^1 \int_{\mathbb{R}^3} v_3 (\mathbf{I}-\mathbf{P})\hat{f} \sqrt{\mu} (|v|^2-10) (-i\bar{v}\cdot k + v_3 \p_{x_3}) \p_{x_3}\phi_a \dd v \dd x_3 \dd t}_{E_2}\\
    &  = 5\int_0^T \int_{-1}^1 \p_{x_3}^2 \phi_a \hat{a} \dd x_3 \dd t +E_2.
\end{align*}

Then $\eqref{weak_formulation}_1$ and $\eqref{weak_formulation}_2$ combine to be
\begin{align}
    & \eqref{weak_formulation}_1 + \eqref{weak_formulation}_2 = -5\int_0^T \int_{-1}^1 (|k|^2 - \p_{x_3}^2) \phi_a \hat{a} \dd x_3 \dd t + E_1 + E_2  \notag \\
    &= 5\Big\Vert \frac{|k|}{\sqrt{1+|k|^2}} \hat{a} \Big\Vert_{L^2_{T,x_3}}^2 + E_1 + E_2.   \label{est_a_LHS}
\end{align}
Here $E_1+E_2$ corresponds to the contribution of $(\mathbf{I}-\mathbf{P})\hat{f}$, which is bounded as
\begin{align}
    &  |E_1+E_2| \lesssim \Vert (\mathbf{I}-\mathbf{P})\hat{f}\Vert_{L^2_{T,x_3,v}}^2 + o(1)\Big[\Vert |k|^2 \phi_a\Vert_{L^2_{T,x_3}}^2 + \Vert |k| \p_{x_3}\phi_a \Vert_{L^2_{T,x_3}}^2 + \Vert \p_{x_3}^2 \phi_a \Vert_{L^2_{T,x_3}}^2  \Big] \notag \\
    &\lesssim o(1)\Big\Vert \frac{|k|}{\sqrt{1+|k|^2}} \hat{a} \Big\Vert_{L^2_{T,x_3}}^2  + \Vert (\mathbf{I}-\mathbf{P})\hat{f}\Vert_{L^2_{T,x_3,v}}^2.  \label{est_a_RHS_1}
\end{align}
Here we have used \eqref{H2_est_a}.

Then we compute the boundary term $\eqref{weak_formulation}_3$. For the contribution of $P_\gamma \hat{f}$, we have
\begin{align*}
    &   \int_0^T \int_{\mathbb{R}^3} v_3   P_\gamma\hat{f}(k,1) \psi_a(1) \dd v \dd t =  \int_0^T \int_{\mathbb{R}^3} v_3 P_\gamma \hat{f}(k,1) \sqrt{\mu} (|v|^2 -10) (-i\bar{v}\cdot k + v_3 \p_{x_3}) \phi_a  \dd v \dd t = 0.
\end{align*}
Here we have used the oddness to have
\begin{align*}
    &     \int_{\mathbb{R}^3} v_3 \mu (|v|^2 - 10) (-i \bar{v}\cdot k) \phi_a \dd v = 0,
\end{align*}
and the boundary condition $\p_{x_3} \phi_a = 0$ to have
\begin{align*}
    &   \int_{\mathbb{R}^3} v_3^2 \mu (|v|^2-10) \p_{x_3} \phi_a \dd v = 0.
\end{align*}
Thus, to estimate \eqref{weak_formulation}$_3$, from the trace estimates \eqref{L2_trace_a} and \eqref{H1_trace_a}, we derive that for $x_3=1$,
\begin{align}
    &     \Big|\int_0^T \int_{v_3>0} (I-P_\gamma)\hat{f}(k,1) \sqrt{\mu} (|v|^2-10) (-i\bar{v}\cdot k + v_3 \p_{x_3}) \phi_a \dd v \dd t  \Big|  \notag \\
    & \lesssim o(1) [\Vert |k| \phi_a(k,1) \Vert_{L^2_T}^2 +  \Vert \p_{x_3} \phi_a(k,1)\Vert_{L^2_T }^2  ]   +    |(I-P_\gamma)\hat{f}|_{L^2_{T,\gamma_+}}^2 \notag \\
    & \lesssim o(1) \Big\Vert \frac{|k|}{\sqrt{1+|k|^2}} \hat{a} \Big\Vert_{L^2_{T,x_3}}^2 + |(I-P_\gamma)\hat{f}|_{L^2_{T,\gamma_+}}^2.   \notag
\end{align}
Similarly, for $x_3=-1$ we have the same estimate. We conclude that
\begin{align}
    & | \eqref{weak_formulation}_3| \lesssim o(1) \Big\Vert \frac{|k|}{\sqrt{1+|k|^2}} \hat{a} \Big\Vert_{L^2_{T,x_3}}^2 + |(I-P_\gamma)\hat{f}|_{L^2_{T,\gamma_+}}^2.  \label{est_a_RHS_2}
\end{align}

Next, we compute the time derivative term $\eqref{weak_formulation}_4$. For this, we denote $\Phi_a$ as the solution to the elliptic equation
\begin{align*}
\begin{cases}
    &\dis    (|k|^2 - \p_{x_3}^2)\Phi_a(k,x_3) = -\p_t \bar{\hat{a}}(t,k,x_3) \frac{|k|^2}{1+|k|^2}, \ x_3 \in (-1,1) ,\\
    &\dis  \p_{x_3} \Phi_a(k,\pm 1) = 0.    
\end{cases}
\end{align*}
\hide
for $k\neq 0$, and
\begin{align*}
    & -\p_{x_3}^2 \Phi_a(0,x_3) = \p_t \hat{a}(t,0,x_3) , \ x_3 \in (-1,1) ,\\
    &\p_{x_3} \Phi_a(0,\pm 1) = 0, \ \int_{-1}^1 \p_t \hat{a}(t,0,x_3) \dd x_3 = 0.
\end{align*}
\unhide

Integration by part leads to
\begin{align}
    &   \int_0^T \int_{-1}^1 |k|^2 |\Phi_a|^2 \dd x_3 \dd t+ \int_0^T \int_{-1}^1 |\p_{x_3} \Phi_a|^2 \dd x_3 \dd t   = \int_0^T \int_{-1}^1 \frac{|k|^2}{1+|k|^2} \p_t \bar{\hat{a}}(t,k,x_3) \bar{\Phi}_a \dd x_3 \dd t.   \label{energy_p_t_a}
\end{align}
From the conservation of mass
\begin{align*}
    & \p_t \hat{a} + i k_1 \hat{b}_1 + ik_2 \hat{b}_2 + \p_{x_3} \hat{b}_3 = 0,
\end{align*}
we have
\begin{align}
    &    \int_0^T \int_{-1}^1 \frac{|k|^2}{1+|k|^2} \p_t \bar{\hat{a}}(t,k,x_3) \bar{\Phi}_a \dd x_3 \dd t \notag \\
    & = \int_0^T \int_{-1}^1 \frac{|k|^2}{1+|k|^2} \Big[ -(ik_1 \hat{b}_1 + ik_2 \hat{b}_2)\bar{\Phi}_a + \hat{b}_3 \p_{x_3}\bar{\Phi}_a   \Big] \dd x_3 \dd t  - \int_0^T  \frac{|k|^2}{1+|k|^2} \hat{b}_3 \bar{\Phi}_a \Big|_{-1}^1 \dd t. \label{energy_p_t_a_2}
\end{align}
The boundary term can be computed as
\begin{align*}
    & \Big|\int_0^T  \frac{|k|^2}{1+|k|^2} \bar{\Phi}_a \int_{\mathbb{R}^3} v_3 \hat{f} \sqrt{\mu} \dd v \Big|_{-1}^1 \dd t \Big|\\
    & =  \Big|\int_0^T  \frac{|k|^2}{1+|k|^2} \bar{\Phi}_a \Big(\Big[ \int_{v_3> 0} v_3 \sqrt{\mu} (P_\gamma \hat{f} + (I-P_\gamma)\hat{f})  \dd v + \int_{v_3<0} v_3 \sqrt{\mu} P_\gamma \hat{f} \dd v \Big](x_3=1)  \\
    & -   \Big[\int_{v_3< 0} v_3 \sqrt{\mu} (P_\gamma \hat{f} + (I-P_\gamma)\hat{f})  \dd v + \int_{v_3>0} v_3 \sqrt{\mu} P_\gamma \hat{f} \dd v \Big](x_3=-1) \Big) \dd t \Big| \\
    & \lesssim o(1)\Big\Vert \frac{|k|}{\sqrt{1+|k|^2}} \Phi_a(k,\pm 1)\Big\Vert^2_{L^2_T} + \Big|\frac{|k|}{\sqrt{1+|k|^2}} (I-P_\gamma)\hat{f}\Big|^2_{L^2_{T,\gamma_+}} \\
    & \lesssim o(1)\Vert |k|  \Phi_a \Vert^2_{L^2_{T,x_3}} + o(1)\Vert \p_{x_3} \Phi_a\Vert^2_{L^2_{T,x_3}} + \Big|\frac{|k|}{\sqrt{1+|k|^2}} (I-P_\gamma)\hat{f}\Big|^2_{L^2_{T,\gamma_+}}.
\end{align*}
In the third line, the contribution of $P_\gamma \hat{f}$ vanished from the oddness with $v_3$. In the fourth line, we applied the trace theorem.

The other term in \eqref{energy_p_t_a_2} is controlled as
\begin{align*}
    &   \Big| \int_0^T \int_{-1}^1 \frac{|k|^2}{1+|k|^2} \Big[ (ik_1 \hat{b}_1 + ik_2 \hat{b}_2)\bar{\Phi}_a - \hat{b}_3 \p_{x_3}\bar{\Phi}_a   \Big] \dd x_3 \dd t \Big|\\
    & \lesssim o(1) \Vert |k| \Phi_a \Vert_{L^2_{T,x_3}}^2 + o(1) \Vert \p_{x_3}\Phi_a\Vert_{L^2_{T,x_3}}^2 + \Big\Vert  \frac{|k|^2}{1+|k|^2} \hat{\mathbf{b}}\Big\Vert^2_{L^2_{T,x_3}}.
\end{align*}
Plugging the estimates to \eqref{energy_p_t_a}, we obtain
\begin{align}
    &   \Vert |k| \Phi_a \Vert_{L^2_{T,x_3}}^2 +  \Vert \p_{x_3}\Phi_a\Vert_{L^2_{T,x_3}}^2 \lesssim \Big\Vert  \frac{|k|^2}{1+|k|^2} \hat{\mathbf{b}}\Big\Vert^2_{L^2_{T,x_3}} + \Big|\frac{|k|}{\sqrt{1+|k|^2}} (I-P_\gamma)\hat{f}\Big|^2_{L^2_{T,\gamma_+}}.   \label{Phi_a_estimate}
\end{align}

Then we compute $\eqref{weak_formulation}_4$ as
\begin{align}
    &     |\eqref{weak_formulation}_4| = \Big|\int_0^T \int_{-1}^1 \int_{\mathbb{R}^3} \hat{f} \sqrt{\mu} (|v|^2-10) (-i\bar{v}\cdot k + v_3 \p_{x_3}) \Phi_a \dd x_3 \dd v \dd t \Big| \notag\\
    & \lesssim \int_0^T \int_{-1}^1  (|\hat{\mathbf{b}}| + \Vert (\mathbf{I}-\mathbf{P})\hat{f}\Vert_{L^2_v}  ) [|k\Phi_a| + |\p_{x_3}\Phi_a|] \dd x_3 \dd t  \notag\\
    & = \int_0^T \int_{-1}^1 \frac{|k|}{\sqrt{1+|k|^2}}(\hat{\mathbf{b}} + (\mathbf{I}-\mathbf{P})\hat{f})[|k\Phi_a| + |\p_{x_3}\Phi_a|] \frac{\sqrt{1+|k|^2}}{|k|} \dd x_3 \dd t \notag \\
    & \lesssim \Big\Vert  \frac{|k|}{\sqrt{1+|k|^2}} \hat{\mathbf{b}}  \Big\Vert^2_{L^2_{T,x_3}} + \Vert (\mathbf{I}-\mathbf{P})\hat{f}\Vert^2_{L^2_{T,x_3,v}} + \frac{1+|k|^2}{|k|^2} \Vert  |k\Phi_a| + |\p_{x_3}\Phi_a| \Vert_{L^2_{T,x_3}}^2  \notag\\
    & \lesssim \Big\Vert  \frac{|k|}{\sqrt{1+|k|^2}} \hat{\mathbf{b}}  \Big\Vert^2_{L^2_{T,x_3}} + \Vert (\mathbf{I}-\mathbf{P})\hat{f}\Vert^2_{L^2_{T,x_3,v}} + |(I-P_\gamma)\hat{f}|^2_{L^2_{T,\gamma_+}}.   \label{est_a_RHS_3}
\end{align}
In the second line, $\hat{a},\hat{c}$ vanish from oddness. In the last line, we have used \eqref{Phi_a_estimate}.

Next we compute $\eqref{weak_formulation}_5$ as
\begin{align}
    &    |\eqref{weak_formulation}_5| \lesssim o(1)[\Vert k \phi_a \Vert_{L^2_{T,x_3}}^2 + \Vert \p_{x_3} \phi_a \Vert_{L^2_{T,x_3}}^2 ] + \Vert (\mathbf{I}-\mathbf{P})\hat{f}\Vert_{L^2_{T,x_3,v}}^2   \notag \\
    & \lesssim o(1)\Big\Vert \frac{|k|}{\sqrt{1+|k|^2}} \hat{a}  \Big\Vert^2_{L^2_{T,x_3}}  + \Vert (\mathbf{I}-\mathbf{P})\hat{f}\Vert^2_{L^2_{T,x_3,v}}  .   \label{est_a_RHS_4}
\end{align}

Then we compute the contribution of the source term using the same computation in \eqref{minkovski}:
\begin{align}
    & |\eqref{weak_formulation}_6| =  \Big| \int_0^T   \int_{-1}^1 \int_{\mathbb{R}^3} \hat{\Gamma}(\hat{f},\hat{f})\psi_a \dd v \dd x_3 \dd t \Big|\notag \\
    &\lesssim o(1)[\Vert k \phi_a \Vert_{L^2_{T,x_3}}^2 + \Vert \p_{x_3} \phi_a \Vert_{L^2_{T,x_3}}^2 ] + \Big( \int_{\mathbb{R}^2} \Vert \hat{f}(k-\ell)\Vert_{L^\infty_T L^2_{x_3,v}} \Vert \hat{f}(\ell)\Vert_{L^2_T L^\infty_{x_3}L^2_\nu} \dd \ell \Big)^{1/2} \notag\\
    & \lesssim  o(1)\Big\Vert \frac{|k|}{\sqrt{1+|k|^2}} \hat{a}  \Big\Vert^2_{L^2_{T,x_3}} + \Big( \int_{\mathbb{R}^2} \Vert \hat{f}(k-\ell)\Vert_{L^\infty_T L^2_{x_3,v}} \Vert \hat{f}(\ell)\Vert_{L^2_T L^\infty_{x_3}L^2_\nu} \dd \ell \Big)^{2}.   \label{est_a_RHS_5}
\end{align}

Last we compute $\eqref{weak_formulation}_0$ as
\begin{align}
    &     \int_{-1}^1 \int_{\mathbb{R}^3} |\hat{f}(T) \psi_a(T) |\dd v \dd x_3  \notag\\
    & \lesssim \Vert \hat{f}(T) \Vert_{L^\infty_T L^2_{x_3,v}} [\Vert k\phi_a \Vert_{L^\infty_T L^2_{x_3}} + \Vert \p_{x_3}\phi_a \Vert_{L^\infty_T L^2_{x_3}} ]  \lesssim \Vert \hat{f}\Vert_{L^\infty_T L^2_{x_3,v}} \Vert \hat{a}\Vert_{L^\infty_{T}L^2_{x_3}} \lesssim \Vert \hat{f}\Vert_{L^\infty_T L^2_{x_3,v}}^2.  \label{est_a_t}
\end{align}
Similarly,
\begin{align}
    &  \int_{-1}^1 \int_{\mathbb{R}^3} |\hat{f}(0)\psi_a(0)| \dd v \dd x_3 \lesssim \Vert \hat{f}_0\Vert^2_{L^2_{x_3,v}}. \label{est_a_0}
\end{align}

We combine \labelcref{est_a_LHS,est_a_RHS_1,est_a_RHS_2} and \labelcref{est_a_RHS_3,est_a_RHS_4,est_a_RHS_5,est_a_t,est_a_0}  to conclude the estimate of $\hat{a}$: 
\begin{align*}
    &     \Big\Vert \frac{|k|}{\sqrt{1+|k|^2}} \hat{a}  \Big\Vert_{L^2_{T,x_3}} \lesssim \Vert (\mathbf{I}-\mathbf{P})\hat{f}\Vert_{L^2_{T,x_3,v}} + \int_{\mathbb{R}^2} \Vert \hat{f}(k-\ell)\Vert_{L^\infty_T L^2_{x_3,v}} \Vert \hat{f}(\ell)\Vert_{L^2_T L^\infty_{x_3}L^2_\nu} \dd \ell \\
    & + |(I-P_\gamma)\hat{f}|_{L^2_{T,\gamma_+}}+ \Big\Vert \frac{|k|}{\sqrt{1+|k|^2}} \hat{\mathbf{b}} \Big\Vert_{L^2_{T,x_3}} + \Vert \hat{f}\Vert_{L_{T}^\infty L^2_{x_3,v}} + \Vert \hat{f}_0\Vert_{L^2_{x_3,v}} .
\end{align*}
\hide

Note that when $k=0$, we have
\begin{align*}
    &     \Vert  \hat{a}(t,0,x_3)  \Vert^2_{L^2_{T,x_3}} \lesssim \Vert (\mathbf{I}-\mathbf{P})\hat{f}\Vert^2_{L^2_{T,x_3,v}} + \Vert \nu^{-1/2}\hat{\Gamma}(\hat{f},\hat{f})\Vert^2_{L^2_{T,x_3,v}} \\
    & + |(I-P_\gamma)\hat{f}|^2_{L^2_{T,\gamma_+}}+ \Vert  \hat{\mathbf{b}}(t,0,x_3) \Vert^2_{L^2_{T,x_3}} + \Vert \hat{f}\Vert^2_{L_{T}^\infty L^2_{x_3,v}} + \Vert \hat{f}_0\Vert_{L^2_{x_3,v}} .
\end{align*}
\unhide

Further taking integration in $k$, with the Young's convolution inequality and $\Vert f\Vert_{L^2_\nu}\lesssim \Vert wf\Vert_{L^\infty_v}$, we conclude that for some $C_1>0$,
\begin{align}
    &     \Big\Vert \frac{|k|}{\sqrt{1+|k|^2}} \hat{a}  \Big\Vert_{L^1_k L^2_{T,x_3}} \leq C_1\Big[ \Big\Vert \frac{|k|}{\sqrt{1+|k|^2}} \hat{\mathbf{b}} \Big\Vert_{L^1_k L^2_{T,x_3}} + \Vert \hat{f} \Vert_{L^1_k L^\infty_T L^2_{x_3,v}} \Vert w\hat{f}\Vert_{L^1_k L^2_T L^\infty_{x_3,v}} \notag \\
    & + |(I-P_\gamma)\hat{f}|_{L^1_k L^2_{T,\gamma_+}}+  \Vert \hat{f}\Vert_{L^1_k L_{T}^\infty L^2_{x_3,v}}+ \Vert (\mathbf{I}-\mathbf{P})\hat{f}\Vert_{L^1_k L^2_{T,x_3,v}}  + \Vert \hat{f}_0\Vert_{L^1_k L^2_{x_3,v}} \Big].   \label{a_estimate_k}
\end{align}

\medskip 
\textit{Estimate of $\hat{\mathbf{b}}$.}

We choose a test function as
\begin{align}
    & \psi_b =  -\frac{3}{2}\Big(|v_1|^2 - \frac{|v|^2}{3} \Big)\sqrt{\mu} ik_1 \phi_b  - v_1v_2 \sqrt{\mu} ik_2 \phi_b + v_1 v_3 \sqrt{\mu} \p_{x_3} \phi_b. \label{psi_b_k}
\end{align}

We let $\phi_b$ satisfy the elliptic system
\begin{align}
\begin{cases}
        & \dis  [2|k_1|^2 + |k_2|^2 - \p_{x_3}^2] \phi_b = \frac{|k|^2}{1+|k|^2}\bar{\hat{b}}_1 , \\
    &\dis  \phi_b = 0 \text{ when } x_3 = \pm 1. 
\end{cases}\label{phi_b_k}
\end{align}
Multiplying \eqref{phi_b_k} by $\bar{\phi}_b$ and taking integration in $x_3$ we obtain
\begin{align*}
    &   \Vert |k| \phi_b \Vert_{L^2_{x_3}}^2 + \Vert \p_{x_3} \phi_b \Vert_{L^2_{x_3}}^2 \lesssim o(1) \frac{|k|^2}{1+|k|^2}\Vert \phi_b\Vert_{L^2_{x_3}}^2 + \frac{|k|^2}{1+|k|^2}\Vert \hat{b}_1\Vert_{L^2_{x_3}}^2, \\
    & \Vert |k|\phi_b \Vert_{L^2_{x_3}}^2 + \Vert \p_{x_3}\phi\Vert^2_{L^2_{x_3}} \lesssim  \frac{|k|^2}{1+|k|^2} \Vert \hat{b}_1 \Vert^2_{L^2_{x_3}}.
\end{align*}

Multiplying \eqref{phi_b_k} by $|k|^2 \bar{\phi}_b$ we obtain
\begin{align*}
    &  \Vert |k|^2\phi_b \Vert_{L^2_{x_3}}^2 + \Vert |k|\p_{x_3}\phi_b\Vert^2_{L^2_{x_3}} \lesssim o(1) \Vert |k|^2 \phi_b\Vert_{L^2_{x_3}}^2 +  \frac{|k|^4}{(1+|k|^2)^2} \Vert \hat{b}_1\Vert_{L^2_{x_3}}^2.
\end{align*}
Thus we conclude
\begin{align}
    & \Vert (|k|+|k|^2) \phi_b\Vert_{L^2_{x_3}} + \Vert (1+|k|)\p_{x_3}\phi_b\Vert_{L^2_{x_3}} \lesssim  \frac{|k|}{\sqrt{1+|k|^2}}\Vert \hat{b}_1\Vert_{L^2_{x_3}}, \label{phib_l2_k}   \\
    & \Vert \p_{x_3}^2 \phi_b\Vert_{L^2_{x_3}} \lesssim \Vert |k|^2 \phi_b\Vert_{L^2_{x_3}} + \frac{|k|^2}{1+|k|^2}\Vert \hat{b}_1\Vert_{L^2_{x_3}} \lesssim \frac{|k|}{\sqrt{1+|k|^2}}\Vert \hat{b}_1\Vert_{L^2_{x_3}}.    \notag
\end{align}

Further by trace theorem, we have
\begin{align}
    & | |k|\phi_b(k,\pm 1) | \lesssim \frac{|k|}{\sqrt{1+|k|^2}} \Vert \hat{b}_1\Vert_{L^2_{x_3}}, \ |\p_{x_3}\phi_b(k,\pm 1)| \lesssim \frac{|k|}{\sqrt{1+|k|^2}}\Vert \hat{b}_1\Vert_{L^2_{x_3}}. \label{phib_trace_k}
\end{align}

We first compute
\begin{align*}
    & \eqref{weak_formulation}_1 =   \int_0^T \int_{-1}^1 \int_{\mathbb{R}^3} i \bar{v}\cdot k (\hat{\mathbf{b}}\cdot v)\sqrt{\mu} \psi_b  \dd v \dd x_3 \dd t  \underbrace{+\int_0^T \int_{-1}^1 \int_{\mathbb{R}^3} (\mathbf{I}-\mathbf{P})\hat{f} \psi_b \dd v \dd x_3 \dd t}_{E_3}  \\
    &  = \int_0^T \int_{-1}^1 \int_{\mathbb{R}^3} i (v_1k_1 + v_2k_2)(\hat{b}_1 v_1 + \hat{b}_2 v_2 + \hat{b}_3v_3)\sqrt{\mu}\psi_b \dd v \dd x_3 \dd t + E_3 \\
    & = \int_0^T \int_{-1}^1 \int_{\mathbb{R}^3} \Big[ \frac{3}{2}|k_1|^2 |v_1|^2 \hat{b}_1 \Big( |v_1|^2 - \frac{|v|^2}{3}\Big) \mu \phi_b + v_1^2 v_2^2 k_1 k_2 \hat{b}_2 \mu \phi_b + v_1^2 v_2^2 k_2^2 \hat{b}_1 \mu \phi_b \\
    & \ \ \ \ \ \ +\frac{3}{2} \Big(|v_1|^2 - \frac{|v|^2}{3} \Big) v_2^2 k_2 k_1 \hat{b}_2 \mu \phi_b \Big] \dd v \dd x_3 \dd t \\
    &+  \int_0^T \int_{-1}^1 \int_{\mathbb{R}^3} i k_1 \hat{b}_3 v_1^2 v_3^2 \mu \p_{x_3}\phi_b \dd v \dd x_3 \dd t + E_3 \\
    &  = \int_0^T \int_{-1}^1  [2 |k_1|^2 \hat{b}_1 + |k_2|^2 \hat{b}_1] \phi_b \dd x_3 \dd t + ik_1 \hat{b}_3 \phi_b  +  E_3.
\end{align*}
The contribution of $\hat{a},\hat{c}$ vanished from the oddness.

Here, by \eqref{phib_l2_k},
\begin{align}
    &   |E_3| \lesssim o(1)[\Vert |k|^2\phi_b\Vert_{L^2_{T,x_3}}^2 + \Vert |k|\p_{x_3}\phi_b \Vert_{L^2_{T,x_3}}^2] + \Vert (\mathbf{I}-\mathbf{P})\hat{f}\Vert_{L^2_{T,x_3}}^2  \notag\\
    & \lesssim o(1)\Big\Vert \frac{|k|}{\sqrt{1+|k|^2}} \hat{b}_1 \Big\Vert^2_{L^2_{T,x_3}} + \Vert (\mathbf{I}-\mathbf{P})\hat{f}\Vert^2_{L^2_{T,x_3,v}}. \label{est_b_E1_k}
\end{align}

Next, we compute
\begin{align*}
    &    \eqref{weak_formulation}_2 = -\int_0^T \int_{-1}^1 \int_{\mathbb{R}^3} v_3 (\hat{\mathbf{b}}\cdot v)\sqrt{\mu} \p_{x_3}\psi_b \dd v \dd x_3 \dd t \underbrace{-\int_0^T \int_{-1}^1 \int_{\mathbb{R}^3} v_3 (\mathbf{I}-\mathbf{P})\hat{f} \p_{x_3}\psi_b \dd v \dd x_3 \dd t}_{E_4} \\
    & = -\int_0^T \int_{-1}^1 \int_{\mathbb{R}^3} v_1^2 v_3^2 \hat{b}_1  \mu \p_{x_3}^2 \phi_b  \dd v \dd x_3 \dd t + \int_0^T \int_{-1}^1 \int_{\mathbb{R}^3}  \frac{3}{2}ik_1\hat{b}_3 v_3^2\Big( |v_1|^2- \frac{|v|^2}{3}\Big)\mu \dd v \dd x_3 \dd t + E_4 \\
    &= -\int_0^T \int_{-1}^1 \hat{b}_1 \p_{x_3}^2 \phi_b \dd x \dd t - ik_1 \hat{b}_3 \phi_b + E_4.
\end{align*}
The contribution of $\hat{a},\hat{c}$ and $v_2$ vanished from the oddness.

Here, by \eqref{phib_l2_k},
\begin{align}
    & |E_4| \lesssim o(1) [\Vert |k|\p_{x_3}\phi_b \Vert^2_{L^2_{T,x_3}} + \Vert \p_{x_3}^2 \phi_b\Vert^2_{L^2_{T,x_3}}] + \Vert (\mathbf{I}-\mathbf{P})\hat{f}\Vert^2_{L^2_{T,x_3,v}}  \notag \\
        & \lesssim o(1)\Big\Vert \frac{|k|}{\sqrt{1+|k|^2}} \hat{b}_1 \Big\Vert^2_{L^2_{T,x_3}} + \Vert (\mathbf{I}-\mathbf{P})\hat{f}\Vert^2_{L^2_{T,x_3,v}}.\label{est_b_E2_k}
\end{align}
Then we have
\begin{align}
    &     \eqref{weak_formulation}_1 + \eqref{weak_formulation}_2 = \int_0^T \int_{-1}^1 [2|k_1|^2 + |k_2|^2 - \p_{x_3}^2]\phi_b \hat{b}_1 \dd x_3 \dd t + E_3 + E_4  \notag\\
    & = \Big\Vert \frac{|k|}{\sqrt{1+|k|^2}} \hat{b}_1 \Big\Vert^2_{L^2_{T,x_3}} + E_3 + E_4. \label{est_b_LHS_k}
\end{align}

Then we compute the boundary term $\eqref{weak_formulation}_3$. For the contribution of $P_\gamma \hat{f}$, we have
\begin{align*}
    &   \int_0^T \int_{\mathbb{R}^3} v_3   P_\gamma \hat{f}(k,1) \psi_b(1) \dd v \dd t \\
    &=  \int_0^T \int_{\mathbb{R}^3} v_3 P_\gamma \hat{f}(k,1)\Big[ -\frac{3}{2}\Big(|v_1|^2 - \frac{|v|^2}{3} \Big)\sqrt{\mu} ik_1 \phi_b  - v_1v_2 \sqrt{\mu} ik_2 \phi_b + v_1 v_3 \sqrt{\mu} \p_{x_3} \phi_b \Big] \dd v \dd t = 0.
\end{align*}
Here we have used the oddness.

For the part with $(I-P_\gamma)\hat{f}$, we derive that
\begin{align}
    & \Big|\int_0^T \int_{v_3>0} (I-P_\gamma)\hat{f}(k,1) \Big[ -\frac{3}{2}\Big(|v_1|^2 - \frac{|v|^2}{3} \Big)\sqrt{\mu} ik_1 \phi_b  - v_1v_2 \sqrt{\mu} ik_2 \phi_b + v_1 v_3 \sqrt{\mu} \p_{x_3} \phi_b \Big] \dd v \dd t\Big| \notag \\
    & \lesssim o(1) [| |k| \phi_b(k,1) |^2_{L^2_T} +  | \p_{x_3} \phi_b(k,1)|^2_{L^2_T}  ]   +    |(I-P_\gamma)\hat{f}|_{L^2_{T,\gamma_+}}^2 \notag \\
    & \lesssim o(1) \Big\Vert \frac{|k|}{\sqrt{1+|k|^2}}\hat{b}_1 \Big\Vert_{L^2_{T,x_3}}^2 + |(I-P_\gamma)\hat{f}|_{L^2_{T,\gamma_+}}^2.  \notag
\end{align}
In the last line, we have used the trace estimate \eqref{phib_trace_k}.

Similarly, for $x_3=-1$ we have the same estimate. Thus we conclude that
\begin{align}
    &  |\eqref{weak_formulation}_3| \lesssim o(1)\Big\Vert \frac{|k|}{\sqrt{1+|k|^2}} \hat{b}_1 \Big\Vert^2_{L^2_{T,x_3}} + |(I-P_\gamma)\hat{f}|_{L^2_{T,\gamma_+}}^2.   \label{est_b_bdr_k}
\end{align}

Next, we compute the time derivative $\eqref{weak_formulation}_4$. We denote $\Phi_b$ as the solution to the elliptic equation
\begin{align*}
\begin{cases}
        &\dis    (2|k_1|^2 + |k_2|^2 - \p_{x_3}^2) \Phi_b(k,x_3) = \p_t \bar{\hat{b}}_1(t,k,x_3) \frac{|k|^2}{1+|k|^2}, \ x_3 \in (-1,1) ,\\
    & \dis \Phi_b(k,\pm 1) = 0.
\end{cases}
\end{align*}
Integration by part leads to
\begin{align}
    &   \int_0^T \int_{-1}^1 (2|k_1|^2 + |k_2|^2) |\Phi_b|^2 \dd x_3 \dd t+ \int_0^T \int_{-1}^1 |\p_{x_3} \Phi_b|^2 \dd x_3 \dd t   \notag\\
    & = \int_0^T \int_{-1}^1 \frac{|k|^2}{1+|k|^2} \p_t \bar{\hat{b}}_1(t,k,x_3) \bar{\Phi}_b \dd x_3 \dd t.   \label{energy_p_t_b}
\end{align}
Denote $\Theta_{ij}(f):= ((v_iv_j-1)\sqrt{\mu},f)_v$. From the conservation of momentum, we have
\begin{align}
    & \p_t \hat{b}_1 + ik_1 (\hat{a}+2\hat{c}) + i k_1 \Theta_{11}((\mathbf{I}-\mathbf{P})\hat{f}) + i k_2 \Theta_{12}((\mathbf{I}-\mathbf{P})\hat{f}) + \p_{x_3} \Theta_{13}((\mathbf{I}-\mathbf{P})\hat{f}) = 0.   \label{conservation_momentum}
\end{align}
Then \eqref{energy_p_t_b} becomes
\begin{align}
    &    \int_0^T \int_{-1}^1 \frac{|k|^2}{1+|k|^2} \p_t \bar{\hat{b}}_1(t,k,x_3) \bar{\Phi}_b \dd x_3 \dd t \notag \\
    & = \int_0^T \int_{-1}^1 \frac{|k|^2}{1+|k|^2} \Big[ -ik_1 (\hat{a}+2\hat{c} + \Theta_{11}((\mathbf{I}-\mathbf{P})\hat{f}))\bar{\Phi}_b - ik_2 \Theta_{12}((\mathbf{I}-\mathbf{P})\hat{f})\bar{\Phi}_b     \notag\\
    & \ \ \ \ \ \  + \Theta_{13}((\mathbf{I}-\mathbf{P})\hat{f})\p_{x_3}\bar{\Phi}_b \Big] \dd x_3 \dd t - \int_0^T  \frac{|k|^2}{1+|k|^2} \bar{\Phi}_b \Theta_{13}((\mathbf{I}-\mathbf{P})\hat{f}) \Big|_{-1}^1 \dd t. \label{energy_p_t_b_2}
\end{align}
The boundary term vanishes from the boundary condition $\Phi_b(k,\pm 1) = 0$:
\begin{align*}
    & \int_0^T  \frac{|k|^2}{1+|k|^2} \bar{\Phi}_b \Theta_{13}((\mathbf{I}-\mathbf{P})\hat{f}) \Big|_{-1}^1 \dd t = 0.
\end{align*}

The other term in \eqref{energy_p_t_b_2} is controlled as
\begin{align*}
    &   \Big| \int_0^T \int_{-1}^1 \frac{|k|^2}{1+|k|^2} \Big[ -ik_1 (\hat{a}+2\hat{c} + \Theta_{11}((\mathbf{I}-\mathbf{P})\hat{f}))\bar{\Phi}_b - ik_2 \Theta_{12}((\mathbf{I}-\mathbf{P})\hat{f})\bar{\Phi}_b     \notag\\
    & \ \ \ \ \ \  + \Theta_{13}((\mathbf{I}-\mathbf{P})\hat{f})\p_{x_3}\bar{\Phi}_b \Big]\dd x_3 \dd t \Big| \\
    & \lesssim o(1) \Vert |k| \Phi_b \Vert_{L^2_{T,x_3}}^2 + o(1) \Vert \p_{x_3}\Phi_b\Vert_{L^2_{T,x_3}}^2 \\
    &+ \Big\Vert  \frac{|k|^2}{1+|k|^2} \hat{a}\Big\Vert^2_{L^2_{T,x_3}}+ \Big\Vert  \frac{|k|^2}{1+|k|^2} \hat{c}\Big\Vert^2_{L^2_{T,x_3}}+ \Big\Vert  \frac{|k|^2}{1+|k|^2} (\mathbf{I}-\mathbf{P})\hat{f}\Big\Vert^2_{L^2_{T,x_3,v}}.
\end{align*}
Plugging this estimate to \eqref{energy_p_t_b}, we obtain
\begin{align}
    &   \Vert |k| \Phi_b \Vert_{L^2_{T,x_3}}^2 +  \Vert \p_{x_3}\Phi_b\Vert_{L^2_{T,x_3}}^2  \notag \\
    & \lesssim \Big\Vert  \frac{|k|^2}{1+|k|^2} \hat{a}\Big\Vert^2_{L^2_{T,x_3}}+ \Big\Vert  \frac{|k|^2}{1+|k|^2} \hat{c}\Big\Vert^2_{L^2_{T,x_3}}+ \Big\Vert  \frac{|k|^2}{1+|k|^2} (\mathbf{I}-\mathbf{P})\hat{f}\Big\Vert^2_{L^2_{T,x_3,v}}.    \label{Phi_b_estimate}
\end{align}

Then we compute $\eqref{weak_formulation}_4$,  as
\begin{align}
    &    | \eqref{weak_formulation}_4 |= \Big|\int_0^T \int_{-1}^1 \int_{\mathbb{R}^3} \hat{f} \sqrt{\mu} \Big[ -\frac{3}{2}\Big( |v_1|^2 - \frac{|v|^2}{3} \Big) i k_1 - v_1 v_2 i k_2 + v_1 v_3 \p_{x_3} \Big] \Phi_b \dd x_3 \dd v \dd t  \Big|\notag\\
    & \lesssim \int_0^T \int_{-1}^1  \Vert (\mathbf{I}-\mathbf{P})\hat{f}\Vert_{L^2_v}   [|k\Phi_b| + |\p_{x_3}\Phi_b|] \dd x_3 \dd t  \notag\\
    & \lesssim  \Vert (\mathbf{I}-\mathbf{P})\hat{f}\Vert^2_{L^2_{T,x_3,v}} + o(1) \Vert  |k\Phi_b| + |\p_{x_3}\Phi_b| \Vert_{L^2_{T,x_3}}^2  \notag\\
    & \lesssim  \Vert (\mathbf{I}-\mathbf{P})\hat{f}\Vert^2_{L^2_{T,x_3,v}}   + o(1)\Big\Vert  \frac{|k|}{\sqrt{1+|k|^2}} \hat{a}\Big\Vert^2_{L^2_{T,x_3}}+ o(1)\Big\Vert  \frac{|k|}{\sqrt{1+|k|^2}} \hat{c}\Big\Vert^2_{L^2_{T,x_3}} .   \label{est_b_RHS_3}
\end{align}
In the last line, we have used \eqref{Phi_b_estimate}.

Next we compute $\eqref{weak_formulation}_5$, $\eqref{weak_formulation}_6$ and $\eqref{weak_formulation}_0$ as
\begin{align}
    &    |\eqref{weak_formulation}_5 |\lesssim o(1)[\Vert k \phi_b \Vert_{L^2_{T,x_3}}^2 + \Vert \p_{x_3} \phi_b \Vert_{L^2_{T,x_3}}^2 ] + \Vert (\mathbf{I}-\mathbf{P})\hat{f}\Vert_{L^2_{T,x_3,v}}^2   \notag \\
    & \lesssim o(1)\Big\Vert \frac{|k|}{\sqrt{1+|k|^2}} \hat{b}_1  \Big\Vert^2_{L^2_{T,x_3}}  + \Vert (\mathbf{I}-\mathbf{P})\hat{f}\Vert^2_{L^2_{T,x_3,v}},   \label{est_b_RHS_4_k}
\end{align}
\begin{align}
    & |\eqref{weak_formulation}_6 |=  \Big| \int_0^T   \int_{-1}^1 \int_{\mathbb{R}^3} \hat{\Gamma}(\hat{f},\hat{f})\psi_b \dd v \dd x_3 \dd t \Big|\notag \\
    &\lesssim o(1)[\Vert k \phi_b \Vert_{L^2_{T,x_3}}^2 + \Vert \p_{x_3} \phi_b \Vert_{L^2_{T,x_3}}^2 ] + \Big(\int_{\mathbb{R}^2} \Vert \hat{f}(k-\ell)\Vert_{L^\infty_T L^2_{x_3,v}} \Vert \hat{f}(\ell)\Vert_{L^2_T L^\infty_{x_3}L^2_\nu} \dd \ell\Big)^2  \notag\\
    & \lesssim  o(1)\Big\Vert \frac{|k|}{\sqrt{1+|k|^2}} \hat{b}_1  \Big\Vert^2_{L^2_{T,x_3}} + \Big(\int_{\mathbb{R}^2} \Vert \hat{f}(k-\ell)\Vert_{L^\infty_T L^2_{x_3,v}} \Vert \hat{f}(\ell)\Vert_{L^2_T L^\infty_{x_3}L^2_\nu} \dd \ell\Big)^2,   \label{est_b_RHS_5_k}
\end{align}
and
\begin{align}
    &     \int_{-1}^1 \int_{\mathbb{R}^3} |\hat{f}(T) \psi_b(T)| \dd v \dd x_3   \lesssim \Vert \hat{f} \Vert_{L^\infty_T L^2_{x_3,v}} [\Vert k\phi_b \Vert_{L^\infty_T L^2_{x_3}} + \Vert \p_{x_3}\phi_b \Vert_{L^\infty_T L^2_{x_3}} ] \notag \\
    & \lesssim \Vert \hat{f}\Vert_{L^\infty_T L^2_{x_3,v}} \Vert \hat{b}_1\Vert_{L^\infty_{T}L^2_{x_3}} \lesssim \Vert \hat{f}\Vert_{L^\infty_T L^2_{x_3,v}}^2,  \label{est_b_t}\\
    &      \int_{-1}^1 \int_{\mathbb{R}^3} |\hat{f}(0)\psi_b(0)| \dd v \dd x_3 \lesssim \Vert \hat{f}_0\Vert^2_{L^2_{x_3,v}}. \label{est_b_0}
\end{align}

We combine \labelcref{est_b_E1_k,est_b_E2_k,est_b_LHS_k,est_b_bdr_k} and \labelcref{est_b_RHS_3,est_b_RHS_4_k,est_b_RHS_5_k,est_b_t,est_b_0} to conclude the estimate of $\hat{b}_1$: 
\begin{align*}
    &     \Big\Vert \frac{|k|}{\sqrt{1+|k|^2}} \hat{b}_1  \Big\Vert_{L^2_{T,x_3}} \lesssim \Vert (\mathbf{I}-\mathbf{P})\hat{f}\Vert_{L^2_{T,x_3,v}} +  \int_{\mathbb{R}^2} \Vert \hat{f}(k-\ell)\Vert_{L^\infty_T L^2_{x_3,v}} \Vert \hat{f}(\ell)\Vert_{L^2_T L^\infty_{x_3}L^2_\nu} \dd \ell \\
    & + |(I-P_\gamma)\hat{f}|_{L^2_{T,\gamma_+}}+ o(1)\Big\Vert \frac{|k|}{\sqrt{1+|k|^2}} \hat{a} \Big\Vert_{L^2_{T,x_3}}  + o(1)\Big\Vert \frac{|k|}{\sqrt{1+|k|^2}} \hat{c} \Big\Vert_{L^2_{T,x_3}} + \Vert \hat{f}\Vert_{L_{T}^\infty L^2_{x_3,v}} + \Vert \hat{f}_0\Vert_{L^2_{x_3,v}} .
\end{align*}
Further taking integration in $k$, with the same computation in \eqref{a_estimate_k} we conclude that
\begin{align*}
    &     \Big\Vert \frac{|k|}{\sqrt{1+|k|^2}} \hat{b}_1  \Big\Vert_{L^1_k L^2_{T,x_3}} \lesssim \Vert (\mathbf{I}-\mathbf{P})\hat{f}\Vert_{L^1_k L^2_{T,x_3,v}} + \Vert \hat{f} \Vert_{L^1_k L^\infty_T L^2_{x_3,v}} \Vert w\hat{f}\Vert_{L^1_k L^2_T L^\infty_{x_3,v}}   + |(I-P_\gamma)\hat{f}|_{L^1_k L^2_{T,\gamma_+}} \\
    &+ o(1)\Big\Vert \frac{|k|}{\sqrt{1+|k|^2}} \hat{a} \Big\Vert_{L^1_k L^2_{T,x_3}}+ o(1)\Big\Vert \frac{|k|}{\sqrt{1+|k|^2}} \hat{c} \Big\Vert_{L^1_k L^2_{T,x_3}} + \Vert \hat{f}\Vert_{L^1_k L_{T}^\infty L^2_{x_3,v}} + \Vert \hat{f}_0\Vert_{L^1_k L^2_{x_3,v}} .
\end{align*}

For $\hat{b}_2$, we choose the test function as
\begin{align*}
\begin{cases}
      &\dis  \psi_b = -v_1v_2 \sqrt{\mu} ik_1 \phi_b - \frac{3}{2}\Big( |v_2|^2 -\frac{|v|^2}{3} \Big) \sqrt{\mu} ik_2 \phi_b + v_2 v_3 \sqrt{\mu} \p_{x_3} \phi_b, \\
    &\dis  [|k_1|^2 + 2|k_2|^2 - \p_{x_3}^2]\phi_b = \frac{|k|^2}{1+|k|^2}\bar{\hat{b}}_2, \\
    &\dis  \phi_b = 0 \text{ when } x_3 = \pm 1.  
\end{cases}
\end{align*}
The estimate for $\hat{b}_2$ can be done by the same computation.

For $\hat{b}_3$, we choose the test function as
\begin{align*}
\begin{cases}
        &\dis \psi_b = -v_1 v_3 \sqrt{\mu} i k_1\phi_b - v_2 v_3 \sqrt{\mu} i k_2 \phi_b + \frac{3}{2} \Big( |v_3|^2 - \frac{|v|^2}{3} \Big) \sqrt{\mu} \p_{x_3} \phi_b, \\
    &\dis  [|k_1|^2 + |k_2|^2 - 2\p_{x_3}^2]\phi_b = \frac{|k|^2}{1+|k|^2} \bar{\hat{b}}_3, \\
    &\dis  \phi_b = 0 \text{ when } x_3 = \pm 1.
\end{cases}
\end{align*}
The difference of the estimate of $\hat{b}_3$ lies in the control of $\eqref{weak_formulation}_4$, since we have a different representation of the conservation law. We only compute this term, and the computation for the other terms can be done in the same manner.

We let $\Phi_b$ satisfy the elliptic equation
\begin{align*}
\begin{cases}
       & \dis   (|k|^2 - 2\p_{x_3}^2) \Phi_b(k,x_3) = \p_t \bar{\hat{b}}_3(t,k,x_3) \frac{|k|^2}{1+|k|^2}, \ x_3 \in (-1,1) ,\\
    &\dis \Phi_b(k,\pm 1) = 0. 
\end{cases}
\end{align*}
Integration by part leads to
\begin{align}
    &   \int_0^T \int_{-1}^1 |k|^2 |\Phi_b|^2 \dd x_3 \dd t+ 2\int_0^T \int_{-1}^1 |\p_{x_3} \Phi_b|^2 \dd x_3 \dd t   \notag\\
    & = \int_0^T \int_{-1}^1 \frac{|k|^2}{1+|k|^2} \p_t \bar{\hat{b}}_3(t,k,x_3) \bar{\Phi}_b \dd x_3 \dd t. \label{energy_p_t_b3}  
\end{align}
The conservation of momentum in $\hat{b}_3$ behaves different to $\hat{b}_1$ and $\hat{b}_2$:
\begin{align}
    & \p_t \hat{b}_3 + \p_{x_3} (\hat{a}+2\hat{c}) + i k_1 \Theta_{31}((\mathbf{I}-\mathbf{P})\hat{f}) + i k_2 \Theta_{32}((\mathbf{I}-\mathbf{P})\hat{f}) + \p_{x_3} \Theta_{33}((\mathbf{I}-\mathbf{P})\hat{f}) = 0.  \label{momentum_conservation}
\end{align}
Then \eqref{energy_p_t_b3} becomes
\begin{align}
    &    \int_0^T \int_{-1}^1 \frac{|k|^2}{1+|k|^2} \p_t \bar{\hat{b}}_3(t,k,x_3) \bar{\Phi}_b \dd x_3 \dd t \notag \\
    & = \int_0^T \int_{-1}^1 \frac{|k|^2}{1+|k|^2} \Big[  (-ik_1 \Theta_{31}((\mathbf{I}-\mathbf{P})\hat{f}))\bar{\Phi}_b - ik_2 \Theta_{32}((\mathbf{I}-\mathbf{P})\hat{f})\bar{\Phi}_b     \notag\\
    &  + [\hat{a} + 2\hat{c} + \Theta_{33}((\mathbf{I}-\mathbf{P})\hat{f})]\p_{x_3}\bar{\Phi}_b \Big] \dd x_3 \dd t - \int_0^T  \frac{|k|^2}{1+|k|^2} \bar{\Phi}_b [\hat{a}+2\hat{c}+\Theta_{33}((\mathbf{I}-\mathbf{P})\hat{f})] \Big|_{-1}^1 \dd t. \label{energy_p_t_b3_2}
\end{align}
The boundary term vanishes from the boundary condition $\Phi_b(k,\pm 1) = 0$:
\begin{align*}
    & \int_0^T  \frac{|k|^2}{1+|k|^2} \bar{\Phi}_b [\hat{a}+2\hat{c}+\Theta_{33}((\mathbf{I}-\mathbf{P})\hat{f})] \Big|_{-1}^1 \dd t = 0.
\end{align*}

The other term in \eqref{energy_p_t_b3_2} is controlled as
\begin{align*}
    &    \int_0^T \int_{-1}^1 \frac{|k|^2}{1+|k|^2} \Big| -ik_1 \Theta_{31}((\mathbf{I}-\mathbf{P})\hat{f})\bar{\Phi}_b - ik_2 \Theta_{32}((\mathbf{I}-\mathbf{P})\hat{f})\bar{\Phi}_b     \notag\\
    & \ \ \ \ \ \  + [ \hat{a} + \hat{c}  + \Theta_{33}((\mathbf{I}-\mathbf{P})\hat{f})]\p_{x_3}\bar{\Phi}_b \Big|\dd x_3 \dd t \\
    & \lesssim o(1) \Vert |k| \Phi_b \Vert_{L^2_{T,x_3}}^2 + o(1) \Vert \p_{x_3}\Phi_b\Vert_{L^2_{T,x_3}}^2 \\
    &+ \Big\Vert  \frac{|k|^2}{1+|k|^2} \hat{a}\Big\Vert^2_{L^2_{T,x_3}}+ \Big\Vert  \frac{|k|^2}{1+|k|^2} \hat{c}\Big\Vert^2_{L^2_{T,x_3}}+ \Big\Vert  \frac{|k|^2}{1+|k|^2} (\mathbf{I}-\mathbf{P})\hat{f}\Big\Vert^2_{L^2_{T,x_3,v}}.
\end{align*}
Plugging the estimates to \eqref{energy_p_t_b3}, we obtain
\begin{align}
    &   \Vert |k| \Phi_b \Vert_{L^2_{T,x_3}}^2 +  \Vert \p_{x_3}\Phi_b\Vert_{L^2_{T,x_3}}^2 \lesssim \Big|\frac{|k|}{\sqrt{1+|k|^2}} (I-P_\gamma)\hat{f}\Big|^2_{L^2_{T,\gamma_+}} \notag \\
    &+ \Big\Vert  \frac{|k|^2}{1+|k|^2} \hat{a}\Big\Vert^2_{L^2_{T,x_3}}+ \Big\Vert  \frac{|k|^2}{1+|k|^2} \hat{c}\Big\Vert^2_{L^2_{T,x_3}}+ \Big\Vert  \frac{|k|^2}{1+|k|^2} (\mathbf{I}-\mathbf{P})\hat{f}\Big\Vert^2_{L^2_{T,x_3,v}}.    \notag
\end{align}
Thus we can derive the same estimate as \eqref{est_b_RHS_3}.

In summary, we obtain the following estimate for $\hat{\mathbf{b}}$. For $0<\delta_2\ll 1$ and $C_2(\delta_2)>1$, it holds that
\begin{align}
    &     \Big\Vert \frac{|k|}{\sqrt{1+|k|^2}} \hat{\mathbf{b}}  \Big\Vert_{L^1_k L^2_{T,x_3}} \leq \delta_2\Big[ \Big\Vert \frac{|k|}{\sqrt{1+|k|^2}} \hat{a} \Big\Vert_{L^1_k L^2_{T,x_3}}+ \Big\Vert \frac{|k|}{\sqrt{1+|k|^2}} \hat{c} \Big\Vert_{L^1_k L^2_{T,x_3}}\Big] \notag \\
    &+ C_2\big[ \Vert (\mathbf{I}-\mathbf{P})\hat{f}\Vert_{L^1_k L^2_{T,x_3,v}} + \Vert \hat{f} \Vert_{L^1_k L^\infty_T L^2_{x_3,v}} \Vert w\hat{f}\Vert_{L^1_k L^2_T L^\infty_{x_3,v}}  \notag \\
    & \ \ \ \ + |(I-P_\gamma)\hat{f}|_{L^1_k L^2_{T,\gamma_+}} + \Vert \hat{f}\Vert_{L^1_k L_{T}^\infty L^2_{x_3,v}} + \Vert \hat{f}_0\Vert_{L^1_k L^2_{x_3,v}}\big] . \label{b_estimate_k}
\end{align}
Note that we can choose $\delta_2>0$ to be arbitrarily small.

\medskip
\textit{Estimate of $\hat{c}$.}

We choose the test function as
\begin{align}
    &\dis \psi_c = (-ik_1v_1 \phi_c - ik_2 v_2 \phi_c + v_3 \p_{x_3}\phi_c)(|v|^2-5)\sqrt{\mu}, \label{psi_c_k}
\end{align}
with $\phi_c$ satisfying
\begin{align}
\begin{cases}
    &\dis   |k|^2 \phi_c - \p_{x_3}^2 \phi_c = \frac{|k|^2}{1+|k|^2}\bar{\hat{c}},  \\
    & \phi_c = 0 \ \text{ when } x_3 = \pm 1.   
\end{cases}\label{phi_c_k}
\end{align}

Multiplying \eqref{phi_c_k} by $\bar{\phi}_c$ and taking integration in $x_3$ we obtain
\begin{align*}
    &   \Vert |k| \phi_c \Vert_{L^2_{x_3}}^2 + \Vert \p_{x_3} \phi_c \Vert_{L^2_{x_3}}^2 \lesssim o(1) \frac{|k|^2}{1+|k|^2}\Vert \phi_c\Vert_{L^2_{x_3}}^2 + \frac{|k|^2}{1+|k|^2}\Vert \hat{c}\Vert_{L^2_{x_3}}^2, \\
    & \Vert |k|\phi_c \Vert_{L^2_{x_3}}^2 + \Vert \p_{x_3}\phi_c\Vert^2_{L^2_{x_3}} \lesssim  \frac{|k|^2}{1+|k|^2} \Vert \hat{c} \Vert^2_{L^2_{x_3}}.
\end{align*}

Multiplying \eqref{phi_c_k} by $|k|^2 \bar{\phi}_c$ we obtain
\begin{align*}
    &  \Vert |k|^2\phi_c \Vert_{L^2_{x_3}}^2 + \Vert |k|\p_{x_3}\phi_c\Vert^2_{L^2_{x_3}} \lesssim o(1) \Vert |k|^2 \phi_c\Vert_{L^2_{x_3}}^2 +  \frac{|k|^4}{(1+|k|^2)^2} \Vert \hat{c}\Vert_{L^2_{x_3}}^2.
\end{align*}
Thus we conclude
\begin{align}
    & \Vert (|k|+|k|^2) \phi_c\Vert_{L^2_{x_3}} + \Vert (1+|k|)\p_{x_3}\phi_c\Vert_{L^2_{x_3}} \lesssim  \frac{|k|}{\sqrt{1+|k|^2}}\Vert \hat{c}\Vert_{L^2_{x_3}}, \label{phic_l2_k}   \\
    & \Vert \p_{x_3}^2 \phi_c\Vert_{L^2_{x_3}} \lesssim \Vert |k|^2 \phi\Vert_{L^2_{x_3}} + \frac{|k|^2}{1+|k|^2}\Vert \hat{c}\Vert_{L^2_{x_3}} \lesssim \frac{|k|}{\sqrt{1+|k|^2}}\Vert \hat{c}\Vert_{L^2_{x_3}}.    \notag
\end{align}

Further by trace theorem, we have
\begin{align}
    & | |k|\phi_c(k,\pm 1) | \lesssim \frac{|k|}{\sqrt{1+|k|^2}} \Vert \hat{c}\Vert_{L^2_{x_3}}, \ | \p_{x_3}\phi_c(k,\pm 1)| \lesssim \frac{|k|}{\sqrt{1+|k|^2}}\Vert \hat{c}\Vert_{L^2_{x_3}}. \label{phic_trace_k}
\end{align}

We first compute
\begin{align*}
    & \eqref{weak_formulation}_1  = \int_0^T \int_{-1}^1 \int_{\mathbb{R}^3} i \bar{v}\cdot k \Big( \hat{a}+\hat{c} \frac{|v|^2-3}{2}\Big) \sqrt{\mu} \psi_c \dd v \dd x_3 \dd t \underbrace{+ \int_0^T \int_{-1}^1 \int_{\mathbb{R}^3} i\bar{v}\cdot k (\mathbf{I}-\mathbf{P})\hat{f}\psi_c \dd v \dd x_3 \dd t}_{E_5} \\
    &  = \int_0^T \int_{-1}^1 \int_{\mathbb{R}^3}  (v_1k_1 + v_2k_2) \Big( \hat{a}+\hat{c} \frac{|v|^2-3}{2}\Big) \mu (v_1\phi_c + v_2 \phi_c)(|v|^2-5) \dd v \dd x_3 \dd t + E_5 \\
    & = \int_0^T \int_{-1}^1  5|k|^2 \hat{c}\phi_c \dd x_3 \dd t + E_5   .
\end{align*}
In the first equality, the contribution of $\hat{\mathbf{b}}$ vanished from the oddness. In the second equality, the contribution of $v_3$ in $\psi_c$ vanished from the oddness. In the third equality, we used $\int_{\mathbb{R}^3} v_i^2 \frac{|v|^2-3}{2}(|v|^2-5)\mu \dd v = 5$, and the contribution of $\hat{a}$ vanished by the orthogonality,
\begin{align*}
    &   \int_{\mathbb{R}^3} v_i^2 (|v|^2-5)\mu \dd v = 0, \ i=1,2,3.
\end{align*}

Note that $E_5$ corresponds to the contribution of $(\mathbf{I}-\mathbf{P})\hat{f}$. By \eqref{phic_l2_k}, it holds that
\begin{align}
    &   |E_5| \lesssim o(1)[\Vert |k|^2\phi_c\Vert_{L^2_{T,x_3}}^2 + \Vert |k|\p_{x_3}\phi_c \Vert_{L^2_{T,x_3}}^2] + \Vert (\mathbf{I}-\mathbf{P})\hat{f}\Vert_{L^2_{T,x_3,v}}^2  \notag\\
    & \lesssim o(1)\Big\Vert \frac{|k|}{\sqrt{1+|k|^2}}\hat{c}\Big\Vert^2_{L^2_{T,x_3}} + \Vert (\mathbf{I}-\mathbf{P})\hat{f}\Vert^2_{L^2_{T,x_3,v}}. \label{est_c_E1_k}
\end{align}

Next, we compute
\begin{align*}
    &    \eqref{weak_formulation}_2 = -\int_0^T \int_{-1}^1 \int_{\mathbb{R}^3} v_3 \Big(\hat{a} + \hat{c}\frac{|v|^2-3}{2} \Big)\sqrt{\mu} \p_{x_3}\psi_c \dd v \dd x_3 \dd t \underbrace{-\int_0^T \int_{-1}^1 \int_{\mathbb{R}^3} v_3 (\mathbf{I}-\mathbf{P})\hat{f} \p_{x_3}\psi_c \dd v \dd x_3 \dd t}_{E_6} \\
    & = -\int_0^T \int_{-1}^1 \int_{\mathbb{R}^3} v_3^2 \hat{c}\frac{|v|^2-3}{2} (|v|^2-5) \mu \p_{x_3}^2 \phi_c  \dd v \dd x_3 \dd t + E_6 \\
    &= -5\int_0^T \int_{-1}^1 \hat{c} \p_{x_3}^2 \phi_c \dd x \dd t + E_6.
\end{align*}
Here, by \eqref{phic_l2_k},
\begin{align}
    & |E_6| \lesssim o(1) [\Vert |k|\p_{x_3}\phi_c \Vert^2_{L^2_{T,x_3}} + \Vert \p_{x_3}^2 \phi_c \Vert^2_{L^2_{T,x_3}}] + \Vert (\mathbf{I}-\mathbf{P})\hat{f}\Vert^2_{L^2_{T,x_3,v}}  \notag \\
        & \lesssim o(1) \Big\Vert \frac{|k|}{\sqrt{1+|k|^2}}\hat{c}\Big\Vert^2_{L^2_{T,x_3}} + \Vert (\mathbf{I}-\mathbf{P})\hat{f}\Vert^2_{L^2_{T,x_3,v}}.\label{est_c_E2_k}
\end{align}
Then we have
\begin{align}
    &     \eqref{weak_formulation}_1 + \eqref{weak_formulation}_2 = 5\int_0^T \int_{-1}^1 [|k|^2  - \p_{x_3}^2]\phi_c \hat{c} \dd x_3 \dd t + E_5 + E_6  \notag\\
    & = 5\Big\Vert \frac{|k|}{\sqrt{1+|k|^2}}\hat{c} \Big\Vert_{L^2_{T,x_3}}^2 + E_5 + E_6. \label{est_c_LHS_k}
\end{align}

Then we compute the boundary term $\eqref{weak_formulation}_3$. For the contribution of $P_\gamma \hat{f}$, we have
\begin{align*}
    &   \int_0^T \int_{\mathbb{R}^3} v_3   P_\gamma \hat{f}(k,1) \psi_c(1) \dd v \dd t \\
    &=  \int_0^T \int_{\mathbb{R}^3} v_3 P_\gamma \hat{f}(k,1) (-ik_1v_1 \phi_c - ik_2 v_2 \phi_c + v_3 \p_{x_3}\phi_c)(|v|^2-5)\sqrt{\mu}  \dd v \dd t = 0.
\end{align*}
Here we have used the oddness and $\int_{\mathbb{R}^3} v_3^2 (|v|^2-5)\mu \dd v = 0$.

For the part with $(I-P_\gamma)\hat{f}$, we derive that
\begin{align}
    &\int_0^T \int_{v_3>0} |(I-P_\gamma)\hat{f}(k,1) (-ik_1v_1 \phi_c - ik_2 v_2 \phi_c + v_3 \p_{x_3}\phi_c)(|v|^2-5)\sqrt{\mu}| \dd v \dd t \notag \\
    & \lesssim o(1) [\Vert |k| \phi_c(k,1) \Vert_{L^2_{T}}^2 +  \Vert \p_{x_3} \phi_c(k,1)\Vert_{L^2_{T} }^2  ]   +    |(I-P_\gamma)\hat{f}|_{L^2_{T,\gamma_+}}^2 \notag \\
    & \lesssim o(1) \Big\Vert \frac{|k|}{\sqrt{1+|k|^2}}\hat{c} \Big\Vert_{L^2_{T,x_3}}^2 + |(I-P_\gamma)\hat{f}|_{L^2_{T,\gamma_+}}^2.  \notag
\end{align}
In the last line, we have used the trace estimate \eqref{phic_trace_k}.

Similarly, for $x_3=-1$ we have the same estimate. Thus we conclude that
\begin{align}
    &  |\eqref{weak_formulation}_3| \lesssim o(1) \Big\Vert \frac{|k|}{\sqrt{1+|k|^2}}\hat{c} \Big\Vert_{L^2_{T,x_3}}^2 + |(I-P_\gamma)\hat{f}|_{L^2_{T,\gamma_+}}^2.   \label{est_c_bdr_k}
\end{align}

Next, we compute the time derivative $\eqref{weak_formulation}_4$. We denote $\Phi_c$ as the solution to the elliptic equation
\begin{align*}
\begin{cases}
    &\dis    (|k|^2 - \p_{x_3}^2) \Phi_c(k,x_3) = \p_t \bar{\hat{c}}(t,k,x_3) \frac{|k|^2}{1+|k|^2}, \ x_3 \in (-1,1) ,\\
    & \Phi_c(k,\pm 1) = 0.    
\end{cases}
\end{align*}
Integration by part leads to
\begin{align}
    &   \int_0^T \int_{-1}^1 |k|^2 |\Phi_c|^2 \dd x_3 \dd t+ \int_0^T \int_{-1}^1 |\p_{x_3} \Phi_c|^2 \dd x_3 \dd t   \notag\\
    & = \int_0^T \int_{-1}^1 \frac{|k|^2}{1+|k|^2} \p_t \bar{\hat{c}}(t,k,x_3) \bar{\Phi}_c \dd x_3 \dd t.   \label{energy_p_t_c}
\end{align}
Denote $\Lambda_{j}(f):= \frac{1}{10}((|v|^2-5)v_j\sqrt{\mu},f)_v$. From the conservation of energy, we have
\begin{align}
    & \p_t \hat{c} + \frac{1}{3}(ik_1 \hat{b}_1 + ik_2 \hat{b}_2) + \frac{1}{3}\p_{x_3}\hat{b}_3 \notag\\
    &+ \frac{1}{6}\Big(ik_1 \Lambda_1((\mathbf{I}-\mathbf{P})\hat{f}) + ik_2 \Lambda_2((\mathbf{I}-\mathbf{P})\hat{f}) 
+ \p_{x_3} \Lambda_3((\mathbf{I}-\mathbf{P})\hat{f}) \Big) = 0. \label{energy_conservation}
\end{align}
Then \eqref{energy_p_t_c} becomes
\begin{align}
    &    \int_0^T \int_{-1}^1 \frac{|k|^2}{1+|k|^2} \p_t \bar{\hat{c}}(t,k,x_3) \bar{\Phi}_c \dd x_3 \dd t \notag \\
    & = \int_0^T \int_{-1}^1 \frac{|k|^2}{1+|k|^2} \Big[ -\frac{1}{3}i(k_1 \hat{b}_1 + k_2 \hat{b}_2)\bar{\Phi}_c   - \frac{1}{6}i(k_1 \Lambda_1((\mathbf{I}-\mathbf{P})\hat{f})+k_2 \Lambda_2((\mathbf{I}-\mathbf{P})\hat{f}))\bar{\Phi}_c      \notag\\
    & + \frac{1}{3}\hat{b}_3 \p_{x_3}\bar{\Phi}_c  +  \frac{1}{6} \Lambda_3((\mathbf{I}-\mathbf{P})\hat{f})\p_{x_3}\bar{\Phi}_c \Big] \dd x_3 \dd t \notag \\
    &- \int_0^T  \frac{|k|^2}{1+|k|^2} \Big(\frac{1}{3}\hat{b}_3 \bar{\Phi}_c + \frac{1}{6} \Lambda_3((\mathbf{I}-\mathbf{P})\hat{f}) \bar{\Phi}_c\Big) \Big|_{-1}^1 \dd t. \label{energy_p_t_c_2}
\end{align}
The boundary term vanishes from the boundary condition $\Phi_c(\pm 1) = 0$:
\begin{align*}
    & \int_0^T  \frac{|k|^2}{1+|k|^2} \Big(\frac{1}{3}\hat{b}_3 \Phi_c + \frac{1}{6} \Lambda_3((\mathbf{I}-\mathbf{P})\hat{f}) \bar{\Phi}_c\Big) \Big|_{-1}^1 \dd t = 0.
\end{align*}

The other term in \eqref{energy_p_t_c_2} is controlled as
\begin{align*}
    &    \int_0^T \int_{-1}^1 \frac{|k|^2}{1+|k|^2} \Big| -\frac{1}{3}i(k_1 \hat{b}_1 + k_2 \hat{b}_2)\bar{\Phi}_c   - \frac{1}{6}i(k_1 \Lambda_1((\mathbf{I}-\mathbf{P})\hat{f})+k_2 \Lambda_2((\mathbf{I}-\mathbf{P})\hat{f}))\bar{\Phi}_c      \notag\\
    & + \frac{1}{3}\hat{b}_3 \p_{x_3}\bar{\Phi}_c  +  \frac{1}{6} \Lambda_3((\mathbf{I}-\mathbf{P})\hat{f})\p_{x_3}\bar{\Phi}_c \Big| \dd x_3 \dd t \\
    &  \lesssim o(1)\Vert  |k|\Phi_c  \Vert_{L^2_{T,x_3}}^2 + o(1)\Vert \p_{x_3}\Phi_c \Vert_{L^2_{T,x_3}}^2 + \Big\Vert \frac{|k|^2}{1+|k|^2} \hat{\mathbf{b}}\Big\Vert_{L^2_{T,x_3}}^2 + \Big\Vert \frac{|k|^2}{1+|k|^2} (\mathbf{I}-\mathbf{P})\hat{f}\Big\Vert^2_{L^2_{T,x_3,v}}.
\end{align*}
Plugging the estimates to \eqref{energy_p_t_c}, we obtain
\begin{align}
    &   \Vert |k| \Phi_c \Vert_{L^2_{T,x_3}}^2 +  \Vert \p_{x_3}\Phi_c\Vert_{L^2_{T,x_3}}^2  \lesssim \Big\Vert  \frac{|k|^2}{1+|k|^2} \hat{\mathbf{b}}\Big\Vert^2_{L^2_{T,x_3}} + \Big\Vert  \frac{|k|^2}{1+|k|^2} (\mathbf{I}-\mathbf{P})\hat{f}\Big\Vert^2_{L^2_{T,x_3,v}}.    \label{Phi_c_estimate}
\end{align}

Then we compute $\eqref{weak_formulation}_4$ as
\begin{align}
    &     |\eqref{weak_formulation}_4| = \Big|\int_0^T \int_{-1}^1 \int_{\mathbb{R}^3} \hat{f}\sqrt{\mu}(|v|^2-5) (-ik_1v_1  - ik_2 v_2  + v_3 \p_{x_3})\Phi_c \dd x_3 \dd v \dd t \Big| \notag\\
    & \lesssim \int_0^T \int_{-1}^1  \Vert (\mathbf{I}-\mathbf{P})\hat{f}\Vert_{L^2_v}   [|k\Phi_c| + |\p_{x_3}\Phi_c|] \dd x_3 \dd t  \notag\\
    & \lesssim  \Vert (\mathbf{I}-\mathbf{P})\hat{f}\Vert^2_{L^2_{T,x_3,v}} + o(1) \Vert  |k\Phi_c| + |\p_{x_3}\Phi_c| \Vert_{L^2_{T,x_3}}^2  \notag\\
    & \lesssim  \Vert (\mathbf{I}-\mathbf{P})\hat{f}\Vert^2_{L^2_{T,x_3,v}}  + o(1)\Big\Vert  \frac{|k|}{\sqrt{1+|k|^2}} \hat{\mathbf{b}}\Big\Vert^2_{L^2_{T,x_3}} .   \label{est_c_RHS_3}
\end{align}
In the last line, we have used \eqref{Phi_c_estimate}.

Next we compute $\eqref{weak_formulation}_5$, $\eqref{weak_formulation}_6$ and $\eqref{weak_formulation}_0$ as
\begin{align}
    &    |\eqref{weak_formulation}_5| \lesssim o(1)[\Vert k \phi_c \Vert_{L^2_{T,x_3}}^2 + \Vert \p_{x_3} \phi_c \Vert_{L^2_{T,x_3}}^2 ] + \Vert (\mathbf{I}-\mathbf{P})\hat{f}\Vert_{L^2_{T,x_3,v}}^2   \notag \\
    & \lesssim o(1)\Big\Vert \frac{|k|}{\sqrt{1+|k|^2}} \hat{c}  \Big\Vert^2_{L^2_{T,x_3}}  + \Vert (\mathbf{I}-\mathbf{P})\hat{f}\Vert^2_{L^2_{T,x_3,v}},   \label{est_c_RHS_4_k}
\end{align}
\begin{align}
    & |\eqref{weak_formulation}_6| =  \Big| \int_0^T   \int_{-1}^1 \int_{\mathbb{R}^3} \hat{\Gamma}(\hat{f},\hat{f})\psi_c \dd v \dd x_3 \dd t \Big|\notag \\
    &\lesssim  o(1)[\Vert k \phi_c \Vert_{L^2_{T,x_3}}^2 + \Vert \p_{x_3} \phi_c \Vert_{L^2_{T,x_3}}^2 ] +  \Big(\int_{\mathbb{R}^2} \Vert \hat{f}(k-\ell)\Vert_{L^\infty_T L^2_{x_3,v}} \Vert \hat{f}(\ell)\Vert_{L^2_T L^\infty_{x_3}L^2_\nu} \dd \ell\Big)^2   \notag\\
    & \lesssim  o(1)\Big\Vert \frac{|k|}{\sqrt{1+|k|^2}} \hat{c}  \Big\Vert^2_{L^2_{T,x_3}} +  \Big(\int_{\mathbb{R}^2} \Vert \hat{f}(k-\ell)\Vert_{L^\infty_T L^2_{x_3,v}} \Vert \hat{f}(\ell)\Vert_{L^2_T L^\infty_{x_3}L^2_\nu} \dd \ell\Big)^2,   \label{est_c_RHS_5_k}
\end{align}
and
\begin{align}
    &     \int_{-1}^1 \int_{\mathbb{R}^3} |\hat{f}(T) \psi_c(T)| \dd v \dd x_3 \lesssim \Vert \hat{f} \Vert_{L^\infty_T L^2_{x_3,v}} [\Vert k\phi_c \Vert_{L^\infty_T L^2_{x_3}} + \Vert \p_{x_3}\phi_c \Vert_{L^\infty_T L^2_{x_3}} ] \notag \\
    & \lesssim \Vert \hat{f}\Vert_{L^\infty_T L^2_{x_3,v}} \Vert \hat{c}\Vert_{L^\infty_{T}L^2_{x_3}} \lesssim \Vert \hat{f}\Vert_{L^\infty_T L^2_{x_3,v}}^2,  \label{est_c_t} \\
    &\int_{-1}^1 \int_{\mathbb{R}^3} |\hat{f}(0)\psi_c(0)| \dd v \dd x_3 \lesssim \Vert \hat{f}_0\Vert^2_{L^2_{x_3,v}}. \label{est_c_0}
\end{align}

We combine \labelcref{est_c_E1_k,est_c_E2_k,est_c_LHS_k,est_c_bdr_k} and \labelcref{est_c_RHS_3,est_c_RHS_4_k,est_c_RHS_5_k,est_c_t,est_c_0} to conclude the estimate of $\hat{c}$: 
\begin{align*}
    &     \Big\Vert \frac{|k|}{\sqrt{1+|k|^2}} \hat{c}  \Big\Vert_{L^2_{T,x_3}} \lesssim \Vert (\mathbf{I}-\mathbf{P})\hat{f}\Vert_{L^2_{T,x_3,v}} +  \int_{\mathbb{R}^2} \Vert \hat{f}(k-\ell)\Vert_{L^\infty_T L^2_{x_3,v}} \Vert \hat{f}(\ell)\Vert_{L^2_T L^\infty_{x_3}L^2_\nu} \dd \ell \\
    & + |(I-P_\gamma)\hat{f}|_{L^2_{T,\gamma_+}}+ o(1)\Big\Vert \frac{|k|}{\sqrt{1+|k|^2}} \hat{\mathbf{b}} \Big\Vert_{L^2_{T,x_3}}   + \Vert \hat{f}\Vert_{L_{T}^\infty L^2_{x_3,v}} + \Vert \hat{f}_0\Vert_{L^2_{x_3,v}} .
\end{align*}
Further taking integration in $k$, by the same computation in \eqref{a_estimate_k} we conclude that for some $C_3>1$ and $\delta_3\ll 1$, 
\begin{align}
    &     \Big\Vert \frac{|k|}{\sqrt{1+|k|^2}} \hat{c}  \Big\Vert_{L^1_k L^2_{T,x_3}} \leq \delta_3 \Big\Vert \frac{|k|}{\sqrt{1+|k|^2}} \hat{\mathbf{b}} \Big\Vert_{L^1_k L^2_{T,x_3}}  +  C_3 \big[\Vert (\mathbf{I}-\mathbf{P})\hat{f}\Vert_{L^1_k L^2_{T,x_3,v}}  \notag\\
    &+  \Vert \hat{f} \Vert_{L^1_k L^\infty_T L^2_{x_3,v}} \Vert w\hat{f}\Vert_{L^1_k L^2_T L^\infty_{x_3,v}}   + |(I-P_\gamma)\hat{f}|_{L^1_k L^2_{T,\gamma_+}}+ \Vert \hat{f}\Vert_{L^1_k L_{T}^\infty L^2_{x_3,v}} + \Vert \hat{f}_0\Vert_{L^1_k L^2_{x_3,v}} \big] . \label{c_estimate_k}
\end{align}
Note that we can choose arbitrarily small $\delta_3>0$.

\textit{Conclusion.} Now we choose $\delta_2$ in \eqref{b_estimate_k} as $\delta_2 = \frac{1}{|C_1|^2}$. For $\delta_3$ in \eqref{c_estimate_k} we choose $\delta_3 = \frac{1}{|C_1|^2}$. Then we evaluate the  summation as $ \eqref{a_estimate_k} + 2C_1 \times \eqref{b_estimate_k} + \eqref{c_estimate_k} $. This leads to
\begin{align*}
    &  \Big\Vert \frac{|k|}{\sqrt{1+|k|^2}} \hat{a}  \Big\Vert_{L^1_k L^2_{T,x_3}} +  2C_1\Big\Vert \frac{|k|}{\sqrt{1+|k|^2}} \hat{\mathbf{b}}  \Big\Vert_{L^1_k L^2_{T,x_3}} +   \Big\Vert \frac{|k|}{\sqrt{1+|k|^2}} \hat{c}  \Big\Vert_{L^1_k L^2_{T,x_3}} \\
    & \leq C_1  \Big\Vert \frac{|k|}{\sqrt{1+|k|^2}} \hat{\mathbf{b}}  \Big\Vert_{L^1_k L^2_{T,x_3}} + \frac{2}{C_1} \Big[ \Big\Vert \frac{|k|}{\sqrt{1+|k|^2}} \hat{c}  \Big\Vert_{L^1_k L^2_{T,x_3}} +  \Big\Vert \frac{|k|}{\sqrt{1+|k|^2}} \hat{a}  \Big\Vert_{L^1_k L^2_{T,x_3}} \Big] \\
    & + \frac{1}{C_1^2}  \Big\Vert \frac{|k|}{\sqrt{1+|k|^2}} \hat{\mathbf{b}}  \Big\Vert_{L^1_k L^2_{T,x_3}} + C\Big[ \Vert (\mathbf{I}-\mathbf{P})\hat{f}\Vert_{L^1_k L^2_{T,x_3,v}}  \notag\\
    &+  \Vert \hat{f} \Vert_{L^1_k L^\infty_T L^2_{x_3,v}} \Vert w\hat{f}\Vert_{L^1_k L^2_T L^\infty_{x_3,v}}   + |(I-P_\gamma)\hat{f}|_{L^1_k L^2_{T,\gamma_+}}+ \Vert \hat{f}\Vert_{L^1_k L_{T}^\infty L^2_{x_3,v}} + \Vert \hat{f}_0\Vert_{L^1_k L^2_{x_3,v}}   \Big].
\end{align*}
We conclude the estimate in the lemma %for $k\neq 0$ 
by setting $C_1>10$. 
%The case of $k=0$ is addressed in the following remark.
%\Red{HI Hongxu: I still get lost in understanding whether or not it is really necessary to obtain the boundedness of $\|\hat{a}(t,0)\|_{L^2_{T,x_3}}$. In fact, all non-zero-frequency estimates are uniform for $k\neq 0$. To consider the weighted estimate on $|k|\|\hat{a}(t,k)\|_{L^2_{T,x_3}}$ in $L^1_k=L^1(\R^2_k)$, I think there is no need to prove boundedness of $\|\hat{a}(t,0)\|_{L^2_{T,x_3}}$, right? For instance, we know that for $0<s<2$, $\frac{1}{|k|^s}$ is integrable in $|k|\leq 1$ but of course it is unbounded at $k=0$. I mean that when we talk about the $L^1$ integrability, we do not need to care about what happens to the function at $k=0$. On the other hand I think Remark \ref{rmk:k=0} below still should be CORRECT. This only means that we can obtain the boundedness of $\|\hat{a}(t,0)\|_{L^2_{T,x_3}}$ under an extra condition that  $\int_{\O}\int_{\mathbb{R}^3} \sqrt{\mu}f(0,x,v) \dd v \dd x = 0$. Without such condition, we may allow that $\|\hat{a}(t,k)\|_{L^2_{T,x_3}}$ is singular in $0<|k|\leq 1$, so there is no contradiction. This is a feature for the problem in the whole space. I feel that  Remark \ref{rmk:k=0} gives us a hint for how to consider the EXTRA time-decay. Please let me know if you agree with these. Thanks. }
\end{proof}

\hide
\begin{remark}\label{rmk:k=0}
In the case of $0$-frequency, the test function \eqref{elliptic_a}, \eqref{phi_b_k} and \eqref{phi_c_k} are constructed as follows:
\begin{align*}
    & \begin{cases}
        & - \p_{x_3}^2 \phi_a(0,x_3) = \hat{a}(0,x_3), \ x_3\in(-1,1), \\
    & \p_3 \phi_a (0,\pm 1) = 0, \ \int_{-1}^1 \phi_a(0,x_3) \dd x_3 = 0, \ \int_{-1}^{-1} \hat{a}(t,0,x_3) \dd x_3 = 0, 
    \end{cases}  \\
    &  \begin{cases}
       & -\p_{x_3}^2 \phi_b(0,x_3) = \hat{b}_i(0,x_3), \\
    & \phi_b(0, \pm 1) = 0, 
    \end{cases}\\
    & \begin{cases}
       & -\p_{x_3}^2 \phi_c (0,x_3) = \hat{c}(0,x_3),\\
    & \phi_c(0,\pm 1) = 0. 
    \end{cases}
\end{align*}
For the $\phi_a$ we have the compatibility condition from conservation of mass the $0$-average condition on $f_0$ \eqref{average_0}
\begin{align*}
    &   \int_{-1}^{-1} \hat{a}(t,0,x_3) \dd x_3  = \int_{\O} a(t,x_1,x_2,x_3) \dd x  = \int_{\O}\int_{\mathbb{R}^3} \sqrt{\mu}f(t,x,v) \dd v \dd x = \int_{\O}\int_{\mathbb{R}^3} \sqrt{\mu}f(0,x,v) \dd v \dd x = 0.
\end{align*}
Then Poincar\'e inequality holds for all these test functions in $x_3\in (-1,1)$, and we can obtain the $H^2_{x_3}$ estimate for them. Thus we can apply the same test function argument to obtain the macroscopic estimate at $k=0$:
\begin{align*}
    &\Vert \hat{a}(0)\Vert_{L^2_{T,x_3}}+\Vert \hat{\mathbf{b}}(0)\Vert_{L^2_{T,x_3}}+\Vert \hat{c}(0)\Vert_{L^2_{T,x_3}} \lesssim  \Vert (\mathbf{I}-\mathbf{P})\hat{f}(0)\Vert_{L^2_{T,x_3,v}}\\
    & + \Vert \hat{f}(0)\Vert_{L^\infty_T L^2_{x_3,v}}\Vert w\hat{f}(0)\Vert_{L^2_T L^\infty_{x_3,v}} + |(I-P_\gamma)\hat{f}(0)|_{L^2_{T,\gamma_+}} + \Vert \hat{f}(0)\Vert_{L^\infty_T L^2_{x_3,v}} + \Vert \hat{f}_0(0)\Vert_{L^2_{x_3,v}}.
\end{align*}
This control at $k=0$ leads to
\begin{align*}
    &   \Big\Vert \frac{|k|}{\sqrt{1+|k|^2}}(\hat{a},\hat{\mathbf{b}},\hat{c})\Big\Vert_{L^1_k L^2_{T,x_3}} \\
    & \leq  \int_{\mathbb{R}^2} \Big[ \mathbf{1}_{k=0} \Vert (\hat{a},\hat{\mathbf{b}},\hat{c})\Vert_{L^2_{T,x_3}} + \mathbf{1}_{k\neq 0}\Big\Vert  \frac{|k|}{\sqrt{1+|k|^2}}(\hat{a},\hat{\mathbf{b}},\hat{c})\Big\Vert_{ L^2_{T,x_3}} \Big]  \dd k \\
    & \lesssim \int_{\mathbb{R}^2}  \Vert (\mathbf{I}-\mathbf{P})\hat{f}\Vert_{L^2_{T,x_3,v}}  + \Vert \hat{f}\Vert_{L^\infty_T L^2_{x_3,v}}\Vert w\hat{f}\Vert_{L^2_T L^\infty_{x_3,v}} + |(I-P_\gamma)\hat{f}|_{L^2_{T,\gamma_+}} + \Vert \hat{f}\Vert_{L^\infty_T L^2_{x_3,v}} + \Vert \hat{f}_0\Vert_{L^2_{x_3,v}}   \dd k.
\end{align*}
Here the dissipation estimate for $k\neq 0$ is already given in the proof of Lemma \ref{lemma:macroscopic} above. This verifies Lemma \ref{lemma:macroscopic}.
\end{remark}

\unhide

\section{Time decay estimate and $L^1_k L^\infty_{T,x_3,v}$ estimate}\label{sec:time_decay}

In this section, we focus on proving two estimates. The first one is the time-weighted energy estimate in Proposition \ref{prop: full_energy_decay}. We will conclude this estimate in Section \ref{sec:energy_decay}. In Section \ref{sec:refined}, we provide a refined estimate of $\hat{\mathbf{b}},\hat{c}$ by leveraging the Poincar\'e inequality into the proof of the macroscopic dissipation estimate. The second estimate is the term $\Vert w\hat{f}\Vert_{L^1_k L^2_T L^\infty_{x_3,v}}$ in Proposition \ref{prop: full_energy_decay}, where such term originates from the nonlinear estimate \eqref{gamma_est_time}. We utilize the time decay factor in the energy estimate and further obtain a term to be controlled as $\Vert (1+t)^{\sigma/2} w\hat{f}\Vert_{L^1_k L^\infty_{T,x_3,v}}$, see \eqref{time_decay_nonlinear}. In Section \ref{sec:linfty}, we construct estimates in Proposition \ref{prop:linfty} by the method of characteristic in the physical space $x_3\in (-1,1)$. In Section \ref{sec:proof_thm1}, we collect both estimates and conclude Theorem \ref{thm:l1k_lpk}.

\subsection{Energy estimate and macroscopic dissipation estimate with time decay}\label{sec:energy_decay}

In this section, we include the time weight into the energy estimate obtained in Proposition \ref{prop:full energy}. We leverage the $L^p_k$ estimate to control the extra contribution from this time weight. We mainly prove the following result.

\begin{proposition}[\textbf{Energy estimate with time decay}]\label{prop: full_energy_decay}
Let $p>2$ and $\sigma = 2(1-1/p)-2\e$ with $\e>0$ small enough, then under the assumption in Proposition \ref{prop:full energy}, we have
\begin{align*}
    &     \Vert (1+t)^{\sigma/2} \hat{f}\Vert_{L^1_k L^\infty_T L^2_{x_3,v}} + \Vert (1+t)^{\sigma/2} (\mathbf{I}-\mathbf{P}) \hat{f} \Vert_{L^1_k L^2_T L^2_{x_3,\nu}} + |(1+t)^{\sigma/2}(I-P_\gamma)\hat{f}|^2_{L^1_k L^2_{T,\gamma_+}}  \\
    & +  \Big\Vert (1+t)^{\sigma/2} \frac{|k|}{\sqrt{1+|k|^2}}(\hat{a},\hat{\mathbf{b}},\hat{c})\Big\Vert_{L^1_k  L^2_{T,x_3}} \\
    & \lesssim \Vert \hat{f}_0\Vert_{L^1_k L^2_{x_3,v}} + \Vert (1+t)^{\sigma/2} \hat{f}\Vert_{L^1_k L^\infty_T L^2_{x_3,v}} \Vert w \hat{f}\Vert_{L^1_k L^2_T L^\infty_{x_3,v}} \\
    & + \Vert \hat{f}_0\Vert_{L^p_k L^2_{x_3,v}} + \Vert \hat{f}\Vert_{L^p_k L^\infty_T L^2_{x_3,v}} \Vert w \hat{f}\Vert_{L^1_k L^2_T L^\infty_{x_3,v}}. 
\end{align*}    
\end{proposition}

The proof of Proposition \ref{prop: full_energy_decay} follows by combining Lemma \ref{lemma:energy_decay} and Lemma \ref{lemma:macro_time} below, which control $ \Vert (1+t)^{\sigma/2} \hat{f}\Vert_{L^1_k L^\infty_T L^2_{x_3,v}}$ and the macroscopic dissipation $\Big\Vert (1+t)^{\sigma/2} \frac{|k|}{\sqrt{1+|k|^2}}(\hat{a},\hat{\mathbf{b}},\hat{c})\Big\Vert_{L^1_k  L^2_{T,x_3}}$, respectively.

\begin{lemma}\label{lemma:energy_decay}
Let $p>2$ and $\sigma = 2(1-1/p)-2\e$, then under the assumption in Proposition \ref{prop: full_energy_decay}, we have
\begin{align*}
    &    \Vert (1+t)^{\sigma/2} \hat{f}\Vert_{L^1_k L^\infty_T L^2_{x_3,v}} + \Vert (1+t)^{\sigma/2} (\mathbf{I}-\mathbf{P}) \hat{f} \Vert_{L^1_k L^2_T L^2_{x_3,\nu}} + |(1+t)^{\sigma/2}(I-P_\gamma)\hat{f}|^2_{L^1_k L^2_{T,\gamma_+}} \\
& \lesssim \Vert \hat{f}_0\Vert_{L^1_k L^2_{x_3,v}} + o(1) \Big\Vert (1+t)^{\sigma/2} \frac{|k|}{\sqrt{1+|k|^2}}(\hat{a},\hat{\mathbf{b}},\hat{c})\Big\Vert_{L^1_k L^2_{T,x_3}} + \Vert (1+t)^{\sigma/2} \hat{f}\Vert_{L^1_k L^\infty_T L^2_{x_3,v}} \Vert w \hat{f}\Vert_{L^1_k L^2_T L^\infty_{x_3,v}} \\
& + \Vert \hat{f}_0\Vert_{L^p_k L^2_{x_3,v}} + \Vert \hat{f}\Vert_{L^p_k L^\infty_T L^2_{x_3,v}} \Vert w \hat{f}\Vert_{L^1_k L^2_T L^\infty_{x_3,v}}.
\end{align*}

\end{lemma}

\begin{proof}
The equation of $(1+t)^{\sigma}\hat{f}$ satisfies
\begin{align}
    &    \p_t [(1+t)^{\sigma} \hat{f} ] + i\bar{v}\cdot k (1+t)^{\sigma} \hat{f} + v_3 \p_{x_3} [(1+t)^{\sigma} \hat{f}] + \mathcal{L}((1+t)^\sigma \hat{f}) \notag\\
    &= \sigma(1+t)^{\sigma - 1} \hat{f}+  (1+t)^{\sigma} \hat{\Gamma}(\hat{f},\hat{f}). \label{weight_f_energy}
\end{align}
We only need to compute one extra term in the energy estimate:
\begin{align*}
    &  \int_{\mathbb{R}^2} \Big( \int_0^T \int_{-1}^1 \int_{\mathbb{R}^3} (1+t)^{\sigma-1} |\hat{f}|^2 \dd v \dd x_3 \dd t \Big)^{1/2} \dd k.
\end{align*}
First, we consider the case of $|k|\geq 1$. In this case, we bound
\begin{align}
    & (1+t)^{\sigma-1} \lesssim o(1)(1+t)^\sigma + 1 \lesssim o(1)(1+t)^{\sigma}\frac{|k|^2}{1+|k|^2} + \frac{|k|^2}{1+|k|^2}.
    \label{add.pes1}
\end{align}
The contribution of $o(1)(1+t)^\sigma \frac{|k|^2}{1+|k|^2}$ is bounded by
\begin{align*}
    & o(1)\int_{\mathbb{R}^2}  \Big( \int_0^T \int_{-1}^1 \int_{\mathbb{R}^3} (1+t)^{\sigma} \frac{|k|^2}{1+|k|^2}|\hat{f}|^2 \dd v \dd x_3 \dd t  \Big)^{1/2} \dd k \\
    & \lesssim o(1) \Big\Vert (1+t)^{\sigma/2} \frac{|k|}{\sqrt{1+|k|^2}} (\hat{a},\hat{\mathbf{b}},\hat{c})\Big\Vert_{L^1_k L^2_{T,x_3}} + o(1) \Vert (1+t)^{\sigma/2} (\mathbf{I}-\mathbf{P})\hat{f}\Vert_{L^1_k L^2_T L^2_{x_3,\nu}}.
\end{align*}

The contribution of the part $\frac{|k|^2}{1+|k|^2}$ in \eqref{add.pes1} is bounded by
\begin{align}
    &   \int_{\mathbb{R}^2} \Big(\int_0^T \int_{-1}^1 \int_{\mathbb{R}^3} \frac{|k|^2}{1+|k|^2} |\hat{f}|^2 \dd v \dd x_3 \dd t   \Big)^{1/2} \dd k \notag\\
    & \lesssim \Big\Vert  \frac{|k|}{\sqrt{1+|k|^2}} (\hat{a},\hat{\mathbf{b}},\hat{c})\Big\Vert_{L^1_k L^2_{T,x_3}} +  \Vert  (\mathbf{I}-\mathbf{P})\hat{f}\Vert_{L^1_k L^2_T L^2_{x_3,\nu}}.  \label{contribution_1}
\end{align}

Next, we consider the case of $|k|\leq 1$. For this, we apply the interpolation
\begin{align*}
    & (1+t)^{\sigma-1} = (1+t)^{(1-\theta)\sigma} (|k|^2)^{1-\theta} (1+t)^{\theta(\sigma-\eta)} (|k|^2)^{-(1-\theta)} \\
    &\lesssim o(1)(1+t)^\sigma |k|^2 + (1+t)^{\sigma-\eta} (|k|^2)^{-\frac{1-\theta}{\theta}}.
\end{align*}
Here $\eta = 3-\frac{2}{p}-\e>1$, $\theta = \frac{1}{\eta}\in (0,1)$, and thus $\frac{1-\theta}{\theta} = \eta - 1 = \sigma+ \e = 2(1-1/p)-\e$.

The contribution of $o(1)(1+t)^\sigma |k|^2$ is bounded in the same way. For the contribution of $(1+t)^{\sigma - \eta} |k|^{-\frac{2(1-\theta)}{\theta}}$, since $-p'\frac{1-\theta}{\theta} >-2$, we have
\begin{align}
    &   \int_{|k|\leq 1} \Big( \int_0^T \int_{-1}^1 \int_{\mathbb{R}^3} (1+t)^{\sigma - \eta} (|k|^2)^{-\frac{1-\theta}{\theta}} |\hat{f}|^2 \dd v \dd x_3 \dd t\Big)^{1/2} \dd k \notag\\
    & \lesssim \Vert \hat{f}\Vert_{L_k^p L^\infty_T L^2_{x_3,v}} \Big( \int_{|k|\leq 1} |k|^{-p' \frac{1-\theta}{\theta}} \dd k \Big)^{1/p'} \Big(\int_0^T (1+t)^{\sigma-\eta} \dd t\Big)^{1/2} \lesssim \Vert \hat{f}\Vert_{L^p_k L^\infty_T L^2_{x_3,v}}.  \label{contribution_p}
\end{align}

Applying Proposition \ref{prop:full energy} to \eqref{contribution_1} and \eqref{Lpk_energy} to \eqref{contribution_p}, we conclude the lemma.
\end{proof}

\begin{lemma}\label{lemma:macro_time}
Let $p>2$ and $\sigma = 2(1-1/p)-2\e$, then under the assumption in Proposition \ref{prop: full_energy_decay}, we have
\begin{align*}
    &     \Big\Vert (1+t)^{\sigma/2} \frac{|k|}{\sqrt{1+|k|^2}}(\hat{a},\hat{\mathbf{b}},\hat{c}) \Big\Vert_{L^1_k L^2_{T,x_3}} \\
    &\lesssim \Vert (1+t)^{\sigma/2} (\mathbf{I}-\mathbf{P})\hat{f}\Vert_{L^1_k L^2_{T,x_3,v}} + \Vert \hat{f}_0\Vert_{L^1_k L^2_{x_3,v}} + \Vert (1+t)^{\sigma/2} \hat{f}\Vert_{L^1_k L^\infty_T L^2_{x_3,v}} \\
    & +\Vert (1+t)^{\sigma/2}\hat{f}\Vert_{L^1_k L^\infty_T L^2_{x_3,v}} \Vert w\hat{f}\Vert_{L^1_kL^2_T L^\infty_{x_3,v}}  + 
  |(1+t)^{\sigma/2}(I-P_\gamma)\hat{f}|_{L^2_{T,\gamma_+}} \\
  &  + \Vert \hat{f}_0\Vert_{L^p_k L^2_{x_3,v}} + \Vert \hat{f}\Vert_{L^p_k L^\infty_T L^2_{x_3,v}} \Vert w \hat{f}\Vert_{L^1_k L^2_T L^\infty_{x_3,v}}.
\end{align*}

\end{lemma}

\begin{proof}
Note that the equation of $(1+t)^{\sigma}\hat{f}$ is given in \eqref{weight_f_energy}.

First, we estimate $\hat{a}$. Following the proof of Lemma \ref{lemma:macroscopic}, we use the test function $\psi_a$ defined in \eqref{test_a}. The contribution of the nonlinear term is bounded using \eqref{gamma_est_time} as 
\begin{align*}
    &  \Big\Vert  \Big| \int_0^T \int_{-1}^1 \int_{\mathbb{R}^3} (1+t)^\sigma \hat{\Gamma}(\hat{f},\hat{f}) \psi_a   \dd v \dd x_3 \dd t  \Big|^{1/2}\Big\Vert_{L^1_k}\\
    & \lesssim o(1) \Vert (1+t)^{\sigma/2}\psi_a\Vert_{L^1_k L^2_{T,x_3,\nu}} + \Vert (1+t)^{\sigma/2} \hat{f}\Vert_{L^1_k L^\infty_T L^2_{x_3,v}} \Vert w\hat{f}\Vert_{L^1_k L^2_T L^\infty_{x_3,v}}  \\
    & \lesssim o(1) \Big\Vert (1+t)^{\sigma/2}  \frac{|k|}{\sqrt{1+|k|^2}} \hat{a}  \Big\Vert_{L^2_{T,x_3}}^2 + \Vert (1+t)^{\sigma/2} \hat{f}\Vert_{L^1_k L^\infty_T L^2_{x_3,v}} \Vert w\hat{f}\Vert_{L^1_k L^2_T L^\infty_{x_3,v}}.
\end{align*}
Here we have used \eqref{H2_est_a}.

Then we only need to compute one extra term:
\begin{align}
    &   \int_0^T \int_{-1}^1 \int_{\mathbb{R}^3} (1+t)^{\sigma-1} |\hat{f} \psi_a| \dd v \dd x_3 \dd t \notag \\
    & \lesssim \int_0^T \int_{-1}^1  (1+t)^{\sigma-1} (|\hat{\mathbf{b}}| + \Vert (\mathbf{I}-\mathbf{P})\hat{f}\Vert_{L^2_v} )[|k\phi_a| + |\p_{x_3}\phi_a| ] \dd x_3 \dd t . \label{extra_a}
\end{align}

First, we consider the case $|k|\geq 1$. We bound
\begin{align*}
    &  (1+t)^{\sigma-1} \leq (1+t)^\sigma, \    \mathbf{1}_{|k|\geq 1} 1\lesssim \mathbf{1}_{|k|\geq 1}\frac{|k|}{\sqrt{1+|k|^2}}.
\end{align*}
Then by \eqref{H2_est_a}, we have
\begin{align}
    &     \eqref{extra_a} \mathbf{1}_{|k|\geq 1} \lesssim \int_0^T \int_{-1}^1 (1+t)^\sigma \frac{|k|}{\sqrt{1+|k|^2}} (|\hat{\mathbf{b}}| + \Vert (\mathbf{I}-\mathbf{P})\hat{f}\Vert_{L^2_v}) [|k\phi_a| + |\p_{x_3}\phi_a| ] \dd x_3 \dd t  \notag\\
    & \lesssim \Big\Vert  (1+t)^{\sigma/2} \frac{|k|}{\sqrt{1+|k|^2}} |\hat{\mathbf{b}}|  \Big\Vert^2_{L^2_{T,x_3}} + \Vert (1+t)^{\sigma/2} (\mathbf{I}-\mathbf{P})\hat{f}\Vert_{L^2_{T,x_3,v}}^2   \notag\\
    & + o(1)[\Vert (1+t)^{\sigma/2} k \phi_a \Vert^2_{L^2_{T,x_3}} + \Vert (1+t)^{\sigma/2} \p_{x_3}\phi_a\Vert_{L^2_{T,x_3}}^2 ]  \notag\\
    & \lesssim \Big\Vert  (1+t)^{\sigma/2} \frac{|k|}{\sqrt{1+|k|^2}} |\hat{\mathbf{b}}|  \Big\Vert^2_{L^2_{T,x_3}} + \Vert (1+t)^{\sigma/2} (\mathbf{I}-\mathbf{P})\hat{f}\Vert_{L^2_{T,x_3,v}}^2 \notag\\
    &+ o(1)\Big\Vert  (1+t)^{\sigma/2} \frac{|k|}{\sqrt{1+|k|^2}} |\hat{a}|  \Big\Vert^2_{L^2_{T,x_3}}.  \notag
\end{align}

And thus we conclude that
\begin{align}
    &     \int_{\mathbb{R}^2} \mathbf{1}_{|k|\geq 1} \Big(\int_0^T \int_{-1}^1 \int_{\mathbb{R}^3} (1+t)^{\sigma-1} |\hat{f}\psi_a| \dd v \dd x_3 \dd t  \Big)^{1/2} \dd k \notag \\
    & \lesssim  o(1)\Big\Vert  (1+t)^{\sigma/2} \frac{|k|}{\sqrt{1+|k|^2}} |\hat{a}|\Big\Vert_{L^1_k L^2_{T,x_3}} + \Big\Vert (1+t)^{\sigma/2} \frac{|k|}{\sqrt{1+|k|^2}} |\hat{\mathbf{b}}| \Big\Vert_{L^1_k L^2_{T,x_3}}  \notag\\
    &+ \Vert (1+t)^{\sigma/2} (\mathbf{I}-\mathbf{P})\hat{f}\Vert_{L^1_k L^2_{T,x_3,v}}. \label{extra_a_k_geq_1_bdd}
\end{align}

Next we consider the case $|k|<1$ in \eqref{extra_a}. For this, we apply the interpolation
\begin{align}
    & (1+t)^{\sigma-1} = (1+t)^{(1-\theta)\sigma} (|k|^2)^{1-\theta} (1+t)^{\theta(\sigma-\eta)} (|k|^2)^{-(1-\theta)} \notag \\
    &\lesssim o(1)(1+t)^\sigma |k|^2 + (1+t)^{\sigma-\eta} (|k|^2)^{-\frac{1-\theta}{\theta}}.  \label{interpolation}
\end{align}
Here $\eta = 3-\frac{2}{p}-\e>1$, $\theta = \frac{1}{\eta}\in (0,1)$, and thus $\frac{1-\theta}{\theta} = \eta - 1 = \sigma+ \e = 2(1-1/p)-\e$.

For the first term in \eqref{interpolation}, we apply the same computation as 
\eqref{extra_a_k_geq_1_bdd} and obtain a bound as
\begin{align}
    & \int_{\mathbb{R}^2}\mathbf{1}_{|k|\leq 1} \Big(\int_0^T \int_{-1}^1 \int_{\mathbb{R}^3} o(1)(1+t)^\sigma |k|^2 |\hat{f}\psi_a| \dd v \dd x_3 \dd t   \Big)^{1/2} \dd k \notag\\
    & \lesssim  o(1)\Big\Vert  (1+t)^{\sigma/2} \frac{|k|}{\sqrt{1+|k|^2}} |\hat{a}|\Big\Vert_{L^1_k L^2_{T,x_3}} + \Big\Vert (1+t)^{\sigma/2} \frac{|k|}{\sqrt{1+|k|^2}} |\hat{\mathbf{b}}| \Big\Vert_{L^1_k L^2_{T,x_3}} \notag\\
    &+ \Vert (1+t)^{\sigma/2} (\mathbf{I}-\mathbf{P})\hat{f}\Vert_{L^1_k L^2_{T,x_3,v}}.   \label{extra_a_k_leq_1_bdd}
\end{align}

For the second term in \eqref{interpolation}, first, we apply Young's inequality to \eqref{extra_a} and obtain a bound as
\begin{align*}
    &    \int_0^T \int_{-1}^1 (1+t)^{\sigma-1} \big[ |\hat{\mathbf{b}}|^2 + \Vert (\mathbf{I}-\mathbf{P})\hat{f}\Vert^2_{L^2_v} + o(1) (|k \phi_a|^2 + |\p_{x_3}\phi_a|^2)\big] \dd x_3 \dd t .
\end{align*}
Then we further take the $k$ integration to have
\begin{align}
    &   \int_{|k|\leq 1} \Big( \int_0^T \int_{-1}^1 (1+t)^{\sigma - \eta} (|k|^2)^{-\frac{1-\theta}{\theta}} \notag\\
    & \times \big[o(1) (|k \phi_a|^2 + |\p_{x_3}\phi_a|^2) + |\hat{\mathbf{b}}|^2 + \Vert (\mathbf{I}-\mathbf{P})\hat{f}\Vert_{L^2_v}^2 \big] \dd x_3 \dd t\Big)^{1/2} \dd k \notag\\
    & \lesssim  \int_{|k|\leq 1} \Big( \int_0^T  (1+t)^{\sigma - \eta} (|k|^2)^{-\frac{1-\theta}{\theta}}  \big[\Vert \hat{a}\Vert_{L^2_{x_3}}^2 + \Vert \hat{\mathbf{b}}\Vert_{L^2_{x_3}}^2 + \Vert (\mathbf{I}-\mathbf{P})\hat{f}\Vert_{L^2_{x_3,v}}^2  \big] \dd t\Big)^{1/2} \dd k  \notag \\
    & \lesssim \Vert \hat{f}\Vert_{L_k^p L^\infty_T L^2_{x_3,v}} \Big( \int_{|k|\leq 1} |k|^{-p' \frac{1-\theta}{\theta}} \dd k \Big)^{1/p'} \Big(\int_0^T (1+t)^{\sigma-\eta} \dd t\Big)^{1/2} \lesssim \Vert \hat{f}\Vert_{L^p_k L^\infty_T L^2_{x_3,v}}\notag \\
    & \lesssim   \Vert \hat{f}_0\Vert_{L^p_k L^2_{x_3,v}} + \Vert \hat{f}\Vert_{L^p_k L^\infty_T L^2_{x_3,v}} \Vert w \hat{f}\Vert_{L^1_k L^2_T L^\infty_{x_3,v}}. \label{extra_a_k_leq_2_bdd}
\end{align}
In the last line, we applied \eqref{Lpk_energy}. In the fourth line, we applied the H\"older inequality with $\frac{1}{p}+\frac{1}{p'}=1$ and used $-p'\frac{1-\theta}{\theta} = -p'(2\frac{1}{p'}-\e)>-2$. In the third line, we applied \eqref{H2_est_a} to have
\begin{align*}
    &     o(1)\Vert |k|\phi_a \Vert_{L^2_{x_3}} + o(1)\Vert \p_{x_3}\phi_a \Vert_{L^2_{x_3}} \lesssim \frac{|k|}{\sqrt{1+|k|^2}}\Vert \hat{a}\Vert_{L^2_{x_3}} \lesssim \Vert \hat{a}\Vert_{L^2_{x_3}}.
\end{align*}
We combine \eqref{extra_a_k_geq_1_bdd}, \eqref{extra_a_k_leq_1_bdd} and \eqref{extra_a_k_leq_2_bdd} to conclude that
\begin{align}
    &   \int_{\mathbb{R}^2}\Big(  \int_0^T \int_{-1}^1 \int_{\mathbb{R}^3} (1+t)^{\sigma-1} |\hat{f} \psi_a| \dd v \dd x_3 \dd t \Big)^{1/2} \dd k \notag \\
    & \lesssim  o(1)\Big\Vert  (1+t)^{\sigma/2} \frac{|k|}{\sqrt{1+|k|^2}} |\hat{a}|\Big\Vert_{L^1_k L^2_{T,x_3}} + \Big\Vert (1+t)^{\sigma/2} \frac{|k|}{\sqrt{1+|k|^2}} |\hat{\mathbf{b}}| \Big\Vert_{L^1_k L^2_{T,x_3}} \notag\\
    &+ \Vert (1+t)^{\sigma/2} (\mathbf{I}-\mathbf{P})\hat{f}\Vert_{L^1_k L^2_{T,x_3,v}} +  \Vert \hat{f}_0\Vert_{L^p_k L^2_{x_3,v}} + \Vert \hat{f}\Vert_{L^p_k L^\infty_T L^2_{x_3,v}} \Vert w \hat{f}\Vert_{L^1_k L^2_T L^\infty_{x_3,v}}.  \label{extra_a_bdd}  
\end{align}

Then we estimate $\hat{\mathbf{b}}$ and $\hat{c}$. Again, we only need to compute the extra contribution of $(1+t)^{\sigma-1}\hat{f}$ in the weak formulation \eqref{weak_formulation}. From the choice of $\psi_b$ in \eqref{psi_b_k} and $\psi_c$ in \eqref{psi_c_k}, we have
\begin{align}
    &   \int_0^T \int_{-1}^1 \int_{\mathbb{R}^3} (1+t)^{\sigma-1} (|\hat{f} \psi_b| + |\hat{f}\psi_c|) \dd v \dd x_3 \dd t \notag \\
    & \lesssim \int_0^T \int_{-1}^1  (1+t)^{\sigma}  \Vert (\mathbf{I}-\mathbf{P})\hat{f}\Vert_{L^2_v} [|k\phi_b| + |\p_{x_3}\phi_b| + |k\phi_c| + |\p_{x_3}\phi_c| ] \dd x_3 \dd t  \notag \\
    &\lesssim o(1) \big[\Vert (1+t)^{\sigma/2} k \phi_b \Vert^2_{L^2_{T,x_3}} + \Vert (1+t)^{\sigma/2} \p_{x_3}\phi_b\Vert_{L^2_{T,x_3}}  \notag\\
    &  + \Vert (1+t)^{\sigma/2} k \phi_c \Vert^2_{L^2_{T,x_3}} + \Vert (1+t)^{\sigma/2} \p_{x_3}\phi_c\Vert_{L^2_{T,x_3}}    \big] + \Vert (1+t)^{\sigma/2} (\mathbf{I}-\mathbf{P})\hat{f}\Vert_{L^2_{T,x_3,v}}^2 \notag \\
    &\lesssim o(1)\Big\Vert  (1+t)^{\sigma/2} \frac{|k|}{\sqrt{1+|k|^2}} |\hat{\mathbf{b}}|  \Big\Vert^2_{L^2_{T,x_3}} + o(1)\Big\Vert  (1+t)^{\sigma/2} \frac{|k|}{\sqrt{1+|k|^2}} |\hat{c}|  \Big\Vert^2_{L^2_{T,x_3}} \notag \\
    & + \Vert (1+t)^{\sigma/2} (\mathbf{I}-\mathbf{P})\hat{f}\Vert_{L^2_{T,x_3,v}}^2 . \label{extra_b_c}
\end{align}

We combine \eqref{extra_a_bdd} and \eqref{extra_b_c} to conclude the lemma.
\end{proof}

\subsection{Refined macroscopic estimate for $\hat{\mathbf{b}}$ and $\hat{c}$ with time decay}\label{sec:refined}

Since the Poincar\'e inequality holds in our domain $\O = \mathbb{R}^2 \times (-1,1)$, we expect a better control of $\hat{\mathbf{b}}$ and $\hat{c}$ in the low-frequency regime, as they essentially satisfy the elliptic equations, see \cite{guo2004boltzmann}. In the following lemma, we provide the refined estimate of $\hat{\mathbf{b}},\hat{c}$ in order to justify \eqref{b_1_b_2_refined}, \eqref{b3_refined} and \eqref{c_refined} in Theorem \ref{thm:bc_refined}.

\begin{lemma}\label{lemma:b_c_no_time_derivative}
For $i=1,2$, under the assumption in Proposition \ref{prop: full_energy_decay}, we have the time-weighted dissipation estimate for $\hat{b}_i$ as
\begin{align}
    & \Vert (1+t)^{\sigma/2} \hat{b}_i \Vert_{L^1_k L^2_{T,x_3}}   \notag \\
    &\lesssim \Vert (1+t)^{\sigma/2} (\mathbf{I}-\mathbf{P})\hat{f}\Vert_{L^1_k L^2_{T,x_3,v}} + \Vert \hat{f}_0\Vert_{L^1_k L^2_{x_3,v}} + \Vert (1+t)^{\sigma/2} \hat{f}\Vert_{L^1_k L^\infty_T L^2_{x_3,v}} \notag \\
    & +\Vert (1+t)^{\sigma/2}\hat{f}\Vert_{L^1_k L^\infty_T L^2_{x_3,v}} \Vert w\hat{f}\Vert_{L^1_kL^2_T L^\infty_{x_3,v}}  + 
  |(1+t)^{\sigma/2}(I-P_\gamma)\hat{f}|_{L^2_{T,\gamma_+}}  \notag\\
  &  + \Vert \hat{f}_0\Vert_{L^p_k L^2_{x_3,v}} + \Vert \hat{f}\Vert_{L^p_k L^\infty_T L^2_{x_3,v}} \Vert w \hat{f}\Vert_{L^1_k L^2_T L^\infty_{x_3,v}}. \label{b_1_b_2}
\end{align}

For $\hat{b}_3$, we have
\begin{align}
    & \Big\Vert (1+t)^{\sigma/2} \frac{\sqrt{|k|}}{(1+|k|^2)^{1/4}} \hat{b}_3 \Big\Vert_{L^1_k L^2_{T,x_3}}    \notag\\
    &\lesssim \Vert (1+t)^{\sigma/2} (\mathbf{I}-\mathbf{P})\hat{f}\Vert_{L^1_k L^2_{T,x_3,v}} + \Vert \hat{f}_0\Vert_{L^1_k L^2_{x_3,v}} + \Vert (1+t)^{\sigma/2} \hat{f}\Vert_{L^1_k L^\infty_T L^2_{x_3,v}} \notag\\
    & +\Vert (1+t)^{\sigma/2}\hat{f}\Vert_{L^1_k L^\infty_T L^2_{x_3,v}} \Vert w\hat{f}\Vert_{L^1_kL^2_T L^\infty_{x_3,v}}  + 
  |(1+t)^{\sigma/2}(I-P_\gamma)\hat{f}|_{L^2_{T,\gamma_+}} \notag\\
  &  + \Vert \hat{f}_0\Vert_{L^p_k L^2_{x_3,v}} + \Vert \hat{f}\Vert_{L^p_k L^\infty_T L^2_{x_3,v}} \Vert w \hat{f}\Vert_{L^1_k L^2_T L^\infty_{x_3,v}} . \label{b_3}
\end{align}

For $\hat{c}$, we have
\begin{align}
    & \Big\Vert (1+t)^{\sigma/2} \frac{|k|^{1/4}}{(1+|k|^2)^{1/8}} \hat{c} \Big\Vert_{L^1_k L^2_{T,x_3}}    \notag\\
    &\lesssim \Vert (1+t)^{\sigma/2} (\mathbf{I}-\mathbf{P})\hat{f}\Vert_{L^1_k L^2_{T,x_3,v}} + \Vert \hat{f}_0\Vert_{L^1_k L^2_{x_3,v}} + \Vert (1+t)^{\sigma/2} \hat{f}\Vert_{L^1_k L^\infty_T L^2_{x_3,v}} \notag\\
    & +\Vert (1+t)^{\sigma/2}\hat{f}\Vert_{L^1_k L^\infty_T L^2_{x_3,v}} \Vert w\hat{f}\Vert_{L^1_kL^2_T L^\infty_{x_3,v}}  + 
  |(1+t)^{\sigma/2}(I-P_\gamma)\hat{f}|_{L^2_{T,\gamma_+}} \notag\\
  &  + \Vert \hat{f}_0\Vert_{L^p_k L^2_{x_3,v}} + \Vert \hat{f}\Vert_{L^p_k L^\infty_T L^2_{x_3,v}} \Vert w \hat{f}\Vert_{L^1_k L^2_T L^\infty_{x_3,v}} . \label{c_less_k}
\end{align}

\end{lemma}

\begin{proof}

\textit{Proof of \eqref{b_1_b_2}.}

First, we prove \eqref{b_1_b_2}. Without loss of generality, we only prove the case $i=1$. As in the proof of Lemma \ref{lemma:macroscopic}, we use the weak formulation \eqref{weak_formulation}. We choose the same $\psi_b$ as \eqref{psi_b_k}, however, we choose a different $\phi_b$:
\begin{align}
\begin{cases}
    &   [2|k_1|^2 + |k_2|^2 - \p_{x_3}^2] \phi_b = \bar{\hat{b}}_1, \ \ -1<x_3<1,\\
    & \phi_b(k,\pm1) = 0.
    %\text{ at } x_3 = \pm 1.    
\end{cases} \label{phi_b}  
\end{align}
Multiplying \eqref{phi_b} by $\bar{\phi}_b$ and taking integration in $x_3$ we obtain
\begin{align*}
    &   \Vert |k| \phi_b \Vert_{L^2_{x_3}}^2 + \Vert \p_{x_3} \phi_b \Vert_{L^2_{x_3}}^2 \lesssim o(1) \Vert \phi_b\Vert_{L^2_{x_3}}^2 + \Vert \hat{b}_1\Vert_{L^2_{x_3}}^2.
\end{align*}
From the Poincar\'e inequality, we further have
\begin{align}
    & \Vert (1+|k|)\phi_b\Vert^2_{L^2_{x_3}}+ \Vert \p_{x_3}\phi_b\Vert^2_{L^2_{x_3}} \lesssim \Vert \hat{b}_1\Vert^2_{L^2_{x_3}}.  \label{phib_l2_2}
\end{align}
Multiplying \eqref{phi_b} by $|k|^2 \bar{\phi}_b$ we obtain
\begin{align*}
    &  \Vert |k|^2\phi_b \Vert_{L^2_{x_3}}^2 + \Vert |k|\p_{x_3}\phi_b\Vert^2_{L^2_{x_3}} \lesssim o(1) \Vert |k|^2 \phi_b\Vert_{L^2_{x_3}}^2 +  \Vert \hat{b}_1\Vert_{L^2_{x_3}}^2.
\end{align*}
Thus we conclude
\begin{align}
    & \Vert (1+|k|+|k|^2)\phi_b \Vert_{L^2_{x_3}}^2 + \Vert (1+|k|)\p_{x_3}\phi_b\Vert^2_{L^2_{x_3}} \lesssim \Vert \hat{b}_1\Vert^2_{L^2_{x_3}},   \notag \\
    &  \Vert \p_{x_3}^2 \phi_b\Vert_{L^2_{x_3}}^2 \lesssim \Vert |k|^2 \phi_b\Vert_{L^2_{x_3}}^2 + \Vert \hat{b}_1\Vert_{L^2_{x_3}}^2 \lesssim \Vert \hat{b}_1\Vert^2_{L^2_{x_3}}.    \label{phib_h2_2}
\end{align}

By trace theorem, we have
\begin{align}
    & | |k|\phi_b(k,\pm 1) |^2 \lesssim \Vert \hat{b}_1\Vert^2_{L^2_{x_3}}, \ | \p_{x_3}\phi_b(k,\pm 1)|^2 \lesssim \Vert \hat{b}_1\Vert_{L^2_{x_3}}^2. \label{phib_trace_2}
\end{align}

By the same computation of the estimate of $\hat{\mathbf{b}}$ in Lemma \ref{lemma:macroscopic}, we compute
\begin{align*}
    & \eqref{weak_formulation}_1   = \int_0^T \int_{-1}^1  [2 |k_1|^2 \hat{b}_1 + |k_2|^2 \hat{b}_1] \phi_b \dd x_3 \dd t  +  \underbrace{\int_0^T \int_{-1}^1 \int_{\mathbb{R}^3} i \bar{v}\cdot k (\mathbf{I}-\mathbf{P})\hat{f}\psi_b \dd v \dd x_3 \dd t}_{E_1}.
\end{align*}
Here, by \eqref{phib_h2_2},
\begin{align}
    &   |E_1| \lesssim o(1)[\Vert |k|^2\phi_b\Vert_{L^2_{T,x_3}}^2 + \Vert |k|\p_{x_3}\phi_b \Vert_{L^2_{T,x_3}}^2] + \Vert (\mathbf{I}-\mathbf{P})\hat{f}\Vert_{L^2_{T,x_3,v}}^2  \notag\\
    & \lesssim o(1)\Vert \hat{b}_1\Vert^2_{L^2_{T,x_3}} + \Vert (\mathbf{I}-\mathbf{P})\hat{f}\Vert^2_{L^2_{T,x_3,v}}. \label{est_b_E1_2}
\end{align}

Next, by the same computation of Lemma \ref{lemma:macroscopic}, we compute
\begin{align*}
    &    \eqref{weak_formulation}_2 = -\int_0^T \int_{-1}^1 \hat{b}_1 \p_{x_3}^2 \phi_b \dd x \dd t  \underbrace{-\int_0^T \int_{-1}^1 \int_{\mathbb{R}^3} v_3 (\mathbf{I}-\mathbf{P})\hat{f} \p_{x_3}\psi_b \dd v \dd x_3 \dd t}_{E_2}.
\end{align*}
Here, by \eqref{phib_h2_2},
\begin{align}
    & |E_2| \lesssim o(1) [\Vert |k|\p_{x_3}\phi_b \Vert^2_{L^2_{T,x_3}} + \Vert \p_{x_3}^2 \phi_b\Vert^2_{L^2_{T,x_3}}] + \Vert (\mathbf{I}-\mathbf{P})\hat{f}\Vert^2_{L^2_{T,x_3,v}}  \notag \\
        & \lesssim o(1)\Vert \hat{b}_1\Vert^2_{L^2_{T,x_3}} + \Vert (\mathbf{I}-\mathbf{P})\hat{f}\Vert^2_{L^2_{T,x_3,v}}.\label{est_b_E2_2}
\end{align}
Then we have
\begin{align}
    &     \eqref{weak_formulation}_1 + \eqref{weak_formulation}_2 = \int_0^T \int_{-1}^1 [2|k_1|^2 + |k_2|^2 - \p_{x_3}^2]\phi_b \hat{b}_1 \dd x_3 \dd t + E_1 + E_2  \notag\\
    & = \Vert \hat{b}_1\Vert_{L^2_{T,x_3}}^2 + E_1 + E_2. \label{est_b_LHS_2}
\end{align}

Then we compute the boundary term $\eqref{weak_formulation}_3$. For the contribution of $P_\gamma \hat{f}$, by the same computation of Lemma \ref{lemma:macroscopic} we have
\begin{align*}
    &   \int_0^T \int_{\mathbb{R}^3} v_3   P_\gamma \hat{f}(k,1) \psi_b(1) \dd v \dd t = 0.
\end{align*}

For the part with $(I-P_\gamma)\hat{f}$, we derive that
\begin{align}
    & \Big|\int_0^T \int_{v_3>0} (I-P_\gamma)\hat{f}(k,1) \Big[ -\frac{3}{2}\Big(|v_1|^2 - \frac{|v|^2}{3} \Big)\sqrt{\mu} ik_1 \phi_b  - v_1v_2 \sqrt{\mu} ik_2 \phi_b + v_1 v_3 \sqrt{\mu} \p_{x_3} \phi_b \Big] \dd v \dd t\Big| \notag \\
    & \lesssim o(1) [[| |k| \phi_b(k,1) |^2_{L^2_T} +  | \p_{x_3} \phi_b(k,1)|^2_{L^2_T}  ]  ]   +    |(I-P_\gamma)\hat{f}|_{L^2_{T,\gamma_+}}^2  \lesssim o(1) \Vert \hat{b}_1 \Vert_{L^2_{T,x_3}}^2 + |(I-P_\gamma)\hat{f}|_{L^2_{T,\gamma_+}}^2.  \notag
\end{align}
In the last line, we have used the trace estimate \eqref{phib_trace_2}.

Similarly, for $x_3=-1$ we have the same estimate. Thus we conclude that
\begin{align}
    &  |\eqref{weak_formulation}_3 |\lesssim o(1) \Vert \hat{b}_1 \Vert_{L^2_{T,x_3}}^2 + |(I-P_\gamma)\hat{f}|_{L^2_{T,\gamma_+}}^2.   \label{est_b_bdr_2}
\end{align}

Next, we compute the time derivative $\eqref{weak_formulation}_4$. We denote $\Phi_b$ as the solution to the elliptic equation
\begin{align*}
\begin{cases}
    &    (2|k_1|^2 + |k_2|^2 - \p_{x_3}^2) \Phi_b(k,x_3) = \p_t \bar{\hat{b}}_1(t,k,x_3) , \ x_3 \in (-1,1) ,\\
    & \Phi_b(k,\pm 1) = 0.    
\end{cases}
\end{align*}
Integration by part leads to
\begin{align}
    &   \int_0^T \int_{-1}^1 (2|k_1|^2 + |k_2|^2) |\Phi_b|^2 \dd x_3 \dd t+ \int_0^T \int_{-1}^1 |\p_{x_3} \Phi_b|^2 \dd x_3 \dd t   \notag\\
    & = \int_0^T \int_{-1}^1  \p_t \bar{\hat{b}}_1(t,k,x_3) \bar{\Phi}_b \dd x_3 \dd t.   \label{energy_p_t_b_12}
\end{align}
From the conservation of momentum \eqref{conservation_momentum}, \eqref{energy_p_t_b_12} becomes
\begin{align}
    &    \int_0^T \int_{-1}^1  \p_t \bar{\hat{b}}_1(t,k,x_3) \bar{\Phi}_b \dd x_3 \dd t \notag \\
    & = \int_0^T \int_{-1}^1  \Big[ -ik_1 (\hat{a}+2\hat{c} + \Theta_{11}((\mathbf{I}-\mathbf{P})\hat{f}))\bar{\Phi}_b - ik_2 \Theta_{12}((\mathbf{I}-\mathbf{P})\hat{f})\bar{\Phi}_b     \notag\\
    & \ \ \ \ \ \  + \Theta_{13}((\mathbf{I}-\mathbf{P})\hat{f})\p_{x_3}\bar{\Phi}_b \Big] \dd x_3 \dd t - \int_0^T   \bar{\Phi}_b \Theta_{13}((\mathbf{I}-\mathbf{P})\hat{f}) \Big|_{-1}^1 \dd t. \label{energy_p_t_b_22}
\end{align}
The boundary term vanishes from the boundary condition $\bar{\Phi}_b(k,\pm 1) = 0$:
\begin{align*}
    & \int_0^T   \bar{\Phi}_b \Theta_{13}((\mathbf{I}-\mathbf{P})\hat{f}) \Big|_{-1}^1 \dd t = 0.
\end{align*}

The other term in \eqref{energy_p_t_b_22} is controlled as
\begin{align*}
    &    \int_0^T \int_{-1}^1  \Big| -ik_1 (\hat{a}+2\hat{c} + \Theta_{11}((\mathbf{I}-\mathbf{P})\hat{f}))\bar{\Phi}_b - ik_2 \Theta_{12}((\mathbf{I}-\mathbf{P})\hat{f})\bar{\Phi}_b    + \Theta_{13}((\mathbf{I}-\mathbf{P})\hat{f})\p_{x_3}\bar{\Phi}_b \Big|\dd x_3 \dd t \\
    & \lesssim o(1) \Vert (1+|k|) \Phi_b \Vert_{L^2_{T,x_3}}^2 + o(1) \Vert \p_{x_3}\Phi_b\Vert_{L^2_{T,x_3}}^2 + \Vert  |k|^2 \hat{a}\Vert^2_{L^2_{T,x_3}}+ \Vert  |k|^2 \hat{c}\Vert^2_{L^2_{T,x_3}}+ \Vert   (\mathbf{I}-\mathbf{P})\hat{f}\Vert^2_{L^2_{T,x_3,v}}.
\end{align*}
Plugging this estimates to \eqref{energy_p_t_b_12}, with the Poincar\'e inequality $\Vert \Phi_b\Vert_{L^2_{x_3}} \lesssim \Vert \p_{x_3}\Phi_b\Vert_{L^2_{x_3}}$, we obtain
\begin{align}
    &   \Vert (1+|k|) \Phi_b \Vert_{L^2_{T,x_3}}^2 +  \Vert \p_{x_3}\Phi_b\Vert_{L^2_{T,x_3}}^2  \lesssim \Vert  |k|^2 \hat{a}\Vert^2_{L^2_{T,x_3}}+ \Vert  |k|^2 \hat{c}\Vert^2_{L^2_{T,x_3}}+ \Vert   (\mathbf{I}-\mathbf{P})\hat{f}\Vert^2_{L^2_{T,x_3,v}}.    \label{Phi_b_estimate_2}
\end{align}

Then we compute $\eqref{weak_formulation}_4$ as
\begin{align}
    &     |\eqref{weak_formulation}_4 |\leq \int_0^T \int_{-1}^1 \int_{\mathbb{R}^3} \Big| \hat{f} \sqrt{\mu} \Big[ -\frac{3}{2}\Big( |v_1|^2 - \frac{|v|^2}{3} \Big) i k_1 - v_1 v_2 i k_2 + v_1 v_3 \p_{x_3} \Big] \Phi_b \Big| \dd x_3 \dd v \dd t  \notag\\
    & \lesssim \int_0^T \int_{-1}^1  \Vert (\mathbf{I}-\mathbf{P})\hat{f}\Vert_{L^2_v}   [|k\Phi_b| + |\p_{x_3}\Phi_b|] \dd x_3 \dd t  \notag\\
    & \lesssim  \Vert (\mathbf{I}-\mathbf{P})\hat{f}\Vert^2_{L^2_{T,x_3,v}} + o(1) \Vert  |k\Phi_b| + |\p_{x_3}\Phi_b| \Vert_{L^2_{T,x_3}}^2  \notag\\
    & \lesssim  \Vert (\mathbf{I}-\mathbf{P})\hat{f}\Vert^2_{L^2_{T,x_3,v}}   + o(1)\Vert  |k|^2 \hat{a}\Vert^2_{L^2_{T,x_3}}+ o(1)\Vert  |k|^2 \hat{c}\Vert^2_{L^2_{T,x_3}} .   \label{est_b_RHS_3_2}
\end{align}
In the last line, we have used \eqref{Phi_b_estimate_2}.

Next we compute $\eqref{weak_formulation}_5$, $\eqref{weak_formulation}_6$ and $\eqref{weak_formulation}_0$ as
\begin{align}
    &  |  \eqref{weak_formulation}_5 |\lesssim o(1)[\Vert k \phi_b \Vert_{L^2_{T,x_3}}^2 + \Vert \p_{x_3} \phi_b \Vert_{L^2_{T,x_3}}^2 ] + \Vert (\mathbf{I}-\mathbf{P})\hat{f}\Vert_{L^2_{T,x_3,v}}^2   \lesssim o(1)\Vert \hat{b}_1  \Vert^2_{L^2_{T,x_3}}  + \Vert (\mathbf{I}-\mathbf{P})\hat{f}\Vert^2_{L^2_{T,x_3,v}},   \label{est_b_RHS_5_2}
\end{align}
\begin{align}
    & |\eqref{weak_formulation}_6| \leq   \int_0^T   \int_{-1}^1 \int_{\mathbb{R}^3} |\hat{\Gamma}(\hat{f},\hat{f})\psi_b| \dd v \dd x_3 \dd t \notag \\
    &\lesssim  o(1)[\Vert k \phi_b \Vert_{L^2_{T,x_3}}^2 + \Vert \p_{x_3} \phi_b \Vert_{L^2_{T,x_3}}^2 ] + \Big(\int_{\mathbb{R}^2} \Vert \hat{f}(k-\ell)\Vert_{L^\infty_T L^2_{x_3,v}} \Vert \hat{f}(\ell)\Vert_{L^2_T L^\infty_{x_3}L^2_\nu} \dd \ell\Big)^2 \notag \\ 
    &\lesssim  o(1)\Vert \hat{b}_1  \Vert^2_{L^2_{T,x_3}} + \Big(\int_{\mathbb{R}^2} \Vert \hat{f}(k-\ell)\Vert_{L^\infty_T L^2_{x_3,v}} \Vert \hat{f}(\ell)\Vert_{L^2_T L^\infty_{x_3}L^2_\nu} \dd \ell\Big)^2,   \label{est_b_RHS_6_2}
\end{align}
and
\begin{align}
    &     \int_{-1}^1 \int_{\mathbb{R}^3} |\hat{f}(T) \psi_b(T)| \dd v \dd x_3   \lesssim \Vert \hat{f}(t) \Vert_{L^\infty_T L^2_{x_3,v}} [\Vert k\phi_b \Vert_{L^\infty_T L^2_{x_3}} + \Vert \p_{x_3}\phi_b \Vert_{L^\infty_T L^2_{x_3}} ] \notag \\
    & \lesssim \Vert \hat{f}\Vert_{L^\infty_T L^2_{x_3,v}} \Vert \hat{b}_3\Vert_{L^\infty_{T}L^2_{x_3}} \lesssim \Vert \hat{f}\Vert_{L^\infty_T L^2_{x_3,v}}^2,  \label{est_b_t_1}\\
        &  \int_{-1}^1 \int_{\mathbb{R}^3} |\hat{f}(0)\psi_b(0)| \dd v \dd x_3 \lesssim \Vert \hat{f}_0\Vert^2_{L^2_{x_3,v}}. \label{est_b_0_1}
\end{align}

We combine \labelcref{est_b_E1_2,est_b_E2_2,est_b_LHS_2,est_b_bdr_2} and \labelcref{est_b_RHS_3_2,est_b_RHS_5_2,est_b_RHS_6_2,est_b_t_1,est_b_0_1} to conclude the estimate for $\hat{b}_1$:
\begin{align}
    &  \Vert \hat{b}_1\Vert_{L^2_{T,x_3}} \lesssim |(I-P_\gamma)\hat{f}|_{L^2_{T,\gamma_+}} + \Vert (\mathbf{I}-\mathbf{P})\hat{f}\Vert_{L^2_{T,x_3,v}} + \int_{\mathbb{R}^2} \Vert \hat{f}(k-\ell)\Vert_{L^\infty_T L^2_{x_3,v}} \Vert \hat{f}(\ell)\Vert_{L^2_T L^\infty_{x_3}L^2_\nu} \dd \ell\notag \\
    &+ o(1)\Vert |k|^2 \hat{a}\Vert_{L^2_{T,x_3}} + o(1)\Vert |k|^2 \hat{c}\Vert_{L^2_{T,x_3}}  +  \Vert \hat{f}_0\Vert_{L^2_{x_3,v}} + \Vert \hat{f}\Vert_{L^\infty_T L^2_{x_3,v}}   \label{b1_bdd_small_k}   .
\end{align}

Now we take the $k-$integration in the following way:
\begin{align*}
    &    \Vert \hat{b}_1 \Vert_{L^1_k L^2_{T,x_3}} \lesssim \Vert \mathbf{1}_{|k|<1} \hat{b}_1 \Vert_{L^1_k L^2_{T,x_3}} + \Big\Vert \mathbf{1}_{|k|>1} \frac{|k|}{\sqrt{1+|k|^2}} \hat{b}_1 \Big\Vert_{L^1_k L^2_{T,x_3}} \\
    & \lesssim      \Vert  \mathbf{1}_{|k|<1}\eqref{b1_bdd_small_k} \Vert_{L^1_k} + \Big\Vert \frac{|k|}{\sqrt{1+|k|^2}} \hat{b}_1 \Big\Vert_{L^1_k L^2_{T,x_3}}  .
\end{align*}
For the first term, since $|k|<1$, in \eqref{b1_bdd_small_k}, we have
\begin{align*}
    & o(1)\Vert |k|^2 \hat{a}\Vert_{L^2_{T,x_3}} + o(1)\Vert |k|^2 \hat{c}\Vert_{L^2_{T,x_3}} \leq  \Big\Vert \frac{|k|}{\sqrt{1+|k|^2}} (\hat{a},\hat{c})\Big\Vert_{L^1_k L^2_{T,x_3}}.
\end{align*}

Applying Lemma \ref{lemma:macroscopic}, we conclude that
\begin{align*}
    & \Vert \hat{b}_1\Vert_{L^1_k L^2_{T,x_3}} \lesssim \Vert (\mathbf{I}-\mathbf{P})\hat{f}\Vert_{L^1_k L^2_{T,x_3,v}}  \notag\\
    &+ \Vert \hat{f}\Vert_{L^p_k L^\infty_T L^2_{x_3,v}} \Vert w\hat{f}\Vert_{L^1_k L^2_T L^\infty_{x_3,v}} + |(I-P_\gamma)\hat{f}|_{L^1_k L^2_{T,\gamma_+}} + \Vert \hat{f}\Vert_{L^1_k L^\infty_T L^2_{x_3,v}} + \Vert \hat{f}_0\Vert_{L^1_k L^2_{x_3,v}}. 
\end{align*}

To prove \eqref{b_1_b_2}, we include the time weight and use \eqref{weight_f_energy}. Similar to the proof of Lemma \ref{lemma:macro_time}, we only need to compute the contribution of $(1+t)^{\sigma-1}\hat{f}$. By the same computation in \eqref{extra_b_c}, such term is controlled as
\begin{align*}
    &  \int_0^T \int_1^1 \int_{\mathbb{R}^3} (1+t)^{\sigma-1} |\hat{f} \psi_b| \dd v \dd x_3 \dd t \\
    & \lesssim o(1) \Vert (1+t)^{\sigma/2} |\hat{b}_1| \Vert^2_{L^2_{T,x_3}} + \Vert (1+t)^{\sigma/2} (\mathbf{I}-\mathbf{P})\hat{f}\Vert_{L^2_{T,x_3,v}}^2.
\end{align*}
This concludes \eqref{b_1_b_2}.

\medskip
\textit{Proof of \eqref{b_3}.}

We choose 
\begin{align}
\begin{cases}
&\dis \psi_b = -v_1 v_3 \sqrt{\mu} i k_1\phi_b - v_2 v_3 \sqrt{\mu} i k_2 \phi_b + \frac{3}{2} \Big( |v_3|^2 - \frac{|v|^2}{3} \Big) \sqrt{\mu} \p_{x_3} \phi_b, \\
    &\dis    [|k_1|^2 + |k_2|^2 - 2\p_{x_3}^2] \phi_b = \frac{|k|}{\sqrt{1+|k|^2}}\bar{\hat{b}}_3, \\
    & \phi_b = 0 \text{ when } x_3 = \pm 1.   
\end{cases}\label{phi_b_3}    
\end{align}
Multiplying \eqref{phi_b_3} by $\bar{\phi}_b$ and taking integration in $x_3$ we obtain
\begin{align*}
    &\dis   \Vert |k| \phi_b \Vert_{L^2_{x_3}}^2 + \Vert \p_{x_3} \phi_b \Vert_{L^2_{x_3}}^2 \lesssim o(1) \Vert \phi_b\Vert_{L^2_{x_3}}^2 + \Big\Vert \frac{|k|}{\sqrt{1+|k|^2}}\hat{b}_3  \Big\Vert_{L^2_{x_3}}^2.
\end{align*}
From the Poincar\'e inequality, we further have
\begin{align}
    & \Vert (1+|k|)\phi_b\Vert^2_{L^2_{x_3}}+ \Vert \p_{x_3}\phi_b\Vert^2_{L^2_{x_3}} \lesssim \Big\Vert  \frac{\sqrt{|k|}}{(1+|k|^2)^{1/4}} \hat{b}_3 \Big\Vert^2_{L^2_{x_3}}.  \notag
\end{align}
Multiplying \eqref{phi_b_3} by $|k|^2 \bar{\phi}_b$ we obtain
\begin{align*}
    &  \Vert |k|^2\phi_b \Vert_{L^2_{x_3}}^2 + \Vert |k|\p_{x_3}\phi_b\Vert^2_{L^2_{x_3}} \lesssim o(1) \Vert |k|^2 \phi_b\Vert_{L^2_{x_3}}^2 +  \Big\Vert  \frac{\sqrt{|k|}}{(1+|k|^2)^{1/4}} \hat{b}_3 \Big\Vert_{L^2_{x_3}}^2.
\end{align*}
Thus we conclude
\begin{align}
    & \Vert |k|\p_{x_3}\phi_b\Vert^2_{L^2_{x_3}} \lesssim  \Big\Vert  \frac{\sqrt{|k|}}{(1+|k|^2)^{1/4}} \hat{b}_3 \Big\Vert_{L^2_{x_3}}^2,   \notag \\
    & \Vert \p_{x_3}^2 \phi_b\Vert_{L^2_{x_3}}^2 \lesssim \Vert |k|^2 \phi_b\Vert_{L^2_{x_3}}^2 +  \Big\Vert  \frac{\sqrt{|k|}}{(1+|k|^2)^{1/4}} \hat{b}_3 \Big\Vert_{L^2_{x_3}}^2 \lesssim  \Big\Vert  \frac{\sqrt{|k|}}{(1+|k|^2)^{1/4}} \hat{b}_3 \Big\Vert_{L^2_{x_3}}^2.    \notag
\end{align}

By trace theorem, we have
\begin{align}
    & | |k|\phi_b(k,\pm 1) |^2 + | \p_{x_3}\phi_b(k,\pm 1)|^2 \lesssim  \Big\Vert  \frac{\sqrt{|k|}}{(1+|k|^2)^{1/4}} \hat{b}_3 \Big\Vert_{L^2_{x_3}}^2. \notag
\end{align}

By the same computation of the estimate of $\hat{\mathbf{b}}$ in Lemma \ref{lemma:macroscopic}, we have
\begin{align}
    &     \eqref{weak_formulation}_1 + \eqref{weak_formulation}_2 = \int_0^T \int_{-1}^1 [|k_1|^2 + |k_2|^2 - 2\p_{x_3}^2]\phi_b \hat{b}_3 \dd x_3 \dd t \notag \\
    &\underbrace{+ \int_0^T\int_{-1}^1\int_{\mathbb{R}^3}i\bar{v}\cdot k (\mathbf{I}-\mathbf{P})\hat{f}\psi_b \dd v \dd x_3 \dd t}_{E_3} \underbrace{-  \int_0^T \int_{-1}^1 \int_{\mathbb{R}^3} v_3 (\mathbf{I}-\mathbf{P})\hat{f}\p_{x_3}\psi_b \dd v \dd x_3 \dd t}_{E_4} \notag\\
    & =  \Big\Vert  \frac{\sqrt{|k|}}{(1+|k|^2)^{1/4}} \hat{b}_3 \Big\Vert_{L^2_T L^2_{x_3}}^2 + E_3 + E_4, \label{est_b_LHS_3}
\end{align}
with
\begin{align}
    & |E_3| + |E_4| \lesssim o(1) [\Vert |k|^2 \phi_b\Vert_{L^2_{T,x_3}}^2 +\Vert |k|\p_{x_3}\phi_b \Vert^2_{L^2_{T,x_3}} + \Vert \p_{x_3}^2 \phi_b\Vert^2_{L^2_{T,x_3}}] + \Vert (\mathbf{I}-\mathbf{P})\hat{f}\Vert^2_{L^2_{T,x_3,v}}  \notag \\
        & \lesssim o(1)\Big\Vert  \frac{\sqrt{|k|}}{(1+|k|^2)^{1/4}} \hat{b}_3 \Big\Vert_{L^2_T L^2_{x_3}}^2 + \Vert (\mathbf{I}-\mathbf{P})\hat{f}\Vert^2_{L^2_{T,x_3,v}}.\label{est_b_E2_3}
\end{align}

Then we compute the boundary term $\eqref{weak_formulation}_3$. By the same computation of Lemma \ref{lemma:macroscopic}, 
\begin{align}
    &  |\eqref{weak_formulation}_3| \lesssim o(1)[| |k| \phi_b(k,\pm 1) |^2_{L^2_T} +  | \p_{x_3} \phi_b(k,\pm 1)|^2_{L^2_T}  ] + |(I-P_\gamma)\hat{f}|_{L^2_{T,\gamma_+}}^2 \notag \\
    &\lesssim o(1) \Big\Vert  \frac{\sqrt{|k|}}{(1+|k|^2)^{1/4}} \hat{b}_3 \Big\Vert_{L^2_T L^2_{x_3}}^2 + |(I-P_\gamma)\hat{f}|_{L^2_{T,\gamma_+}}^2.   \label{est_b_bdr_3}
\end{align}

Next, we compute the time derivative $\eqref{weak_formulation}_4$. We denote $\Phi_b$ as the solution to the elliptic equation
\begin{align*}
\begin{cases}
      &\dis    (|k_1|^2 + |k_2|^2 - 2\p_{x_3}^2) \Phi_b(k,x_3) = \frac{|k|}{\sqrt{1+|k|^2}}\p_t \bar{\hat{b}}_3(t,k,x_3) , \ x_3 \in (-1,1) ,\\
    & \Phi_b(k,\pm 1) = 0.  
\end{cases}
\end{align*}
Integration by part leads to
\begin{align}
    &   \int_0^T \int_{-1}^1 (|k_1|^2 + |k_2|^2) |\Phi_b|^2 \dd x_3 \dd t+2 \int_0^T \int_{-1}^1 |\p_{x_3} \Phi_b|^2 \dd x_3 \dd t   \notag\\
    & = \int_0^T \int_{-1}^1 \frac{|k|}{\sqrt{1+|k|^2}} \p_t \bar{\hat{b}}_3(t,k,x_3) \bar{\Phi}_b \dd x_3 \dd t.   \label{energy_p_t_b_13}
\end{align}
From the conservation of momentum in $\hat{b}_3$ in \eqref{momentum_conservation}, \eqref{energy_p_t_b_13} becomes
\begin{align*}
    &    \int_0^T \int_{-1}^1 \frac{|k|}{\sqrt{1+|k|^2}} |\p_t \bar{\hat{b}}_3(t,k,x_3) \bar{\Phi}_b| \dd x_3 \dd t \notag \\
    & \leq \int_0^T \int_{-1}^1 \frac{|k|}{\sqrt{1+|k|^2}} \Big|  (-ik_1 \Theta_{31}((\mathbf{I}-\mathbf{P})\hat{f}))\bar{\Phi}_b - ik_2 \Theta_{32}((\mathbf{I}-\mathbf{P})\hat{f})\bar{\Phi}_b     \notag\\
    &  + [\hat{a} + 2\hat{c} + \Theta_{33}((\mathbf{I}-\mathbf{P})\hat{f})]\p_{x_3}\bar{\Phi}_b \Big| \dd x_3 \dd t \lesssim o(1) \Vert |k|\Phi_b \Vert_{L^2_{T,x_3}}^2 + o(1) \Vert \p_{x_3}\Phi_b\Vert_{L^2_{T,x_3}}^2 \\
    &+ \Big\Vert  \frac{|k|}{\sqrt{1+|k|^2}}\hat{a}\Big\Vert^2_{L^2_{T,x_3}}+ \Big\Vert  \frac{|k|}{\sqrt{1+|k|^2}} \hat{c}\Big\Vert^2_{L^2_{T,x_3}}+ \Vert   (\mathbf{I}-\mathbf{P})\hat{f}\Vert^2_{L^2_{T,x_3,v}}.
\end{align*}
Plugging this estimates to \eqref{energy_p_t_b_12}, with the Poincar\'e inequality $\Vert \Phi_b\Vert_{L^2_{x_3}} \lesssim \Vert \p_{x_3}\Phi_b\Vert_{L^2_{x_3}}$, we obtain
\begin{align}
    &   \Vert (1+|k|) \Phi_b \Vert_{L^2_{T,x_3}}^2 +  \Vert \p_{x_3}\Phi_b\Vert_{L^2_{T,x_3}}^2  \lesssim \Big\Vert  \frac{|k|}{\sqrt{1+|k|^2}}\hat{a}\Big\Vert^2_{L^2_{T,x_3}}+ \Big\Vert  \frac{|k|}{\sqrt{1+|k|^2}} \hat{c}\Big\Vert^2_{L^2_{T,x_3}}+ \Vert   (\mathbf{I}-\mathbf{P})\hat{f}\Vert^2_{L^2_{T,x_3,v}}.    \label{Phi_b_estimate_3}
\end{align}

Then we compute $\eqref{weak_formulation}_4$ using the same computation as Lemma \ref{lemma:macroscopic} as
\begin{align}
    &     |\eqref{weak_formulation}_4|  \lesssim  \Vert (\mathbf{I}-\mathbf{P})\hat{f}\Vert^2_{L^2_{T,x_3,v}} + o(1) \Vert  |k\Phi_b| + |\p_{x_3}\Phi_b| \Vert_{L^2_{T,x_3}}^2  \notag\\
    & \lesssim  \Vert (\mathbf{I}-\mathbf{P})\hat{f}\Vert^2_{L^2_{T,x_3,v}}   +\Big\Vert  \frac{|k|}{\sqrt{1+|k|^2}}\hat{a}\Big\Vert^2_{L^2_{T,x_3}}+ \Big\Vert  \frac{|k|}{\sqrt{1+|k|^2}} \hat{c}\Big\Vert^2_{L^2_{T,x_3}} .   \label{est_b_RHS_3_3}
\end{align}
In the last line, we have used \eqref{Phi_b_estimate_3}.

Moreover, $\eqref{weak_formulation}_5$, $\eqref{weak_formulation}_6$ and $\eqref{weak_formulation}_0$ are computed similarly:
\begin{align}
    &    |\eqref{weak_formulation}_5| \lesssim o(1)[\Vert k \phi_b \Vert_{L^2_{T,x_3}}^2 + \Vert \p_{x_3} \phi_b \Vert_{L^2_{T,x_3}}^2 ] + \Vert (\mathbf{I}-\mathbf{P})\hat{f}\Vert_{L^2_{T,x_3,v}}^2   \lesssim o(1)\Vert \hat{b}_1  \Vert^2_{L^2_{T,x_3}}  + \Vert (\mathbf{I}-\mathbf{P})\hat{f}\Vert^2_{L^2_{T,x_3,v}},  
    \label{est_b_RHS_5_3}
\end{align}
\begin{align}
    & |\eqref{weak_formulation}_6| \leq   \int_0^T   \int_{-1}^1 \int_{\mathbb{R}^3} |\hat{\Gamma}(\hat{f},\hat{f})\psi_b| \dd v \dd x_3 \dd t \notag \\
    &\lesssim o(1)[\Vert k \phi_b \Vert_{L^2_{T,x_3}}^2 + \Vert \p_{x_3} \phi_b \Vert_{L^2_{T,x_3}}^2 ] + \Big(\int_{\mathbb{R}^2} \Vert \hat{f}(k-\ell)\Vert_{L^\infty_T L^2_{x_3,v}} \Vert \hat{f}(\ell)\Vert_{L^2_T L^\infty_{x_3}L^2_\nu} \dd \ell\Big)^2 \notag \\
    &\lesssim  o(1)\Vert \hat{b}_1  \Vert^2_{L^2_{T,x_3}}  + \Big(\int_{\mathbb{R}^2} \Vert \hat{f}(k-\ell)\Vert_{L^\infty_T L^2_{x_3,v}} \Vert \hat{f}(\ell)\Vert_{L^2_T L^\infty_{x_3}L^2_\nu} \dd \ell\Big)^2,   \label{est_b_RHS_6_3}
\end{align}
and
\begin{align}
    &     \int_{-1}^1 \int_{\mathbb{R}^3} |\hat{f}(T) \psi_b(T)| \dd v \dd x_3   \lesssim \Vert \hat{f}(t) \Vert_{L^\infty_T L^2_{x_3,v}} [\Vert k\phi_b \Vert_{L^\infty_T L^2_{x_3}} + \Vert \p_{x_3}\phi_b \Vert_{L^\infty_T L^2_{x_3}} ] \notag \\
    & \lesssim \Vert \hat{f}\Vert_{L^\infty_T L^2_{x_3,v}} \Vert \hat{b}_3\Vert_{L^\infty_{T}L^2_{x_3}} \lesssim \Vert \hat{f}\Vert_{L^\infty_T L^2_{x_3,v}}^2,  \label{est_b_t_3}\\
        &  \int_{-1}^1 \int_{\mathbb{R}^3} |\hat{f}(0)\psi_b(0)| \dd v \dd x_3 \lesssim \Vert \hat{f}_0\Vert^2_{L^2_{x_3,v}}. \label{est_b_0_3}
\end{align}

We combine \labelcref{est_b_LHS_3,est_b_E2_3,est_b_bdr_3} and \labelcref{est_b_RHS_3_3,est_b_RHS_5_3,est_b_RHS_6_3,est_b_t_3,est_b_0_3} to conclude the estimate for $\hat{b}_3$:
\begin{align*}
    &  \Big\Vert \frac{\sqrt{|k|}}{(1+|k|^2)^{1/4}}\hat{b}_3 \Big\Vert_{L^2_{T,x_3}} \lesssim |(I-P_\gamma)\hat{f}|_{L^2_{T,\gamma_+}} + \Vert (\mathbf{I}-\mathbf{P})\hat{f}\Vert_{L^2_{T,x_3,v}} + \int_{\mathbb{R}^2} \Vert \hat{f}(k-\ell)\Vert_{L^\infty_T L^2_{x_3,v}} \Vert \hat{f}(\ell)\Vert_{L^2_T L^\infty_{x_3}L^2_\nu} \dd \ell\notag \\
    &+ \Big\Vert  \frac{|k|}{\sqrt{1+|k|^2}}\hat{a}\Big\Vert_{L^2_{T,x_3}}+ \Big\Vert  \frac{|k|}{\sqrt{1+|k|^2}} \hat{c}\Big\Vert_{L^2_{T,x_3}}   +   \Vert \hat{f}_0\Vert_{L^2_{x_3,v}} + \Vert \hat{f}\Vert_{L^\infty_T L^2_{x_3,v}}  .
\end{align*}

Applying Lemma \ref{lemma:macroscopic}, we conclude that
\begin{align*}
    & \Big\Vert \frac{\sqrt{|k|}}{(1+|k|^2)^{1/4}}\hat{b}_3 \Big\Vert_{L^1_k L^2_{T,x_3}} \lesssim \Vert (\mathbf{I}-\mathbf{P})\hat{f}\Vert_{L^1_k L^2_{T,x_3,v}}  \notag\\
    &+ \Vert \hat{f}\Vert_{L^p_k L^\infty_T L^2_{x_3,v}} \Vert w\hat{f}\Vert_{L^1_k L^2_T L^\infty_{x_3,v}} + |(I-P_\gamma)\hat{f}|_{L^1_k L^2_{T,\gamma_+}} + \Vert \hat{f}\Vert_{L^1_k L^\infty_T L^2_{x_3,v}} + \Vert \hat{f}_0\Vert_{L^1_k L^2_{x_3,v}}. 
\end{align*}

To include the time weight, by the same computation in \eqref{extra_b_c}, the extra term is controlled as
\begin{align*}
    &  \int_0^T \int_1^1 \int_{\mathbb{R}^3} (1+t)^{\sigma-1} |\hat{f} \psi_b| \dd v \dd x_3 \dd t \\
    & \lesssim o(1) \Big\Vert (1+t)^{\sigma/2} \frac{\sqrt{|k|}}{(1+|k|^2)^{1/4}} \hat{b}_3 \Big\Vert^2_{L^2_{T,x_3}} + \Vert (1+t)^{\sigma/2} (\mathbf{I}-\mathbf{P})\hat{f}\Vert_{L^2_{T,x_3,v}}^2.
\end{align*}
This concludes \eqref{b_3}.

\medskip
\textit{Proof of \eqref{c_less_k}.}

We choose a test function as
\begin{align}
    & \psi_c = (-ik_1v_1 \phi_c - ik_2 v_2 \phi_c + v_3 \p_{x_3}\phi_c)(|v|^2-5)\sqrt{\mu}, \notag
\end{align}
with $\phi_c$ satisfying
\begin{align}
\begin{cases}
    & \dis  |k|^2 \phi_c - \p_{x_3}^2 \phi_c = \frac{|k|^{1/2}}{(1+|k|^2)^{1/4}}\bar{\hat{c}},  \\
    & \phi_c = 0 \ \text{ when } x_3 = \pm 1.   
\end{cases}\label{phi_c_2}
\end{align}

Multiplying \eqref{phi_c_2} by $\bar{\phi}_c$ and taking integration in $x_3$ we obtain
\begin{align*}
    &   \Vert |k| \phi_c \Vert_{L^2_{x_3}}^2 + \Vert \p_{x_3} \phi_c \Vert_{L^2_{x_3}}^2 \lesssim o(1) \Vert \phi_c\Vert_{L^2_{x_3}}^2 + \Big\Vert \frac{|k|^{1/4}}{(1+|k|^2)^{1/8}} \hat{c} \Big\Vert_{L^2_{x_3}}^2.
\end{align*}
From the Poincar\'e inequality, we further have
\begin{align}
    & \Vert (1+|k|)\phi_c\Vert^2_{L^2_{x_3}}+ \Vert \p_{x_3}\phi_c\Vert^2_{L^2_{x_3}} \lesssim \Big\Vert \frac{|k|^{1/4}}{(1+|k|^2)^{1/8}} \hat{c} \Big\Vert_{L^2_{x_3}}^2.  \label{phic_l2_2}
\end{align}
Multiplying \eqref{phi_c_2} by $|k|^2 \bar{\phi}_c$ we obtain
\begin{align*}
    &  \Vert |k|^2\phi_c \Vert_{L^2_{x_3}}^2 + \Vert |k|\p_{x_3}\phi_c\Vert^2_{L^2_{x_3}} \lesssim o(1) \Vert |k|^2 \phi_c\Vert_{L^2_{x_3}}^2 +  \Big\Vert \frac{|k|^{1/4}}{(1+|k|^2)^{1/8}} \hat{c} \Big\Vert_{L^2_{x_3}}^2.
\end{align*}
Thus we conclude
\begin{align}
    & \Vert |k|\p_{x_3}\phi_c\Vert^2_{L^2_{x_3}} \lesssim \Vert \hat{c}\Vert^2_{L^2_{x_3}},  \notag  \\
    & \Vert \p_{x_3}^2 \phi_c\Vert_{L^2_{x_3}}^2 \lesssim \Vert |k|^2 \phi_c\Vert_{L^2_{x_3}}^2 + \Vert \hat{c}\Vert_{L^2_{x_3}}^2 \lesssim \Big\Vert \frac{|k|^{1/4}}{(1+|k|^2)^{1/8}} \hat{c} \Big\Vert_{L^2_{x_3}}^2.    \label{phic_h2_2}
\end{align}

By trace theorem, we have
\begin{align}
    & | |k|\phi_c(k,\pm 1) |^2 + | \p_{x_3}\phi_c(k,\pm 1)|^2 \lesssim \Big\Vert \frac{|k|^{1/4}}{(1+|k|^2)^{1/8}} \hat{c} \Big\Vert_{L^2_{x_3}}^2. \label{phic_trace_2}
\end{align}

By the same computation of Lemma \ref{lemma:macroscopic}, we have
\begin{align}
    &     \eqref{weak_formulation}_1 + \eqref{weak_formulation}_2 = 5\int_0^T \int_{-1}^1 [|k|^2 - \p_{x_3}^2]\phi_c \hat{c} \dd x_3 \dd t \notag\\
    & \underbrace{+ \int_0^T\int_{-1}^1\int_{\mathbb{R}^3}i\bar{v}\cdot k (\mathbf{I}-\mathbf{P})\hat{f}\psi_c \dd v \dd x_3 \dd t}_{E_5} \underbrace{-  \int_0^T \int_{-1}^1 \int_{\mathbb{R}^3} v_3 (\mathbf{I}-\mathbf{P})\hat{f}\p_{x_3}\psi_c \dd v \dd x_3 \dd t}_{E_6}\notag\\
    & = 5\Big\Vert \frac{|k|^{1/4}}{(1+|k|^2)^{1/8}} \hat{c} \Big\Vert_{L^2_{T,x_3}}^2 + E_5 + E_6, \label{est_c_LHS_2}
\end{align}
where $E_5$ and $E_6$ are controlled by applying \eqref{phic_h2_2} as
\begin{align}
    & |E_5| + |E_6| \lesssim o(1) [\Vert |k|^2 \phi_c\Vert_{L^2_{T,x_3}}^2 + \Vert |k|\p_{x_3}\phi_c \Vert^2_{L^2_{T,x_3}} + \Vert \p_{x_3}^2 \phi_c\Vert^2_{L^2_{T,x_3}}] + \Vert (\mathbf{I}-\mathbf{P})\hat{f}\Vert^2_{L^2_{T,x_3,v}}  \notag \\
        & \lesssim o(1)\Big\Vert \frac{|k|^{1/4}}{(1+|k|^2)^{1/8}} \hat{c} \Big\Vert_{L^2_{T,x_3}}^2 + \Vert (\mathbf{I}-\mathbf{P})\hat{f}\Vert^2_{L^2_{T,x_3,v}}.\label{est_c_E2_2}
\end{align}

Then we compute the boundary term $\eqref{weak_formulation}_3$. By the same computation of Lemma \ref{lemma:macroscopic} we have
\begin{align}
    & | \eqref{weak_formulation}_3  |    \lesssim o(1) [[| |k| \phi_c(k,\pm 1) |^2_{L^2_T} +  | \p_{x_3} \phi_c(k,\pm 1)|^2_{L^2_T}  ]  ]   +    |(I-P_\gamma)\hat{f}|_{L^2_{T,\gamma_+}}^2 \notag \\
    & \lesssim o(1)\Big\Vert \frac{|k|^{1/4}}{(1+|k|^2)^{1/8}} \hat{c} \Big\Vert_{L^2_{T,x_3}}^2 + |(I-P_\gamma)\hat{f}|_{L^2_{T,\gamma_+}}^2.    \label{est_c_bdr_2}
\end{align}
Here we used \eqref{phic_trace_2}.

Moreover, $\eqref{weak_formulation}_5$, $\eqref{weak_formulation}_6$ and $\eqref{weak_formulation}_0$ are computed similarly using \eqref{phic_l2_2}:
\begin{align}
    &   | \eqref{weak_formulation}_5 | \lesssim o(1)[\Vert k \phi_c \Vert_{L^2_{T,x_3}}^2 + \Vert \p_{x_3} \phi_c \Vert_{L^2_{T,x_3}}^2 ] + \Vert (\mathbf{I}-\mathbf{P})\hat{f}\Vert_{L^2_{T,x_3,v}}^2 \notag \\
    &\lesssim o(1)\Big\Vert \frac{|k|^{1/4}}{(1+|k|^2)^{1/8}} \hat{c} \Big\Vert_{L^2_{T,x_3}}^2  + \Vert (\mathbf{I}-\mathbf{P})\hat{f}\Vert^2_{L^2_{T,x_3,v}},   \label{est_c_RHS_5_2}
\end{align}
\begin{align}
    & |\eqref{weak_formulation}_6| =   \int_0^T   \int_{-1}^1 \int_{\mathbb{R}^3} |\hat{\Gamma}(\hat{f},\hat{f})\psi_c| \dd v \dd x_3 \dd t \notag \\
    &\lesssim o(1)[\Vert k \phi_c \Vert_{L^2_{T,x_3}}^2 + \Vert \p_{x_3} \phi_c \Vert_{L^2_{T,x_3}}^2 ] + \Big(\int_{\mathbb{R}^2} \Vert \hat{f}(k-\ell)\Vert_{L^\infty_T L^2_{x_3,v}} \Vert \hat{f}(\ell)\Vert_{L^2_T L^\infty_{x_3}L^2_\nu} \dd \ell\Big)^2  \notag \\
    &\lesssim  \Big\Vert \frac{|k|^{1/4}}{(1+|k|^2)^{1/8}} \hat{c} \Big\Vert_{L^2_{T,x_3}}^2 + \Big(\int_{\mathbb{R}^2} \Vert \hat{f}(k-\ell)\Vert_{L^\infty_T L^2_{x_3,v}} \Vert \hat{f}(\ell)\Vert_{L^2_T L^\infty_{x_3}L^2_\nu} \dd \ell\Big)^2,   \label{est_c_RHS_6_2}
\end{align}
\begin{align}
    &     \int_{-1}^1 \int_{\mathbb{R}^3} |\hat{f}(T) \psi_c(T)| \dd v \dd x_3   \lesssim \Vert \hat{f}(t) \Vert_{L^\infty_T L^2_{x_3,v}} [\Vert k\phi_c \Vert_{L^\infty_T L^2_{x_3}} + \Vert \p_{x_3}\phi_c \Vert_{L^\infty_T L^2_{x_3}} ] \notag \\
    & \lesssim \Vert \hat{f}\Vert_{L^\infty_T L^2_{x_3,v}} \Vert \hat{c} \Vert_{L^\infty_{T}L^2_{x_3}} \lesssim \Vert \hat{f}\Vert_{L^\infty_T L^2_{x_3,v}}^2,  \label{est_c_t_3}
\end{align}
and 
\begin{align}
        &  \int_{-1}^1 \int_{\mathbb{R}^3} |\hat{f}(0)\psi_c(0)| \dd v \dd x_3 \lesssim \Vert \hat{f}_0\Vert^2_{L^2_{x_3,v}}. \label{est_c_0_3}
\end{align}

We focus on computing $\eqref{weak_formulation}_4$. We denote $\Phi_c$ as the solution to the elliptic equation
\begin{align*}
\begin{cases}
    &\dis    (|k|^2 - \p_{x_3}^2) \Phi_c(k,x_3) = \p_t \bar{\hat{c}}(t,k,x_3) \frac{\sqrt{|k|}}{(1+|k|^2)^{1/4}}, \ x_3 \in (-1,1) ,\\
    & \Phi_c(k,\pm 1) = 0.    
\end{cases}
\end{align*}
Integration by part leads to
\begin{align}
    &   \int_0^T \int_{-1}^1 |k|^2 |\Phi_c|^2 \dd x_3 \dd t+ \int_0^T \int_{-1}^1 |\p_{x_3} \Phi_c|^2 \dd x_3 \dd t    = \int_0^T \int_{-1}^1 \frac{\sqrt{|k|}}{(1+|k|^2)^{1/4}} \p_t \bar{\hat{c}}(t,k,x_3) \bar{\Phi}_c \dd x_3 \dd t.   \label{energy_p_t_c_22}
\end{align}
By the same computation of Lemma \ref{lemma:macroscopic}, \eqref{energy_p_t_c_22} becomes
\begin{align*}
    &    \int_0^T \int_{-1}^1 \frac{\sqrt{|k|}}{(1+|k|^2)^{1/4}} |\p_t \bar{\hat{c}}(t,k,x_3) \bar{\Phi}_c| \dd x_3 \dd t \notag \\
    & = \int_0^T \int_{-1}^1 \frac{\sqrt{|k|}}{(1+|k|^2)^{1/4}} \Big| -\frac{1}{3}i(k_1 \hat{b}_1 + k_2 \hat{b}_2)\bar{\Phi}_c   - \frac{1}{6}i(k_1 \Lambda_1((\mathbf{I}-\mathbf{P})\hat{f})+k_2 \Lambda_2((\mathbf{I}-\mathbf{P})\hat{f}))\bar{\Phi}_c      \notag\\
    & + \frac{1}{3}\hat{b}_3 \p_{x_3}\bar{\Phi}_c  +  \frac{1}{6} \Lambda_3((\mathbf{I}-\mathbf{P})\hat{f})\p_{x_3}\bar{\Phi}_c \Big| \dd x_3 \dd t \\
    &  \lesssim o(1)\Vert  |k|\Phi_c  \Vert_{L^2_{T,x_3}}^2 + o(1)\Vert \p_{x_3}\Phi_c \Vert_{L^2_{T,x_3}}^2 + \Big\Vert \frac{\sqrt{|k|}}{(1+|k|^2)^{1/4}} \hat{\mathbf{b}}\Big\Vert_{L^2_{T,x_3}}^2 + \Vert (\mathbf{I}-\mathbf{P})\hat{f}\Vert^2_{L^2_{T,x_3,v}}.
\end{align*}
Plugging the estimates to \eqref{energy_p_t_c_22}, we obtain
\begin{align}
    &   \Vert |k| \Phi_c \Vert_{L^2_{T,x_3}}^2 +  \Vert \p_{x_3}\Phi_c\Vert_{L^2_{T,x_3}}^2  \lesssim \Big\Vert \frac{\sqrt{|k|}}{(1+|k|^2)^{1/4}} \hat{\mathbf{b}}\Big\Vert_{L^2_{T,x_3}}^2 + \Vert   (\mathbf{I}-\mathbf{P})\hat{f}\Vert^2_{L^2_{T,x_3,v}}.   \notag
\end{align}

By the same computation in Lemma \ref{lemma:macroscopic}, we compute $\eqref{weak_formulation}_4$ as
\begin{align}
    &   |  \eqref{weak_formulation}_4 | \lesssim  \Vert (\mathbf{I}-\mathbf{P})\hat{f}\Vert^2_{L^2_{T,x_3,v}} + o(1) \Vert  |k\Phi_c| + |\p_{x_3}\Phi_c| \Vert_{L^2_{T,x_3}}^2  \notag\\
    & \lesssim  \Vert (\mathbf{I}-\mathbf{P})\hat{f}\Vert^2_{L^2_{T,x_3,v}}  + o(1)\Big\Vert \frac{\sqrt{|k|}}{(1+|k|^2)^{1/4}} \hat{\mathbf{b}}\Big\Vert_{L^2_{T,x_3}}^2 .   \label{est_c_RHS_3_2}
\end{align}

We combine \labelcref{est_c_LHS_2,est_c_E2_2,est_c_bdr_2,est_c_RHS_5_2,est_c_RHS_6_2,est_c_t_3,est_c_0_3} together with \eqref{est_c_RHS_3_2} to conclude the estimate for $\hat{c}$:
\begin{align}
    &  \Big\Vert \frac{|k|^{1/4}}{(1+|k|^2)^{1/8}} \hat{c} \Big\Vert_{L^2_{T,x_3}} \lesssim |(I-P_\gamma)\hat{f}|_{L^2_{T,\gamma_+}} + \Vert (\mathbf{I}-\mathbf{P})\hat{f}\Vert_{L^2_{T,x_3,v}} \notag \\
    &+ \Big\Vert \frac{\sqrt{|k|}}{(1+|k|^2)^{1/4}} \hat{\mathbf{b}}\Big\Vert_{L^2_{T,x_3}} + \int_{\mathbb{R}^2} \Vert \hat{f}(k-\ell)\Vert_{L^\infty_T L^2_{x_3,v}} \Vert \hat{f}(\ell)\Vert_{L^2_T L^\infty_{x_3}L^2_\nu} \dd \ell    .
\notag
\end{align}
Applying \eqref{b_3}, we conclude that
\begin{align*}
    & \Big\Vert \frac{|k|^{1/4}}{(1+|k|^2)^{1/8}}\hat{c} \Big\Vert_{L^1_k L^2_{T,x_3}} \lesssim \Vert (\mathbf{I}-\mathbf{P})\hat{f}\Vert_{L^1_k L^2_{T,x_3,v}}  \notag\\
    &+ \Vert \hat{f}\Vert_{L^p_k L^\infty_T L^2_{x_3,v}} \Vert w\hat{f}\Vert_{L^1_k L^2_T L^\infty_{x_3,v}} + |(I-P_\gamma)\hat{f}|_{L^1_k L^2_{T,\gamma_+}} + \Vert \hat{f}\Vert_{L^1_k L^\infty_T L^2_{x_3,v}} + \Vert \hat{f}_0\Vert_{L^1_k L^2_{x_3,v}}. 
\end{align*}
To include the time weight, by the same computation in \eqref{extra_b_c}, such term is controlled as
\begin{align*}
    &  \int_0^T \int_{-1}^1 \int_{\mathbb{R}^3} | (1+t)^{\sigma-1} \hat{f} \psi_b | \dd v \dd x_3 \dd t \\
    & \lesssim o(1) \Big\Vert (1+t)^{\sigma/2} \frac{|k|^{1/4}}{(1+|k|^2)^{1/8}} \hat{b}_3 \Big\Vert^2_{L^2_{T,x_3}} + \Vert (1+t)^{\sigma/2} (\mathbf{I}-\mathbf{P})\hat{f}\Vert_{L^2_{T,x_3,v}}^2.
\end{align*}
This concludes \eqref{c_less_k}. We then conclude Lemma \ref{lemma:b_c_no_time_derivative}.
\end{proof}

\subsection{$L^1_k L^\infty_{T,x_3,v}$ estimate with time decay}\label{sec:linfty}

In this section, based on Proposition \ref{prop: full_energy_decay}, we bootstrap the estimate from $L^1_k L^\infty_T L^2_{x_3,v}$ to  $L^1_k L^\infty_{T,x_3,v}$ with time decay.

\begin{proposition}[\textbf{$L^1_k L^\infty_{T,x_3,v}$ estimate with time decay}]\label{prop:linfty}
Let $\hat{f}$ be the solution to \eqref{f_eqn} with initial data $f_0$ satisfying \eqref{initial_assumption}, then we have the following $L^1_k L^\infty_{T,x_3,v}$ control with time decay:
\begin{align*}
    &     \Vert (1+t)^{\sigma/2} w\hat{f}\Vert_{L^1_k L^\infty_{T,x_3,v}} \lesssim \Vert w\hat{f}_0\Vert_{L^1_k L^\infty_{x_3,v}} + \Vert (1+t)^{\sigma/2} \hat{f}\Vert_{L^1_k L^\infty_T L^2_{x_3,v}} + \Vert (1+t)^{\sigma/2} w \hat{f}\Vert_{L^1_kL^\infty_{T,x_3,v}}^2 .
\end{align*}

\end{proposition}

To prove the proposition we first define the stochastic cycle.
We use standard notations for the backward exit time and backward exit position in the physical space $x_3\in (-1,1)$:
\begin{equation*}%\label{xb_tb}
\begin{split}
\tb(x_3,v) :   &  = \sup\{s\geq 0, x_3-sv_3 \in (-1,1)\} , \\
  \xb(x_3,v)  : & = x_3 - \tb(x_3,v)v_3 \in \{-1,1\}.
\end{split}
\end{equation*}

We denote $t^0=T_0$, a fixed starting time. First, we define the stochastic cycle as follows.

\begin{definition}
We define a stochastic cycles as $(x^0_3,v^0)= (x_3,v) \in (-1,1) \times \R^3$ and inductively
\begin{align}
&x^1_3:= \xb(x_3,v), \   v^1 \in \mathcal{V}_1 := \{v^1\in \mathbb{R}^3: v^1_3 \times  \text{sign}(x_3^1) > 0\} , \notag\\
&  v^{n}\in \mathcal{V}_n:=\{v^{n}\in \mathbb{R}^3: v^n_3 \times \text{sign}(x_3^n)>0\}, \ \ \text{for} \  n \geq 1,
\notag
\\
 &x^{n+1}_3 := \xb(x^n_3, v^n) \in \{-1,1\} , \ \tb^{n}:= \tb(x^n_3,v^n) \ \ \text{for} \  v^n_3 \gtrless 0, \ x^n_3 = \pm 1 , \notag \\
 & t^n =  t^0-  \{ \tb + \tb^1 + \cdots + \tb^{n-1}\},  \ \ \text{for} \  n \geq 1. \notag
\end{align}
\end{definition}

We rewrite \eqref{f_eqn} into the following formulation:
\begin{align*}
    &   \p_t \hat{f} + i\bar{v}\cdot k \hat{f} + v_3 \p_{x_3} \hat{f} + \nu(v) \hat{f} = K(\hat{f}) + \hat{\Gamma}(\hat{f},\hat{f}).
\end{align*}
We apply the method of characteristics to have
\begin{align}
    &   w(v) \hat{f}(t,k_1,k_2,x_3,v) \notag\\
    &= \mathbf{1}_{\tb>t} e^{-\nu(v) t - i(\bar{v}\cdot k)t} w(v)\hat{f}_0(k_1,k_2,x_3 - t v_3, v)  \label{chara:initial}  \\
    & + \mathbf{1}_{\tb \leq t} e^{-\nu(v)\tb - i(\bar{v}\cdot k)\tb}w(v)\hat{f}(t^1,k_1,k_2,x_3-\tb v_3,v)  \label{chara:bdr}\\
    & + \int^t_{\max\{0, t-\tb\}} e^{-\nu(v)(t-s) - i(\bar{v}\cdot k)(t-s)} w(v)\int_{\mathbb{R}^3} \mathbf{k}(v,u) \hat{f}(s,k_1,k_2,x_3-(t-s)v_3,u) \dd s  \label{chara:K}\\
    & + \int^t_{\max\{0,t-\tb\}} e^{-\nu(v)(t-s)-i(\bar{v}\cdot k)(t-s)} w(v)\hat{\Gamma}(\hat{f},\hat{f})(s,k_1,k_2,x_3-(t-s)v_3,v) \dd s. \label{chara:Gamma}
\end{align}
Here the boundary term is bounded as
\begin{align}
   & |\eqref{chara:bdr}|\leq   e^{-\nu(v) (t-t^1)} w(v)\sqrt{\mu(v)}  \notag\\
   & \times \int_{\prod_{j=1}^{n}\mathcal{V}_j} \bigg\{ \sum_{i=1}^{n}\mathbf{1}_{t^{i+1}\leq 0 < t^i} e^{-\nu(v^i) t^i} w(v^i)|\hat{f}_0(x^i_3 - t^i v^i_3, v^i)| \dd \Sigma_i     \label{bdr_initial}\\
  &\qquad\qquad\qquad + \mathbf{1}_{t^{n+1}>0}  w(v^{n})|\hat{f}(t^{n+1},x^{n+1}_3,v^{n})| \dd \Sigma_{n} \label{bdr_tk}\\
  & + \sum_{i=1}^{n} \mathbf{1}_{t^{i+1}\leq 0 < t^i}  \int^{t^i}_0 e^{-\nu(v^i) (t^i-s)}   w(v^i)  \int_{\mathbb{R}^3}  \mathbf{k}(v^i,u) |\hat{f}(s,x^i_3-(t^i-s)v^i_3, u)| \dd u \dd s \dd \Sigma_i   \label{bdr_K_0} \\
  & + \sum_{i=1}^{n} \mathbf{1}_{t^{i+1}>0} \int_{t^{i+1}}^{t^i} e^{-\nu(v^i) (t^i-s)}  w(v^i)
  \int_{\mathbb{R}^3} \mathbf{k}(v^i,u) |\hat{f}(s,x^i_3-(t^i-s)v^i_3, u)| \dd u \dd s \dd \Sigma_i   \label{bdr_K_i} \\
  &  + \sum_{i=1}^{n} \mathbf{1}_{t^{i+1}\leq 0 < t^i}  \int^{t^i}_0 e^{-\nu(v^i) (t^i-s)} w(v^i) |\hat{\Gamma}(\hat{f},\hat{f})(s,x^i_3-(t^i-s)v^i_3,v^i)| \dd s \dd \Sigma_i     \label{bdr_g_0}   \\
  &  + \sum_{i=1}^{n} \mathbf{1}_{t^{i+1}>0} \int_{t^{i+1}}^{t^i} e^{-\nu(v^i) (t^i-s)} w(v^i) |\hat{\Gamma}(\hat{f},\hat{f})(s,x^i_3-(t^i-s)v^i_3,v^i)| \dd s \dd \Sigma_i  \bigg\}.  \label{bdr_g_i}  
\end{align}
Here $\dd \Sigma_i$ is defined as
\begin{equation}\label{Sigma_i}
\dd \Sigma_i = \Big\{\prod_{j=i+1}^{n} \dd \sigma_j \Big\}  \times  \Big\{  \frac{1}{w(v^i)\sqrt{\mu(v^i)}}   \dd \sigma_i \Big\} \times \Big\{\prod_{j=1}^{i-1}  e^{-\nu(v^j)(t^j-t^{j+1})}   \dd \sigma_j \Big\} ,
\end{equation}
where $\dd \sigma_i$ is a probability measure on $\mathcal{V}_i$ given by
\begin{equation}\label{sigma_i}
\dd \sigma_i =  \sqrt{2\pi}\mu(v^i) |v^i_3|\dd v^i.
\end{equation}

Note that \eqref{bdr_tk} corresponds to the scenario in which the backward trajectory interacts with the boundary many times. Such term is controlled by the following lemma.

\begin{lemma}\label{lemma:tk}
For $T_0>0$ sufficiently large, there exist constants $C_1,C_2>0$ independent of $T_0$ such that for $n = C_1 T_0^{5/4}$, and $(t^0,x^0_3,v^0) = (t,x_3,v)\in [0,T_0]\times (-1,1)\times \mathbb{R}^3$,
\begin{equation*}%\label{tk_small}
\int_{\prod_{j=1}^{n-1} \mathcal{V}_j}\mathbf{1}_{t^n>0}  \prod_{j=1}^{n-1} \dd \sigma_j \leq \Big( \frac{1}{2}\Big)^{C_2 T_0^{5/4}}.
\end{equation*}

\end{lemma}

\begin{proof}
The proof is similar to \cite{G} since the domain is bounded in $x_3$.
\end{proof}

First, we control the boundary term \eqref{chara:bdr} in the following lemma.
\begin{lemma}\label{lemma:linfty_bdr}
There exists a constant $C=C(T_0)$ such that for $t\leq T_0$, the boundary term \eqref{chara:bdr} is controlled as
\begin{align*}
    &    |\eqref{chara:bdr}|   \leq 4e^{-\nu_0t}\Vert w\hat{f}_0\Vert_{ L^\infty_{x_3,v}}   + o(1)(1+t)^{-\sigma/2}  \sup_{0\leq s\leq t}\Vert (1+s)^{\sigma/2} w \hat{f}(s)\Vert_{ L^\infty_{x_3,v}} \\
    &+  C(1+t)^{-\sigma/2}  \big[\sup_{0\leq s\leq t}\Vert (1+s)^{\sigma/2} \hat{f}(s)\Vert_{L^2_{x_3,v}} + \sup_{0\leq s\leq t}\Vert \nu^{-1}(1+s)^{\sigma/2} w \hat{\Gamma}(\hat{f},\hat{f})(s)\Vert_{ L^\infty_{x_3,v}}  \big].
\end{align*}
\end{lemma}

\begin{proof}
Since $\dd \sigma_i$ in~\eqref{sigma_i} is a probability measure on $\mathcal{V}_i$, \eqref{bdr_initial} is directly bounded as
\begin{equation}\label{bdr_initial_bdd}
|\eqref{bdr_initial}|\leq 4e^{-\nu_0 t^1}   \Vert w \hat{f}_0\Vert_{L^\infty_{x_3,v}}.
\end{equation}
Here the constant $4$ comes from $\sqrt{2\pi}\int_{\mathcal{V}_i}   |v^i_3| \sqrt{\mu(v^i)} w^{-1}(v^i) \dd v^i \leq 4.$
The exponential decay factor $e^{-\nu_0 t^1}$ comes from the decay factor in \eqref{Sigma_i}, and the computation
\begin{align*}
    & e^{-\nu_0 t^i} e^{-\nu_0 (t^{i-1}-t^i)} \leq e^{-\nu_0 t^{i-1}}, \ e^{-\nu_0 t^{i-1}} e^{-\nu_0 (t^{i-2}-t^{i-1})} \leq e^{-\nu_0t^{i-2}} \cdots.
\end{align*}

For~\eqref{bdr_tk}, since $t\leq T_0$, with $n=C_1 T_0^{5/4}$, and $t^{n+1}>0$ implies $t^n>0$, we apply Lemma \ref{lemma:tk} to have
\begin{equation}\label{bdr_tk_bdd}
\begin{split}
    &  |\eqref{bdr_tk}|\\
    &\leq  \int_{\prod_{j=1}^{n-1} \mathcal{V}_j}  \int_{\mathcal{V}_n} \mathbf{1}_{t^{n+1}>0} w(v^n)|\hat{f}(t^{n+1},x_3^{n+1},v^n)|w^{-1}(v^n)\sqrt{\mu(v^n)}|v^n_3| \dd v^n   \prod_{j=1}^{n-1} e^{-\nu(v^j)(t^j-t^{j+1})}  \dd \sigma_j   \\
 & \lesssim (1+t^1)^{-\sigma/2}\sup_{0\leq s\leq t}\Vert (1+s)^{\sigma/2} w \hat{f}(s)\Vert_{L^\infty_{x_3,v}} \int_{\prod^{n-1}_{j=1}\mathcal{V}_j} \mathbf{1}_{t^{n}>0}   \prod_{j=1}^{n-1} \dd \sigma_j   \\
 &\leq o(1) (1+t^1)^{-\sigma/2} \sup_{0\leq s\leq t}\Vert (1+s)^{\sigma/2} w \hat{f}(s)\Vert_{L^\infty_{x_3,v}}.
\end{split}
\end{equation}
Here the polynomial decay factor $(1+t^1)^{-\sigma/2}$ comes from the following computation:
\begin{align}
    & (1+t^{n+1})^{-\sigma/2} e^{-\nu_0 (t^{n}-t^{n+1})} e^{-\nu_0(t^{n-1}-t^n)}\cdots e^{-\nu_0(t^1-t^2)} = (1+t^{n+1})^{-\sigma/2} e^{-\nu_0(t^1-t^{n+1})} \notag\\
    &= (1+t^{1})^{-\sigma/2} \frac{(1+t^{1})^{\sigma/2}}{(1+t^{n+1})^{\sigma/2}} e^{-\nu_0(t^{1}-t^{n+1})}\notag \\
    &\lesssim (1+t^{1})^{-\sigma/2} [1+|t^{1}-t^{n+1}|^{\sigma/2}] e^{-\nu_0(t^{1}-t^{n+1})} \lesssim (1+t^{1})^{-\sigma/2}. \label{time_induction}
\end{align}

Then we estimate \eqref{bdr_g_0} and \eqref{bdr_g_i}. For each $i$, we compute
\begin{align*}
    &   \int_{\prod_{j=1}^{n}\mathcal{V}_j} \int^{t^i}_0 e^{-\nu(v^i)(t^i-s)}w(v^i) |\hat{\Gamma}(\hat{f},\hat{f})(s,x_3^i - (t^i-s)v_3^i,v^i)| \dd s \dd \Sigma_i \\
    & \lesssim \sup_{0\leq s\leq t} \Vert \nu^{-1} w (1+s)^{\sigma/2}\hat{\Gamma}(\hat{f},\hat{f})(s)\Vert_{L^\infty_{x_3,v}} \\
    &\times \int_{\prod_{j=1}^{n}\mathcal{V}_j} \int^{t^i}_0 e^{-\nu(v^i)(t^i-s)/2} e^{-\nu(v^i)(t^i-s)/2} (1+s)^{-\sigma/2} \nu(v^i) \dd s \dd \Sigma_i \\
    & \lesssim  \sup_{0\leq s\leq t} \Vert \nu^{-1} w (1+s)^{\sigma/2}\hat{\Gamma}(\hat{f},\hat{f})(s)\Vert_{L^\infty_{x_3,v}}  \int_{\prod_{j=1}^{n}\mathcal{V}_j} (1+t^i)^{-\sigma/2} \dd \Sigma_i \\
    & \lesssim \sup_{0\leq s\leq t} \Vert \nu^{-1} w (1+s)^{\sigma/2}\hat{\Gamma}(\hat{f},\hat{f})(s)\Vert_{L^\infty_{x_3,v}} (1+t^1)^{-\sigma/2}  .
\end{align*}
Here we used
\begin{align*}
    &  \int_0^{t^i} e^{-\nu(v^i)(t^i-s)/2} \nu(v^i) \dd s \lesssim 1,  \\
    & (1+s)^{-\sigma/2} e^{-\nu_0 (t^i-s)/2} \lesssim (1+t^{i})^{-\sigma/2}. 
\end{align*}
In the last line, we obtain the decay term $(1+t^1)^{-\sigma/2}$ by the same computation as \eqref{time_induction}.

Then we conclude that
\begin{equation}\label{bdr_g_bdd}
\begin{split}
|\eqref{bdr_g_0} + \eqref{bdr_g_i}| & \leq      Cn(1+t^1)^{-\sigma/2} \sup_{0\leq s\leq t} \Vert \nu^{-1} w (1+s)^{\sigma/2}\hat{\Gamma}(\hat{f},\hat{f})(s)\Vert_{L^\infty_{x_3,v}}.
\end{split}
\end{equation}

Then we estimate \eqref{bdr_K_i}. Recall the notation $\mathbf{k}_\theta(v,u)= \mathbf{k}(v,u)\frac{e^{\theta|v|^2}}{e^{\theta |u|^2}}$ in Lemma \ref{lemma:k_theta}. We focus on estimating
\begin{equation}\label{iteration_i}
\begin{split}
    & \int_{\prod_{j=1}^{i-1} \mathcal{V}_j}\mathbf{1}_{t^{i+1}>0 } \prod_{j=1}^{i-1}e^{-\nu(v^{j})(t^j-t^{j+1})}\dd \sigma_j \int_{\mathcal{V}_i} |v^i_3| w^{-1}(v^i) \dd v^i   \\
    & \times     \int_{t^{i+1}}^{t^i} e^{-\nu(v^i)(t^i -s)} \int_{\mathbb{R}^3} \dd u \mathbf{k}_\theta(v^i,u) w(u)|\hat{f}(s,x^i_3-(t^i-s)v^i_3,u)| \dd s .
\end{split}
\end{equation}

First we decompose the $\dd s$ integral into $\mathbf{1}_{s\geq t^i-\delta} + \mathbf{1}_{s< t^i-\delta}$. By \eqref{k_theta} in Lemma \ref{lemma:k_theta}, the contribution of the first term reads
\begin{align}
& |\eqref{iteration_i} \mathbf{1}_{s\geq t^i-\delta} |\notag \\
& \lesssim \int_{\prod_{j=1}^{i-1} \mathcal{V}_j}\mathbf{1}_{t^{i+1}>0 }  \prod_{j=1}^{i-1} e^{-\nu(v^{j})(t^j-t^{j+1})} \dd \sigma_j \int_{\mathcal{V}_i} |v^i_3| w^{-1}(v^i) \dd v^i \notag\\ 
 & \times \int^{t^i}_{\max\{t^{i+1},t^i-\delta\}} e^{-\nu(v^i) (t^i-s)} \int_{\mathbb{R}^3} \dd u\mathbf{k}_\theta(v^i,u) w(u) |\hat{f}(s,x^i_3-(t^i-s)v^i_3,u)| \dd s  \notag\\
 &\leq  \sup_{0\leq s\leq t}\Vert (1+s)^{\sigma/2} w \hat{f}(s)\Vert_{L^\infty_{x_3,v}} \int_{\prod_{j=1}^{i-1} \mathcal{V}_j}\mathbf{1}_{t^{i+1}>0 }  \prod_{j=1}^{i-1} e^{-\nu(v^{j})(t^j-t^{j+1})}\dd \sigma_j  \notag\\ 
 & \times \int_{\mathcal{V}_i} |v^i_3| w^{-1}(v^i) \dd v^i \int^{t^i}_{\max\{t^{i+1},t^i-\delta\}} e^{-\nu_0 (t^i-s)/2} e^{-\nu_0(t^i-s)/2} (1+s)^{-\sigma/2} \dd s    \notag\\
 &  \leq o(1) \sup_{0\leq s\leq t}\Vert (1+s)^{\sigma/2} w \hat{f}(s)\Vert_{L^\infty_{x_3,v}}  \int_{\prod_{j=1}^{i-1} \mathcal{V}_j}\mathbf{1}_{t^{i+1}>0 }  \prod_{j=1}^{i-1}e^{-\nu(v^{j})(t^j-t^{j+1})}\dd \sigma_j  (1+t^i)^{-\sigma/2} \notag  \\
 & \leq o(1) (1+t^1)^{-\sigma/2} \sup_{0\leq s\leq t}\Vert (1+s)^{\sigma/2} w \hat{f}(s)\Vert_{L^\infty_{x_3,v}} . \label{bdr_s_small_bdd}
\end{align}
In the last line, we applied the same computation \eqref{time_induction}.

Next, we decompose the $v^i$ integral into $\mathbf{1}_{|v^i|\geq N} + \mathbf{1}_{|v^i|<N}$. By \eqref{k_theta} in Lemma \ref{lemma:k_theta}, and using the computation \eqref{time_induction}, the contribution of the first term reads
\begin{align}
    & |\eqref{iteration_i} \mathbf{1}_{|v^i|\geq N}|  \leq \int_{\prod_{j=1}^{i-1} \mathcal{V}_j}\mathbf{1}_{t^{i+1}>0 }  \prod_{j=1}^{i-1}e^{-\nu(v^{j})(t^j-t^{j+1})}\dd \sigma_j \int_{|v^i|\geq N} |v^i_3| w^{-1}(v^i) \dd v^i   \notag \\
    &\ \ \ \ \ \ \ \ \ \ \ \ \ \ \ \ \  \times \int_{t^{i+1}}^{t^i} e^{-\nu(v^i)(t^i -s)} \int_{\mathbb{R}^3}\dd u \mathbf{k}_\theta(v^i,u) w(u) |\hat{f}(s,x^i_3-(t^i-s)v^i_3,u)| \dd s \notag\\
    & \leq o(1) (1+t^1)^{-\sigma/2} \sup_{0\leq s\leq t}\Vert (1+s)^{\sigma/2} w \hat{f}(s)\Vert_{L^\infty_{x_3,v}}.   \label{bdr_v_large_bdd}
\end{align}

Then we decompose the $u$ integral into $\mathbf{1}_{|u|\geq N \text{ or } |v^i-u|\leq \frac{1}{N}} + \mathbf{1}_{|u|<N, \ |v^i-u|> \frac{1}{N}}$. By \eqref{K_N_small} in Lemma \ref{lemma:k_theta}, and using the computation \eqref{time_induction}, the contribution of the first term reads
\begin{align}
    & |\eqref{iteration_i} \mathbf{1}_{|u|\geq N \text{ or } |v^i-u|\leq \frac{1}{N}} |   \notag \\
    &\leq \int_{\prod_{j=1}^{i-1} \mathcal{V}_j}\mathbf{1}_{t^{i+1}>0 }  \prod_{j=1}^{i-1}e^{-\nu(v^{j})(t^j-t^{j+1})}\dd \sigma_j \int_{\mathcal{V}_i} |v^i_3| w^{-1}(v^i) \dd v^i  \notag\\
    &\ \ \times  \int_{t^{i+1}}^{t^i} e^{-\nu(v^i)(t^i -s)} \int_{\mathbb{R}^3}  \dd u\mathbf{1}_{|u|\geq N \text{ or } |v^i-u|\leq \frac{1}{N}} \mathbf{k}_\theta(v^i,u) w(u) |\hat{f}(s,x^i_3-(t^i-s)v^i_3,u)| \dd s \notag\\
    & \leq o(1) (1+t^1)^{-\sigma/2} \sup_{0\leq s\leq t}\Vert (1+s)^{\sigma/2} w \hat{f}(s)\Vert_{L^\infty_{x_3,v}}.   \label{bdr_u_large_bdd}
\end{align}

Now we consider the intersection of all the other cases, where we have $|v^i|\leq N, \ s<t^i-\delta$, and $|u|<N, \ |v^i-u|>\frac{1}{N}$. The conditions of $v^i$ and $u$ imply that $\mathbf{k}(v^i,u)\leq C_N$ from \eqref{k_theta_bdd} in Lemma \ref{lemma:k_theta}.

In the last line, we have applied the change of variable $v^i_3 \to y_3 = x^i_3-(t^i-s)v^i_3 \in (-1,1)$ with Jacobian
\[\Big|\det\Big(\frac{\p (x^i_3-(t^i-s)v^i_3)}{\p v^i_3} \Big) \Big| = (t^i-s) \geq \delta .\]

Then we apply this change of variable with the H\"older inequality to have
\begin{align}
 & |\eqref{iteration_i} \mathbf{1}_{|u|< N , |v^i-u|> \frac{1}{N}, s<t^i-\delta, |v^i|\leq N}|     \notag\\
 & \leq \frac{1}{\delta}\int_{\prod_{j=1}^{i-1}\mathcal{V}_j}  \prod_{j=1}^{i-1}e^{-\nu(v^{j})(t^j-t^{j+1})}\dd \sigma_j  \int^{t^i-\delta}_{0}  e^{-\nu_0(t^i-s)}    \int_{-1}^1 \int_{|u|\leq N}  |\hat{f}(s,y_3,u)| \dd u \dd y \dd s   \notag\\
 & \leq  C_{N,\delta,T_0}   \int_{\prod_{j=1}^{i-1}\mathcal{V}_j}  \prod_{j=1}^{i-1}e^{-\nu(v^{j})(t^j-t^{j+1})}\dd \sigma_j \times    \int_{0}^{t^i} e^{-\nu_0(t^i-s)} (1+s)^{-\sigma/2} \Vert (1+s)^{\sigma/2} \hat{f}(s)\Vert_{L^2_{x_3,v}} \dd s \notag\\
  & \leq  C_{N,\delta,T_0} (1+t^1)^{-\sigma/2} \sup_{0\leq s\leq t}\Vert (1+s)^{\sigma/2} \hat{f}(s)\Vert_{L^2_{x_3,v}}  . \label{other_case_bdd}
\end{align}

Collecting \labelcref{bdr_s_small_bdd,bdr_v_large_bdd,bdr_u_large_bdd,other_case_bdd}, we conclude that
\begin{equation*}
\begin{split}
    & |\eqref{bdr_K_i}| \leq o(1)n (1+t^1)^{-\sigma/2} \sup_{0\leq s\leq t} \Vert (1+s)^{\sigma/2} w \hat{f}(s)\Vert_{L^\infty_{x_3,v}} \\
    & +   C_{N,\delta,n,T_0} (1+t^1)^{-\sigma/2}   \sup_{0\leq s\leq t}\Vert (1+s)^{\sigma/2}\hat{f}(s)\Vert_{L^2_{x_3,v}}. 
\end{split}
\end{equation*}
Here, $o(1)$ corresponds to $\delta$ and $\frac{1}{N}$. Since $n=n(T_0)$ is fixed, we choose $\delta$ and $\frac{1}{N}$ to be small enough such that $\Big(\delta+\frac{1}{N}\Big) n \leq o(1).$

By the same computation, we have the same bound for \eqref{bdr_K_0}. Thus we conclude that
\begin{equation}\label{bdr_K_0_bdd}
\begin{split}
    & |\eqref{bdr_K_0} + \eqref{bdr_K_i}| \leq o(1) (1+t^1)^{-\sigma/2} \sup_{0\leq s\leq t} \Vert (1+s)^{\sigma/2} w \hat{f}(s)\Vert_{L^\infty_{x_3,v}} \\
    &+ C_{N,\delta,n,T_0} (1+t^1)^{-\sigma/2} \sup_{0\leq s\leq t}  \Vert (1+s)^{\sigma/2}\hat{f}(s)\Vert_{L^2_{x_3,v}}.
\end{split}
\end{equation}

Summarizing \eqref{bdr_initial_bdd}, \eqref{bdr_tk_bdd}, \eqref{bdr_g_bdd} and \eqref{bdr_K_0_bdd}, with
\begin{align*}
    &   e^{-\nu_0 t^1} e^{-\nu_0(t-t^1)} = e^{-\nu_0 t}, \ e^{-\nu_0 (t-t^1)}(1+t^1)^{-\sigma/2} \lesssim (1+t)^{-\sigma/2}, \ w(v)\sqrt{\mu(v)} \lesssim 1,
\end{align*}
we conclude the lemma.
\end{proof}

\begin{proof}[\textbf{Proof of Proposition \ref{prop:linfty}}]

We first fix $t\leq T_0$.

The first term \eqref{chara:initial} is controlled as
\begin{align}
    &    |\eqref{chara:initial}| \leq  e^{-\nu_0 t} \Vert w\hat{f}_0\Vert_{L^\infty_{x_3,v}} .   \label{chara:initial_bdd}
\end{align}

By Lemma \ref{lemma:linfty_bdr} with $t\leq T_0$, we control the second term as
\begin{align}
    &    |\eqref{chara:bdr}|  \leq 4e^{-\nu_0t}\Vert w\hat{f}_0\Vert_{ L^\infty_{x_3,v}}   + o(1)(1+t)^{-\sigma/2}  \sup_{0\leq s\leq t}\Vert (1+s)^{\sigma/2} w \hat{f}(s)\Vert_{ L^\infty_{x_3,v}} \notag\\
    &+  C(T_0)(1+t)^{-\sigma/2}  \big[\sup_{0\leq s\leq t}\Vert (1+t)^{\sigma/2} \hat{f}(s)\Vert_{L^2_{x_3,v}} + \sup_{0\leq s\leq t}\Vert \nu^{-1}(1+s)^{\sigma/2} w \hat{\Gamma}(\hat{f},\hat{f})(s)\Vert_{ L^\infty_{x_3,v}}  \big].  \label{chara:bdr_bdd}
\end{align}

For \eqref{chara:K}, we apply the method of characteristic \labelcref{chara:initial,chara:bdr,chara:K,chara:Gamma} to iterate $\hat{f}(s,k_1,k_2,x_3-(t-s)v_3,u)$ along the velocity $u$. We denote $t_u: = s-\tb(x_3-(t-s)v_3,u)$, $y_3=x_3-(t-s)v_3$. We have
\begin{align}
    & |\eqref{chara:K}| \leq  \int^{t}_{\max\{0,t^1\}} \dd s e^{-\nu(v) (t -s)}  \int_{\mathbb{R}^3} \dd u \frac{w(v)}{w(u)}\mathbf{k}(v,u) \notag \\
    &\times \Big\{ \mathbf{1}_{t_u \leq 0} e^{-\nu(u) s} w(u) |\hat{f}(0,y_3 -su_3,u)|       \label{chara:K_initial} \\
    &  + \int_{\max\{0,t_u\}}^{s}  e^{-\nu(u)(s-s')} \dd s'\int_{\mathbb{R}^3} w(u)\mathbf{k}(u,u') |\hat{f}(s',y_3-(s-s')u_3, u') |   \dd u' \label{chara:K_K} \\
    & +  \int_{\max\{0,t_u\}}^{s} e^{-\nu(u)(s-s')} w(u)|\hat{\Gamma}(\hat{f},\hat{f})(s',y_3-(s-s')u_3,u)| \dd s' \label{chara:K_gamma} \\
    & + \mathbf{1}_{t_u > 0} e^{-\nu(u)(s-t_u)} w(u) |\hat{f}(t_u,y-\tb(y_3,u)u_3,u)| \Big\}.   \label{chara:K_bdr}
\end{align}

We first compute the contribution of the initial condition as
\begin{align}
    &   |\eqref{chara:K_initial}| \leq \Vert w\hat{f}_0\Vert_{L^\infty_{x_3,v}} \int^t_{\max\{0,t^1\}} \dd s e^{-\nu(v)(t-s)/2} e^{-\nu(v)(t-s)/2} e^{-\nu_0 s} \int_{\mathbb{R}^3} \dd u    \mathbf{k}_\theta(v,u) \notag\\
    & \leq C_\theta e^{-\nu_0 t/2} \Vert w \hat{f}_0\Vert_{L^\infty_{x_3,v}}. \label{chara:K_initial_bdd}
\end{align}

Then we compute the contribution of $\hat{\Gamma}$ as
\begin{align}
    &   |\eqref{chara:K_gamma}|  \leq \sup_{0\leq s\leq t}\Vert \nu^{-1}w (1+s)^{\sigma/2}\hat{\Gamma}(\hat{f},\hat{f})\Vert_{L^\infty_{x_3,v}} \int^t_{\max\{0,t^1\}} \dd s e^{-\nu_0(t-s)} \int_{\mathbb{R}^3} \dd u \mathbf{k}_\theta(v,u) \notag\\
    & \times \int^s_{\max\{0,t_u\}} \dd s' e^{-\nu(u)(s-s')/2} \nu(u)  e^{-\nu(u)(s-s')/2} (1+s')^{-\sigma/2}   \notag\\
    & \leq C_\theta  \sup_{0\leq s\leq t}\Vert \nu^{-1}w (1+s)^{\sigma/2}\hat{\Gamma}(\hat{f},\hat{f})\Vert_{L^\infty_{x_3,v}} \int^t_{\max\{0,t^1\}} \dd s e^{-\nu_0(t-s)} (1+s)^{-\sigma/2}  \notag\\
    & \leq C_\theta  (1+t)^{-\sigma/2} \sup_{0\leq s\leq t}\Vert \nu^{-1}w (1+s)^{\sigma/2}\hat{\Gamma}(\hat{f},\hat{f})\Vert_{L^\infty_{x_3,v}}. \label{chara:K_gamma_bdd}
\end{align}

We control the contribution of the boundary term by applying Lemma \ref{lemma:linfty_bdr} to $w(u)\hat{f}(t_u,y-\tb(y_3,u)u_3,u)$:
\begin{align}
    &   | \eqref{chara:K_bdr}| \leq   \int^t_{\max\{0,t^1\}} \dd s e^{-\nu_0(t-s)} \int_{\mathbb{R}^3} \dd u \mathbf{k}_\theta(v,u) e^{-\nu_0(s-t_u)} \notag\\
    & \times   \Big[ 4 e^{-\nu_0 t_u}\Vert w\hat{f}_0\Vert_{ L^\infty_{x_3,v}}   + o(1)(1+t_u)^{-\sigma/2}  \sup_{0\leq s\leq t}\Vert (1+s)^{\sigma/2} w \hat{f}(s)\Vert_{ L^\infty_{x_3,v}} \notag\\
    &+  C(T_0)(1+t_u)^{-\sigma/2}  \big[\Vert (1+t)^{\sigma/2} \hat{f}\Vert_{L^\infty_T L^2_{x_3,v}} + \sup_{0\leq s\leq t}\Vert \nu^{-1}(1+s)^{\sigma/2} w \hat{\Gamma}(\hat{f},\hat{f})(s)\Vert_{ L^\infty_{x_3,v}}  \big]  \Big]  \notag\\
    & \leq \int^t_{\max\{0,t^1\}} \dd s e^{-\nu_0(t-s)/2} e^{-\nu_0(t-s)/2} \notag\\
    &\times \Big[ 4 e^{-\nu_0 s}\Vert w\hat{f}_0\Vert_{ L^\infty_{x_3,v}}   + o(1)(1+s)^{-\sigma/2}  \sup_{0\leq s\leq t}\Vert (1+s)^{\sigma/2} w \hat{f}(s)\Vert_{ L^\infty_{x_3,v}} \notag\\
    &+  C(T_0)(1+s)^{-\sigma/2}  \big[\sup_{0\leq s\leq t}\Vert (1+s)^{\sigma/2} \hat{f}(s)\Vert_{L^2_{x_3,v}} + \sup_{0\leq s\leq t}\Vert \nu^{-1}(1+s)^{\sigma/2} w \hat{\Gamma}(\hat{f},\hat{f})(s)\Vert_{ L^\infty_{x_3,v}}  \big]  \Big] \notag\\
    & \leq C_{\nu_0} e^{-\nu_0t/2}\Vert w\hat{f}_0\Vert_{ L^\infty_{x_3,v}}   + o(1)(1+t)^{-\sigma/2}  \sup_{0\leq s\leq t}\Vert (1+s)^{\sigma/2} w \hat{f}(s)\Vert_{ L^\infty_{x_3,v}} \notag\\
    &+  C(T_0)(1+t)^{-\sigma/2}  \big[\sup_{0\leq s\leq t}\Vert (1+s)^{\sigma/2} \hat{f}(s)\Vert_{L^2_{x_3,v}} + \sup_{0\leq s\leq t}\Vert \nu^{-1}(1+s)^{\sigma/2} w \hat{\Gamma}(\hat{f},\hat{f})(s)\Vert_{ L^\infty_{x_3,v}}  \big]. \label{chara:K_bdr_bdd}
\end{align}

Next, we compute \eqref{chara:K_K}. We consider several cases. When $s-\e<s'<s$, we have
\begin{align}
    &   |\eqref{chara:K_K}| \leq C  (1+t)^{-\sigma/2} \sup_{0\leq s\leq t}\Vert (1+s)^{\sigma/2} w\hat{f}(s)\Vert_{L^\infty_{x_3,v}} \notag\\
    &\times \int^t_{\max\{0,t-\tb\}}\dd s e^{-\nu(v)(t-s)} \frac{(1+t)^{\sigma/2}}{(1+s)^{\sigma/2}} \int_{\mathbb{R}^3}\dd u \mathbf{k}_\theta(v,u)  \notag\\
    & \times   \int_{\max\{0,s-\e\}}^s \dd s' e^{-\nu(u)(s-s')/2} e^{-\nu(u)(s-s')/2} \frac{(1+s)^{\sigma/2}}{(1+s')^{\sigma/2}} \int_{\mathbb{R}^3}\dd u' \mathbf{k}_\theta(u,u')  \notag\\
    & \leq    C \e (1+t)^{-\sigma/2} \sup_{0\leq s\leq t}\Vert (1+s)^{\sigma/2} w\hat{f}(s)\Vert_{L^\infty_{x_3,v}} \int^t_{\max\{0,t-\tb\}} e^{-\nu(v)(t-s)} \frac{(1+t)^{\sigma/2}}{(1+s)^{\sigma/2}}\dd s \notag\\ 
    &  \leq   o(1)  (1+t)^{-\sigma/2} \sup_{0\leq s\leq t}\Vert (1+s)^{\sigma/2} w\hat{f}(s)\Vert_{L^\infty_{x_3,v}}   .   \label{chara_K_K_1_bdd}
\end{align}

When $|u|>N$ or $|u-v|\leq \frac{1}{N}$, we have
\begin{align}
    &   |\eqref{chara:K_K}| \leq C  (1+t)^{-\sigma/2} \sup_{0\leq s\leq t}\Vert (1+s)^{\sigma/2} w\hat{f}(s)\Vert_{L^\infty_{x_3,v}} \notag\\
    &\times \int^t_{\max\{0,t-\tb\}}\dd s e^{-\nu(v)(t-s)} \frac{(1+t)^{\sigma/2}}{(1+s)^{\sigma/2}} \int_{|u|>N \text{ or }|u-v|\leq \frac{1}{N}}\dd u \mathbf{k}(v,u) \frac{w(v)}{w(u)} \notag\\
    & \times   \int_{\max\{0,s-\tb(y_3,u)\}}^s \dd s' e^{-\nu(u)(s-s')} \frac{(1+s)^{\sigma/2}}{(1+s')^{\sigma/2}} \int_{\mathbb{R}^3}\dd u' \mathbf{k}(u,u')\frac{w(u)}{w(u')}  \notag\\
    &  \leq \frac{C}{N}  (1+t)^{-\sigma/2} \sup_{0\leq s\leq t}\Vert (1+s)^{\sigma/2} w\hat{f}(s)\Vert_{L^\infty_{x_3,v}} \leq o(1) (1+t)^{-\sigma/2} \sup_{0\leq s\leq t}\Vert (1+s)^{\sigma/2} w\hat{f}(s)\Vert_{L^\infty_{x_3,v}}   .   \label{chara_K_K_2_bdd}
\end{align}

When $|u'|>N$ or $|u'-u|\leq \frac{1}{N}$, we have
\begin{align}
    &   |\eqref{chara:K_K}| \leq C  (1+t)^{-\sigma/2} \sup_{0\leq s\leq t}\Vert (1+s)^{\sigma/2} w\hat{f}(s)\Vert_{L^\infty_{x_3,v}} \notag\\
    &\times \int^t_{\max\{0,t-\tb\}}\dd s e^{-\nu(v)(t-s)} \frac{(1+t)^{\sigma/2}}{(1+s)^{\sigma/2}} \int_{\mathbb{R}^3}\dd u \mathbf{k}(v,u) \frac{w(v)}{w(u)} \notag\\
    & \times   \int_{\max\{0,s-\tb(y,u)\}}^s \dd s' e^{-\nu(u)(s-s')} \frac{(1+s)^{\sigma/2}}{(1+s')^{\sigma/2}} \int_{|u'|>N \text{ or } |u-u'|\leq \frac{1}{N}}\dd u' \mathbf{k}(u,u')\frac{w(u)}{w(u')} \notag \\
    &  \leq   o(1)  (1+t)^{-\sigma/2} \sup_{0\leq s\leq t}\Vert (1+s)^{\sigma/2} w\hat{f}(s)\Vert_{L^\infty_{x_3,v}}    .   \label{chara_K_K_3_bdd}
\end{align}

Then we consider the rest case, where we have $s'<s-\e$, $|u'|\leq N$, $|u'-u|>\frac{1}{N}$, $|u|\leq N$ and $|u-v|> \frac{1}{N}$. We apply the change of variable
\begin{align}
    &  u_3 \to z_3 = x_3 - (t-s)v_3 - (s-s')u_3 \in (-1,1). \label{cov}
\end{align}
The Jacobian is given by
\begin{align*}
    &   \Big| \frac{\p (x_3 - (t-s)v_3 - (s-s')u_3)}{\p u_3}\Big| = (s-s')>\e.
\end{align*}
Since $|u'|\leq N$, $|u'-u|>\frac{1}{N}$, $|u|\leq N$ and $|u-v|> \frac{1}{N}$, from Lemma \ref{lemma:k_theta}, we have 
\begin{align*}
    &   \mathbf{k}(u,u')\frac{w(u)}{w(u')} \lesssim 1, \ \mathbf{k}(v,u)\frac{w(u)}{w(v)} \lesssim 1.
\end{align*}

Then we compute
\begin{align}
    & |\eqref{chara:K_K}| \leq C \int^t_{\max\{0,t-\tb\}} \dd s e^{-\nu(v)(t-s)} \int_{|u|\leq N,|v-u|>\frac{1}{N}} \dd u \mathbf{k}(v,u) \frac{w(v)}{w(u)} \notag\\
    &\times   \int^s_{\max\{0,t-\tb(y_3,u)\}} \dd s' e^{-\nu(u)(s-s')} \int_{|u'|\leq N, |u-u'|>\frac{1}{N}} \dd u'\mathbf{k}(u,u')\frac{w(u)}{w(u')} w(u') \notag\\
    &\times |\hat{f}(s',k_1,k_2,y_3-(s-s')u_3,u')|  \notag \\
    &\leq C(\e) \int^t_{\max\{0,t-\tb\}} \dd s e^{-\nu(v)(t-s)} \int_{-1}^1 \dd z_3 \notag \\
    & \times \int^s_{\max\{0,s-\tb(y_3,u)\}} \dd s' e^{-\nu_0(s-s')} \int_{|u'|\leq N} \dd u' |\hat{f}(s',k_1,k_2,z_3,u')| \notag\\
    &  \leq C(\e)  (1+t)^{-\sigma/2} \int^t_{\max\{0,t-\tb\}} \dd s e^{-\nu(v)(t-s)}\frac{(1+t)^{\sigma/2}}{(1+s)^{\sigma/2}}  \notag\\
    & \times \sup_{0\leq s\leq t}\Vert (1+s)^{\sigma/2} \hat{f}(s)\Vert_{L^2_{x_3,v}}\int^s_{\max\{0,s-\tb(y,u)\}} \dd s' e^{-\nu_0(s-s')} \frac{(1+s)^{\sigma/2}}{(1+s')^{\sigma/2}} \notag\\
    & \leq C(\e) (1+t)^{-\sigma/2} \sup_{0\leq s\leq t}\Vert (1+s)^{\sigma/2} \hat{f}(s)\Vert_{L^2_{x_3,v}}.   \label{chara_K_K_4_bdd}
\end{align}
In the fourth line, we have applied the change of variable \eqref{cov}. In the sixth line, we applied the H\"older inequality. 

We combine \eqref{chara_K_K_1_bdd}, \eqref{chara_K_K_2_bdd}, \eqref{chara_K_K_3_bdd} and \eqref{chara_K_K_4_bdd} to conclude that
\begin{align}
    &   |\eqref{chara:K_K}| \leq (1+t)^{-\sigma/2} \big[o(1)\sup_{0\leq s\leq t}\Vert (1+s)^{\sigma/2} w\hat{f}(s) \Vert_{L^\infty_{x_3,v}}+ C(\e)\sup_{0\leq s\leq t}\Vert (1+s)^{\sigma/2} \hat{f}(s)\Vert_{L^2_{x_3,v}} \big]. \label{chara:K_K_bdd}
\end{align}

Now we can conclude the estimate of \eqref{chara:K} by combining \eqref{chara:K_initial_bdd}, \eqref{chara:K_gamma_bdd}, \eqref{chara:K_bdr_bdd} and \eqref{chara:K_K_bdd}:
\begin{align}
    &    |\eqref{chara:K}| \leq (C_\theta + C_{\nu_0}) e^{-\nu_0 t/ 2} \Vert w\hat{f}_0\Vert_{L^\infty_{x_3,v}} + o(1)(1+t)^{-\sigma/2}\sup_{0\leq s\leq t}\Vert (1+s)^{\sigma/2} w \hat{f}(s)\Vert_{L^\infty_{x_3,v}} \notag\\
    & + C(T_0,\e,\theta) (1+t)^{-\sigma/2} \big[\sup_{0\leq s\leq t}\Vert (1+s)^{\sigma/2} \hat{f}(s)\Vert_{L^\infty L^2_{x_3,v}} + \sup_{0\leq s\leq t}\Vert \nu^{-1}(1+s)^{\sigma/2} w \hat{\Gamma}(\hat{f},\hat{f})(s)\Vert_{ L^\infty_{x_3,v}}  \big].  \label{chara:K_bdd}
\end{align}

Last we compute \eqref{chara:Gamma} as
\begin{align}
    & |\eqref{chara:Gamma}| \lesssim \int^t_{\max\{0,t-\tb\}} \dd s e^{-\nu(v)(t-s)} w(v)|\hat{\Gamma}(\hat{f},\hat{f})(s,k_1,k_2,x_3-(t-s)v_3,v) | \notag\\
    & \leq C (1+t)^{-\sigma/2}   \int^t_{\max\{0,t-\tb\}} \dd s e^{-\nu(v)(t-s)} \nu(v) \frac{(1+t)^{\sigma/2}}{(1+s)^{\sigma/2}} \notag\\
    &\times \sup_{0\leq s\leq t}\Vert \nu^{-1}w(1+s)^{\sigma/2}\hat{\Gamma}(\hat{f},\hat{f})\Vert_{L^\infty_{x_3,v}} \notag\\
    & \leq C (1+t)^{-\sigma/2} \sup_{0\leq s\leq t} \Vert \nu^{-1}w(1+s)^{\sigma/2}\hat{\Gamma}(\hat{f},\hat{f})\Vert_{L^\infty_{x_3,v}}. \label{chara:gamma_bdd}
\end{align}

Now we collect \eqref{chara:initial_bdd}, \eqref{chara:bdr_bdd}, \eqref{chara:K_bdd} and \eqref{chara:gamma_bdd} to obtain that, for any $0<t\leq T_0$,
\begin{align*}
    &   \Vert w\hat{f}(t)\Vert_{L^\infty_{x_3,v}} \leq [C_\theta + C_{\nu_0}+5] e^{-\nu_0 t/ 2} \Vert w\hat{f}_0\Vert_{L^\infty_{x_3,v}} + o(1)(1+t)^{-\sigma/2}\sup_{0\leq s\leq t}\Vert (1+s)^{\sigma/2} w \hat{f}(s)\Vert_{L^\infty_{x_3,v}} \notag\\
    & + C(T_0,\e,\theta) (1+t)^{-\sigma/2} \big[\sup_{0\leq s\leq t}\Vert (1+s)^{\sigma/2} \hat{f}\Vert_{L^2_{x_3,v}} + \sup_{0\leq s\leq t}\Vert \nu^{-1}(1+s)^{\sigma/2} w \hat{\Gamma}(\hat{f},\hat{f})(s)\Vert_{ L^\infty_{x_3,v}}  \big]. 
\end{align*}
We absorb the term with $o(1)$ and further obtain
\begin{align}
    &   \Vert w\hat{f}(t)\Vert_{L^\infty_{x_3,v}} \leq C_{\theta,\nu_0} e^{-\nu_0 t/ 2} \Vert w\hat{f}_0\Vert_{L^\infty_{x_3,v}}  \notag\\
    & + C(T_0,\e,\theta) (1+t)^{-\sigma/2} \big[\sup_{0\leq s\leq t}\Vert (1+s)^{\sigma/2} \hat{f}(s)\Vert_{L^2_{x_3,v}} + \sup_{0\leq s\leq t}\Vert \nu^{-1}(1+s)^{\sigma/2} w \hat{\Gamma}(\hat{f},\hat{f})(s)\Vert_{ L^\infty_{x_3,v}}  \big].  \label{f_t_bdd}
\end{align}
Since $T_0\gg 1$, we further simplify the upper bound at time $T_0$:
\begin{align*}
    &   \Vert w\hat{f}(T_0)\Vert_{L^\infty_{x_3,v}} \leq e^{-\nu_0 t/4} \Vert w\hat{f}_0\Vert_{L^\infty_{x_3,v}}  \notag\\
    & + C(T_0,\e,\theta) (1+T_0)^{-\sigma/2} \big[\sup_{0\leq s\leq T_0}\Vert (1+s)^{\sigma/2} \hat{f}\Vert_{L^\infty_T L^2_{x_3,v}} + \sup_{0\leq s\leq T_0}\Vert \nu^{-1}(1+s)^{\sigma/2} w \hat{\Gamma}(\hat{f},\hat{f})(s)\Vert_{ L^\infty_{x_3,v}}  \big]. 
\end{align*}

For $t=mT_0$, we inductively compute
\begin{align}
  &\Vert w\hat{f}(mT_0)\Vert_{L^\infty_{x_3,v}} \leq e^{-\frac{\nu_0 T_0}{4}} \Vert w \hat{f}((m-1)T_0)\Vert_{L^\infty_{x_3,v}}  \notag\\
  &   + C(T_0,\e,\theta) (1+T_0)^{-\sigma/2} \sup_{0\leq s\leq T_0} \Vert (1+s)^{\sigma/2} \hat{f}((m-1)T_0+s)\Vert_{L^2_{x_3,v}}  \notag\\
  &  + C(T_0,\e,\theta)(1+T_0)^{-\sigma/2} \sup_{0 \leq s\leq T_0} \Vert \nu^{-1}(1+s)^{\sigma/2} w \hat{\Gamma}(\hat{f},\hat{f})((m-1)T_0+s)\Vert_{ L^\infty_{x_3,v}} \notag\\
  &  \leq e^{-\frac{\nu_0 T_0}{4}} \Vert w \hat{f}((m-1)T_0)\Vert_{L^\infty_{x_3,v}}  + C(T_0,\e,\theta) (1+mT_0)^{-\sigma/2} \sup_{0\leq s\leq mT_0}\Vert (1+s)^{\sigma/2} \hat{f}(s)\Vert_{L^2_{x_3,v}}  \notag\\
  &  + C(T_0,\e,\theta)(1+m T_0)^{-\sigma/2} \sup_{0 \leq s\leq mT_0} \Vert \nu^{-1}(1+s)^{\sigma/2} w \hat{\Gamma}(\hat{f},\hat{f})(s)\Vert_{ L^\infty_{x_3,v}} \notag\\
  & \leq e^{-2\frac{\nu_0 T_0}{4}} \Vert w\hat{f}((m-2)T_0)\Vert_{L^\infty_{x,v}}  +   C(T_0,\e,\theta)(1+mT_0)^{-\sigma/2} \times \big[1 + e^{-\frac{T_0}{4}} (1+T_0)^{\sigma/2} \big] \notag \\
  & \times \Big[\sup_{0\leq s\leq mT_0}\Vert (1+s)^{\sigma/2} \hat{f}(s)\Vert_{L^2_{x_3,v}} + \sup_{0 \leq s\leq mT_0} \Vert \nu^{-1}(1+s)^{\sigma/2} w \hat{\Gamma}(\hat{f},\hat{f})(s)\Vert_{ L^\infty_{x_3,v}} \Big] \notag\\
  & \leq \cdots \leq e^{-\frac{m\nu_0 T_0}{4}} \Vert w \hat{f}_0\Vert_{L^\infty_{x_3,v}} +  C(T_0,\e,\theta)(1+mT_0)^{-\sigma/2} \times \sum_{j=0}^{m-1} e^{-j\nu_0 T_0/4} (1+T_0)^{j\sigma/2} \notag\\
  & \times  \Big[\sup_{0\leq s\leq mT_0}\Vert (1+s)^{\sigma/2} \hat{f}(s)\Vert_{L^2_{x_3,v}} + \sup_{0 \leq s\leq mT_0} \Vert \nu^{-1}(1+s)^{\sigma/2} w \hat{\Gamma}(\hat{f},\hat{f})(s)\Vert_{ L^\infty_{x_3,v}} \Big] \notag\\
  & \leq e^{-\frac{m\nu_0 T_0}{4}} \Vert w \hat{f}_0\Vert_{L^\infty_{x_3,v}} + C(T_0,\e,\theta)(1+mT_0)^{-\sigma/2}  \notag\\
  & \times \Big[\sup_{0\leq s\leq mT_0}\Vert (1+s)^{\sigma/2} \hat{f}(s)\Vert_{L^2_{x_3,v}} + \sup_{0 \leq s\leq mT_0} \Vert \nu^{-1}(1+s)^{\sigma/2} w \hat{\Gamma}(\hat{f},\hat{f})(s)\Vert_{ L^\infty_{x_3,v}} \Big] . \label{mT0}
\end{align}
In the fourth line, we applied the following computation
\begin{align*}
    & (1+T_0)^{-\sigma/2} \frac{(1+mT_0)^{\sigma/2}}{(1+mT_0)^{\sigma/2}} (1+s)^{\sigma/2}  \leq \frac{(1+(m-1)T_0+s)^{\sigma/2}}{(1+mT_0)^{\sigma/2}}.
\end{align*}
Such inequality holds since
\begin{align*}
    &  (1+(m-1)T_0+s)(1+T_0) - (1+mT_0)(1+s)  \\
    &= 1+mT_0+s+(m-1)T_0^2 + T_0 s - 1 - mT_0 -s -smT_0\\
    & = (m-1)T_0^2 + T_0 s - smT_0 = (m-1)T_0^2 - (m-1)sT_0 \geq 0.
\end{align*}

In the second-last line in \eqref{mT0}, the summation $\sum_{j=0}^{m-1} e^{-j\nu_0 T_0/4} (1+T_0)^{j\sigma/2}$ converges since $T_0\gg 1$ and thus $e^{-\nu_0 T_0/4}(1+T_0)^{\sigma/2} \leq e^{-\nu_0 T_0/8}$.

For any $t>0$, we can find $m\in \mathbb{Z}^+$ such that $t=mT_0 +t'$ for $0\leq t'\leq T_0$. Then we apply \eqref{f_t_bdd} to have
\begin{align*}
  &\Vert w\hat{f}(mT_0+t')\Vert_{L^\infty_{x_3,v}} \leq e^{-\frac{\nu_0 t'}{4}} \Vert w \hat{f}(mT_0)\Vert_{L^\infty_{x_3,v}}  \notag\\
  &   + C(T_0,\e,\theta) (1+t')^{-\sigma/2} \sup_{0\leq s\leq t'} \Vert (1+s)^{\sigma/2} \hat{f}(mT_0+s)\Vert_{L^2_{x_3,v}}  \notag\\
  &  + C(T_0,\e,\theta)(1+t')^{-\sigma/2} \sup_{0 \leq s\leq t'} \Vert \nu^{-1}(1+s)^{\sigma/2} w \hat{\Gamma}(\hat{f},\hat{f})(mT_0+s)\Vert_{ L^\infty_{x_3,v}} \notag\\
  & \leq e^{-\frac{\nu_0(mT_0+t')}{4}} \Vert w\hat{f}_0\Vert_{L^\infty_{x_3,v}} + C(T_0,\e)(1+mT_0+t')^{-\sigma/2} \\
  & \times \Big[\sup_{0\leq s\leq mT_0+t'}\Vert (1+s)^{\sigma/2} \hat{f}(s)\Vert_{L^2_{x_3,v}} + \sup_{0 \leq s\leq mT_0+t'} \Vert \nu^{-1}(1+s)^{\sigma/2} w \hat{\Gamma}(\hat{f},\hat{f})(s)\Vert_{ L^\infty_{x_3,v}} \Big] .
\end{align*}
Here we applied \eqref{mT0} to $\Vert w \hat{f}(mT_0)\Vert_{L^\infty_{x_3,v}}$, and also applied the following computations:
\begin{align*}
    &    (1+t')^{-\sigma/2} \frac{(1+mT_0+t')^{\sigma/2}}{(1+mT_0+t')^{\sigma/2}} (1+s)^{\sigma/2} \leq \frac{(1+mT_0+s)^{\sigma/2}}{(1+mT_0+t')^{\sigma/2}}, \ s\leq t' \\
    & \text{ for } (1+t')^{-\sigma/2} \sup_{0\leq s\leq t'} \Vert (1+s)^{\sigma/2} \hat{f}(mT_0+s)\Vert_{L^2_{x_3,v}};\\
    & e^{-\frac{\nu_0 t'}{4}} (1+mT_0)^{-\sigma/2} = (1+mT_0+t')^{-\sigma/2} e^{-\frac{\nu_0 t'}{4}} \Big( 1+ \frac{t'}{1+mT_0} \Big)^{\sigma/2} \leq C(1+mT_0+t')^{-\sigma/2}.
\end{align*}

Finally, we conclude that, for any $T$ and some constant $C$ that does not depend on $T$, 
\begin{align*}
    &   \Vert (1+t)^{\sigma/2} w\hat{f}\Vert_{L^\infty_{T,x_3,v}} \leq C \Vert w\hat{f}_0\Vert_{L^\infty_{x_3,v}}  \notag\\
    & + C\big[\Vert (1+t)^{\sigma/2} \hat{f}\Vert_{L^\infty_T L^2_{x_3,v}} + \Vert \nu^{-1}(1+t)^{\sigma/2} w \hat{\Gamma}(\hat{f},\hat{f})\Vert_{  L^\infty_{T,x_3,v}}  \big]. 
\end{align*}
We take the $k$ integration and apply \eqref{gamma_est_time_linfty} to conclude Proposition \ref{prop:linfty}.
\end{proof}

\subsection{Proof of Theorem \ref{thm:l1k_lpk}}\label{sec:proof_thm1}

To prove Theorem \ref{thm:l1k_lpk}, we collect all previous estimates to obtain the following a priori estimate.
\begin{proposition}\label{prop:aprioi_f}
Let $\hat{f}$ be the solution to \eqref{f_eqn} such that the initial condition $f_0$ satisfy \eqref{initial_assumption}, and 
\begin{align}\label{ad.prop.ap}
    &\Vert \hat{f}\Vert_{L^p_k L^\infty_T L^2_{x_3,v}}+\Vert (1+t)^{\sigma/2} w\hat{f}\Vert_{L^1_k  L^\infty_{T,x_3,v}}<\infty,
\end{align}
then for some $C>1$,
\begin{align*}
    &    \Vert (1+t)^{\sigma/2} w\hat{f} \Vert_{L^1_k L^\infty_{T,x_3,v}} + \Vert \hat{f}\Vert_{L^p_k L^\infty_T L^2_{x_3,v}} \\
& \leq C\big[ \Vert w\hat{f}_0\Vert_{L^1_k L^\infty_{x_3,v}} + \Vert \hat{f}_0\Vert_{L^p_k L^2_{x_3,v}} + \Vert (1+t)^{\sigma/2} w\hat{f}\Vert_{L^1_k  L^\infty_{T,x_3,v}}^2 + \Vert \hat{f}\Vert_{L^p_k L^\infty_T L^2_{x_3,v}} \Vert (1+t)^{\sigma/2} w\hat{f}\Vert_{L^1_k  L^\infty_{T,x_3,v}}\big].
\end{align*}
We also have the following estimate:
\begin{align*}
    & \Vert (1+t)^{\sigma/2} \hat{f}\Vert_{L^1_k L^\infty_T L^2_{x_3,v}} + \Vert (1+t)^{\sigma/2} (\mathbf{I}-\mathbf{P}) \hat{f} \Vert_{L^1_k L^2_T L^2_{x_3,\nu}} + |(1+t)^{\sigma/2}(I-P_\gamma)\hat{f}|^2_{L^1_k L^2_{T,\gamma_+}}  \\
    &  + \Big\Vert (1+t)^{\sigma/2} \frac{|k|}{\sqrt{1+|k|^2}}(\hat{a},\hat{\mathbf{b}},\hat{c})\Big\Vert_{L^1_k L^2_{T,x_3}} \\
    & \leq C \big[ \Vert w\hat{f}_0\Vert_{L^1_k L^\infty_{x_3,v}} + \Vert \hat{f}_0\Vert_{L^p_k L^2_{x_3,v}} + \Vert (1+t)^{\sigma/2} w\hat{f}\Vert_{L^1_k  L^\infty_{T,x_3,v}}^2 + \Vert \hat{f}\Vert_{L^p_k L^\infty_T L^2_{x_3,v}} \Vert (1+t)^{\sigma/2} w\hat{f}\Vert_{L^1_k  L^\infty_{T,x_3,v}} \big].
\end{align*}

\end{proposition}

\begin{proof}

Since $\sigma = 2(1-1/p)-2\e$ for $p>2$, we have $\sigma>1$, and thus
\begin{align}
    &    \Vert w\hat{f}\Vert_{L^1_k L^2_T L^\infty_{x_3,v}} \leq \Vert (1+t)^{\sigma/2}w \hat{f}\Vert_{L^1_k L^\infty_{T,x_3,v}} \Big(\int_0^T (1+t)^{\sigma/2} \dd t\Big)^{1/2} \lesssim \Vert (1+t)^{\sigma/2}w \hat{f}\Vert_{L^1_k L^\infty_{T,x_3,v}} . \label{time_decay_nonlinear}
\end{align}

Combining this with Proposition \ref{prop:linfty}, Proposition \ref{prop: full_energy_decay} and Lemma \ref{lemma: energy}, we conclude Proposition \ref{prop:aprioi_f}.
\end{proof}

\begin{proof}[\textbf{Proof of Theorem \ref{thm:l1k_lpk}}]
With the a priori estimate in Proposition \ref{prop:aprioi_f}, we can apply the standard sequential argument to construct a unique solution to \eqref{f_eqn} that satisfies \eqref{f_estimate} and \eqref{f_estimate_2}. The positivity also follows from a standard sequential argument, we refer detail to \cite{duan_SIMA}. Note that the a priori assumption \eqref{ad.prop.ap} can be closed due to the smallness of initial data as in \eqref{initial_assumption}. 
\end{proof}

\section{Time derivative estimate and time-weighted dissipation estimate of $\hat{\mathbf{b}},\hat{c}$}\label{sec:time_derivative}

To conclude Theorem \ref{thm:bc_refined}, we need to obtain the time-weighted dissipation for the low-frequency regime of $\hat{b}_3,\hat{c}$ \eqref{b3_c_refined}, see Lemma \ref{lemma:b_c_no_time_derivative} for $\hat{b}_1$ and $\hat{b}_2$.

We can apply similar arguments to obtain the following estimate to $\p_t \hat{f}$ in \eqref{p_t_f_eqn}.

\begin{proposition}\label{prop:time_derivative}
Let the assumptions in Theorem \ref{thm:l1k_lpk} be satisfied. If we further assume the condition \eqref{t_derivative_assumption}, then there exists a unique solution $\p_t \hat{f}(t,k,x_3,v)$ to \eqref{p_t_f_eqn} and the following estimate is satisfied:
   \begin{align*}
    &    \Vert (1+t)^{\sigma/2} w\p_t \hat{ f} \Vert_{L^1_k L^\infty_{T,x_3,v}} + \Vert \p_t \hat{f}\Vert_{L^p_k L^\infty_T L^2_{x_3,v}} \\
& \lesssim \Vert w \p_t \hat{f}_0\Vert_{L^1_k L^\infty_{x_3,v}} + \Vert \p_t \hat{f}_0\Vert_{L^p_k L^2_{x_3,v}} + \Vert (1+t)^{\sigma/2} w \p_t \hat{f}\Vert_{L^1_k  L^\infty_{T,x_3,v}}\Vert (1+t)^{\sigma/2} w  \hat{f}\Vert_{L^1_k  L^\infty_{T,x_3,v}} \\
& + \Vert \p_t \hat{f}\Vert_{L^p_k L^\infty_T L^2_{x_3,v}} \Vert (1+t)^{\sigma/2} w\hat{f}\Vert_{L^1_k  L^\infty_{T,x_3,v}} + \Vert  \hat{f}\Vert_{L^p_k L^\infty_T L^2_{x_3,v}} \Vert (1+t)^{\sigma/2} w\p_t \hat{f}\Vert_{L^1_k  L^\infty_{T,x_3,v}},
\end{align*}
for any $T>0$. Moreover, it also holds that
\begin{align*}
    & \Vert (1+t)^{\sigma/2} \p_t \hat{f}\Vert_{L^1_k L^\infty_T L^2_{x_3,v}} + \Vert (1+t)^{\sigma/2} (\mathbf{I}-\mathbf{P}) \p_t \hat{f} \Vert_{L^1_k L^2_T L^2_{x_3,\nu}} + |(1+t)^{\sigma/2}(I-P_\gamma)\p_t \hat{f}|^2_{L^1_k L^2_{T,\gamma_+}}  \\
    &  + \Big\Vert (1+t)^{\sigma/2} \frac{|k|}{\sqrt{1+|k|^2}}\p_t(\hat{a},\hat{\mathbf{b}},\hat{c})\Big\Vert_{L^1_k L^2_{T,x_3}} \\
    & \lesssim \Vert w \p_t \hat{f}_0\Vert_{L^1_k L^\infty_{x_3,v}} + \Vert \p_t \hat{f}_0\Vert_{L^p_k L^2_{x_3,v}} + \Vert (1+t)^{\sigma/2} w \p_t \hat{f}\Vert_{L^1_k  L^\infty_{T,x_3,v}}\Vert (1+t)^{\sigma/2} w  \hat{f}\Vert_{L^1_k  L^\infty_{T,x_3,v}} \\
& + \Vert \p_t \hat{f}\Vert_{L^p_k L^\infty_T L^2_{x_3,v}} \Vert (1+t)^{\sigma/2} w\hat{f}\Vert_{L^1_k  L^\infty_{T,x_3,v}} + \Vert  \hat{f}\Vert_{L^p_k L^\infty_T L^2_{x_3,v}} \Vert (1+t)^{\sigma/2} w\p_t \hat{f}\Vert_{L^1_k  L^\infty_{T,x_3,v}}.
\end{align*}

\end{proposition}

\begin{proof}
The proof is almost identical to the proof of Theorem \ref{thm:l1k_lpk}. The only difference is that the nonlinear operator $\hat{\Gamma}(\hat{f},\hat{f})$ now becomes $\hat{\Gamma}(\p_t \hat{f}, \hat{f}) + \hat{\Gamma}(\hat{f},\p_t \hat{f})$. The contribution of this term can be controlled using Lemma \ref{lemma:gamma_est}.
\end{proof}

We estimate $\hat{b}_3$ and $\hat{c}$ in Section \ref{sec:b_3_c}. This will conclude Theorem \ref{thm:bc_refined} in Section \ref{sec:proof_thm2}.

\subsection{Time-weighted dissipation estimate of $\hat{\mathbf{b}},\hat{c}$}\label{sec:b_3_c}

First, we derive the dissipation estimate of $\hat{\mathbf{b}},\hat{c}$ without weight in time.
\begin{lemma}\label{lemma:bc_l2}
Under the assumption in Proposition \ref{prop:time_derivative}, it holds that
\begin{align*}
    &  \Vert (\hat{\mathbf{b}},\hat{c}) \Vert_{L^1_k L^2_{T,x_3}} \lesssim \Vert (\mathbf{I}-\mathbf{P})\hat{f}\Vert_{L^1_k L^2_{T,x_3,v}} + \Vert (\mathbf{I}-\mathbf{P}) \p_t \hat{f}\Vert_{L^1_k L^2_{T,x_3,v}} \\
    & + \Vert \hat{f}\Vert_{L^1_k L^\infty_T L^2_{x_3,v}} \Vert w\hat{f}\Vert_{L^1_k L^2_T L^\infty_{x_3,v}} + |(I-P_\gamma)\hat{f}|_{L^1_k L^2_{T,\gamma_+}} .
\end{align*}
\end{lemma}

\begin{proof}
We rewrite the weak formulation \eqref{weak_formulation} as
\begin{align}
    &  \underbrace{\int_0^{T}\int_{-1}^{1}\int_{\mathbb{R}^3} i\bar{v}\cdot k \hat{f} \psi \dd v \dd x_3 \dd t}_{\eqref{weak_formulation_2}_1} \underbrace{- \int_0^T \int_{-1}^1 \int_{\mathbb{R}^3} v_3 \hat{f} \p_{x_3}\psi \dd v \dd x_3 \dd t}_{\eqref{weak_formulation_2}_2} \notag\\
    &+ \underbrace{ \int_0^T \int_{\mathbb{R}^3} v_3 [\hat{f}(k,1)  \psi(1) - \hat{f}(k,-1) \psi(-1) ]\dd v \dd t}_{\eqref{weak_formulation_2}_3} \underbrace{- \int_0^T \int_{\mathbb{R}^3} \int_{-1}^1 \p_t \hat{f} \psi \dd v \dd x_3 \dd t}_{\eqref{weak_formulation_2}_4} \notag  \\
    &+ \underbrace{\int_0^T \int_{-1}^1 \int_{\mathbb{R}^3} \mathcal{L}(\hat{f}) \psi \dd v \dd x_3 \dd t}_{\eqref{weak_formulation_2}_5}  = \underbrace{\int_0^T\int_{-1}^1 \int_{\mathbb{R}^3} \hat{\Gamma}(\hat{f},\hat{f}) \psi \dd v \dd x_3 \dd t}_{\eqref{weak_formulation_2}_6}. \label{weak_formulation_2}
\end{align}

\textit{Estimate of $\hat{b}_3$.} 

From Lemma \ref{lemma:b_c_no_time_derivative}, we only need to estimate $\hat{b}_3$. We choose a test function as
\begin{align}
\begin{cases}
    &\dis  \psi_b = -v_1 v_3 \sqrt{\mu} i k_1\phi_b - v_2 v_3 \sqrt{\mu} i k_2 \phi_b + \frac{3}{2} \Big( |v_3|^2 - \frac{|v|^2}{3} \Big) \sqrt{\mu} \p_{x_3} \phi_b, \\
    &\dis  [|k_1|^2 + |k_2|^2 - 2\p_{x_3}^2]\phi_b = \bar{\hat{b}}_3, \\
    & \phi_b = 0 \text{ when } x_3 = \pm 1.    
\end{cases}\label{psi_b}
\end{align}

Similar to \eqref{phib_l2_2}, \eqref{phib_h2_2} and \eqref{phib_trace_2}, we have
\begin{align*}
    & \Vert (1+|k|)\phi_b\Vert^2_{L^2_{x_3}}+ \Vert \p_{x_3}\phi_b\Vert^2_{L^2_{x_3}} \lesssim \Vert \hat{b}_3\Vert^2_{L^2_{x_3}},  
\end{align*}
\begin{align}
    & \Vert |k|\p_{x_3}\phi_b\Vert^2_{L^2_{x_3}} \lesssim \Vert \hat{b}_3\Vert^2_{L^2_{x_3}},  \notag  \\
    & \Vert \p_{x_3}^2 \phi_b\Vert_{L^2_{x_3}}^2 \lesssim \Vert |k|^2 \phi_b\Vert_{L^2_{x_3}}^2 + \Vert \hat{b}_3\Vert_{L^2_{x_3}}^2 \lesssim \Vert \hat{b}_3\Vert^2_{L^2_{x_3}}. \notag
\end{align}
\begin{align}
    & | |k|\phi_b(k,\pm 1) |^2 \lesssim \Vert \hat{b}_3\Vert^2_{L^2_{x_3}}, \ | \p_{x_3}\phi_b(k,\pm 1)|^2 \lesssim \Vert \hat{b}_3\Vert_{L^2_{x_3}}^2. \notag
\end{align}

The computations for the terms in \eqref{weak_formulation_2} are all the same as those in the proof of Lemma \ref{lemma:b_c_no_time_derivative}, except $\eqref{weak_formulation_2}_4$. We only compute this term.

Due to the choice of $\psi_b$ in \eqref{psi_b}, we have
\begin{align}
    &  | \eqref{weak_formulation_2}_4 |= \Big|\int_0^T \int_{\mathbb{R}^3} \int_{-1}^1 \p_t (\mathbf{I}-\mathbf{P})\hat{f} \psi_b \dd v \dd x_3 \dd t \Big| \notag \\
    & \lesssim \Vert  (\mathbf{I}-\mathbf{P}) \p_t\hat{f} \Vert_{L^2_{T,x_3,v}}^2 + o(1)[\Vert |k|\phi_b\Vert_{L^2_{T,x_3}}^2 + \Vert \p_{x_3}\phi_b \Vert_{L^2_{T,x_3}}^2 ]  \lesssim o(1)\Vert \hat{b}_3\Vert_{L^2_{T,x_3}}^2 + \Vert  (\mathbf{I}-\mathbf{P}) \p_t\hat{f} \Vert_{L^2_{T,x_3,v}}^2 .  \notag
\end{align}

We conclude that
\begin{align}
    &  \Vert \hat{b}_3\Vert_{L^2_{T,x_3}} \lesssim |(I-P_\gamma)\hat{f}|_{L^2_{T,\gamma_+}} + \Vert (\mathbf{I}-\mathbf{P})\hat{f}\Vert_{L^2_{T,x_3,v}} + \Vert (\mathbf{I}-\mathbf{P})\p_t \hat{f}\Vert_{L^2_{T,x_3,v}} + \Vert \nu^{-1/2}\hat{\Gamma}(\hat{f},\hat{f})\Vert_{L^2_{T,x_3,v}}.  \label{b_bdd_l2}
\end{align}

\medskip
\textit{Estimate of $\hat{c}$.}

We choose a test function as
\begin{align}
    & \psi_c = (-ik_1v_1 \phi_c - ik_2 v_2 \phi_c + v_3 \p_{x_3}\phi_c)(|v|^2-5)\sqrt{\mu}, \label{psi_c}
\end{align}
with $\phi_c$ satisfying
\begin{align}
\begin{cases}
    &   |k|^2 \phi_c - \p_{x_3}^2 \phi_c = \bar{\hat{c}},  \\
    & \phi_c = 0 \ \text{ when } x_3 = \pm 1.    
\end{cases}\label{phi_c}
\end{align}

Multiplying \eqref{phi_c} by $\bar{\phi}_c$ and taking integration in $x_3$ we obtain
\begin{align*}
    &   \Vert |k| \phi_c \Vert_{L^2_{x_3}}^2 + \Vert \p_{x_3} \phi_c \Vert_{L^2_{x_3}}^2 \lesssim o(1) \Vert \phi_c\Vert_{L^2_{x_3}}^2 + \Vert \hat{c}\Vert_{L^2_{x_3}}^2.
\end{align*}
From the Poincar\'e inequality, we further have
\begin{align}
    & \Vert (1+|k|)\phi_c\Vert^2_{L^2_{x_3}}+ \Vert \p_{x_3}\phi_c\Vert^2_{L^2_{x_3}} \lesssim \Vert \hat{c}\Vert^2_{L^2_{x_3}}.  \label{phic_l2}
\end{align}
Multiplying \eqref{phi_c} by $|k|^2 \bar{\phi}_c$ we obtain
\begin{align*}
    &  \Vert |k|^2\phi_c \Vert_{L^2_{x_3}}^2 + \Vert |k|\p_{x_3}\phi_c\Vert^2_{L^2_{x_3}} \lesssim o(1) \Vert |k|^2 \phi_c\Vert_{L^2_{x_3}}^2 +  \Vert \hat{c}\Vert_{L^2_{x_3}}^2.
\end{align*}
Thus we conclude
\begin{align}
    & \Vert |k|\p_{x_3}\phi_c\Vert^2_{L^2_{x_3}} \lesssim \Vert \hat{c}\Vert^2_{L^2_{x_3}},  \notag  \\
    & \Vert \p_{x_3}^2 \phi_c\Vert_{L^2_{x_3}}^2 \lesssim \Vert |k|^2 \phi_c\Vert_{L^2_{x_3}}^2 + \Vert \hat{c}\Vert_{L^2_{x_3}}^2 \lesssim \Vert \hat{c}\Vert^2_{L^2_{x_3}}. \notag   
\end{align}

By trace theorem, we have
\begin{align}
    & | |k|\phi_c(k,\pm 1) |^2 \lesssim \Vert \hat{c}\Vert^2_{L^2_{x_3}}, \ | \p_{x_3}\phi_c(k,\pm 1)|^2 \lesssim \Vert \hat{c}\Vert_{L^2_{x_3}}^2. \label{phic_trace}
\end{align}

By the same computation of Lemma \ref{lemma:macroscopic}, we have
\begin{align}
    &     \eqref{weak_formulation_2}_1 + \eqref{weak_formulation_2}_2 = 5\int_0^T \int_{-1}^1 [ |k|^2 - \p_{x_3}^2]\phi_c \hat{c} \dd x_3 \dd t + E_3 + E_4  = 5\Vert \hat{c}\Vert_{L^2_{T,x_3}}^2 + E_3 + E_4, \label{est_c_LHS}
\end{align}
with 
\begin{align}
    &   |E_3| + |E_4| \lesssim o(1)[\Vert |k|\phi_c\Vert_{L^2_{T,x_3,v}}^2 + \Vert \p_{x_3}\phi_c \Vert_{L^2_{T,x_3}}^2] + \Vert (\mathbf{I}-\mathbf{P})\hat{f}\Vert_{L^2_{T,x_3}}^2  \notag\\
    & \lesssim o(1)\Vert \hat{c}\Vert^2_{L^2_{T,x_3}} + \Vert (\mathbf{I}-\mathbf{P})\hat{f}\Vert^2_{L^2_{T,x_3,v}}. \label{est_c_E1}
\end{align}

Then we compute the boundary term $\eqref{weak_formulation_2}_3$. For the contribution of $P_\gamma \hat{f}$, by the same computation of Lemma \ref{lemma:macroscopic} we have
\begin{align*}
    &   \int_0^T \int_{\mathbb{R}^3} v_3   P_\gamma \hat{f}(k,1) \psi_c(1) \dd v \dd t = 0.
\end{align*}

For the part with $(I-P_\gamma)\hat{f}$, we derive that
\begin{align}
    & \int_0^T \int_{v_3>0} |(I-P_\gamma)\hat{f}(k,1) (-ik_1v_1 \phi_c - ik_2 v_2 \phi_c + v_3 \p_{x_3}\phi_c)(|v|^2-5)\sqrt{\mu} |\dd v \dd t \notag \\
    & \lesssim o(1) [| |k| \phi_c(k,1) |^2 +  | \p_{x_3} \phi_c(k,1)|^2  ]   +    |(I-P_\gamma)\hat{f}|_{L^2_{T,\gamma_+}}^2 \notag \\
    & \lesssim o(1) \Vert \hat{c} \Vert_{L^2_{T,x_3}}^2 + |(I-P_\gamma)\hat{f}|_{L^2_{T,\gamma_+}}^2.  \notag
\end{align}
In the last line, we have used the trace estimate \eqref{phic_trace}.

Similarly, for $x_3=-1$ we have the same estimate. Thus we conclude that
\begin{align}
    &  |\eqref{weak_formulation_2}_3| \lesssim o(1) \Vert \hat{c} \Vert_{L^2_{T,x_3}}^2 + |(I-P_\gamma)\hat{f}|_{L^2_{T,\gamma_+}}^2.   \label{est_c_bdr}
\end{align}

Then we compute the contribution of the time derivative $\eqref{weak_formulation_2}_4$. Due to the choice of $\psi_c$ in \eqref{psi_c}, we have
\begin{align}
    &   |\eqref{weak_formulation_2}_4| \leq \int_0^T \int_{\mathbb{R}^3} \int_{-1}^1 |\p_t (\mathbf{I}-\mathbf{P})\hat{f} \psi_c| \dd v \dd x_3 \dd t \notag \\
    & \lesssim \Vert  (\mathbf{I}-\mathbf{P}) \p_t\hat{f} \Vert_{L^2_{T,x_3,v}}^2 + o(1)[\Vert |k|\phi_c\Vert_{L^2_{T,x_3}}^2 + \Vert \p_{x_3}\phi_c \Vert_{L^2_{T,x_3}}^2 ]  \lesssim o(1)\Vert \hat{c}\Vert_{L^2_{T,x_3}}^2 + \Vert  (\mathbf{I}-\mathbf{P}) \p_t\hat{f} \Vert_{L^2_{T,x_3,v}}^2 .  \label{est_c_RHS_4}
\end{align}

Last we compute $\eqref{weak_formulation_2}_5$ and $\eqref{weak_formulation_2}_6$ as
\begin{align}
    &    |\eqref{weak_formulation_2}_5| \lesssim o(1)[\Vert k \phi_c \Vert_{L^2_{T,x_3}}^2 + \Vert \p_{x_3} \phi_c \Vert_{L^2_{T,x_3}}^2 ] + \Vert (\mathbf{I}-\mathbf{P})\hat{f}\Vert_{L^2_{T,x_3,v}}^2   \lesssim o(1)\Vert \hat{c}  \Vert^2_{L^2_{T,x_3}}  + \Vert (\mathbf{I}-\mathbf{P})\hat{f}\Vert^2_{L^2_{T,x_3,v}},   \label{est_c_RHS_5}
\end{align}
and
\begin{align}
    & |\eqref{weak_formulation_2}_6| \leq   \int_0^T   \int_{-1}^1 \int_{\mathbb{R}^3} |\hat{\Gamma}(\hat{f},\hat{f})\psi_c| \dd v \dd x_3 \dd t \notag \\
    &\lesssim o(1)[\Vert k \phi_c \Vert_{L^2_{T,x_3}}^2 + \Vert \p_{x_3} \phi_c \Vert_{L^2_{T,x_3}}^2 ] + \Big(\int_{\mathbb{R}^2} \Vert \hat{f}(k-\ell)\Vert_{L^\infty_T L^2_{x_3,v}} \Vert \hat{f}(\ell)\Vert_{L^2_T L^\infty_{x_3}L^2_\nu} \dd \ell\Big)^2 \notag \\
    &\lesssim  o(1)\Vert \hat{c}  \Vert^2_{L^2_{T,x_3}} + \Big(\int_{\mathbb{R}^2} \Vert \hat{f}(k-\ell)\Vert_{L^\infty_T L^2_{x_3,v}} \Vert \hat{f}(\ell)\Vert_{L^2_T L^\infty_{x_3}L^2_\nu} \dd \ell\Big)^2.   \label{est_c_RHS_6}
\end{align}

We combine \labelcref{est_c_LHS,est_c_E1,est_c_bdr,est_c_RHS_4,est_c_RHS_5,est_c_RHS_6} to conclude the estimate for $\hat{c}$:
\begin{align}
    &  \Vert \hat{c}\Vert_{L^2_{T,x_3}} \lesssim |(I-P_\gamma)\hat{f}|_{L^2_{T,\gamma_+}} + \Vert (\mathbf{I}-\mathbf{P})\hat{f}\Vert_{L^2_{T,x_3,v}} + \Vert (\mathbf{I}-\mathbf{P})\p_t \hat{f}\Vert_{L^2_{T,x_3,v}} + \Vert \nu^{-1/2}\hat{\Gamma}(\hat{f},\hat{f})\Vert_{L^2_{T,x_3,v}}.  \label{c_bdd_l2}
\end{align}

Combining \eqref{b_bdd_l2} and \eqref{c_bdd_l2}, we take the $k$-integration and obtain
\begin{align*}
    &   \Vert (\hat{\mathbf{b}},\hat{c})\Vert_{L^1_k L^2_{T,x_3}} \lesssim |(I-P_\gamma)\hat{f}|_{L^1_k L^2_{T,\gamma_+}} + \Vert (\mathbf{I}-\mathbf{P})\hat{f}\Vert_{L^1_k L^2_{T,x_3,v}} + \Vert (\mathbf{I}-\mathbf{P})\p_t \hat{f}\Vert_{L^1_k L^2_{T,x_3,v}} \\
    & +\Vert \hat{f} \Vert_{L^1_k L^\infty_T L^2_{x_3,v}} \Vert w\hat{f}\Vert_{L^1_k L^2_T L^\infty_{x_3,v}}.
\end{align*}
Here we have used the same computation in \eqref{convolution}.

We conclude the proof of Lemma \ref{lemma:bc_l2}.
\end{proof}

In the next lemma, we further include the weight in time.
 
\begin{lemma}[\textbf{Time-weighted dissipation estimate of $\hat{\mathbf{b}}$ and $\hat{c}$}]\label{lemma:bc_l2_time}
The time-weight can be included into Lemma \ref{lemma:bc_l2} as follows:
\begin{align*}
    &  \Vert (1+t)^{\sigma/2}(\hat{\mathbf{b}},\hat{c}) \Vert_{L^1_k L^2_{T,x_3}} \lesssim \Vert (1+t)^{\sigma/2}(\mathbf{I}-\mathbf{P})\hat{f}\Vert_{L^1_k L^2_{T,x_3,v}} + \Vert (1+t)^{\sigma/2}(\mathbf{I}-\mathbf{P}) \p_t \hat{f}\Vert_{L^1_k L^2_{T,x_3,v}} \\
    & + \Vert (1+t)^{\sigma/2}\hat{f}\Vert_{L^1_k L^\infty_T L^2_{x_3,v}} \Vert w\hat{f}\Vert_{L^1_k L^2_T L^\infty_{x_3,v}} + |(1+t)^{\sigma/2}(I-P_\gamma)\hat{f}|_{L^1_k L^2_{T,\gamma_+}} .
\end{align*}

\end{lemma}

\begin{proof}
The proof is the same by rewriting \eqref{weak_formulation_2} into the following form:
\begin{align*}
    &  \int_0^{T}\int_{-1}^{1}\int_{\mathbb{R}^3} i\bar{v}\cdot k (1+t)^\sigma \hat{f} \psi \dd v \dd x_3 \dd t - \int_0^T \int_{-1}^1 \int_{\mathbb{R}^3} v_3 (1+t)^\sigma\hat{f} \p_{x_3}\psi \dd v \dd x_3 \dd t \notag\\
    &+  \int_0^T \int_{\mathbb{R}^3} v_3 [(1+t)^\sigma\hat{f}(k,1)  \psi(1) - (1+t)^\sigma\hat{f}(k,-1) \psi(-1) ]\dd v \dd t \\
    &- \int_0^T  \int_{-1}^1 \int_{\mathbb{R}^3} (1+t)^\sigma\p_t \hat{f} \psi \dd v \dd x_3 \dd t \notag  \\
    &+ \int_0^T \int_{-1}^1 \int_{\mathbb{R}^3} \mathcal{L}((1+t)^\sigma\hat{f}) \psi \dd v \dd x_3 \dd t  = \int_0^T\int_{-1}^1 \int_{\mathbb{R}^3} (1+t)^\sigma\hat{\Gamma}(\hat{f},\hat{f}) \psi \dd v \dd x_3 \dd t. 
\end{align*}
As in the proof of Lemma \ref{lemma:bc_l2}, we still choose the same test functions $\psi_b,\psi_c \perp \ker \mathcal{L}$. Then the third line becomes
\begin{align*}
    & -\int_0^T  \int_{-1}^1 \int_{\mathbb{R}^3} (1+t)^\sigma (\mathbf{I}-\mathbf{P}) \p_t \hat{f} \psi \dd v \dd x_3 \dd t .
\end{align*}
Here we note that we do not make commutation of the time weight $(1+t)^\sigma$ with the $\p_t$-differentiation, and thus such term is bounded by 
\begin{align*}
    & \Vert (1+t)^{\sigma/2} (\mathbf{I}-\mathbf{P})\p_t \hat{f}\Vert_{L^2_{T,x_3,v}}^2 + o(1)\Vert (1+t)^{\sigma/2} \psi\Vert_{L^2_{T,x_3,v}}^2.
\end{align*}

For the nonlinear term $\hat{\Gamma}(\hat{f},\hat{f})$ we can apply \eqref{gamma_est_time}.

The proof of the rest terms is the same.
\end{proof}

\subsection{Proof of Theorem \ref{thm:bc_refined}}\label{sec:proof_thm2}

The proof of Theorem \ref{thm:bc_refined} follows from combining Lemma \ref{lemma:b_c_no_time_derivative} and Lemma \ref{lemma:bc_l2_time}.

\section{$L^2_{x,v}$-$L^\infty_{x,v}$ argument in physical space for two-dimensional problem}\label{sec:2d}

In this section, we will prove Theorem \ref{thm:2d}. For this purpose, we consider the two-dimensional infinite layer problem \eqref{2dprob} or equivalently \eqref{f_eqn_2d} on the Boltzmann equation for rarefied gas in $\O=\R\times (-1,1)$ confined between two parallel plates. Through this section we use the simplified notations for convenience: $x=(x_1,x_3)\in \O$,  $\partial_i=\partial_{x_i}$ for $i=1,3$, $\nabla=(\p_1,\p_3),$ and $\Delta = \p_{11}+\p_{33}$. Since $x$ is a two-dimensional variable and $v$ is a three-dimensional variable, along the characteristic, we will use the notation $x-(t-s)(v_1,v_3)$ in Section \ref{sec:linfty_2d}.

To establish the global-in-time existence of solutions with the corresponding estimates \eqref{thm:2d:e1} and \eqref{thm:2d:e2} under the smallness condition \eqref{f_initial_2d}, we define the norm of solutions as
\begin{align*}
    & \Vert f\Vert_T := \Vert f\Vert_{L^\infty_T L^2_{x,v}}^2 + \Vert \p_t f\Vert_{L^\infty_T L^2_{x,v}}^2 +  |(I-P_\gamma)f|_{L^2_{T,\gamma_+}}^2  +  |(I-P_\gamma)\p_t f|_{L^2_{T,\gamma_+}}^2   \\
    & +  \Vert \nu^{1/2}(\mathbf{I}-\mathbf{P})f\Vert_{L^2_{T,x,v}}^2  +  \Vert \nu^{1/2}(\mathbf{I}-\mathbf{P})\p_t f\Vert_{L^2_{T,x,v}}^2  + \Vert \mathbf{b}\Vert_{L^2_{T,x}}^2 + \Vert c\Vert_{L^2_{T,x}}^2 +  \Vert wf\Vert_{L^\infty_{T,x,v}}^2 +  \Vert w\p_t f\Vert_{L^\infty_{T,x,v}}^2,
\end{align*}
with $T>0$.

We mainly focus on the following a priori estimate.

\begin{proposition}\label{prop:apriori}
Suppose $f,\p_t f$ are the solutions to \eqref{f_eqn_2d} and \eqref{p_t_eqn_2d} such that $\Vert f\Vert_T < \infty$ for any $T$. Then there exist $0<\delta\ll 1$ and $C$, which are independent of $T$, such that if
\begin{align*}
    & \Vert f(0)\Vert_{L^2_{x,v}}^2 + \Vert \p_t f(0)\Vert_{L^2_{x,v}}^2 + \Vert wf(0)\Vert_{L^\infty_{x,v}}^2 + \Vert w\p_t f(0)\Vert_{L^\infty_{x,v}}^2 <\delta,
\end{align*}
then the solution $f$ satisfies the uniform estimate
\begin{align*}
    & \Vert f\Vert_T \leq  C\delta + C\Vert f\Vert_T^2.
\end{align*}

\end{proposition}

This section is organized as follows: Section \ref{sec:energy_2d} is devoted to the $L^2_{x,v}$ energy estimate of both $f$ and $\p_t f$, where only $\mathbf{b},c$ dissipation estimates are computed. Section \ref{sec:linfty_2d} is devoted to the $L^\infty$ estimate of $f,\p_t f$ using the method of characteristics.
We conclude Proposition \ref{prop:apriori} and Theorem \ref{thm:2d} in Section \ref{sec:thm_2d_proof}.

\subsection{$L^\infty_T L^2_{x,v}$ energy estimate and dissipation estimate of $\mathbf{b},c$}\label{sec:energy_2d}
In this subsection, we construct the energy estimate to both $f$ and $\p_t f$. The estimate is given by the following lemma.

\begin{lemma}\label{lemma:energy}
Suppose the assumptions in Proposition \ref{prop:apriori} are satisfied, then
\begin{align*}
    &    \Vert \p_t f(T)\Vert_{L^2_{x,v}}^2 + \Vert f(T)\Vert_{L^2_{x,v}}^2 + \int_0^T |(I-P_\gamma)f|_{L^2_{\gamma_+}}^2 \dd t + \int_0^T |(I-P_\gamma)\p_t f|_{L^2_{\gamma_+}}^2 \dd t  \\
    &+ \int_0^T \Vert \nu^{1/2}(\mathbf{I}-\mathbf{P})f\Vert_{L^2_{x,v}}^2 \dd t+ \int_0^T \Vert \nu^{1/2}(\mathbf{I}-\mathbf{P})\p_t f\Vert_{L^2_{x,v}}^2 \dd t + \int_0^T [\Vert \mathbf{b}\Vert_{L^2_{x}}^2 + \Vert c\Vert_{L^2_{x}}^2] \dd t  \\
    & \lesssim \Vert \p_t f(0)\Vert_{L^2_{x,v}}^2 + \Vert f(0)\Vert_{L^2_{x,v}}^2 +[\Vert wf\Vert_{L^\infty_{T,x,v}}^2 + \Vert w\p_t f\Vert_{L^\infty_{T,x,v}}^2]\\
    & \times \Big[\int_0^T \Vert \nu^{1/2}(\mathbf{I}-\mathbf{P})f\Vert_{L^2_{x,v}}^2 \dd t + \int_0^T \Vert \nu^{1/2}(\mathbf{I}-\mathbf{P})\p_t f\Vert_{L^2_{x,v}}^2 \dd t + \int_0^T [\Vert \mathbf{b}\Vert_{L^2_{x}}^2 + \Vert c\Vert_{L^2_{x}}^2] \dd t  \Big].
\end{align*}

\end{lemma}

This lemma follows from a basic energy estimate in Lemma \ref{lemma:basic_energy_2d} and the macroscopic dissipation estimate in Lemma \ref{lemma:b_c}.

\begin{lemma}\label{lemma:basic_energy_2d}
Under the assumption in Lemma \ref{lemma:energy}, we have the following basic $L^2$ energy estimates to $f$ and $\p_t f$:
\begin{align}
    &     \Vert f(T)\Vert_{L^2_{x,v}}^2 + \int_0^T |(I-P_\gamma)f|_{L^2_{\gamma_+}}^2 \dd t + \int_0^T \Vert \nu^{1/2} (\mathbf{I}-\mathbf{P})f\Vert_{L^2_{x,v}}^2 \dd t \notag\\
    & \lesssim \Vert f(0)\Vert_{L^2_{x,v}}^2 + \Vert wf\Vert_{L^\infty_{T,x,v}}^2 \int_0^T [\Vert \mathbf{b}\Vert_{L^2_{x}}^2 + \Vert c\Vert_{L^2_{x}}^2] \dd t + \Vert wf\Vert_{L^\infty_{T,x,v}}^2\int_0^T \Vert \nu^{1/2}(\mathbf{I}-\mathbf{P})f\Vert_{L^2_{x,v}}^2 \dd t, \label{f_energy}
\end{align}
and
\begin{align}
    &  \Vert \p_t f(T)\Vert_{L^2_{x,v}}^2 + \int_0^T |(I-P_\gamma)\p_t f|_{L^2_{\gamma_+}}^2 \dd t + \int_0^T \Vert \nu^{1/2}(\mathbf{I}-\mathbf{P})\p_t f\Vert_{L^2_{x,v}}^2 \dd t  \notag\\
    & \lesssim   \Vert \p_t f(0)\Vert_{L^2_{x,v}}^2 + \Vert w\p_t f\Vert_{L^\infty_{T,x,v}}^2\int_0^T \Vert \nu^{1/2}(\mathbf{I}-\mathbf{P})f\Vert_{L^2_{x,v}}^2 \dd t    \label{p_t_energy}\\
    & + \Vert wf\Vert_{L^\infty_{T,x,v}}^2 \int_0^T \Vert \p_t (\mathbf{I}-\mathbf{P})f\Vert_{L^2_{x,v}}^2 + \Vert w\p_t f\Vert_{L^\infty_{T,x,v}}^2 \int_0^T [\Vert \mathbf{b}\Vert_{L^2_{x}}^2 + \Vert c\Vert_{L^2_{x}}^2] \dd t. \notag
\end{align}
\end{lemma}

\begin{proof}
The $L^2$ energy estimate of \eqref{f_eqn_2d} leads to
\begin{align*}
    &      \Vert f(T)\Vert_{L^2_{x,v}}^2 + \int_0^T |(I-P_\gamma)f|_{L^2_{\gamma_+}}^2 \dd t +  \int_0^T \Vert \nu^{1/2} (\mathbf{I}-\mathbf{P})f\Vert_{L^2_{x,v}}^2 \dd t \\
    &\lesssim \Vert f(0)\Vert_{L^2_{x,v}}^2 + \int_0^T \int_{\mathbb{R}^3}\int_{\O} |\Gamma(f,f)(\mathbf{I}-\mathbf{P})f| \dd x \dd v \dd t \\
    & \lesssim \Vert f(0)\Vert_{L^2_{x,v}}^2 + \int_0^T \Vert \nu^{-1/2}\Gamma(f,f)\Vert_{L^2_{x,v}}^2 \dd t + o(1) \int_0^T \Vert \nu^{1/2}(\mathbf{I} - \mathbf{P})f \Vert_{L^2_{x,v}}^2 \dd t. 
\end{align*}

We compute the nonlinear term as
\begin{align}
    &    \Vert \nu^{-1/2}\Gamma(f,f)\Vert_{L^2_{x,v}}^2 \leq \Vert \nu^{-1/2}\Gamma(\mathbf{P}f,\mathbf{P}f)\Vert_{L^2_{x,v}}^2 + \Vert \nu^{-1/2}\Gamma(f,(\mathbf{I}-\mathbf{P})f) \Vert_{L^2_{x,v}}^2 + \Vert \nu^{-1/2}\Gamma((\mathbf{I}-\mathbf{P})f,f)\Vert_{L^2_{x,v}}^2 \notag\\
    & \lesssim \Vert wf\Vert_{L^\infty_{T,x,v}}^2[\Vert \mathbf{b}\Vert_{L^2_{x}}^2 + \Vert c\Vert_{L^2_{x}}^2 + \Vert \nu^{1/2}(\mathbf{I}-\mathbf{P})f\Vert_{L^2_{x,v}}^2]. \label{gamma_l2_est}
\end{align}
This concludes \eqref{f_energy}.

The $L^2$ energy estimate of $\p_t f$ leads to
\begin{align*}
    &    \Vert \p_t f(T)\Vert_{L^2_{x,v}}^2 + \int_0^T |(I-P_\gamma)\p_t f|_{L^2_{\gamma_+}}^2 \dd t + \int_0^T \Vert \nu^{1/2} (\mathbf{I}-\mathbf{P})\p_t f\Vert_{L^2_{x,v}}^2 \dd t \\
    & \lesssim \Vert \p_t f(0)\Vert_{L^2_{x,v}}^2 + \int_0^T \int_{\mathbb{R}^3}\int_{\O} |[\Gamma(f,\p_t f) + \Gamma(\p_t f,f)](\mathbf{I}-\mathbf{P})\p_t f| \dd x \dd v \dd t \\
    & \lesssim \Vert \p_t f(0)\Vert_{L^2_{x,v}}^2 + \int_0^T  \Vert \nu^{-1/2}[\Gamma(f,\p_t f) + \Gamma(\p_t f,f)] \Vert_{L^2_{x,v}}^2 \dd t \\
    &+ o(1)\int_0^T  \Vert \nu^{1/2}(\mathbf{I}-\mathbf{P})\p_t f\Vert_{L^2_{x,v}}^2 \dd t         .
\end{align*}

We control the nonlinear operator as
\begin{align*}
    &  \Vert \nu^{-1/2}[\Gamma(f,\p_t f) + \Gamma(\p_t f,f)]\Vert_{L^2_{x,v}}^2 \\
    &\lesssim \Vert \nu^{-1/2} [\Gamma((\mathbf{I}-\mathbf{P})f, \p_t f) + \Gamma(f, \p_t (\mathbf{I}-\mathbf{P})f) + \Gamma((\mathbf{I}-\mathbf{P})\p_t f,  f) \\
    & \ \ + \Gamma(\p_t f,  (\mathbf{I}-\mathbf{P})f) + \Gamma(\mathbf{P}f , \p_t \mathbf{P}f) + \Gamma(\p_t \mathbf{P}f, \mathbf{P}f) ]\Vert_{L^2_{x,v}}^2 \\
    & \lesssim \Vert w\p_t f\Vert_{L^\infty_{T,x,v}}^2 \Vert \nu^{1/2}(\mathbf{I}-\mathbf{P})f\Vert_{L^2_{x,v}}^2 + \Vert wf\Vert_{L^\infty_{T,x,v}}^2 \Vert \p_t (\mathbf{I}-\mathbf{P})f\Vert_{L^2_{x,v}}^2 + [\Vert \mathbf{b}\Vert_{L^2_{x}}^2 + \Vert c\Vert_{L^2_{x}}^2] \Vert w\p_t f\Vert_{L^\infty_{T,x,v}}^2.
\end{align*}
This concludes \eqref{p_t_energy}.
\end{proof}

In RHS of the basic energy estimates \eqref{f_energy} and \eqref{p_t_energy}, it suffices to control the dissipation estimate for $\mathbf{b}$ and $c$. We note that we do not need to estimate $\p_t \mathbf{b}, \p_t c$. In the following lemma, we derive these estimates by using the Poincar\'e inequality in the weak formulation.

\begin{lemma}\label{lemma:b_c}
It holds that
    \begin{align*}
    &    \Vert c\Vert_{L^2_{x}}^2+  \Vert \mathbf{b}\Vert_{L^2_{x}}^2 \lesssim \Vert (\mathbf{I}-\mathbf{P})f\Vert_{L^2_{x,v}}^2 + |(I-P_\gamma)f|_{L^2_{\gamma_+}}^2 + \Vert \nu^{-1/2}\Gamma(f,f)\Vert_{L^2_{x,v}}^2 + \Vert \p_t (\mathbf{I}-\mathbf{P})f\Vert_{L^2_{x,v}}^2.
\end{align*}

From \eqref{gamma_l2_est}, this implies that
\begin{align}
    &    \int_0^T [\Vert \mathbf{b}\Vert_{L^2_{x}}^2 + \Vert c\Vert_{L^2_{x}}^2] \dd t \lesssim [1+\Vert wf\Vert_{L^\infty_{T,x,v}}^2]\int_0^T \Vert (\mathbf{I}-\mathbf{P})f\Vert_{L^2_{x,v}}^2 \dd t + \int_0^T |(I-P_\gamma)f|_{L^2_{\gamma_+}}^2 \dd t \notag \\
    & \ \ \ \ \ \ \ \ \ \ \ \ \ \  \ \ \ \ \ \ \ \ \ \ \ \ + \Vert wf\Vert_{L^\infty_{T,x,v}}^2 \int_0^T [\Vert \mathbf{b}\Vert_{L^2_{x}}^2 + \Vert c\Vert_{L^2_{x}}^2] \dd t + \int_0^T \Vert \p_t (\mathbf{I}-\mathbf{P})f\Vert_{L^2_{x,v}}^2 \dd t.   \notag
\end{align}

\end{lemma}

\begin{proof}
We use the weak formulation:
\begin{align}
  & \underbrace{-\int_{\mathbb{R}^3}\int_{\O} f (v_1\p_1 \psi + v_3 \p_3 \psi) \dd x \dd v}_{\eqref{weak_formula}_1} + \underbrace{\int_{\mathbb{R}^3} \int_{\mathbb{R}} [f \psi(x_3=1) - f\psi(x_3=-1)]v_3  \dd x_1 \dd v}_{\eqref{weak_formula}_2}  \notag \\
  &+ \underbrace{\int_{\mathbb{R}^3}\int_{\O} \mathcal{L}f \psi \dd x \dd v}_{\eqref{weak_formula}_3} = \underbrace{\int_{\mathbb{R}^3}\int_{\O} \Gamma(f,f) \psi \dd x \dd v}_{\eqref{weak_formula}_4} - \underbrace{\int_{\mathbb{R}^3}\int_{\O} \p_t f \psi \dd x \dd v}_{\eqref{weak_formula}_5}.   \label{weak_formula}
\end{align}

\textit{Estimate of $\mathbf{b}$.} 

We choose a test function as
\begin{align*}
    &   \psi =  \frac{3}{2}\Big( |v_1|^2 - \frac{|v|^2}{3} \Big)\sqrt{\mu} \p_1 \phi_1 + v_1 v_3 \sqrt{\mu} \p_3 \phi_1.
\end{align*}
Then $\eqref{weak_formula}_1$ becomes
\begin{align*}
   \eqref{weak_formula}_1 = -\int_{\mathbb{R}^3}\int_{\O} \sum_{i,j=1,3} v_i v_j \p_i \psi b_j \sqrt{\mu} \dd x \dd v \underbrace{- \int_{\mathbb{R}^3}\int_{\O} (\mathbf{I}-\mathbf{P})f (v_1\p_1 \psi + v_3\p_3 \psi) \dd x \dd v}_{E_1}.
\end{align*}
In the first term, the contribution of $a$ and $c$ vanish due to oddness.  

Again from the oddness, for the first term, we have
\begin{align*}
    &     \int_{\mathbb{R}^3}\int_{\O} \big[\frac{3}{2}v_1^2 \Big(|v_1|^2 - \frac{|v|^2}{3} \Big) \mu \p_{11} \phi_1 b_1  + v_1^2 v_3^2  \mu\p_{13} \phi_1 b_3 \\
    & + v_1^2 v_3^2 \mu \p_{33} \phi_1 b_1 + \frac{3}{2} v_3^2 \Big( |v_1|^2 - \frac{|v|^2}{3}\Big) \mu \p_{13} \phi_1 b_3\big] \dd x \dd v \\
    & = 2 \p_{11}\phi_1 b_1+  \p_{33}\phi_1 b_1 +  \p_{13}\phi_1 b_3  - \p_{13}\phi_1 b_3  = (\Delta \phi_1 + \p_{11}\phi_1) b_1.
 \end{align*}

\hide
We choose a test function as
\begin{align*}
    &   \psi =  \frac{3}{2}\Big( |v_1|^2 - \frac{|v|^2}{3} \Big)\sqrt{\mu} \p_1 \phi_1 + v_1 v_2  \sqrt{\mu} \p_2 \phi_1 + v_1 v_3 \sqrt{\mu} \p_3 \phi_1.
\end{align*}
Then $\eqref{weak_formula}_1$ becomes
\begin{align*}
   \eqref{weak_formula}_1 = \int_{\mathbb{R}^3}\int_{\O} \sum_{i,j=1}^3 v_i v_j \p_i \psi b_j \sqrt{\mu} \dd x \dd v + E_1.
\end{align*}
In the first term, the contribution of $a$ and $c$ vanish due to oddness. Here $E_1:= \int_{\mathbb{R}^3}\int_{\O} (\mathbf{I}-\mathbf{P})f (v\cdot \nabla_x \psi) \dd x \dd v $. 

Again from the oddness, for the first term, we have
\begin{align*}
    &     \int_{\mathbb{R}^3}\int_{\O} \frac{3}{2}v_1^2 \Big(|v_1|^2 - \frac{|v|^2}{3} \Big) \mu \p_{11} \phi_1 b_1 + v_1^2 v_2^2 \mu \p_{12} \phi_1 b_2 + v_1^2 v_3^2  \mu\p_{13} \phi_1 b_3 \\
    & + v_1^2 v_2^2 \mu \p_{22}\phi_1 b_1 + \frac{3}{2} v_2^2 \Big(|v_1|^2 - \frac{|v|^2}{3} \Big) \mu \p_{12}\phi_1 b_2  \\
    & + v_1^2 v_3^2 \mu \p_{33} \phi_1 b_1 + \frac{3}{2} v_3^2 \Big( |v_1|^2 - \frac{|v|^2}{3}\Big) \mu \p_{13} \phi_1 b_3 \\
    & = 2 \p_{11}\phi_1 b_1+ \p_{22}\phi_1 b_1 + \p_{33}\phi_1 b_1 + \p_{12}\phi_{1}b_2 + \p_{13}\phi_1 b_3 - \p_{12}\phi_1 b_2 - \p_{13}\phi_1 b_3 \\
    & = (\Delta \phi_1 + \p_{11}\phi_1) b_1.
 \end{align*}
\unhide

We let $\phi_1$ solve the elliptic equation
\begin{align*}
\begin{cases}
    &     \Delta \phi_1 + \p_{11}\phi_1  =  -b_1, \\
    &     \phi_1 = 0 \text{ when }x_3 = \pm 1.    
\end{cases}
\end{align*}
Note that such elliptic equation can be reduced to the Poisson equation by setting $x_1 \to \sqrt{2}x_1$. From the Poincar\'e inequality, there exists a unique $\phi_1$ such that
\begin{align*}
    &    \Vert \phi_1 \Vert_{H^2_x} \lesssim    \Vert b_1\Vert_{L^2_{x}}.
\end{align*}

We conclude that
\begin{align*}
    &    \eqref{weak_formula}_1 = \Vert b_1\Vert_{L^2_{x}}^2 + E_1,
\end{align*}
with
\begin{align*}
    & |E_1| \lesssim o(1)\Vert b_1\Vert_{L^2_{x}}^2 + \Vert (\mathbf{I}-\mathbf{P})f\Vert_{L^2_{x,v}}^2.
\end{align*}

For the boundary term $\eqref{weak_formula}_2$, we apply the diffuse boundary condition to have
\begin{align*}
    & \int_{\mathbb{R}^3}\int_{\mathbb{R}} f\psi(x_3=1) v_3 \dd x_1 \dd v =  \int_{\mathbb{R}^3}\int_{\mathbb{R}} P_\gamma f \psi v_3 \dd x_1 \dd v + \int_{v_3>0}\int_{\mathbb{R}} (I-P_\gamma)f \psi(1) v_3 \dd x_1 \dd v \\
    & = \int_{v_3>0}\int_{\mathbb{R}} (I-P_\gamma)f \psi(1) v_3 \dd x_1 \dd v  \lesssim |(I-P_\gamma)f|_{L^2_{\gamma_+}}^2 + o(1) \int_{\mathbb{R}} |\nabla \phi(x_1,1)|^2 \dd x_1 \\
    & \lesssim |(I-P_\gamma)f|_{L^2_{\gamma_+}}^2 + o(1)\Vert \phi\Vert^2_{H^2_x} \lesssim |(I-P_\gamma)f|_{L^2_{\gamma_+}}^2 + o(1)\Vert b_1\Vert_{L^2_{x}}^2 .
\end{align*}
In the second line, the contribution of $P_\gamma f$ vanished due to oddness. In the third line, we applied the trace theorem.

For $x_3=-1$ we have a similar estimate, and thus we conclude that
\begin{align*}
    & \eqref{weak_formula}_2 \lesssim |(I-P_\gamma)f|_{L^2_{\gamma_+}}^2 + o(1)\Vert b_1\Vert_{L^2_{x}}^2.
\end{align*}

Note that $\eqref{weak_formula}_3$ is bounded as
\begin{align*}
    &   | \eqref{weak_formula}_3 |\lesssim \Vert (\mathbf{I}-\mathbf{P})f\Vert_{L^2_{x,v}}^2 + o(1)\Vert \phi_1\Vert_{H^1_x}^2 \lesssim \Vert (\mathbf{I}-\mathbf{P})f\Vert_{L^2_{x,v}}^2 + o(1)\Vert b_1\Vert_{L^2_{x}}^2,
\end{align*}
and
$\eqref{weak_formula}_4$ is bounded as
\begin{align*}
    &    |\eqref{weak_formula}_4| \lesssim \Vert \nu^{-1/2}\Gamma(f,f)\Vert_{L^2_{x,v}}^2 + o(1)\Vert \phi_1\Vert_{H^1_x}^2 \lesssim \Vert \nu^{-1/2}\Gamma(f,f)\Vert_{L^2_{x,v}}^2 + o(1)\Vert b_1\Vert_{L^2_{x}}^2.
\end{align*}

For $\eqref{weak_formula}_5$, since $\psi \perp \ker\mathcal{L}$, we have
\begin{align*}
    & |\eqref{weak_formula}_5|  = \int_{\mathbb{R}^3}\int_{\O} \p_t (\mathbf{I}-\mathbf{P})f \psi \dd x \dd v \lesssim \Vert \p_t (\mathbf{I}-\mathbf{P})f\Vert_{L^2_{x,v}}^2 + o(1)\Vert \phi\Vert_{H^1_x}^2 \\
    & \lesssim \Vert \p_t (\mathbf{I}-\mathbf{P})f\Vert_{L^2_{x,v}}^2 + o(1)\Vert \mathbf{b}\Vert_{L^2_x}^2.
\end{align*}

In summary, we obtain
\begin{align*}
    & \Vert b_1\Vert_{L^2_{x}}^2 \lesssim \Vert (\mathbf{I}-\mathbf{P})f\Vert_{L^2_{x,v}}^2 + |(I-P_\gamma)f|_{L^2_{\gamma_+}}^2 + \Vert \nu^{-1/2}\Gamma(f,f)\Vert_{L^2_{x,v}}^2 + \Vert \p_t (\mathbf{I}-\mathbf{P})f\Vert_{L^2_{x,v}}^2.
\end{align*}

For $b_3$, we can construct test function as
\begin{align*}
& \begin{cases}
         & \psi = \frac{3}{2} \Big(|v_3|^2 - \frac{|v|^2}{3} \Big)\sqrt{\mu}\p_3 \phi_3 + v_1v_3 \sqrt{\mu} \p_1 \phi_3  , \\
    & \Delta \phi_3 + \p_{33} \phi_3 = -b_3, \\
    &\phi_3  = 0 \text{ when }x_3 = \pm 1.
 \end{cases}
\end{align*}
Applying the same computation as $b_1$, we can obtain
\begin{align*}
    &  \Vert b_3\Vert_{L^2_{x}}^2 \lesssim \Vert (\mathbf{I}-\mathbf{P})f\Vert_{L^2_{x,v}}^2 + |(I-P_\gamma)f|_{L^2_{\gamma_+}}^2 + \Vert \nu^{-1/2}\Gamma(f,f)\Vert_{L^2_{x,v}}^2 + \Vert \p_t (\mathbf{I}-\mathbf{P})f\Vert_{L^2_{x,v}}^2.
\end{align*}

For $b_2$, we construct a test function as
\begin{align*}
    & \begin{cases}
        &\psi_2 = v_1 v_2 \sqrt{\mu}\p_1 \phi_2+v_2 v_3 \sqrt{\mu} \p_3 \phi_2, \\
        & \p_{11}\phi_2 + \p_{33}\phi_2 = -b_2, \\
        &\phi_2 = 0 \text{ when } x_3 = \pm 1.
    \end{cases}   
\end{align*}
From the Poincar\'e inequality, we have 
\begin{align*}
    & \Vert \phi_2\Vert_{H^2_x} \lesssim \Vert b_2\Vert_{L^2_x}.
\end{align*}
Since $\psi \perp \ker \mathcal{L}$ and $\psi$ is odd in $v_2$, the difference of the $b_2$ estimate only lies in $\eqref{weak_formula}_1$. We compute that
\begin{align*}
    &   \eqref{weak_formula}_1 = - \int_{\mathbb{R}^3} \int_{\O} (b_2v_1^2 v_3^2 \mu \p_{11} \phi_2 +b_2 v_2^2 v_3^2 \mu \p_{33} \phi_2) \dd x \dd v \underbrace{- \int_{\mathbb{R}^3}\int_{\O} (\mathbf{I}-\mathbf{P})f (v_1\p_1 \psi_2 + v_3\p_3 \psi_2) \dd x \dd v}_{E_2}\\
    & = -\int_\O  b_2 (\p_{11}\phi_2+\p_{33}\phi_2)  \dd x + E_2 = \Vert b_2\Vert_{L^2_x}^2 + E_2 , \\
    & |E_2| \lesssim \Vert (\mathbf{I}-\mathbf{P})f\Vert_{L^2_{x,v}}^2 + o(1)\Vert \phi_2\Vert_{H^2_x}^2 \lesssim \Vert (\mathbf{I}-\mathbf{P})f\Vert_{L^2_{x,v}}^2 + o(1)\Vert b_2\Vert_{L^2_x}^2.
\end{align*}
We can conclude the estimate for $b_2$ as
\begin{align*}
    &\Vert b_2\Vert_{L^2_x}^2 \lesssim \Vert (\mathbf{I}-\mathbf{P})f\Vert_{L^2_{x,v}}^2 + |(I-P_\gamma)f|_{L^2_{\gamma_+}}^2 + \Vert \nu^{-1/2}\Gamma(f,f)\Vert_{L^2_{x,v}}^2 + \Vert \p_t (\mathbf{I}-\mathbf{P})f\Vert_{L^2_{x,v}}^2.
\end{align*}

Hence, we conclude the $\mathbf{b}$ estimate as
\begin{align*}
    & \Vert \mathbf{b}\Vert_{L^2_{x}}^2 \lesssim \Vert (\mathbf{I}-\mathbf{P})f\Vert_{L^2_{x,v}}^2 + |(I-P_\gamma)f|_{L^2_{\gamma_+}}^2 + \Vert \nu^{-1/2}\Gamma(f,f)\Vert_{L^2_{x,v}}^2 + \Vert \p_t (\mathbf{I}-\mathbf{P})f\Vert_{L^2_{x,v}}^2.
\end{align*}

\textit{Estimate of $c$.}

We let $\phi_c$ be a solution of the following problem
\begin{align}
  -\Delta \phi_c   & = c \text{ in } \O , \ \phi_c = 0 \text{ when }x_3 = \pm 1, \label{poisson_c}
\end{align}
and we choose 
\begin{align}
  \psi_c  &  = \sum_{i=1,3} \p_i \phi_c v_i(|v|^2-5) \mu^{1/2} \perp \ker \mathcal{L}. \notag
\end{align}
A direct computation leads to
\begin{align*}
   v_1 \p_1 \psi_c + v_3\p_3 \psi_c &  = 5\Delta \phi_c \frac{|v|^2-3}{2}\mu^{1/2}  - \sum_{i,j=1,3} \p_{ij}^2\phi_c (\mathbf{I}-\P)(v_iv_j (|v|^2-5)\mu^{1/2}).
 %\label{transport_c}
\end{align*}
Thus $\eqref{weak_formula}_1$ is
\begin{align}
  \eqref{weak_formula}_1  &  = 5 \int_{\O} c^2 \dd x \underbrace{-    \int_{\O}\int_{\mathbb{R}^3}  (\mathbf{I}-\mathbf{P})f (v_1\p_1 \psi_c + v_3 \p_3 \psi_c) \dd x  \dd v }_{E_3} , \notag
\end{align}
where, from the Poincar\'e inequality and elliptic estimate to \eqref{poisson_c},
\[|E_3|\lesssim o(1) \Vert c\Vert_{L^2_x}^2 + \Vert (\mathbf{I}-\P)f\Vert_{L^2_{x,v}}^2 .\]

Then we apply boundary condition of $\phi_c$ and $f$ to compute $\eqref{weak_formula}_2$:
\begin{align}
& \int_{\mathbb{R}^3}\int_{\mathbb{R}} f\psi_c(x_3=1) v_3 \dd x_1 \dd v  \notag \\
 & =\int_{\mathbb{R}} \bigg[ \int_{v_3>0} + \int_{v_3<0}\bigg] \big(|v|^2 -5\big)\sqrt{\mu} (v_1 \p_1 \phi_c + v_3 \p_3 \phi_c) v_3 f  \dd v \dd x_1 \notag \\
 & = \int_{\mathbb{R}} \int_{v_3>0} \big(|v|^2 -5 \big) \sqrt{\mu} (v_1 \p_1 \phi_c + v_3 \p_3 \phi_c) v_3(f-P_\gamma f) \dd v \dd x_1 \notag\\
  & + 2 \int_{\mathbb{R}}\int_{v_3>0} \big(|v|^2 -5 \big) \sqrt{\mu} |v_3|^2 \p_3 \phi_c P_\gamma f \dd v \dd x_1 \notag    \\
  & = \int_{\mathbb{R}} \int_{v_3>0} \big(|v|^2 -5 \big) \sqrt{\mu} (v_1 \p_1 \phi_c + v_3 \p_3 \phi_c) v_3(f-P_\gamma f) \dd v \dd x_1 \notag \\
  & \lesssim o(1) \int_{\mathbb{R}} |\nabla \phi_c(x_1,1)|^2 \dd x_1 +   |(I-P_\gamma)f|_{L^2_{\gamma_+}}^2  \lesssim o(1) \Vert c\Vert_{L^2_x}^2 + |(I-P_\gamma)f|_{L^2_{\gamma_+}}^2   .\notag
\end{align}
In the third line, we have applied the change of variable $v\to v-2(0,0,v_3)$. The fourth line vanished by $\int_{v_3>0} (|v|^2-5)|v_3|^2\mu \dd v= 0$. In the last inequality, we applied elliptic estimate to~\eqref{poisson_c} with the trace theorem:
\begin{align*}
    & \int_{\mathbb{R}} |\nabla \phi_c(x_1,1)|^2 \dd x_1 \lesssim \Vert \phi_c \Vert^2_{H^2_x} \lesssim \Vert c\Vert_{L^2_x}^2.
\end{align*}

The estimate for $x=-1$ is similar. We conclude the estimate for $\eqref{weak_formula}_2$ as
\begin{equation*}
\eqref{weak_formula}_2 \lesssim o(1) \Vert c\Vert_{L^2_x}^2  +  |(I-P_\gamma)f|_{L^2_{\gamma_+}}^2 .
\end{equation*}

For the rest $\eqref{weak_formula}_3, \eqref{weak_formula}_4, \eqref{weak_formula}_5$, by the same computation as the estimate of $\mathbf{b}$, with the property $\psi \perp \ker \mathcal{L}$, we obtain
\begin{align*}
  &  |\eqref{weak_formula}_3 + \eqref{weak_formula}_4 + \eqref{weak_formula}_5 | \lesssim o(1)\Vert c\Vert_{L^2_{x}}^2 + \Vert (\mathbf{I}-\mathbf{P})f\Vert_{L^2_{x,v}}^2 + \Vert \nu^{-1/2}\Gamma(f,f)\Vert_{L^2_{x,v}}^2 + \Vert \p_t (\mathbf{I}-\mathbf{P})f\Vert_{L^2_{x,v}}^2. 
\end{align*}
  
In summary, for $c$ we conclude
\begin{align*}
    & \Vert c\Vert_{L^2_{x}}^2 \lesssim \Vert (\mathbf{I}-\mathbf{P})f\Vert_{L^2_{x,v}}^2 + |(I-P_\gamma)f|_{L^2_{\gamma_+}}^2 + \Vert \nu^{-1/2}\Gamma(f,f)\Vert_{L^2_{x,v}}^2 + \Vert \p_t (\mathbf{I}-\mathbf{P})f\Vert_{L^2_{x,v}}^2.
\end{align*}

We conclude the lemma. 
\end{proof}

\subsection{$L^\infty_{T,x,v}$ estimate}\label{sec:linfty_2d}

We control the $L^\infty_{T,x,v}$ estimate in the following lemma.
\begin{lemma}\label{lemma:linfty_2d}
For any $T\geq 0$, it holds that
\begin{align*}
    &    \Vert wf\Vert_{L^\infty_{T,x,v}} \lesssim \Vert wf(0)\Vert_{L^\infty_{x,v}} + \Vert f\Vert_{L^\infty_T L^2_{x,v}} + \Vert wf\Vert_{L^\infty_{T,x,v}}^2, \\
    &   \Vert w\p_t f\Vert_{L^\infty_{T,x,v}} \lesssim \Vert w\p_t f(0)\Vert_{L^\infty_{x,v}} + \Vert \p_t f\Vert_{L^\infty_T L^2_{x,v}} + \Vert wf\Vert_{L^\infty_{T,x,v}} \Vert w\p_t f\Vert_{L^\infty_{T,x,v}}.
\end{align*}

\end{lemma}

To prove the lemma we first define the stochastic cycle. We use standard notations for the backward exit time and backward exit position:
\begin{equation*}%\label{xb_tb}
\begin{split}
\tb(x,v) :   &  = \sup\{s\geq 0, x-s(v_1,v_3) \in \O\} , \\
  \xb(x,v)  : & = x - \tb(x,v)(v_1,v_3).
\end{split}
\end{equation*}

We denote $t_0=T_0$ to be a fixed starting time. First, we define the stochastic cycle as
\begin{definition}
We define a stochastic cycles as $(x^0,v^0)= (x,v) \in \bar{\O} \times \R^3$ and inductively
\begin{align}
&x^1:= \xb(x,v), \   v^1 \in \mathcal{V}_1:=\{v^1\in \mathbb{R}^3: v^1_3 \times \text{sign}(x_3^1)>0\} , \notag\\
&  v^n\in \mathcal{V}_n:= \{v^n\in \mathbb{R}^3: v^n_3 \times \text{sign}(x_3^1)>0\}, \ \ \text{for} \  n \geq 1,
\notag
\\
 &x^{n+1} := \xb(x^n, v^n) , \ \tb^{n}:= \tb(x^n,v^n) \ \ \text{for} \  v^n \in \mathcal{V}_n. \notag
\end{align}
\begin{equation*}
t^n =  t_0-  \{ \tb + \tb^1 + \cdots + \tb^{n-1}\},  \ \ \text{for} \  n \geq 1.    
\end{equation*}

\end{definition}

With the stochastic cycles defined, we apply the method of characteristics to have
\begin{align}
  &w(v) f(T_0,x,v)  \notag \\
  & =  \mathbf{1}_{t^1\leq 0} w(v) e^{- \nu(v) T_0} f(0,x-T_0 (v_1,v_3), v) \label{initial}\\
  & + \mathbf{1}_{t^1 \leq 0} \int_0^{T_0} e^{-\nu(v) (T_0-s)}  w(v) \int_{\mathbb{R}^3}   f(s,x-(t-s)(v_1,v_3),u) \mathbf{k}(v,u) \dd u \dd s \label{K_0}\\
  & +\mathbf{1}_{t^1>0} \int^{T_0}_{t^1} e^{-\nu(v) (T_0 - s)}  w(v)  \int_{\mathbb{R}^3}  f(s,x-(t-s)(v_1,v_3),u) \mathbf{k}(v,u) \dd u \dd s  \label{K_1} \\
  &  + \mathbf{1}_{t^1 \leq 0} \int_0^{T_0} e^{-\nu(v)(T_0-s)}  w(v)  \Gamma(f,f)(s,x-(t-s)(v_1,v_3),v) \dd s  \label{g_0} \\
  &  +\mathbf{1}_{t^1>0} \int^{T_0}_{t^1} e^{-\nu(v) (T_0 - s)} w(v) \Gamma(f,f)(s,x-(t-s)(v_1,v_3),v) \dd s   \label{g_1} \\
  & + \mathbf{1}_{t^1>0} e^{-\nu(v) (T_0 - t^1)} w(v)f(t^1,x^1,v), \label{f_bdr}
\end{align}
where the contribution of the boundary is bounded as
\begin{align}
   & |\eqref{f_bdr}|\leq   e^{-\nu(v) (T_0-t^1)} w(v) \sqrt{\mu(v)}  \notag\\
   & \times \int_{\prod_{j=1}^{n}\mathcal{V}_j} \bigg\{ \sum_{i=1}^{n}\mathbf{1}_{t^{i+1}\leq 0 < t^i} e^{-\nu(v^i) t^i} w(v^i)|f(0, x^i - t^i (v^i_1,v^i_3), v^i)| \dd \Sigma_i     \label{bdr_initial_2d}\\
  & + \mathbf{1}_{t^{n+1}>0}  w(v^{n})|f(t^{n+1},x^{n+1},v^{n})| \dd \Sigma_{n} \label{bdr_tk_2d}\\
  & + \sum_{i=1}^{n} \mathbf{1}_{t^{i+1}\leq 0 < t^i}  \int^{t^i}_0 e^{-\nu(v^i) (t^i-s)}   w(v^i)  \int_{\mathbb{R}^3}   \mathbf{k}(v^i,u)|f(s,x^i-(t^i-s)(v^i_1,v^i_3), u)|  \dd u \dd s \dd \Sigma_i   \label{bdr_K_0_2d} \\
  & + \sum_{i=1}^{n} \mathbf{1}_{t^{i+1}>0} \int_{t^{i+1}}^{t^i} e^{-\nu(v^i) (t^i-s)}  w(v^i)
  \int_{\mathbb{R}^3} \mathbf{k}(v^i,u) |f(s,x^i-(t^i-s)(v^i_1,v^i_3), u)| \dd u \dd s \dd \Sigma_i   \label{bdr_K_i_2d} \\
  &  + \sum_{i=1}^{n} \mathbf{1}_{t^{i+1}\leq 0 < t^i}  \int^{t^i}_0 e^{-\nu(v^i) (t^i-s)} w(v^i) |\Gamma(f,f)(s,x^i-(t^i-s)(v^i_1,v^i_3),v^i)| \dd \Sigma_i     \label{bdr_g_0_2d}   \\
  &  + \sum_{i=1}^{n} \mathbf{1}_{t^{i+1}>0} \int_{t^{i+1}}^{t^i} e^{-\nu(v^i) (t^i-s)} w(v^i) |\Gamma(f,f)(s,x^i-(t^i-s)(v^i_1,v^i_3),v^i)| \dd \Sigma_i  \bigg\}.  \label{bdr_g_i_2d}  
\end{align}
Here $\dd \Sigma_i$ is defined as
\begin{equation}
\dd \Sigma_i = \Big\{\prod_{j=i+1}^{n} \dd \sigma_j \Big\}  \times  \Big\{  \frac{1}{w(v^i)\sqrt{\mu(v^i)}}   \dd \sigma_i \Big\} \times \Big\{\prod_{j=1}^{i-1}  e^{-\nu(v^j)(t^j-t^{j+1})}   \dd \sigma_j \Big\} , \label{Sigma_2d}
\end{equation}
where $\dd \sigma_i$ is a probability measure in $\mathcal{V}_i$ given by
\begin{equation}
\dd \sigma_i =  \sqrt{2\pi}\mu(v^i) |v^i_3|\dd v^i. \label{d_sigma_i_2d}
\end{equation}

\eqref{bdr_tk_2d} corresponds to the scenario that the backward trajectory interacts with the diffuse boundary portion a large number of times. This term is controlled by the following lemma.

\begin{lemma}\label{lemma:tk_2d}
For $T_0>0$ sufficiently large, there exist constants $C_1,C_2>0$ independent of $T_0$ such that for $n = C_1 T_0^{5/4}$ and $(t^0,x^0,v^0) = (t,x,v)\in [0,T_0]\times \bar{\O}\times \mathbb{R}^3$,
\begin{equation*}%\label{tk_small}
\int_{\prod_{j=1}^{n-1} \mathcal{V}_j}\mathbf{1}_{t^n>0}  \prod_{j=1}^{n-1} \dd \sigma_j \leq \Big( \frac{1}{2}\Big)^{C_2 T_0^{5/4}}.
\end{equation*}

\end{lemma}

\begin{proof}
The proof of this lemma is the same as Lemma \ref{lemma:tk}, since the backward exit time $\tb(x,v)$ in both settings are determined by $v_3$ and hence they are equivalent.
\end{proof}

To prove Lemma \ref{lemma:linfty_2d}, we need to estimate every term in the characteristic formula \labelcref{initial,K_0,K_1,g_0,g_1,f_bdr}. First, we estimate the boundary term \eqref{f_bdr} in the following lemma.

\begin{lemma}\label{lemma:bdr}
For the boundary term~\eqref{f_bdr}, it holds that
\begin{equation*}%\label{bdr_bdd}
\begin{split}
 w(v)|f(t^1,x^1,v)| \leq   &  4 e^{-\nu_0 t^1}\Vert w f_0\Vert_{L^\infty_{x,v}}+ o(1) \Vert w f\Vert_{L^\infty_{T_0,x,v}}  \\
    &+   C(T_0)    \Big\Vert \nu^{-1} w \Gamma(f,f) \Big\Vert_{L^\infty_{T_0,x,v}}  +  C(T_0)  \Vert f\Vert_{L^\infty_{T_0}L^2_{x,v}} .
\end{split}
\end{equation*}
\end{lemma}

\begin{proof}
Since $\dd \sigma_i$ in~\eqref{d_sigma_i_2d} is a probability measure, \eqref{bdr_initial_2d} is directly bounded as
\begin{equation}\label{bdr_initial_bdd_2d}
\eqref{bdr_initial_2d}\leq 4e^{-\nu_0 t^1}   \Vert w f_0\Vert_{L^\infty_{x,v}}.
\end{equation}
Here the constant $4$ comes from $\sqrt{2\pi} \int_{\mathcal{V}_i}   |v^i_3| \sqrt{\mu(v^i)} w^{-1}(v^i) \dd v^i < 4.$
The exponential decay factor $e^{-\nu_0 t^1}$ comes from the decay factor from \eqref{Sigma_2d}, and the computation
\begin{align*}
    & e^{-\nu_0 t^i} e^{-\nu_0 (t^{i-1}-t^i)} \leq e^{-\nu_0 t^{i-1}}, \ e^{-\nu_0 t^{i-1}} e^{-\nu_0 (t^{i-2}-t^{i-1})} \leq e^{-\nu_0 t^{i-2}} \cdots.
\end{align*}

For~\eqref{bdr_tk_2d}, with $n=C_1 T_0^{5/4}$, we apply Lemma \ref{lemma:tk_2d} to have
\begin{align}
 |\eqref{bdr_tk_2d} |  &  \leq  \int_{\prod_{j=1}^{n-1} \mathcal{V}_j} \int_{\mathcal{V}_n} \mathbf{1}_{t^{n+1}>0}  | w(v^{n}) f(t^{n+1},x^{n+1},v^{n})| w^{-1}(v^n)\sqrt{\mu(v^n)}|v^n_3| \dd v^n \prod_{j=1}^{n-1}\dd \sigma_j  \notag\\
 & \lesssim   \Vert wf\Vert_{L^\infty_{T_0,x,v}} \int_{\prod_{j=1}^{n-1} \mathcal{V}_j}  \mathbf{1}_{t^n>0}  \prod_{j=1}^{n-1}\dd \sigma_j  \leq o(1)  \Vert wf\Vert_{L^\infty_{T_0,x,v}}.     \label{bdr_tk_bdd_2d}
\end{align}

\eqref{bdr_g_0_2d} and \eqref{bdr_g_i_2d} are directly bounded as
\begin{equation}\label{bdr_g_bdd_2d}
\begin{split}
|\eqref{bdr_g_0_2d} + \eqref{bdr_g_i_2d}| & \leq      Cn \Vert \nu^{-1}w \Gamma(f,f) \Vert_{L^\infty_{T_0,x,v}}   \int_0^{T_0} e^{-\nu(v^i)(T_0-s)/2} \nu(v^i) \dd s \\ 
& \leq  Cn \Vert \nu^{-1}w \Gamma(f,f)\Vert_{L^\infty_{T_0,x,v}} .
\end{split}
\end{equation}

Then we estimate \eqref{bdr_K_i_2d}. Recall the notation $\mathbf{k}_\theta(v,u) = \mathbf{k}(v,u)\frac{e^{\theta|v|^2}}{e^{\theta |u|^2}}.$ We focus on estimating
\begin{equation}\label{iteration_i_2d}
\begin{split}
    & \int_{\prod_{j=1}^{i-1} \mathcal{V}_j}        \prod_{j=1}^{i-1}e^{-\nu(v^j)(t^j-t^{j+1})}\dd \sigma_j \int_{\mathcal{V}_i}\dd v_i\mathbf{1}_{t^{i+1}>0 }\mu^{1/2}(v^i) w^{-1}(v^i) |v_3^i|    \\
    & \times     \int_{t^{i+1}}^{t^i} e^{-\nu(v^i)(t^i -s)} \int_{\mathbb{R}^3} \mathbf{k}_\theta(v^i,u) w(u)f(s,x^i-(t^i-s)(v^i_1,v^i_3),u) \dd s .
\end{split}
\end{equation}

First we decompose the $\dd s$ integral into $\mathbf{1}_{s\geq t^i-\delta} + \mathbf{1}_{s< t^i-\delta}$. By \eqref{k_theta} in Lemma \ref{lemma:k_theta}, the contribution of the first term reads
\begin{align}
& |\eqref{iteration_i_2d} \mathbf{1}_{s\geq t^i-\delta} |    \leq \int_{\prod_{j=1}^{i-1}\mathcal{V}_j} \prod_{j=1}^{i-1} e^{-\nu(v^j)(t^j-t^{j+1})}  \dd \sigma_j \int_{\mathcal{V}_i}\dd v_i \mu^{1/2}(v^i) w^{-1}(v^i)|v_3^i|  \notag\\ 
 & \times \int^{t^i}_{\max\{t^{i+1},t^i-\delta\}} e^{-\nu(v^i) (t^i-s)} \int_{\mathbb{R}^3} \mathbf{k}_\theta(v^i,u) w(u) |f(s,x^i-(t^i-s)(v^i_1,v^i_3),u)| \dd u \dd s  \leq o(1) \Vert wf\Vert_{L^\infty_{T_0,x,v}} . \label{bdr_s_small_bdd_2d}
\end{align}

Next we decompose the $v^i$ integral into $\mathbf{1}_{|v^i|\geq N} + \mathbf{1}_{|v^i|<N}$. By \eqref{k_theta} in Lemma \ref{lemma:k_theta}, the contribution of the first term reads
\begin{align}
    & |\eqref{iteration_i_2d} \mathbf{1}_{|v^i|\geq N} | \leq \int_{\prod_{j=1}^{i-1}\mathcal{V}_j} \prod_{j=1}^{i-1}e^{-\nu(v^j)(t^j-t^{j+1})}\dd \sigma_j  \int_{\mathcal{V}_i} \mathbf{1}_{|v^i|\geq N} \sqrt{\mu(v^i)} w^{-1}(v^i)|v^i_3| \dd v^i   \notag \\
    &\ \ \ \ \ \ \ \ \  \times \int_{t^{i+1}}^{t^i} e^{-\nu(v^i)(t^i -s)} \int_{\mathbb{R}^3} \mathbf{k}_\theta(v^i,u) w(u) |f(s,x^i-(t^i-s)(v^i_1,v^i_3),u)| \dd s  \leq o(1)  \Vert wf\Vert_{L^\infty_{T_0,x,v}}.   \label{bdr_v_large_bdd_2d}
\end{align}

Then we decompose the $u$ integral into $\mathbf{1}_{|u|\geq N \text{ or } |v^i-u|\leq \frac{1}{N}} + \mathbf{1}_{|u|<N, \ |v^i-u|> \frac{1}{N}}$. By \eqref{K_N_small} in Lemma \ref{lemma:k_theta}, the contribution of the first term reads
\begin{align}
    & |\eqref{iteration_i_2d} \mathbf{1}_{|u|\geq N \text{ or } |v^i-u|\leq \frac{1}{N}} | \leq \int_{\prod_{j=1}^{i-1}\mathcal{V}_j} \prod_{j=1}^{i-1}e^{-\nu(v^j)(t^j-t^{j+1})}\dd \sigma_j \int_{\mathcal{V}_i} \dd v^i  \mu^{1/2}(v^i) w^{-1}(v^i)|v_3^i|  \notag\\
    &\ \ \times  \int_{t^{i+1}}^{t^i} e^{-\nu(v^i)(t^i -s)} \int_{\mathbb{R}^3}  \mathbf{1}_{|u|\geq N \text{ or } |v^i-u|\leq \frac{1}{N}} \mathbf{k}_\theta(v^i,u) w(u) |f(s,x^i-(t^i-s)(v^i_1,v^i_3),u)| \dd s \leq o(1)  \Vert wf\Vert_{L^\infty_{T_0,x,v}}.   \label{bdr_u_large_bdd_2d}
\end{align}

Now we consider the intersection of all other cases, where we have $|v^i|\leq N, \ s<t^i-\delta$, and $|u|<N, \ |v^i-u|>\frac{1}{N}$. The conditions of $v^i$ and $u$ imply that $\mathbf{k}(v^i,u)\leq C_N$ from \eqref{k_theta_bdd} in Lemma \ref{lemma:k_theta}.

In the last line, we have applied the change of variable $(v^i_1,v^i_3) \to y = x^i-(t^i-s)(v^i_1,v^i_3) \in \O$ with Jacobian 
\[\Big|\det\Big(\frac{\p (x^i-(t^i-s)(v^i_1,v^i_3))}{\p (v^i_1,v^i_3)} \Big) \Big| = (t^i-s)^2 \geq \delta^2 .\]

Then we apply the H\"older inequality to have
\begin{align}
 & |\eqref{iteration_i_2d} \mathbf{1}_{|u|< N , |v^i-u|> \frac{1}{N}, s<t^i-\delta, |v^i|\leq N}   |  \notag\\
 & \leq \frac{1}{\delta^2}\int_{\prod_{j=1}^{i-1}\mathcal{V}_j} \prod_{j=1}^{i-1}e^{-\nu(v^j)(t^j-t^{j+1})}\dd \sigma_j  \int^{t^i-\delta}_{0}  e^{-\nu_0(t^i-s)}     \int_{|u|\leq N} \int_{\O} |f(s,y,u)| \dd u \dd y \dd s   \notag\\
 & \leq  C_{N,\delta,T_0,\O}   \int_{\prod_{j=1}^{i-1}\mathcal{V}_j} \prod_{j=1}^{i-1}e^{-\nu(v^j)(t^j-t^{j+1})}\dd \sigma_j \times    \int_{0}^{t^1} e^{-\nu_0(t^1-s)} \Vert f(s)\Vert_{L^2_{x,v}} \dd s  \leq  C_{N,\delta,T_0,\O}  \Vert f\Vert_{L^\infty_{T_0}L^2_{x,v}}  . \label{other_case_bdd_2d}
\end{align}

Collecting \labelcref{bdr_s_small_bdd_2d,bdr_v_large_bdd_2d,bdr_u_large_bdd_2d,other_case_bdd_2d}, we conclude that
\begin{equation}\label{bdr_K_i_bdd_2d}
|\eqref{bdr_K_i_2d}| \lesssim \Vert wf\Vert_{L^\infty_{T_0,x,v}} + C_{N,\delta,n,T_0,\O}  \Vert f\Vert_{L^\infty_{T_0}L^2_{x,v}}. 
\end{equation}

By the same computation, we have the same bound for \eqref{bdr_K_0_2d}:
\begin{equation}\label{bdr_K_0_bdd_2d}
|\eqref{bdr_K_0_2d}| \lesssim \Vert wf\Vert_{L^\infty_{T_0,x,v}} + C_{N,\delta,n,T_0,\O}  \Vert f\Vert_{L^\infty_{T_0}L^2_{x,v}}.  
\end{equation}

Summarizing \labelcref{bdr_initial_bdd_2d,bdr_tk_bdd_2d,bdr_g_bdd_2d,bdr_K_i_bdd_2d,bdr_K_0_bdd_2d}, we conclude the lemma.
\end{proof}

\begin{proof}[\textbf{Proof of Lemma \ref{lemma:linfty_2d}}]
First of all, \eqref{initial}, \eqref{g_0} and \eqref{g_1} are bounded as
\begin{equation}\label{initial_g_bdd}
|\eqref{initial}| + |\eqref{g_0}| + |\eqref{g_1}| \leq  e^{-\nu_0 T_0}\Vert w f_0\Vert_{L^\infty_{x,v}} + C\Vert \nu^{-1}w \Gamma(f,f)\Vert_{L^\infty_{T_0,x,v}}.
\end{equation}
Moreover, \eqref{f_bdr} is bounded by Lemma \ref{lemma:bdr} as
\begin{equation}\label{f_bdr_bdd}
\begin{split}
 |\eqref{f_bdr}| \leq   &  4e^{-\nu_0 T_0}  \Vert w f_0\Vert_{L^\infty_{x,v}} + o(1)\Vert wf\Vert_{L^\infty_{T_0,x,v}}  \\
    &  + C(T_0) \big[ \Vert \nu^{-1}w \Gamma(f,f)\Vert_{L^\infty_{T_0,x,v}} +  \Vert f\Vert_{L^\infty_{T_0}L^2_{x,v}} \big].
\end{split}
\end{equation}

Then we focus on~\eqref{K_1}. We expand $f(s,x-(t-s)(v_1,v_3),u)$ using the characteristic form \labelcref{initial,K_0,K_1,g_0,g_1,f_bdr} again along $u$. Denoting $t^u_1:= s - \tb(x-(t-s)(v_1,v_3),u)$ and $y:=x-(t-s)(v_1,v_3)$, we have
\begin{align}
    & \eqref{K_1} = \mathbf{1}_{t^1>0} \int^{T_0}_{t^1} \dd s e^{-\nu(v) (T_0 -s)}  \int_{\mathbb{R}^3} \dd u \frac{w(v)}{w(u)}\mathbf{k}(v,u) \notag \\
    &\times \Big\{ \mathbf{1}_{t_1^u \leq 0} e^{-\nu(u) s} w(u)f(0,y-s(u_1,u_3),u)       \label{u_initial} \\
    &  +\mathbf{1}_{t_1^u\leq 0} \int_0^{s}  e^{-\nu(u)(s-s')} \dd s'\int_{\mathbb{R}^3} w(u)\mathbf{k}(u,u') f(s',y-(s-s')(u_1,u_3), u')    \dd u' \label{u_K0} \\
    &  +\mathbf{1}_{t_1^u> 0} \int_{t_1^u}^{s}  e^{-\nu(u)(s-s')} \dd s'\int_{\mathbb{R}^3} w(u) \mathbf{k}(u,u') f(s',y-(s-s')(u_1,u_3), u')  \dd u' \label{u_K1}\\
    & + \mathbf{1}_{t_1^u\leq 0} \int_{0}^{s} e^{-\nu(u)(s-s')} w(u) \Gamma(f,f)(s',y-(s-s')(u_1,u_3),u) \dd s' \label{u_g0} \\
    & + \mathbf{1}_{t_1^u> 0} \int_{t_1^u}^{s} e^{-\nu(u)(s-s')} w(u) \Gamma(f,f)(s',y-(s-s')(u_1,u_3),u) \dd s' \label{u_g1} \\
    & + \mathbf{1}_{t_1^u > 0} e^{-\nu(u)(s-t_1^u)} w(u) f(t_1^u,y-\tb(y,u)(u_1,u_3),u) \Big\}.   \label{u_bdr}
\end{align}

The contribution of \eqref{u_initial} in \eqref{K_1} is bounded by
\begin{align}
    & \int^{T_0}_{t^1} \dd s e^{-\nu_0 (T_0-s)} \int_{\mathbb{R}^3} \dd u \mathbf{k}_\theta(v,u) e^{-\nu_0 s} \Vert w f_0\Vert_{L^\infty_{x,v}}    \notag \\
    &\leq C_\theta\int_{t^1}^{T_0} \dd s e^{-\nu_0(T_0-s)}e^{-\nu_0 s}  \Vert wf_0\Vert_{L^\infty_{x,v}}  \leq C_\theta e^{-\nu_0 T_0/2}\Vert w f_0\Vert_{L^\infty_{x,v}} . \label{u_initial_bdd}
\end{align}
In the second line, we have used Lemma \ref{lemma:k_theta}.

The contribution of \eqref{u_g0} and \eqref{u_g1} in \eqref{K_1} are bounded by
\begin{align}
    & \int_{t^1}^{T_0} \dd s e^{-\nu_0 (T_0-s)} \int_{\mathbb{R}^3} \dd u  \mathbf{k}_\theta(v,u) \int^s_0 \dd s' e^{-\nu(u)(s-s')} \nu(u)  \Vert \nu^{-1}w \Gamma(f,f)\Vert_{L^\infty_{T_0,x,v}}  \notag  \\
    & \leq C \Vert \nu^{-1}w \Gamma(f,f)\Vert_{L^\infty_{T_0,x,v}} \int_{t^1}^{T_0} \dd s e^{-\nu_0 (T_0-s)} \int_{\mathbb{R}^3} \dd u  \mathbf{k}_\theta(v,u) \leq C_\theta \Vert \nu^{-1}w \Gamma(f,f)\Vert_{L^\infty_{T_0,x,v}}. \label{u_g_bdd}
\end{align}
In the second line, we have used \eqref{nu_bdd} and Lemma \ref{lemma:k_theta}.

The contribution of the boundary term in~\eqref{u_bdr} can be bounded by Lemma \ref{lemma:bdr} as
\begin{align}
  |\eqref{u_bdr} | & \leq \int_{t^1}^{T_0} \dd s e^{-\nu_0(T_0-s)} \int_{\mathbb{R}^3} \dd u \mathbf{k}_\theta(v,u) e^{-\nu_0 (s-t_1^u)} \notag \\
 & \times \Big\{   C_\theta\Vert w f_0\Vert_{L^\infty_{x,v}} + o(1)\Vert wf\Vert_{L^\infty_{T_0,x,v}} + Cn \Big[\Vert \nu^{-1}w \Gamma(f,f)\Vert_{L^\infty_{T_0,x,v}}  +   \Vert f\Vert_{L^\infty_{T_0}L^2_{x,v}}  \Big] \Big\} \notag\\
  &\leq   C_\theta  \Vert w f_0\Vert_{L^\infty_{x,v}} + o(1) \Vert wf\Vert_{L^\infty_{T_0,x,v}}   + C(T_0)\Big[\Vert \nu^{-1}w \Gamma(f,f)\Vert_{L^\infty_{T_0,x,v}}  +   \Vert f\Vert_{L^\infty_{T_0}L^2_{x,v}} \Big].   \label{u_bdr_bdd}
\end{align}
We have used Lemma \ref{lemma:k_theta}.

Then we focus on the contribution of \eqref{u_K1} in \eqref{K_1}. First we decompose the $\dd s' $ integral into $\mathbf{1}_{s-s'< \delta} + \mathbf{1}_{s-s'\geq \delta}$. Applying \eqref{k_theta} in Lemma \ref{lemma:k_theta} twice, the contribution of the first term reads
\begin{align}
    & |\eqref{u_K1} \mathbf{1}_{s-s'<\delta} |\notag \\
   & \leq   \int_{t^1}^{T_0} \dd s e^{-\nu(v)(T_0-s)} \int_{\mathbb{R}^3} \dd u \mathbf{k}_\theta(v,u) \int^s_{\max\{s-\delta,t_1^u\}} e^{-\nu(u)(s-s')} \dd s' \int_{\mathbb{R}^3} \dd u' \mathbf{k}_\theta(u,u')     \Vert wf\Vert_{L^\infty_{T_0,x,v}}   \notag \\
   & \leq \int_{t^1}^{T_0} \dd s e^{-\nu(v)(T_0-s)} \int_{\mathbb{R}^3} \dd u \mathbf{k}_\theta(v,u)  \Vert wf\Vert_{L^\infty_{T_0,x,v}} \leq    o(1)\Vert wf\Vert_{L^\infty_{T_0,x,v}}. \label{u_K1_s_small}
\end{align}

Next we decompose the $\dd u$ integral into $\mathbf{1}_{|u|> N \text{ or } |v-u|\leq \frac{1}{N}} + \mathbf{1}_{|u|\leq N, \ |v-u|>\frac{1}{N}}$. Applying \eqref{k_theta} and \eqref{K_N_small} in Lemma \ref{lemma:k_theta}, the contribution of the first term reads
\begin{align}
   &|\eqref{u_K1}\mathbf{1}_{|u|>N \text{ or }|v-u| \leq \frac{1}{N}} |\notag\\
   & \leq \int_{t^1}^{T_0} \dd s e^{-\nu(v)(T_0-s)} \int_{|u|>N \text{ or }|v-u|\leq \frac{1}{N} } \dd u \mathbf{k}_\theta(v,u) \times \int^s_{t_1^u} e^{-\nu(u)(s-s')}   \dd s' \Vert wf\Vert_{L^\infty_{T_0,x,v}} \notag \\
   & \leq o(1)\int_{t^1}^{T_0} \dd s e^{-\nu_0(T_0-s)}  \Vert wf\Vert_{L^\infty_{T_0,x,v}}  \leq o(1)  \Vert wf\Vert_{L^\infty_{T_0,x,v}}. \label{u_K1_u_small}
\end{align}

Next we decompose the $\dd u'$ integral into $\mathbf{1}_{|u'|\geq N \text{ or } |u-u'|\leq \frac{1}{N}} + \mathbf{1}_{|u'|\leq N, \ |u'-u|>\frac{1}{N}}$. The contribution of the first term reads
\begin{align}
   &|\eqref{u_K1}\mathbf{1}_{|u'|\geq N \text{ or }|u'-u|\leq\frac{1}{N}} |\notag\\
   & \leq o(1)\int_{t^1}^{T_0} \dd s e^{-\nu(v)(T_0-s)} \int_{\mathbb{R}^3} \dd u \mathbf{k}_\theta(v,u) \int^s_{t_1^u} e^{-\nu(u)(s-s')}  \dd s' \Vert wf\Vert_{L^\infty_{T_0,x,v}}   \notag\\
   & \leq   o(1)\int_{t^1}^{T_0} \dd s e^{-\nu_0(T_0-s)}   \Vert wf\Vert_{L^\infty_{T_0,x,v}}     \leq  o(1)  \Vert wf\Vert_{L^\infty_{T_0,x,v}}. \label{u_K1_u_prime_small}
\end{align}

Now we consider the intersection of all other cases, where we have $|u-v|>\frac{1}{N}, \ |u|\leq N,$ $s'<s-\delta$ and $|u'|<N, \  |u-u'|>\frac{1}{N}$. In such case by \eqref{k_N_upper_bdd} we have
$$
\mathbf{k}_\theta(v,u)w(u)\mathbf{k}(u,u')\leq C_N.
$$
We compute such contribution in \eqref{u_K1} as
\begin{align}
    &  |\eqref{u_K1}\mathbf{1}_{|u-v|>\frac{1}{N}, \ |u|\leq N,\ s'<s-\delta, \ |u'|<N, \ |u-u'|>\frac{1}{N}} |\notag\\
    &  \leq C_N  \int_{t^1}^{T_0} \dd s e^{-\nu(v) (T_0 -s)} \int_{|u|\leq N} \dd u \times \int^{s-\delta}_{t_1^u} e^{-\nu(u)(s-s')} \dd s'\notag \\
    &\times \int_{|u'|<N} \dd u'   |f(s', x-(t-s)(v_1,v_3)-(s-s')(u_1,u_3),u')| . \label{u_K1_other}
\end{align}
%{\color{red}
With $|u'|<N$, we apply the same argument in \eqref{other_case_bdd_2d}. We apply the change of variable $(u_1,u_3)\to x-(t-s)(v_1,v_3)-(s-s')(u_1,u_3)$ with Jacobian 
\[\Big|\det\Big(\frac{\p (x-(t-s)(v_1,v_3)-(s-s')(u_1,u_3))}{\p (u_1,u_3)} \Big) \Big| = (s-s')^2 \geq \delta^2 \]
to derive that
\begin{align}
 & |\eqref{u_K1_other}|  \leq C_{T_0,N,\delta,\O} \int_{t^1}^{T_0} \dd s e^{-\nu(v)(T_0-s)}  \int_{0}^{s-\delta} e^{-\nu(u)(s-s')} \Vert f(s')\Vert_{L^2_{x,v}} \dd s'     \notag \\
  &\leq  C_{T_0,N,\delta,\O}  \Vert f\Vert_{L^\infty_{T_0}L^2_{x,v}} .  \label{u_K1_other_bdd}
\end{align}

Collecting \eqref{u_K1_s_small}, \eqref{u_K1_u_small}, \eqref{u_K1_u_prime_small} and \eqref{u_K1_other_bdd}, we have
\begin{equation}\label{u_K1_bdd}
|\eqref{u_K1}| \leq \Vert wf\Vert_{L^\infty_{T_0,x,v}} + C(T_0)  \Vert f\Vert_{L^\infty_{T_0}L^2_{x,v}} .
\end{equation}

By the same computation, we have
\begin{equation}\label{u_K0_bdd}
 |\eqref{u_K0} |\leq  \Vert wf\Vert_{L^\infty_{T_0,x,v}} + C(T_0) \Vert f\Vert_{L^\infty_{T_0}L^2_{x,v}}.   
\end{equation}

We combine \eqref{u_initial_bdd}, \eqref{u_g_bdd}, \eqref{u_bdr_bdd},  \eqref{u_K1_bdd} and \eqref{u_K0_bdd}  to conclude the estimate for \eqref{K_1}:
\begin{equation}\label{K1_bdd}
\begin{split}
   |\eqref{K_1}| & \leq C(\theta) e^{-\nu_0 T_0/2}\Vert w f_0\Vert_{L^\infty_{x,v}} + o(1)\Vert wf\Vert_{L^\infty_{T_0,x,v}} \\
   &+ C(T_0) \Vert \nu^{-1}w \Gamma(f,f)\Vert_{L^\infty_{T_0,x,v}}   +C(T_0) \Vert f\Vert_{L^\infty_{T_0}L^2_{x,v}}.
\end{split}
\end{equation}

Similarly, we can have the same estimate for \eqref{K_0} as
\begin{equation}\label{K0_bdd}
\begin{split}
   |\eqref{K_0}| & \leq C_\theta e^{-\nu_0 T_0/2}\Vert w f_0\Vert_{L^\infty_{x,v}} + o(1)\Vert wf\Vert_{L^\infty_{T_0,x,v}}    \\
    & + C(T_0) \Vert \nu^{-1}w \Gamma(f,f)\Vert_{L^\infty_{T_0,x,v}} +C(T_0)  \Vert f\Vert_{L^\infty_{T_0}L^2_{x,v}}.
\end{split}
\end{equation}

Last we collect \eqref{initial_g_bdd}, \eqref{f_bdr_bdd}, \eqref{K1_bdd} and \eqref{K0_bdd} to conclude that
\begin{align}
 &w(v) |f(T_0,x,v)|     \leq [C_\theta+5] e^{-\nu_0 T_0/2} \Vert w f_0\Vert_{L^\infty_{x,v}} +o(1)\Vert wf\Vert_{L^\infty_{T_0,x,v}}     \label{C_theta} \\
  &  + C(T_0) \Vert \nu^{-1}w \Gamma(f,f)\Vert_{L^\infty_{T_0,x,v}} +C(T_0) \Vert f\Vert_{L^\infty_{T_0}L^2_{x,v}}.  \notag
\end{align}

Since the constant $C_\theta>0$  does not depend on $T_0$, we choose $T_0$ to be large enough such that  $[C_\theta+5]e^{-\frac{\nu_0 T_0}{2}} \leq e^{-\frac{\nu_0T_0}{4}} $. Then we further have
\begin{align}
   \Vert  wf(T_0)\Vert_{L^\infty_{x,v}}  &  \leq  e^{-\nu_0 T_0/4} \Vert wf_0\Vert_{L^\infty_{x,v}}    + o(1)\Vert wf\Vert_{L^\infty_{T_0,x,v}} \notag \\
   &+  C(T_0)  \Vert \nu^{-1}w \Gamma(f,f)\Vert_{L^\infty_{T_0,x,v}} + C(T_0)\Vert f\Vert_{L^\infty_{T_0}L^2_{x,v}}.  \label{est_T0}
\end{align}

For given $0\leq T<\infty$, we denote
\begin{align*}
    \mathcal{R}_T := \Vert wf_0\Vert_{L^\infty_{x,v}} + \sup_{0\leq t\leq T}\Vert f(t)\Vert_{L^2_{x,v}} + \sup_{0\leq t\leq T}\Vert  \nu^{-1}w \Gamma(f,f)(t)\Vert_{L^\infty_{x,v}}.
\end{align*}

For $0\leq T\leq T_0$, with the same choice of $n=C_1 T_0^{5/4}$, it is straightforward to apply the same argument for $w(v)|f(T,x,v)|$ to have
\begin{align}
  \Vert wf(T)\Vert_{L^\infty_{x,v}}  & \leq [C_\theta+5] e^{-\frac{\nu_0 T}{2}}\Vert wf_0\Vert_{L^\infty_{x,v}} + o(1)  \Vert  wf\Vert_{L^\infty_{T,x,v}} \notag\\
  & + C(T_0) \Vert \nu^{-1}w \Gamma(f,f)\Vert_{L^\infty_{T,x,v}} + C(T_0)\Vert f\Vert_{L^\infty_T L^2_{x,v}}. \label{est_t}
\end{align} 

For $T=mT_0$, we apply \eqref{est_T0} to have
\begin{align}
  &\Vert wf(mT_0)\Vert_{L^\infty_{x,v}}  \notag\\
  &  \leq  e^{-\nu_0 T_0/4}\Vert wf((m-1)T_0)\Vert_{L^\infty_{x,v}} + C(T_0) \sup_{0 \leq t\leq T_0} \Vert \nu^{-1}w\Gamma(f,f)((m-1)T_0+t) \Vert_{L^\infty_{x,v}}   \notag\\
  & + o(1) \sup_{0 \leq t\leq T_0} \Vert w f((m-1)T_0+t) \Vert_{L^\infty_{x,v}} + C(T_0)  \sup_{0\leq t\leq T_0}\Vert f((m-1)T_0+t)\Vert_{L^2_{x,v}}  \notag\\
  & \leq  e^{-\nu_0 T_0/4}\Vert wf((m-1)T_0)\Vert_{L^\infty_{x,v}} +o(1) \sup_{0 \leq t\leq mT_0} \Vert w f(t)\Vert_{L^\infty_{x,v}} + C(T_0)  \mathcal{R}_{mT_0} \notag\\
  & \leq e^{-2\frac{\nu_0 T_0}{4}} \Vert wf((m-2)T_0)\Vert_{L^\infty_{x,v}} +   \Big[o(1)\sup_{0\leq t\leq mT_0} \Vert w f(t)\Vert_{L^\infty_{x,v}} + C(T_0) \mathcal{R}_{mT_0} \Big]\times \big[1 + e^{-\frac{\nu_0 T_0}{4}} \big] \notag\\
  & \leq \cdots \leq e^{-\frac{m\nu_0T_0}{4}} \Vert wf_0\Vert_{L^\infty_{x,v}} +  \Big[o(1)\sup_{0\leq t\leq mT_0} \Vert w f(t)\Vert_{L^\infty_{x,v}} + C(T_0) \mathcal{R}_{mT_0} \Big]\times \sum_{i=0}^{m-1}  e^{-\frac{i\nu_0T_0}{4}} \notag\\
  & \leq  o(1)C(\nu_0) \sup_{0\leq t\leq mT_0}\Vert wf(t)\Vert_{L^\infty_{x,v}} + C(T_0) \mathcal{R}_{mT_0}. \notag
\end{align}

For any $T>0$, we can choose $m$ such that $mT_0\leq T\leq (m+1)T_0$. Writing $T=mT_0+t$ with $0\leq t\leq T_0$, we apply \eqref{est_t} to have
\begin{align}
  & \Vert wf(T)\Vert_{L^\infty_{x,v}}  = \Vert wf(mT_0 + t)\Vert_{L^\infty_{x,v}} \notag\\
  & \leq [C_\theta+5]e^{\frac{-\nu_0 t}{2}}\Vert wf(mT_0)\Vert_{L^\infty_{x,v}}  + o(1) \sup_{0\leq s\leq t} \Vert wf(mT_0+s)\Vert_{L^\infty_{x,v}} \notag\\
  &+ C(T_0) \sup_{0\leq s\leq t}   \Vert \nu^{-1}w\Gamma(f,f)(mT_0+s) \Vert_{L^\infty_{x,v}} + C(T_0)  \sup_{0\leq s\leq t}\Vert f(mT_0+s)\Vert_{L^2_{x,v}} \notag\\
  & \leq o(1)C(\nu_0,\theta) \sup_{0\leq t\leq mT_0}\Vert wf(t)\Vert_{L^\infty_{x,v}} + C(T_0) \mathcal{R}_{mT_0+t}  \leq o(1) \sup_{0\leq t\leq T}\Vert wf(t)\Vert_{L^\infty_{x,v}} + C(T_0)  \mathcal{R}_T. \label{f_t_bdd_2d}
\end{align}

Since \eqref{f_t_bdd_2d} holds for all $T$, we conclude that
\begin{align*}
  \Vert wf(T)\Vert_{L^\infty_{x,v}}  & \lesssim \Vert wf_0\Vert_{L^\infty_{x,v}} + \sup_{0\leq t\leq T}\Vert \nu^{-1}w\Gamma(f,f)(t)\Vert_{L^\infty_{x,v}} + \sup_{0\leq t\leq T}\Vert f(t)\Vert_{L^2_{x,v}} \\
  & \lesssim \Vert wf_0\Vert_{L^\infty_{x,v}} + \Vert wf\Vert_{L^\infty_{T,x,v}}^2 + \Vert f\Vert_{L^\infty_T L^2_{x,v}}.
\end{align*}
In the second line, we applied the standard estimate to the nonlinear operator:
\begin{align*}
    &  \Vert \nu^{-1}w\Gamma(f,f) \Vert_{L^\infty_{T,x,v}} \lesssim \Vert wf\Vert_{L^\infty_{T,x,v}}^2.
\end{align*}

We conclude the first part of Lemma \ref{lemma:linfty_2d}.  

The proof of the second part is the same, with replacing $\Gamma(f,f)$ by $\p_t \Gamma(f,f) = \Gamma(f,\p_t f) + \Gamma(\p_t f,f)$. Such term can be controlled as
\begin{align*}
    & \Vert \nu^{-1}w[\Gamma(\p_t f , f )+ \Gamma(f,\p_t f)]\Vert_{L^\infty_{T,x,v}} \lesssim \Vert w\p_t f\Vert_{L^\infty_{T,x,v}} \Vert wf\Vert_{L^\infty_{T,x,v}}. 
\end{align*}

We conclude the lemma.
\end{proof}

\subsection{Proof of Theorem \ref{thm:2d}}\label{sec:thm_2d_proof}

The proof of Theorem \ref{thm:2d} follows from a standard sequential argument together with the a priori estimate in Proposition \ref{prop:apriori}. The positivity also follows from a standard sequential argument approach; we refer detailed construction to \cite{EGKM}. Then we just need to prove Proposition \ref{prop:apriori}.

\begin{proof}[\textbf{Proof of Proposition \ref{prop:apriori}}]
Combining Lemma \ref{lemma:energy} and Lemma \ref{lemma:linfty_2d}, we obtain that for any $T\geq 0$,
\begin{align*}
    & \Vert f\Vert_T \lesssim \Vert w f(0)\Vert_{L^\infty_{x,v}}^2 + \Vert w \p_t f(0)\Vert_{L^\infty_{x,v}}^2 + \Vert f(0)\Vert_{L^2_{x,v}}^2 + \Vert \p_t f(0)\Vert_{L^2_{x,v}}^2 + \Vert f\Vert_T^2.
\end{align*}
We conclude the proposition.
\end{proof}

\noindent {\bf Acknowledgment:}\,
The research of Renjun Duan was partially supported by the General Research Fund (Project No.~14303321) from RGC of Hong Kong and the Direct Grant (4053652) from CUHK. The authors would like to thank the anonymous referees for valuable and helpful comments on the manuscript. Hongxu Chen thanks Professor Chanwoo Kim for his interest and thanks Professor Jiaxin Jin for helpful discussion.

\medskip
\noindent{\bf Conflict of Interest:} The authors declare that they have no conflict of interest.

\medskip
\noindent{\bf Data Availability Statement:} Data sharing not applicable to this article as no datasets were generated or analyzed during the current study.

\bibliographystyle{siam}
%\bibliography{citation}

\end{document}